%% file: expVDC.tex
\documentclass{article}

\usepackage{amsmath,amssymb,amsfonts}%
\usepackage{amsthm}%

\theoremstyle{plain}
\newtheorem{thm}{Theorem}[section]
\newtheorem{lem}[thm]{Lemma}
\newtheorem{prop}[thm]{Proposition}
\newtheorem{cor}[thm]{Corollary}

\theoremstyle{definition}
\newtheorem{defn}[thm]{Definition}
\newtheorem{exmp}[thm]{Example}

\theoremstyle{remark}
\newtheorem{rmk}[thm]{Remark}

\newtheorem{expl}[thm]{Explication}

\newtheorem{cons}[thm]{Construction}

\usepackage[title]{appendix}%
\usepackage{xcolor}%
\usepackage{quiver}
\usepackage{adjustbox}
\usepackage{hyperref}
\usepackage{cleveref}
\usepackage{CJKutf8}

\usetikzlibrary{decorations.markings}
\tikzset{mid vert/.style={/utils/exec=\tikzset{every node/.append style={outer sep=0.8ex}},
postaction=decorate,decoration={markings,
mark=at position 0.5 with {\draw[-] (0,#1) -- (0,-#1);}}},
mid vert/.default=0.75ex}

\usepackage[style=alphabetic,backref=true,giveninits=true]{biblatex}
\usepackage{mathador}
\newcommand{\tet}{\mathpalette\matth{}}
\newcommand{\matth}[2]{%
  \mathchoice
    {\raisebox{-.2ex}{\scalebox{1.0}{$\DeMathador<line width=0.6pt>{white}{Q}{.4cm}{}$}}}
    {\raisebox{-.2ex}{\scalebox{0.8}{$\DeMathador<line width=0.6pt>{white}{Q}{.4cm}{}$}}}
    {\raisebox{-.2ex}{\scalebox{0.6}{$\DeMathador<line width=0.6pt>{white}{Q}{.4cm}{}$}}}
    {\raisebox{-.2ex}{\scalebox{0.45}{$\DeMathador<line width=0.6pt>{white}{Q}{.4cm}{}$}}}
}
\def\p#1{\mathrel{\ooalign{\hfil$\mapstochar\mkern 5mu$\hfil\cr$#1$}}}
\newcommand{\proto}{\p\rightarrow}
\newcommand{\id}{\operatorname{id}}
\newcommand{\op}{\operatorname{op}}
\newcommand{\dom}{\operatorname{dom}}
\newcommand{\cod}{\operatorname{cod} }
\newcommand{\yo}{\mathord{\text{\begin{CJK}{UTF8}{min}よ\end{CJK}}}}

\newcommand{\Cat}[1]{\mathsf{#1}}

\newcommand{\BiCat}[1]{\mathbf{#1}}
\newcommand{\MultCat}[1]{\mathsf{#1}}
\newcommand{\Dbl}[1]{\mathbb{#1}}

\newcommand{\Dbllong}[2]{\Dbl{#1}\mathsf{#2}}
\newcommand{\Sq}[1]{\mathbb{S}\mathsf{q}_{#1}}
\newcommand{\Susp}{\Sigma}
\newcommand{\Tight}{\Dbllong{T}{ight}}
\newcommand{\T}{\Dbllong{T}{}}
\newcommand{\Loose}{\Dbllong{L}{oose}}
\renewcommand{\L}{\Dbllong{L}{}}
\newcommand{\uTight}{\mathsf{Tight}}
\newcommand{\uLoose}{\mathsf{Loose}}
\newcommand{\Ob}{\Dbllong{O}{b}}
\newcommand{\uOb}{\mathsf{Ob}}

\newcommand{\uSq}[1]{\mathsf{Sq}_{#1}}
\newcommand{\prolist}[1]{\boldsymbol{#1}}
\newcommand{\fc}{\Cat{fc}}
\newcommand{\MCell}{\mathsf{MCell}}
\newcommand{\MCelln}[1]{\mathsf{MCell}_{#1}}
\newcommand{\Mod}{\Dbllong{M}{od}}
\newcommand{\Span}{\Dbllong{S}{pan}}
\newcommand{\Prof}{\Dbllong{P}{rof}}

\newcommand{\Cospan}{\Dbllong{C}{ospan}}
\newcommand{\fcsqn}[1]{\Cat{fc}_{#1}{\uSq{}}}
\newcommand{\fcsq}[1]{\Cat{fc}(#1_{1})}
\newcommand{\Vdc}{\Cat{Vdc}}

\newcommand{\define}[1]{\textbf{#1}}
\newcommand{\cell}[4]{\left(\begin{smallmatrix} & #2 & \\ #1 & & #4 \\ & #3 & \end{smallmatrix}\right)}
\newcommand{\scell}[2]{{\left(\begin{smallmatrix} #1 \\ #2 \end{smallmatrix}\right)}}
\usepackage{calc}
\newcommand*{\threefrac}[3]{%
  \begingroup
  \def\rowsep{-8pt}
  \sbox0{$\scriptscriptstyle#1$}%
  \sbox2{$\scriptscriptstyle#2$}%
  \sbox4{$\scriptscriptstyle#3$}%
  \setlength{\dimen0}{\maxof{\wd0}{\maxof{\wd2}{\wd4}}}%
  \begin{array}{@{}c@{}}
    \scriptscriptstyle #1 \\[\rowsep]
    \rule{\dimen0}{0.4pt} \\[\rowsep]
    \scriptscriptstyle #2 \\[\rowsep]
    \rule{\dimen0}{0.4pt} \\[\rowsep]
    \scriptscriptstyle #3
  \end{array}%
  \endgroup
}
\DeclareMathOperator{\colim}{colim}

\begin{document}

\title{Exponentiable Virtual Double Categories and Representability of Exponentials}

\author{
  Ea E. Thompson\thanks{Department of Mathematics, University of Illinois Urbana-Champaign. \href{mailto:eaet2@illinois.edu}{eaet2@illinois.edu}}
  \and
  Kevin D. Carlson\thanks{Topos Institute. \href{mailto:kevin@topos.institute}{kevin@topos.institute}}
}

\date{}
\maketitle
\begin{abstract}Virtual double categories provide an effective framework for formal category theory. 
Recent work has investigated the question of higher morphisms between virtual double categories,
following on from work on higher morphisms between double categories, and building up to 
Arkor's recent conjecture on exponentiable virtual double categories--those virtual 
double categories, morphisms out of which can themselves be enriched to a whole virtual
double category.

In this paper we resolve Arkor's conjecture by providing a number of equivalent characterizations
of the exponentiable virtual double categories in terms of existence of decompositions of cells. 
We also show that virtual double categories of cospans are always exponentiable, as 
are the virtual double categories arising from pseudo double categories or from
exponentiable multicategories, as studied by Pisani. 
We give conditions under which the virtual double category of virtual double functors 
admits composites, following work of Paré for the non-virtual case. 
We base our work on a general approach to exponentiability for categories of models of limit sketches,
which we apply to give new treatments of exponentiability for semicategories, categories, multicategories,
and their functors.\end{abstract}

\tableofcontents

\section{Introduction}\label{sec1}

\input{Sections/1_Introduction}

\section{Background on Virtual Double Categories}\label{sec:VDCs}

\input{Sections/2_VirtualDoubleCategories}

\clearpage

\section{A Sketch Approach to Exponentiability Theorems}\label{sec:ApproachExponentiability}

\input{Sections/3_ApproachExponentiability}

\clearpage

\section{Exponentiability via Decompositions}\label{sec:ExponentiableViaDecompositions}

\input{Sections/4_ExponentiabilityViaDecompositions}

\clearpage

\section{Other Characterizations of Exponentiability}\label{sec:OtherChar}
\input{Sections/5_OtherCharacterizations}

\clearpage

\section{Examples and Counterexamples}\label{sec:Examples}
\input{Sections/6_ExamplesAndCounterexamples}

\section{Representability of Exponentials}\label{sec:RepresentabilityOfExponentials}

\input{Sections/7_RepresentabilityOfExponential}

\clearpage

\printbibliography

\begin{appendices}
\section{Exponentiable Virtual Double Functors}\label{sec:ExponentiabilityOfMorphisms}
\input{Sections/A_ExponentiabilityOfMorphisms}

\clearpage
\input{Sections/Appendix.tex}

\end{appendices}

\end{document}

%% file: Sections/1_Introduction.tex
Double categories, originally defined by Ehresmann in~\cite{Ehresmann1964}, and later developed by Par\'{e} and collaborators in~\cite{Dawson1993,Grandis1999,Grandis2004,Grandis2007,Grandis2008,Pare2011,Pare2013}, have proved effective in a range of areas in mathematics for modelling structures with two interacting classes of maps. Applications of double categories can be found in homotopy theory, such as in algebraic $K$-theory (c.f.~\cite{zakharevich2011,campbellA2023,merlingScissorsCongruenceKtheory2025}) and enriched $(\infty,n)$-category theory (c.f.~\cite{moser2023,abellan2023,gepner2015,haugseng2016}), as well as use cases in modelling physical systems such as in biochemistry~\cite{aduddell2023compositional}, epidemiology~\cite{Baez2023}, and dynamical systems theory more broadly~\cite{libkind2025doubleoperadictheorysystems}. 

One application of particular interest to the authors comes from Lambert and Patterson's work 
on cartesian double theories in~\cite{lambertCartesianDoubleTheories2024}, where double categories 
are used as a foundation for categorical logic. 
Lambert and Patterson are particularly interested in morphisms and higher morphisms among 
``models'' of such a theory, which are given by lax double functors and lax transformations between them.
But as is shown in Theorem 10.3 of~\cite{lambertCartesianDoubleTheories2024}, the collection of models of 
a cartesian double theory need not form a genuine double category, 
but instead only a virtual double category in the sense of~\cite{Cruttwell2010}--roughly, 
a double category in which one direction of the arrows may not admit compositions. 
This work extends a growing understanding, originating from the study of the virtual double
category of lax double functors by Par\'{e} in~\cite{Pare2013}, that virtual double categories
arise unavoidably in the study of double category theory. In this paper, we take the virtual 
double categories as our subject of interest in their own right, and investigate 
in particular the question of internal homs for virtual double categories.

\subsubsection*{Internal homs for category-like structures}

The fact that Par\'{e}'s virtual double category of lax double functors is not a pseudo-double category 
in general is a glimpse into a subtle story about internal homs in categories with category-like objects. 
That the category $\Cat{Cat}$ of small categories is cartesian closed is quite remarkable, 
insofar as such of its near neighbors as the categories of monoids 
and of multicategories (\cite{Pisani2014}) fall well short of this condition. 
An even closer, but still non-cartesian closed, analogue is the category $\Cat{Semicat}$ of 
small semicategories (a semicategory has an associative composition operation, but may lack identity arrows.) 
It is also famously the case (\cite{conduche1972}) that most slices of $\Cat{Cat}$ are not cartesian closed.

Indeed, the only reasonably spacious general class of cartesian closed categories into which $\Cat{Cat}$ 
seems to fall is the class of 
``categories $\mathcal{V}-\Cat{Cat}$ of categories enriched in a cartesian closed category $\mathcal{V}$,'' 
in this case for $\mathcal{V}=\Cat{Set}.$ covers in particular the cases of strict $n$-categories for larger $n.$ 
But for most other kinds of categorical structures, such as the multicategories already
mentioned, but also for pseudo-double categories and virtual double categories, cartesian closure fails. 
These are cases in which, unlike the symmetric multicategories
with their Boardman-Vogt tensor product or $\mathcal{V}$-categories for a non-cartesian monoidal
$\mathcal{V}$, the cartesian product really is the natural monoidal product, so that we are at an 
impasse if we wish to develop the appropriate generalization of functor categories.

We are thus led to look for a partial solution: do certain exponentials $A^B$ exist, in categories
like $\Cat{Semicat}$, $\Cat{Cat}/B$, $\Cat{Multicat}$, or the category $\Vdc$
of virtual double categories? This is the problem
of \emph{exponentiability}, most famous from the case of topological spaces (\cite{Day_Kelly_1970}). An object 
$B$ of a category $\Cat{C}$ with finite products is \emph{exponentiable} if the functor $(-)\times B$ admits
a right adjoint, so that every exponential $A^B$ exists. We will restrict our interest here to
locally presentable categories, whose adjoint functor theorems allow for a much simpler theory 
of exponentiability than the reader may recall from the study of $\Cat{Top}.$

\subsubsection*{Contributions} 

 In this paper we will provide a complete classification of the exponentiable virtual double categories, 
 in particular, resolving all but one branch of Conjecture 3.15 of~\cite{arkor2025exponentiablevirtualdoublecategories} 
 positively. The last branch, regarding the pro-double categories, cannot be quite made sense of as written, and we provide an 
 alternative formulation in terms of pro-double categories with unbiased associators in Appendix \ref{sec:VDCCatOfOper}. 

 The approach is one which 
 should find usability quite generally in questions for exponentiability in 
 locally presentable categories: we give a sketch for virtual 
 double categories and show that exponentiability of an object $A$ may be characterized via preservation, 
 by $A\times (-),$ of the colimits arising from the distinguished cones of a sketch.
 We illustrate how straightforward it is to apply this approach to get complete proofs of the
 characterization of exponentiable objects--not, for instance, leaving the verification of the universal
 property of a putative exponential to the reader--in the preliminary 
 cases of semicategories, categories, and multicategories in Section~\ref{sec:ApproachExponentiability}.

 The technical heart of the paper runs from Subsection \ref{subsec:decompositions} through 
 Subsection \ref{subsec:malleable}. Here we show that the functor $(-)\times \Dbl{D},$ for a VDC $\Dbl{D}$,
 preserves the colimits of VDCs arising from the sketched cones in the sketch for VDCs (Definition \ref{defn:sketchVDC})
 if and only if every multicell of $\Dbl{D}$ admits essentially unique decompositions ``of every possible shape.''
 These possible shapes are parameterized by rooted finite uniform planar trees (Definition \ref{defn:tree}), 
 and the reduction of the conditions for exponentiability from Theorem~\ref{thm:ExpViaDecomp} to 
 Corollary~\ref{cor:ProRepImpliesDecomp} depends on some study of the combinatorics of such trees.
 The idea of this parameterization of decompositions by trees may be indicated by the diagram below; 
 the essential uniqueness is up to unary cells acting on the intermediate layers.

\[\begin{tikzcd}
	{\color{white}} & {\color{white}} & {\color{white}} & {\color{white}} & {\color{white}} & {\color{white}} & {\color{white}} \\
	& \bullet && \bullet & \bullet & \bullet & \bullet \\
	& {\color{white}} &&& \bullet & \bullet \\
	\bullet & \bullet & \bullet & \bullet && \bullet \\
	\bullet &&& \bullet && {\color{white}} \\
	&&& \bullet & \bullet & \bullet & \bullet \\
	&& \bullet & \bullet & \bullet & \bullet & \bullet \\
	&& \bullet && \bullet && \bullet \\
	&& \bullet &&&& \bullet
	\arrow[no head, from=1-1, to=2-2]
	\arrow[no head, from=1-2, to=2-2]
	\arrow[no head, from=1-3, to=2-2]
	\arrow[no head, from=1-5, to=2-5]
	\arrow[no head, from=1-6, to=2-6]
	\arrow[no head, from=1-7, to=2-7]
	\arrow[between={0.3}{0.7}, squiggly, from=2-2, to=2-4]
	\arrow[no head, from=2-2, to=3-2]
	\arrow[no head, from=2-5, to=3-5]
	\arrow[no head, from=2-6, to=3-6]
	\arrow[no head, from=2-7, to=3-6]
	\arrow[no head, from=3-5, to=2-4]
	\arrow[no head, from=3-5, to=4-6]
	\arrow[no head, from=3-6, to=4-6]
	\arrow["\shortmid"{marking}, from=4-1, to=4-2]
	\arrow[from=4-1, to=5-1]
	\arrow["\shortmid"{marking}, from=4-2, to=4-3]
	\arrow["\shortmid"{marking}, from=4-3, to=4-4]
	\arrow[from=4-4, to=5-4]
	\arrow[no head, from=4-6, to=5-6]
	\arrow[""{name=0, anchor=center, inner sep=0}, "\shortmid"{marking}, from=5-1, to=5-4]
	\arrow["\shortmid"{marking}, from=6-4, to=6-5]
	\arrow[""{name=1, anchor=center, inner sep=0}, from=6-4, to=7-3]
	\arrow[from=6-4, to=7-4]
	\arrow["\shortmid"{marking}, from=6-5, to=6-6]
	\arrow[from=6-5, to=7-5]
	\arrow["\shortmid"{marking}, from=6-6, to=6-7]
	\arrow[from=6-6, to=7-6]
	\arrow[from=6-7, to=7-7]
	\arrow["\shortmid"{marking}, from=7-3, to=7-4]
	\arrow[from=7-3, to=8-3]
	\arrow["\shortmid"{marking}, from=7-4, to=7-5]
	\arrow["\shortmid"{marking}, from=7-5, to=7-6]
	\arrow[from=7-5, to=8-5]
	\arrow["\shortmid"{marking}, from=7-6, to=7-7]
	\arrow[from=7-7, to=8-7]
	\arrow["\shortmid"{marking}, from=8-3, to=8-5]
	\arrow[from=8-3, to=9-3]
	\arrow["\shortmid"{marking}, from=8-5, to=8-7]
	\arrow[from=8-7, to=9-7]
	\arrow["\shortmid"{marking}, from=9-3, to=9-7]
	\arrow[between={0.3}{0.7}, squiggly, from=0, to=1]
\end{tikzcd}\]

 We will also extend Par\'{e}'s factorization property and representability results 
 in~\cite{Pare2013} to the setting of exponentiable VDCs,
 thus characterizing when an existing exponential of VDC's $\Dbllong{V}{dc}(\Dbl{A},\Dbl{B})$ 
 admits composites of loose arrows. Roughly speaking, we show that $\mathbb{A}$'s cells must all 
 factor through globular cells, and that $\mathbb{B}$ must admit 
local colimits, which will sound much the same as Par\'{e}'s characterization in the narrower non-virtual case.
 In the process of characterizing exponentiable VDCs and their representability we will show that the 
 theory extends the characterization of exponentiable multicategories proved by Pisani in~\cite{Pisani2014}, 
 while also providing a generalization of Day's classification of pro-monoidal multicategories in terms of 
 bi-cocontinuous monoidal structures on copresheaf categories in~\cite{Day1970}. (Example~\ref{eg:colaxConvStruct} and the paragraph following.)

\subsubsection*{Related and future work}

When $\Cat{A}=\Vdc$ is the category of VDCs, 
exponentiable morphisms were partially investigated in~\cite{fujii2025} where it was 
shown that discrete opfibrations between VDCs, which generalize discrete opfibrations between 
categories, are all exponentiable. Exponentiable VDCs were also partially investigated in recent
work of Arkor~\cite{arkor2025exponentiablevirtualdoublecategories}, where it is shown 
that representable VDCs (i.e.~pseudo-double categories) are exponentiable, while conditions for arbitrary exponentiable VDCs are 
conjectured in~\cite[Conjecture 3.15]{arkor2025exponentiablevirtualdoublecategories}.

Additionally, in recent work of Blom, Loubaton, and Ruit~\cite{blom2026dayconvolutionalgebraicpatterns},
a characterization of exponentiable morphisms between algebrads for $\infty$-categorical algebraic patterns
is provided. These algebrads correspond to $T$-multicategories for certain parametric right adjoint
monads on pre-sheaf $\infty$-categories. Taking $T=\Cat{fc}$ to be the free category monad,
the classification in~\cite{blom2026dayconvolutionalgebraicpatterns} can be modified slightly
to obtain a classification of exponentiable VDCs, essentially 
reproducing our first characterization in Theorem~\ref{thm:ExpViaDecomp}. Our tighter characterizations
do not fall out of Blom, Loubaton, and Ruit's work, however.
 
In future work the first author hopes to extend the results of this paper to classify exponential normal 
VDCs and to understand exponentiability in the context of 
virtual double categories of lax functors and lax transformations, as appearing in~\cite{lambertCartesianDoubleTheories2024}. 

\vspace{10pt}

\noindent \textbf{Size Considerations.} We fix two Grothendieck universes $\mathcal{U}\in\mathcal{V}$. 
Elements of $\mathcal{U}$ are referred to as small, while elements of $\mathcal{V}$ are referred to as large.
 We will write $\Cat{Set}$ and $\Cat{Cat}$ for the large categories of small sets and small categories, 
 respectively, and we will write $\Cat{SET}$ and $\Cat{CAT}$ for the very large categories of large sets and 
 large categories, respectively. All sets, categories, and virtual double categories will be assumed small, 
 unless stated otherwise.

\vspace{10pt}

\noindent \textbf{Notation.} Categories and multicategories will be denoted with sans serif font $\Cat{A},\Cat{B},\Cat{C},\ldots$, 
while (virtual) double categories will be denoted with blackboard bold font $\Dbl{A},\Dbl{B},\Dbl{C},\ldots$. 
We will also denote bicategories using bold font $\BiCat{A},\BiCat{B},\BiCat{C},\ldots.$

\vspace{10pt}

\noindent \textbf{Paper Outline.} 
\begin{itemize} 
  \item In Section~\ref{sec:VDCs} we review the necessary preliminaries on virtual 
double categories (VDCs) and their relation to pseudo-double categories. The reader
with a background in virtual double category theory should be able to simply scan this section for 
notational conventions. 
\item In Section~\ref{sec:ApproachExponentiability} we outline our general approach to studying exponentiability using 
limit sketches. We apply Pultr's theorem on the universal property of the category of models of a sketch
to characterize exponentiable models as those whose product preserves just those colimits of models 
arising from the sketched cones in Corollary~\ref{cor:exponentiable-sketch}. We illustrate 
this technology with the motivating examples of semicategories, categories, and multicategories. 
\item In Section~\ref{sec:ExponentiableViaDecompositions} we obtain a first 
characterization of exponentiable VDCs as those for which every multicell admits essentially 
unique decompositions according to every tree shape, in Theorem~\ref{thm:ExpViaDecomp}.
\item This is then extended to our main result, Theorem~\ref{thm:ExpChar}, in Section~\ref{sec:OtherChar}. Here 
we first show that unique decompositions reduce to those shaped by trees that branch only one cell per layer, 
and only using binary or nullary branches, which we call pro-representable virtual double categories. We
also show that pro-representable VDCs are equivalent to malleable VDCs, those whose composition cell is cartesian 
when viewed in a certain manner as generalized multicategories. Finally, 
we show that the VDC of spans of sets detects exponentiability by giving a 
Yoneda lemma (Lemma \ref{lem:Yoneda}) for VDCs. 
\item In Section~\ref{sec:RepresentabilityOfExponentials} we provide sufficient conditions for the exponential of two 
VDCs to have (weak) composites, as well as conditions for Par\'{e}'s VDC $\mathbb{L}\mathsf{ax}(\mathbb{A},\mathbb{B})$ 
to be representable when $\mathbb{A}$ is an exponentiable VDC, and not necessarily representable. 
\item In Section~\ref{sec:Examples} we give examples of exponentiable VDCs. Perhaps most interesting is that the 
VDC of cospans in any category $\Cat{E}$, which is not a pseudo double category unless $\Cat{E}$ should admit pushouts, 
is nevertheless always exponentiable. We also give a tool for checking exponentiability of certain tiny VDCs in 
Lemma \ref{lem:expOfUnitalization} and show that exponentiable multicategories also provide examples of exponentiable VDCs.
Finally we show that the definition of pro-double category proferred in \cite{arkor2025exponentiablevirtualdoublecategories}
is not workable in Proposition \ref{prop:counterExToProDouble}.
\item The proof of the characterization of exponentiable VDCs is extended in Appendix~\ref{sec:ExponentiabilityOfMorphisms} 
to provide a characterization of exponentiable morphisms of VDCs.  In Appendix~\ref{sec:VDCCatOfOper} we give a presentation of 
virtual double categories in terms of operator categories, which may be familiar from the study of algebraic 
patterns, interpret our exponentiability conditions in this context. In Appendix~\ref{sec:YonedaForMon} we rephrase 
the Yoneda lemma for VDCs in terms of a Yoneda lemma for actions of a monoid in a VDC.
\end{itemize}

\vspace{10pt}

\noindent \textbf{Acknowledgements.} The authors acknowledge helpful comments from Nathanael Arkor. 
During the initial research part of this project in June through August 2025, author Thompson 
was funded by the Advanced Research + Invention Agency (ARIA) through project code MSAI-PR01-P15 
while working at the Topos Institute as a summer research associate. 
During the remainder of the project duration, both authors were funded by the Air Force Office of Scientific
Research (AFOSR) Young Investigator Program (YIP) through Award FA9550-
23-1-0133.

%% file: Sections/2_VirtualDoubleCategories.tex
In this section we recall some basic definitions appearing in the theory of virtual double categories (VDCs), along with certain profunctors associated to virtual double categories which will be important for the main proofs in the paper. The primary references used for the background on VDCs comes from~\cite{LeinsterTom2002Geoc,Leinster2004,Cruttwell2010}. Note that in~\cite{LeinsterTom2002Geoc,Leinster2004} VDCs are referred to as $\Cat{fc}$-multicategories.

First, we recall that for a category $\Cat{C}$ with finite pullbacks and countable coproducts preserved by the finite pullbacks, 
we have a natural free category monad $\Cat{fc}$ on the category $\Cat{Grph}(\Cat{C})$ of graphs in $\Cat{C}$. 
Explicitly, $\Cat{fc}$ sends a graph $y\rightrightarrows x$ in $\Cat{C}$ to the graph with the same object of
vertices $x$, and with object of edges $\coprod_{n\geq 0}(\underbrace{y\times_x\cdots\times_xy}_n)$, 
where for $n=0$ the nullary pullback yields $x$. We will often abuse notation and apply $\Cat{fc}$ directly to the 
object of edges, leaving the object of vertices implicit if it is clear from the context. 
For $n\geq 2$ we will write $\mu_\Cat{fc}^{n-1}:\Cat{fc}^n\Rightarrow \Cat{fc}$ for the composite of multiplication 
transformations for the monad. For each integer $k\geq 0$, will also write 
$\iota_k:\Cat{fc}_k\hookrightarrow \Cat{fc}$ for the sub-endofunctor given by the $k$-fold pullback on the edge 
object, so that on edges $\Cat{fc}=\coprod_{n\geq 0}\Cat{fc}_n$ is the sum 
of paths of all finite lengths in the graph. 
We will most frequently use this monad when $\Cat{C}$ is either $\Cat{Set}$, the category of (small) sets, or 
$\Cat{Cat}$, the category of (small) categories. 

\subsection*{The definition of virtual double category}

\begin{defn}[Virtual Double Categories]\label{defn:VDC}
    A virtual double category (VDC) $\Dbl{D}$ consists of the following data:
    \begin{enumerate}
        \item An underlying category $\Dbl{D}_0$ of objects $x,y,z,w,\ldots \in \uOb(\Dbl{D})$ and tight arrows $a,b,c,d,\ldots$.
        We write $\uTight(\Dbl{D})$ for the set of all tight arrows, and $\uTight(\Dbl{D})_{x,y}$ for the set of tight arrows from $x$ to $y$ in $\Dbl{D}$.
        \item For each pair of objects $x,y \in \uOb (\Dbl{D})$, a set $\uLoose(\Dbl{D})_{x,y}$ of loose arrows $\varphi,\chi,\psi,\ldots:x\proto y$ from $x$ to $y$.
		The set of all loose arrows in $\Dbl{D}$ is denoted by $\uLoose(\Dbl{D})$, and is part of a graph $\uLoose(\Dbl{D})\rightrightarrows \uOb (\Dbl{D})$.
        \item For each integer $n\geq 0$ and each frame of objects, tight arrows, and loose arrows as below 
        \[\begin{adjustbox}{}\begin{tikzcd}
	{x_0} && \ldots && {x_n} \\
	\\
	{y_0} &&&& {y_1}
	\arrow["{{\varphi_1}}"{inner sep=.8ex}, "\shortmid"{marking}, from=1-1, to=1-3]
	\arrow["a"', from=1-1, to=3-1]
	\arrow["{{\varphi_n}}"{inner sep=.8ex}, "\shortmid"{marking}, from=1-3, to=1-5]
	\arrow["b", from=1-5, to=3-5]
	\arrow["\psi"'{inner sep=.8ex}, "\shortmid"{marking}, from=3-1, to=3-5]
\end{tikzcd}\end{adjustbox}\]
            a set of $n$-ary multicells 
            \[\begin{adjustbox}{}\begin{tikzcd}
	{x_0} && \ldots && {x_n} \\
	&& \alpha \\
	{y_0} &&&& {y_1}
	\arrow["{{\varphi_1}}"{inner sep=.8ex}, "\shortmid"{marking}, from=1-1, to=1-3]
	\arrow["a"', from=1-1, to=3-1]
	\arrow["{{\varphi_n}}"{inner sep=.8ex}, "\shortmid"{marking}, from=1-3, to=1-5]
	\arrow["b", from=1-5, to=3-5]
	\arrow["\psi"'{inner sep=.8ex}, "\shortmid"{marking}, from=3-1, to=3-5]
\end{tikzcd}\end{adjustbox}\]
            with the specified boundary, denoted by $\uSq{n}(\Dbl{D})\cell{a}{\varphi_1 \cdots \varphi_n}{\psi}{b}$. If $a$ and $b$ are identity arrows, we also write the set of multicells as $\uSq{n}(\Dbl{D})\scell{\varphi_1 \cdots \varphi_n}{\psi}$, and we refer to such multicells with identity tight source and target as \textit{special} or \textit{globular}. 
            We also refer to $0$-ary multicells as nullary, and $1$-ary multicells as unary.
            For an integer $n\geq 0$, we write $\uSq{n}(\Dbl{D})$ for the set of all $n$-ary multicells, and more generally if $k_1,...,k_n$ is a sequence of non-negative integers, we will write
			\begin{equation*}
				\uSq{k_1,...,k_n}(\Dbl{D}):=\uSq{k_1}(\Dbl{D})\times_{\uTight(\Dbl{D})}\ldots\times_{\uTight(\Dbl{D})}\uSq{k_n}(\Dbl{D})
			\end{equation*}
			for the set of sequences of compatible multicells, with the $i$th multicell in the sequence having arity $k_i$. The special case of the empty sequence is defined by $\uSq{()}(\Dbl{D}):=\uTight(\Dbl{D})$.
        \item For each loose arrow $\varphi:x\proto y$ a loose identity multicell $\text{id}_\varphi:\scell{\varphi}{\varphi}$.
        \item For every $n\geq 0$ and $k_1,...,k_n\geq 0$, a composition function
        \begin{equation*}
            \uSq{n}(\Dbl{D})\times_{\Cat{fc}_n(\uLoose(\Dbl{D}))}\uSq{k_1,...,k_n}(\Dbl{D})\xrightarrow{} \uSq{k_1+\cdots+k_n}(\Dbl{D})
        \end{equation*}
		which we denote using fractions, so that the composite of a diagram of multicells as below:
  \[\begin{tikzcd}
	{x_{1,0}} & \cdots & {x_{2,0}} & \cdots & {x_{n,0}} & \cdots & {x_{n,k_n}} \\
	{y_0} && {y_1} & \cdots & {y_{n-1}} && {y_n} \\
	{z_0} &&&&&& {z_1}
	\arrow["{{{{{\varphi_{1,1}}}}}}"{inner sep=.8ex}, "\shortmid"{marking}, from=1-1, to=1-2]
	\arrow["{{{{{a_0}}}}}"', from=1-1, to=2-1]
	\arrow["{{{{{\varphi_{1,k_1}}}}}}"{inner sep=.8ex}, "\shortmid"{marking}, from=1-2, to=1-3]
	\arrow["{{{{{\varphi_{2,1}}}}}}"{inner sep=.8ex}, "\shortmid"{marking}, from=1-3, to=1-4]
	\arrow["{{{{{a_1}}}}}"', from=1-3, to=2-3]
	\arrow["{{{{{\varphi_{n-1,k_{n-1}}}}}}}"{inner sep=.8ex}, "\shortmid"{marking}, from=1-4, to=1-5]
	\arrow["{{{{{\varphi_{n,1}}}}}}"{inner sep=.8ex}, "\shortmid"{marking}, from=1-5, to=1-6]
	\arrow["{{a_{n-1}}}"', from=1-5, to=2-5]
	\arrow["{{{{{\varphi_{n,k_n}}}}}}"{inner sep=.8ex}, "\shortmid"{marking}, from=1-6, to=1-7]
	\arrow["{{a_n}}", from=1-7, to=2-7]
	\arrow[""{name=0, anchor=center, inner sep=0}, "{{\psi_1}}"'{inner sep=.8ex}, "\shortmid"{marking}, from=2-1, to=2-3]
	\arrow["{{{{{b_0}}}}}"', from=2-1, to=3-1]
	\arrow["{{\psi_2}}"'{inner sep=.8ex}, "\shortmid"{marking}, from=2-3, to=2-4]
	\arrow["{{\psi_{n-1}}}"'{inner sep=.8ex}, "\shortmid"{marking}, from=2-4, to=2-5]
	\arrow[""{name=1, anchor=center, inner sep=0}, "{{\psi_n}}"'{inner sep=.8ex}, "\shortmid"{marking}, from=2-5, to=2-7]
	\arrow["{{{{{b_1}}}}}", from=2-7, to=3-7]
	\arrow[""{name=2, anchor=center, inner sep=0}, "\chi"'{inner sep=.8ex}, "\shortmid"{marking}, from=3-1, to=3-7]
	\arrow["{{{\alpha_1}}}"{description}, draw=none, from=1-2, to=0]
	\arrow["{{\alpha_n}}"{description}, draw=none, from=1-6, to=1]
	\arrow["\beta"{description}, draw=none, from=2-4, to=2]
\end{tikzcd}\]
  is denoted by $\frac{\alpha_1 \cdots \alpha_{n}}{\beta}$.
   As a warning, when $n=0$ this reduces to a composition function
        \begin{equation*}
            \uSq{0}(\Dbl{D})\times_{\uOb (\Dbl{D})}\uTight(\Dbl{D})\xrightarrow{}\uSq{0}(\Dbl{D})
        \end{equation*}
        which says we can whisker nullary multicells by tight arrows. In this case we can think of the tight arrows as 
        ``nullary multicells with nullary loose target," an intuition which is reified in the theory of augmented virtual double categories \cite{koudenberg2020}.
    \end{enumerate}
    This data is subject to the condition that composition of multicells is compatible with composition of tight sources and targets, is unital with respect to the identity multicells, and is associative with respect to the order of vertical pastings. \qed
\end{defn}

In order to simplify notation, 
we will often write $\prolist{\varphi}:=(\varphi_1 \cdots \varphi_n)$ 
to denote the sequence of loose arrows in the source, 
and we write $|\prolist{\varphi}|:= n$ for the length of this sequence. 
A length zero sequence of loose arrows corresponds to an object $a$ in $\Dbl{D}$, denoted $(a).$

Similarly, we will denote a pasting of multicells, as below right, using the notation below left:
\[\begin{tikzcd}[cramped, column sep=small]
	{x_{1,0}} & {x_{n,k_n}} & {x_{1,0}} & \cdots & {x_{2,0}} & \cdots & {x_{n,0}} & \cdots & {x_{n,k_n}} \\
	{y_0} & {y_n} & {y_0} && {y_1} & \cdots & {y_{n-1}} && {y_n} \\
	{z_0} & {z_1} & {z_0} &&&&&& {z_1}
	\arrow[""{name=0, anchor=center, inner sep=0}, "{{{\prolist{\varphi}}}}"{inner sep=.8ex}, "\shortmid"{marking}, from=1-1, to=1-2]
	\arrow["{{a_0}}"', from=1-1, to=2-1]
	\arrow["{{a_n}}", from=1-2, to=2-2]
	\arrow["{{{{{\varphi_{1,1}}}}}}"{inner sep=.8ex}, "\shortmid"{marking}, from=1-3, to=1-4]
	\arrow["{{{{{a_0}}}}}"', from=1-3, to=2-3]
	\arrow["{{{{{\varphi_{1,k_1}}}}}}"{inner sep=.8ex}, "\shortmid"{marking}, from=1-4, to=1-5]
	\arrow["{{{{{\varphi_{2,1}}}}}}"{inner sep=.8ex}, "\shortmid"{marking}, from=1-5, to=1-6]
	\arrow["{{{{{a_1}}}}}"', from=1-5, to=2-5]
	\arrow["{{{{{\varphi_{n-1,k_{n-1}}}}}}}"{inner sep=.8ex}, "\shortmid"{marking}, from=1-6, to=1-7]
	\arrow["{{{{{\varphi_{n,1}}}}}}"{inner sep=.8ex}, "\shortmid"{marking}, from=1-7, to=1-8]
	\arrow["{{a_{n-1}}}"', from=1-7, to=2-7]
	\arrow["{{{{{\varphi_{n,k_n}}}}}}"{inner sep=.8ex}, "\shortmid"{marking}, from=1-8, to=1-9]
	\arrow["{{a_n}}", from=1-9, to=2-9]
	\arrow[""{name=1, anchor=center, inner sep=0}, "{{\prolist{\psi}}}"'{inner sep=.8ex}, "\shortmid"{marking}, from=2-1, to=2-2]
	\arrow["{{b_0}}"', from=2-1, to=3-1]
	\arrow["{:=}"{description}, draw=none, from=2-2, to=2-3]
	\arrow["{{b_n}}", from=2-2, to=3-2]
	\arrow[""{name=2, anchor=center, inner sep=0}, "{{\psi_1}}"'{inner sep=.8ex}, "\shortmid"{marking}, from=2-3, to=2-5]
	\arrow["{{{{{b_0}}}}}"', from=2-3, to=3-3]
	\arrow["{{\psi_2}}"'{inner sep=.8ex}, "\shortmid"{marking}, from=2-5, to=2-6]
	\arrow["{{\psi_{n-1}}}"'{inner sep=.8ex}, "\shortmid"{marking}, from=2-6, to=2-7]
	\arrow[""{name=3, anchor=center, inner sep=0}, "{{\psi_n}}"'{inner sep=.8ex}, "\shortmid"{marking}, from=2-7, to=2-9]
	\arrow["{{{{{b_1}}}}}", from=2-9, to=3-9]
	\arrow[""{name=4, anchor=center, inner sep=0}, "\chi"'{inner sep=.8ex}, "\shortmid"{marking}, from=3-1, to=3-2]
	\arrow[""{name=5, anchor=center, inner sep=0}, "\chi"'{inner sep=.8ex}, "\shortmid"{marking}, from=3-3, to=3-9]
	\arrow["{{\prolist{\alpha}}}"{description}, draw=none, from=0, to=1]
	\arrow["{{{\alpha_1}}}"{description}, draw=none, from=1-4, to=2]
	\arrow["{{\alpha_n}}"{description}, draw=none, from=1-8, to=3]
	\arrow["\beta"{description}, draw=none, from=1, to=4]
	\arrow["\beta"{description}, draw=none, from=2-6, to=5]
\end{tikzcd}\]
As above, a length zero sequence of multicells corresponds to a tight arrow in $\Dbl{D}$. 
It may be enlightening to note that there is a strict double category (the free one on $\Dbl{D}$) 
with the same tight category as 
$\Dbl{D}$, loose arrows paths of loose arrows in $\Dbl{D},$ and multicells arising from paths of 
cells in $\Dbl{D},$ such that the compact notation above corresponds to working in this 
double category. 

Note that for a VDC $\Dbl{D}$, we have a category $\Dbl{D}_1$ whose objects are 
loose arrows and whose morphisms are unary multicells in $\Dbl{D}$. For a VDC $\Dbl{D}$, $\Dbl{D}_1$ fits into a natural graph of categories 
$s,t:\Dbl{D}_1\rightrightarrows \Dbl{D}_0$ which we can use to obtain the category 
$\Cat{fc}(\Dbl{D}_1)$ whose objects are sequences of loose arrows and whose morphisms are sequences of unary multicells. 

In fact, this assignment extends to a $2$-functor fitting into a $2$-adjunction
\[\begin{adjustbox}{}\begin{tikzcd}
	{\BiCat{Vdc}} && {\BiCat{Cat}}
	\arrow[""{name=0, anchor=center, inner sep=0}, "{(-)_1}"', curve={height=18pt}, from=1-1, to=1-3]
	\arrow[""{name=1, anchor=center, inner sep=0}, "{\Susp}"', curve={height=18pt}, from=1-3, to=1-1]
	\arrow["\dashv"{anchor=center, rotate=-90}, draw=none, from=1, to=0]
\end{tikzcd}\end{adjustbox}\]
for the morphisms and 2-cells of VDCs described below in Definition~\ref{defn:VDFunctor} and Definition~\ref{defn:TightTrans}, respectively, 
where we define the ``suspension'' by $\Susp(\Cat{C})_0=\Cat{C}+\Cat{C}$
and $\Susp(\Cat{C})_1=\Cat{C},$ with source and 
target functors $\Susp(\Cat{C}_1)\rightrightarrows \Susp(\Cat{C}_0)$ given by the two coproduct inclusions.
This is a complete specification insofar as 
$\Susp(\Cat{C})$ has no non-unary multicells.

With these categories defined, we have a natural profunctor 
$\MCell(\Dbl{D}):\fcsq{\Dbl{D}}^{op}\times \Dbl{D}_1\to \Cat{Set}$ which at a pair 
$(\prolist{\varphi},\psi)$ gives the set of multicells in $\Dbl{D}$ with loose source $\prolist{\varphi}$ 
and loose target $\psi$. We will write $\MCelln{n}(\Dbl{D}):=\MCell(\Dbl{D})(\iota_n(-),-)$ for the 
profunctor of $n$-ary multicells, the restriction of $\MCell(\Dbl{D})$ 
along the inclusion $\iota_n:\fcsqn{n}(\Dbl{D})\hookrightarrow \fcsq{\Dbl{D}}$. 
It will also be beneficial to work with the associated categories of these profunctors under the 
Grothendieck construction, which by abuse of notation we will denote with the same names. 
For instance, $\MCell(\Dbl{D})$ is the category whose objects are multicells and whose morphisms 
$\cell{a}{\prolist{\varphi}}{\psi}{b}\to \cell{a'}{\prolist{\varphi}'}{\psi'}{b'}$ are a pair of a unary multicell $\sigma=\cell{c}{\psi}{\psi'}{d}$ and a sequence of unary multicells $\prolist{\tau}=\cell{e_0}{\prolist{\varphi}}{\prolist{\varphi}'}{e_n}$ such that we have the pasting equality
\[
\begin{adjustbox}{}
	\begin{tikzcd}
	x & y && x & y \\
	z & w & {=} & {x'} & {y'} \\
	{z'} & {w'} && {z'} & {w'}
	\arrow[""{name=0, anchor=center, inner sep=0}, "{{\prolist{\varphi}}}"{inner sep=.8ex}, "\shortmid"{marking}, from=1-1, to=1-2]
	\arrow["a"', from=1-1, to=2-1]
	\arrow["b", from=1-2, to=2-2]
	\arrow[""{name=1, anchor=center, inner sep=0}, "{{\prolist{\varphi}}}"{inner sep=.8ex}, "\shortmid"{marking}, from=1-4, to=1-5]
	\arrow["{{e_0}}"', from=1-4, to=2-4]
	\arrow["{{e_n}}", from=1-5, to=2-5]
	\arrow[""{name=2, anchor=center, inner sep=0}, "\psi"'{inner sep=.8ex}, "\shortmid"{marking}, from=2-1, to=2-2]
	\arrow["c"', from=2-1, to=3-1]
	\arrow["d", from=2-2, to=3-2]
	\arrow[""{name=3, anchor=center, inner sep=0}, "{{\prolist{\varphi}'}}"'{inner sep=.8ex}, "\shortmid"{marking}, from=2-4, to=2-5]
	\arrow["{{a'}}"', from=2-4, to=3-4]
	\arrow["{{b'}}", from=2-5, to=3-5]
	\arrow[""{name=4, anchor=center, inner sep=0}, "{{\psi'}}"'{inner sep=.8ex}, "\shortmid"{marking}, from=3-1, to=3-2]
	\arrow[""{name=5, anchor=center, inner sep=0}, "{{\psi'}}"'{inner sep=.8ex}, "\shortmid"{marking}, from=3-4, to=3-5]
	\arrow["\alpha"{description}, draw=none, from=0, to=2]
	\arrow["{\prolist{\tau}}"{description}, draw=none, from=1, to=3]
	\arrow["\sigma"{description}, draw=none, from=2, to=4]
	\arrow["\beta"{description}, draw=none, from=3, to=5]
\end{tikzcd}
\end{adjustbox}
\]

\subsection*{Basic examples of VDCs}

Before continuing with definitions, let's recall some important examples of VDCs which will be relevant for our current work.

\begin{exmp}[Quintets]
	For any 2-category $\Cat{E}$, we have a VDC $\Dbl{Q}(\Cat{E})$ whose tight category is the underlying
  category of $\Cat{E}$, 
  whose loose arrows $x\proto y$ are just ordinary arrows $x\to y$ in $\Cat{E}$, and the 
  $n$-ary multicells filling a frame 
  $\cell{a}{\prolist{\varphi}}{\psi}{b}$ correspond to the 2-morphisms
  $\psi\circ a \to b\circ \varphi_n\circ \cdots \circ \varphi_1$ in $\Cat{E}$. 
  Below, we will use the quintet construction only in the case that $\Cat{E}$ is a mere category, 
  so the existence of a multicell comes down to a commutativity condition. 
\end{exmp}

\begin{exmp}[Span and Cospan Constructions]
    For any category $\Cat{E}$ we have a VDC $\Dbl{C}\Cat{ospan}(\Cat{E})$ whose underlying tight category is $\Cat{E}$, whose loose arrows $x\proto y$ for $x,y \in \Cat{E}$ are cospans $x\xrightarrow{a}z\xleftarrow{b}y$, and whose $n$-ary multicells are commuting diagrams of the form 
    \[\begin{adjustbox}{}\begin{tikzcd}
	{x_0} & {y_1} & {x_1} & \ldots & {x_{n-1}} & {y_n} & {x_n} \\
	&&& \ldots \\
	{z_0} &&& {w_1} &&& {z_1}
	\arrow["{{a_1}}", from=1-1, to=1-2]
	\arrow["{{c_0}}"', from=1-1, to=3-1]
	\arrow["{{e_1}}"', from=1-2, to=3-4]
	\arrow["{{b_1}}"', from=1-3, to=1-2]
	\arrow["{{a_2}}", from=1-3, to=1-4]
	\arrow["{{b_{n-1}}}"', from=1-5, to=1-4]
	\arrow["{{a_n}}", from=1-5, to=1-6]
	\arrow["{{e_n}}", from=1-6, to=3-4]
	\arrow["{{b_n}}"', from=1-7, to=1-6]
	\arrow["{{c_1}}", from=1-7, to=3-7]
	\arrow["{d_0}"', from=3-1, to=3-4]
	\arrow["{d_1}", from=3-7, to=3-4]
\end{tikzcd}\end{adjustbox}\]
Composition of multicells is by pasting diagrams.

  If the category $\Cat{E}$ has finite pullbacks, then we can similarly define the VDC $\Dbl{S}\Cat{pan}(\Cat{E})$ whose underlying tight category is once again $\Cat{E}$, but whose loose arrows $x\proto y$ are now spans $x\xleftarrow{a}z\xrightarrow{b}y$ in $\Cat{E}$, and whose $n$-ary multicells 
    \[\begin{adjustbox}{}\begin{tikzcd}
	{x_0} & {y_1} & {x_1} & \ldots & {x_{n-1}} & {y_n} & {x_n} \\
	{z_0} &&& {w_1} &&& {z_1}
	\arrow["{{{{c_0}}}}"', from=1-1, to=2-1]
	\arrow["{{{{a_1}}}}"', from=1-2, to=1-1]
	\arrow["{{{{b_1}}}}", from=1-2, to=1-3]
	\arrow["{{{{a_2}}}}"', from=1-4, to=1-3]
	\arrow["{{{{b_{n-1}}}}}", from=1-4, to=1-5]
	\arrow["\alpha"{description}, draw=none, from=1-4, to=2-4]
	\arrow["{{{{a_n}}}}"', from=1-6, to=1-5]
	\arrow["{{{{b_n}}}}", from=1-6, to=1-7]
	\arrow["{{{{c_1}}}}", from=1-7, to=2-7]
	\arrow["{{{{d_0}}}}", from=2-4, to=2-1]
	\arrow["{{{{d_1}}}}"', from=2-4, to=2-7]
\end{tikzcd}\end{adjustbox}\]
    are maps of spans
    \[
    \begin{adjustbox}{}\begin{tikzcd}
	{x_0} && {y_1\times_{x_1}\ldots\times_{x_{n-1}}y_n} && {x_n} \\
	{z_0} && {w_1} && {z_1}
	\arrow["{{{{c_0}}}}"', from=1-1, to=2-1]
	\arrow[from=1-3, to=1-1]
	\arrow[from=1-3, to=1-5]
	\arrow["\alpha"{description}, from=1-3, to=2-3]
	\arrow["{{{{c_1}}}}", from=1-5, to=2-5]
	\arrow["{{{{d_0}}}}", from=2-3, to=2-1]
	\arrow["{{{{d_1}}}}"', from=2-3, to=2-5]
\end{tikzcd}\end{adjustbox}
    \]
	When $\Cat{E}=\Cat{Set}$ we write $\Span := \Span(\Cat{Set})$ for the large VDC of spans of small sets.
\end{exmp}

\begin{exmp}[Profunctors]
    There exists a (large) VDC $\Prof$ whose underlying category is $\Cat{Cat}$, the category of small categories and functors,
     and whose loose arrows $\Cat{A}\proto \Cat{B}$ are profunctors, i.e. functors 
     $\Cat{A}^{op}\times \Cat{B}\to \Cat{Set}$. 
     A multicell $\alpha:\cell{F_0}{P_1 \cdots  P_n}{Q}{F_1}$
	consists of a family of maps 
	\begin{equation*}
		P_1(a_0,a_1)\times P_2(a_1,a_2)\times \cdots \times P_n(a_{n-1},a_n)\to Q(F_0(a_0),F_1(a_n))
	\end{equation*}
	which are natural in $a_0$ and $a_n$, and extra-natural in $a_1,...,a_{n-1}$. 
  Equivalently, such a multicell is given by a natural transformation 
  $\alpha:P_1\odot \cdots \odot P_n\Rightarrow Q(F_0,F_1)$ where $\odot$ denotes the composition of profunctors 
  and $Q(F_0,F_1):\Cat{A}_0\proto \Cat{A}_n$ is the pro-functor obtained by restricting $Q:\Cat{B}_0\proto \Cat{B}_1$ 
  along the functors $F_0$ and $F_1$; the more elementary definition is, nevertheless, of interest insofar as 
  it can be given in more general settings lacking the colimits on which the existence of $\odot$ depends.
\end{exmp}

\subsection*{Morphisms of VDCs}

We now define the morphisms and higher morphisms of virtual double categories. 

\begin{defn}[Virtual Double Functors]\label{defn:VDFunctor}
    A virtual double functor (VDF) $F:\Dbl{A}\to \Dbl{B}$ between VDCs $\Dbl{A}$ and $\Dbl{B}$ consists of the 
    following data:
    \begin{enumerate}
        \item A functor on underlying tight categories $F_0:\Dbl{A}_0\to \Dbl{B}_0$.
        \item For every pair of objects $x,x' \in \uOb (\Dbl{A})$ an assignment on loose arrows 
        $\uLoose(\Dbl{A})_{x,x'}\to \uLoose(\Dbl{B})_{Fx,Fx'}$.
        \item For every $n\geq 0$, an assignment on $n$-ary multicells $F:\uSq{n}(\Dbl{A})\to \uSq{n}(\Dbl{B})$ respecting tight and loose sources and targets.
    \end{enumerate}
    This data is required to strictly preserve identity multicells and composition of multicells.
\qed
\end{defn}

The resulting $1$-category $\Vdc$ of (small) VDCs is locally finitely presentable, and hence complete and 
cocomplete. 

\begin{defn}[Tight Transformations]\label{defn:TightTrans}
    A \define{transformation} between VDFs $F,G:\Dbl{A}\to \Dbl{B}$, denoted $f:F\Rightarrow G$, 
    consists of the following data:
    \begin{enumerate}
        \item A natural transformation of underlying tight functors $f_0:F_0\Rightarrow G_0$.
        \item For each loose arrow $\varphi:x\proto x'$ in $\Dbl{A}$, a multicell 
        $f_\varphi \in \uSq{1}(\Dbl{B})\cell{f_x}{F\varphi}{G\varphi}{f_{x'}}$.
    \end{enumerate}
    This data must be natural with respect to pasting multicells in the sense that for an $n$-ary multicell $\alpha$ in 
    $\Dbl{A}$ with loose source $(\varphi_1 \cdots \varphi_n)$ and loose target $\psi$, we have the pasting equality
    \begin{equation*}
        \frac{F(\alpha)}{f_\psi} = \frac{f_{\varphi_1} \cdots  f_{\varphi_n}}{G(\alpha)}
    \end{equation*}
\qed
\end{defn}

To understand how virtual double categories generalize the classical notion of (pseudo) double categories, we can use the notion of opcartesian multicells, which act as witnesses to loose composites.

\begin{defn}[Opcartesian multicells]\label{defn:OpCart}
    A globular $n$-ary multicell 
    $\alpha:\cell{\text{id}_x}{\varphi_1 \cdots \varphi_n}{\psi}{\text{id}_y}$ in a VDC $\Dbl{D}$ 
    is said to be \define{opcartesian} if any multicell of the form 
    $\beta:\cell{a}{\prolist{\chi}_0,\prolist{\varphi},\prolist{\chi}_1}{\omega}{b}$
    factors uniquely through $\alpha$ as a composite
    \[
    \begin{adjustbox}{}\begin{tikzcd}
	{x'} & x & y & {y'} \\
	{x'} & x & y & {y'} \\
	z &&& w
	\arrow[""{name=0, anchor=center, inner sep=0}, "{{{{\prolist{\chi}_0}}}}"{inner sep=.8ex}, "\shortmid"{marking}, from=1-1, to=1-2]
	\arrow[equals, from=1-1, to=2-1]
	\arrow[""{name=1, anchor=center, inner sep=0}, "{{{{\prolist{\varphi}}}}}"{inner sep=.8ex}, "\shortmid"{marking}, from=1-2, to=1-3]
	\arrow[equals, from=1-2, to=2-2]
	\arrow[""{name=2, anchor=center, inner sep=0}, "{{{{\prolist{\chi}_1}}}}"{inner sep=.8ex}, "\shortmid"{marking}, from=1-3, to=1-4]
	\arrow[equals, from=1-3, to=2-3]
	\arrow[equals, from=1-4, to=2-4]
	\arrow[""{name=3, anchor=center, inner sep=0}, "{{{{\prolist{\chi}_0}}}}"{inner sep=.8ex}, "\shortmid"{marking}, from=2-1, to=2-2]
	\arrow["a"', from=2-1, to=3-1]
	\arrow[""{name=4, anchor=center, inner sep=0}, "\psi"{inner sep=.8ex}, "\shortmid"{marking}, from=2-2, to=2-3]
	\arrow[""{name=5, anchor=center, inner sep=0}, "{{{{\prolist{\chi}_1}}}}"{inner sep=.8ex}, "\shortmid"{marking}, from=2-3, to=2-4]
	\arrow["b", from=2-4, to=3-4]
	\arrow[""{name=6, anchor=center, inner sep=0}, "\omega"'{inner sep=.8ex}, "\shortmid"{marking}, from=3-1, to=3-4]
	\arrow["{{\text{id}_{\prolist{\chi}_0}}}"{description}, draw=none, from=0, to=3]
	\arrow["\alpha"{description}, draw=none, from=1, to=4]
	\arrow["{{\text{id}_{\prolist{\chi}_1}}}"{description}, draw=none, from=2, to=5]
	\arrow["{{\exists!\beta^\flat}}"{description}, draw=none, from=4, to=6]
\end{tikzcd}\end{adjustbox}
    \]
    In this case we say $\alpha$ witnesses $\psi$ as a loose composite of $(\varphi_1 \cdots \varphi_n)$. 
    If we only have such unique factorizations when $|\prolist{\chi}_0|=|\prolist{\chi}_1|=0$, we say $\alpha$ 
    is \define{weakly opcartesian}, and that $\alpha$ witnesses $\psi$ as a weak loose composite of 
    $(\varphi_1 \cdots \varphi_n)$. We will write a general opcartesian multicell as $\Cat{opcart}$.
\qed
\end{defn}

If a sequence of loose arrows $\prolist{\varphi}=(\varphi_1 \cdots \varphi_n)$ in a VDC $\Dbl{D}$ admits a weak 
loose composite, we denote the weak composite by $\varphi_1\odot \ldots \odot \varphi_n$ or 
$\odot(\prolist{\varphi})$. In the case of an empty sequence of loose arrows $(a)$, with 
$a$ an object of $\Dbl{D},$ we will denote its (weak) loose composite by $I_a:a\proto a$, and refer to it as the (weak) \define{loose unit} at $a$. 

\begin{rmk}[The significance of strong composites]
When we only have weak composites, these composites need not be associative. Indeed, for any triple of loose arrows $(\varphi_1,\varphi_2,\varphi_3)$ we have globular multicells as below comparing associates for the composites, but these need not be isomorphisms:
\[
\begin{adjustbox}{}\begin{tikzcd}
	x && y && x && y \\
	x && y && x && y
	\arrow[""{name=0, anchor=center, inner sep=0}, "{{{\varphi_1\odot\varphi_2\odot\varphi_3}}}"{inner sep=.8ex}, "\shortmid"{marking}, from=1-1, to=1-3]
	\arrow[equals, from=1-1, to=2-1]
	\arrow[equals, from=1-3, to=2-3]
	\arrow[""{name=1, anchor=center, inner sep=0}, "{{{\varphi_1\odot\varphi_2\odot\varphi_3}}}"{inner sep=.8ex}, "\shortmid"{marking}, from=1-5, to=1-7]
	\arrow[equals, from=1-5, to=2-5]
	\arrow[equals, from=1-7, to=2-7]
	\arrow[""{name=2, anchor=center, inner sep=0}, "{{{(\varphi_1\odot \varphi_2)\odot\varphi_3}}}"'{inner sep=.8ex}, "\shortmid"{marking}, from=2-1, to=2-3]
	\arrow[""{name=3, anchor=center, inner sep=0}, "{{{\varphi_1\odot (\varphi_2\odot\varphi_3)}}}"'{inner sep=.8ex}, "\shortmid"{marking}, from=2-5, to=2-7]
	\arrow["{\exists!}"{description}, draw=none, from=0, to=2]
	\arrow["{\exists!}"{description}, draw=none, from=1, to=3]
\end{tikzcd}\end{adjustbox}
\]
Properties and examples of these VDCs with units and weak composites appear in~\cite{grandisOverviewColaxVirtual2025} under the name of colax double categories. The primary examples of such VDCs in~\cite{grandisOverviewColaxVirtual2025} come from VDCs of topological algebraic objects, where the failure of associativity comes from the non-commutativity of the monoidal structure with topological quotients.
\end{rmk}

On the other hand, from~\cite[Theorem 5.2]{Cruttwell2010} pseudo-double categories correspond precisely to those 
VDCs, called \emph{representable}, with all loose composites. 
In particular, loose composites satisfy associativity and unitality axioms via the universal property of the opcartesian multicells witnessing them. 
Additionally, under the identification of pseudo-double categories with VDCs, the co-dual of~\cite[Proposition 2.11]{Dawson2006} states that the VDFs are precisely the lax double functors, while the tight transformations of both structures correspond. 
We will write $\iota_{\BiCat{Dbl}}:\BiCat{Dbl}_\ell\hookrightarrow \BiCat{Vdc}$ for the fully-faithful 2-embedding of the $2$-category of pseudo-double categories, lax double functors, and tight transformations into the 2-category of VDCs.
We will write $\Dbl{H}:\BiCat{BiCat}_\ell\hookrightarrow \BiCat{Dbl}_\ell$ for the horizontal embedding of bicategories into pseudo-double categories.

\begin{defn}\label{defn:looseAndTightEmbeddings}
	We define the loose, tight, and chaotic embeddings into $\BiCat{Vdc}$ as follows:
	\begin{enumerate}
		\item We define the loose embedding $\L:\BiCat{Bicat}_\ell\hookrightarrow \BiCat{Vdc}$
			as the composite of the horizontal embedding $\Dbl{H}$ with $\iota_{\BiCat{Dbl}}$;
		\item We define the tight embedding $\T:\BiCat{Cat}\hookrightarrow \BiCat{Vdc}$
			by sending a category $\Cat{C}$ to the VDC $\T(\Cat{C})$ with tight category $\Cat{C}$ and no loose arrows;
		\item We define the chaotic embedding $\Dbllong{C}{haotic}:\BiCat{Cat}\hookrightarrow \BiCat{Vdc}$ 
			by sending a category $\Cat{C}$ to the VDC $\Dbllong{C}{haotic}(\Cat{C})$ with tight category $\Cat{C}$,
			a single loose between every pair of objects in $\Cat{C}$, 
			and a unique multicell filling every possible boundary of loose and tight arrows.
	\end{enumerate}
	\qed
\end{defn}

The embeddings $\T$ and $\Dbllong{C}{haotic}$ fit into a 2-adjoint triple
with the underlying category functor $(-)_0$ in the center:
\[\begin{tikzcd}
	{\BiCat{Vdc}} && {\BiCat{Cat}}
	\arrow[""{name=0, anchor=center, inner sep=0}, "{(-)_0}"{description}, from=1-1, to=1-3]
	\arrow[""{name=1, anchor=center, inner sep=0}, "{\Tight}"', curve={height=18pt}, from=1-3, to=1-1]
	\arrow[""{name=2, anchor=center, inner sep=0}, "{\Dbl{C}\mathsf{haotic}}", curve={height=-18pt}, from=1-3, to=1-1]
	\arrow["\dashv"{anchor=center, rotate=-92}, draw=none, from=0, to=2]
	\arrow["\dashv"{anchor=center, rotate=-88}, draw=none, from=1, to=0]
\end{tikzcd}\]
In particular, limits and colimits in $\BiCat{Vdc}$ 
are computed by limits and colimits in $\BiCat{Cat}$ on the level of underlying tight categories.

We will also be interested in the 2-category $\BiCat{Vdc}_n$ whose objects are ``normal VDCs'', that is, VDCs equipped with a choice of loose unit and witnessing opcartesian nullary multicell for each object. The morphisms in $\BiCat{Vdc}_n$ are VDFs which strictly preserve choices of loose units and opcartesian multicells, while the $2$-cells in $\BiCat{Vdc}_n$ are arbitrary tight transformations between the VDFs. Then we have a natural forgetful $2$-functor $U:\BiCat{Vdc}_n\to \BiCat{Vdc}$, which from~\cite{fujii2025} fits into an important 2-adjoint triple 
\[\begin{adjustbox}{}
\begin{tikzcd}
	{\BiCat{Vdc}_n} && {\BiCat{Vdc}}
	\arrow[""{name=0, anchor=center, inner sep=0}, "U"{description}, from=1-1, to=1-3]
	\arrow[""{name=1, anchor=center, inner sep=0}, "{\Dbl{F}_u}"', curve={height=18pt}, from=1-3, to=1-1]
	\arrow[""{name=2, anchor=center, inner sep=0}, "{\Dbl{M}\Cat{od}}", curve={height=-18pt}, from=1-3, to=1-1]
	\arrow["\dashv"{anchor=center, rotate=-92}, draw=none, from=0, to=2]
	\arrow["\dashv"{anchor=center, rotate=-88}, draw=none, from=1, to=0]
\end{tikzcd}
\end{adjustbox}\]
We will now explicate the 2-functors $\Mod$ and $\Dbl{F}_u$ appearing in the adjoint triple. First, the 2-functor $\Dbl{F}_u$ universally adds loose units to a VDC. Explicitly, for a VDC $\Dbl{D}$, $\Dbl{F}_u(\Dbl{D})$ has the same underlying tight category and has as loose arrows those in $\Dbl{D}$ along with a new loose arrow $I_x:x\proto x$ for every object $x \in \Dbl{D}$. Multicells in $\Dbl{F}_u(\Dbl{D})$ with target other than $I_x$ are given by multicells (possibly with empty loose target) in $\Dbl{D}$ after removing all the new loose arrows from the source. There
is a nullary multicell into $I_x$ for each $x,$ and the only other multicells with target $I_x$ are given by whiskering this nullary multicell with a tight arrow on the domain.

Next, the 2-functor $\Mod$ produces VDCs of modules as defined in~\cite[Definition 5.3.1]{Leinster2004}:

\begin{defn}[Modules in VDCs]\label{defn:Modules}
    Let $\Dbl{X}$ be a VDC. The VDC $\Mod(\Dbl{X})$ of \define{monoids and modules} consists of the following data:
    \begin{enumerate}
        \item The objects, \textit{monoids}, consist of an object $x$ in $\Dbl{X}$, a loose endoarrow $\varphi:x\proto x$, and multiplication and unit multicells 
        \[\begin{adjustbox}{}\begin{tikzcd}
	x & x & x && x \\
	x && x & x && x
	\arrow["\varphi"{inner sep=.8ex}, "\shortmid"{marking}, from=1-1, to=1-2]
	\arrow[equals, from=1-1, to=2-1]
	\arrow["\varphi"{inner sep=.8ex}, "\shortmid"{marking}, from=1-2, to=1-3]
	\arrow[equals, from=1-3, to=2-3]
	\arrow[equals, from=1-5, to=2-4]
	\arrow[equals, from=1-5, to=2-6]
	\arrow[""{name=0, anchor=center, inner sep=0}, "\varphi"'{inner sep=.8ex}, "\shortmid"{marking}, from=2-1, to=2-3]
	\arrow[""{name=1, anchor=center, inner sep=0}, "\varphi"'{inner sep=.8ex}, "\shortmid"{marking}, from=2-4, to=2-6]
	\arrow["\mu"{description}, draw=none, from=1-2, to=0]
	\arrow["\eta"{description}, draw=none, from=1-5, to=1]
\end{tikzcd}\end{adjustbox}\]
        satisfying associativity $\tfrac{(\mu,\text{id}_{\varphi})}{\mu}=\tfrac{(\text{id}_{\varphi},\mu)}{\mu}$ and unitality $\tfrac{(\eta,\text{id}_{\varphi})}{\mu}=\text{id}_{\varphi}=\tfrac{(\text{id}_{\varphi},\eta)}{\mu}$. 
        \item The tight arrows $(x,\varphi,\mu,\eta)\to (y,\psi,\nu,\theta)$, called \textit{monoid homomorphisms}, consist of a tight arrow $a:x\to y$ in $\Dbl{X}$ and a unary multicell $\alpha:\cell{a}{\varphi}{\psi}{a}$
        which is compatible with the multiplications and units in the sense that $\tfrac{(\alpha,\alpha)}{\nu} = \tfrac{\mu}{\alpha}$ and $\tfrac{a}{\theta} = \tfrac{\eta}{\alpha}$.
        \item The loose arrows $(x,\varphi,\mu,\eta)\proto (y,\psi,\nu,\theta)$, called \textit{modules}, consist of a loose arrow $\chi:x\proto y$ in $\Dbl{X}$, and left and right action multicells
        \[
        \begin{adjustbox}{}\begin{tikzcd}
	x & x & y & x & y & y \\
	x && y & x && y
	\arrow["\varphi"{inner sep=.8ex}, "\shortmid"{marking}, from=1-1, to=1-2]
	\arrow[equals, from=1-1, to=2-1]
	\arrow["\chi"{inner sep=.8ex}, "\shortmid"{marking}, from=1-2, to=1-3]
	\arrow[equals, from=1-3, to=2-3]
	\arrow["\chi"{inner sep=.8ex}, "\shortmid"{marking}, from=1-4, to=1-5]
	\arrow[equals, from=1-4, to=2-4]
	\arrow["\psi"{inner sep=.8ex}, "\shortmid"{marking}, from=1-5, to=1-6]
	\arrow[equals, from=1-6, to=2-6]
	\arrow[""{name=0, anchor=center, inner sep=0}, "\chi"'{inner sep=.8ex}, "\shortmid"{marking}, from=2-1, to=2-3]
	\arrow[""{name=1, anchor=center, inner sep=0}, "\chi"'{inner sep=.8ex}, "\shortmid"{marking}, from=2-4, to=2-6]
	\arrow["{\chi_L}"{description}, draw=none, from=1-2, to=0]
	\arrow["{\chi_R}"{description}, draw=none, from=1-5, to=1]
\end{tikzcd}\end{adjustbox}
        \]
        satisfying the bimodule axioms:
        \begin{align*} 
        \frac{(\mu,\text{id}_\chi)}{\chi_L} &= \frac{(\text{id}_\varphi,\chi_L)}{\chi_L}; &\frac{(\text{id}_\chi,\nu)}{\chi_R}=\frac{(\chi_R,\text{id}_\psi)}{\chi_R} \\ \\
        \frac{(\eta,\text{id}_\chi)}{\chi_L}=\text{id}_\chi&=\frac{(\text{id}_\phi,\theta)}{\chi_R}; &\frac{(\chi_L,\text{id}_\psi)}{\chi_R}=\frac{(\text{id}_\varphi,\chi_R)}{\chi_L}
        \end{align*}

        \item The $n$-ary multicells $\Theta:\cell{(a,\alpha)}{\chi_1 \cdots \chi_n}{\omega}{(b,\beta)},$
        called \emph{modulations}, are multicells $\Theta:\cell{a}{\chi_1 \cdots \chi_n}{\omega}{b}$ in $\Dbl{X}$ 
        which are compatible with left and right actions of the horizontal multicells:
        \begin{enumerate}
            \item (\textbf{Left Outer Equivariance}) 
			\[
			\begin{adjustbox}{}
				\begin{tikzcd}[column sep=small]
	{x_0} & {x_0} & {x_1} && {x_n} && {x_0} & {x_0} && {x_n} \\
	{x_0} && {x_1} && {x_n} & {=} & {y_0} & {y_0} && {y_1} \\
	{y_0} &&&& {y_1} && {y_0} &&& {y_1}
	\arrow["{{{\varphi_0}}}"{inner sep=.8ex}, "\shortmid"{marking}, from=1-1, to=1-2]
	\arrow[equals, from=1-1, to=2-1]
	\arrow["{{{\chi_1}}}"{inner sep=.8ex}, "\shortmid"{marking}, from=1-2, to=1-3]
	\arrow[""{name=0, anchor=center, inner sep=0}, "{{{\chi_2...\chi_n}}}"{inner sep=.8ex}, "\shortmid"{marking}, from=1-3, to=1-5]
	\arrow[equals, from=1-3, to=2-3]
	\arrow[equals, from=1-5, to=2-5]
	\arrow[""{name=1, anchor=center, inner sep=0}, "{{{\varphi_0}}}"{inner sep=.8ex}, "\shortmid"{marking}, from=1-7, to=1-8]
	\arrow["a"', from=1-7, to=2-7]
	\arrow[""{name=2, anchor=center, inner sep=0}, "{{{\chi_1...\chi_n}}}"{inner sep=.8ex}, "\shortmid"{marking}, from=1-8, to=1-10]
	\arrow["a"{description}, from=1-8, to=2-8]
	\arrow["b", from=1-10, to=2-10]
	\arrow[""{name=3, anchor=center, inner sep=0}, "{{{\chi_1}}}"'{inner sep=.8ex}, "\shortmid"{marking}, from=2-1, to=2-3]
	\arrow["a"', from=2-1, to=3-1]
	\arrow[""{name=4, anchor=center, inner sep=0}, "{{{\chi_2...\chi_n}}}"'{inner sep=.8ex}, "\shortmid"{marking}, from=2-3, to=2-5]
	\arrow["b", from=2-5, to=3-5]
	\arrow[""{name=5, anchor=center, inner sep=0}, "{{{\psi_0}}}"'{inner sep=.8ex}, "\shortmid"{marking}, from=2-7, to=2-8]
	\arrow[equals, from=2-7, to=3-7]
	\arrow[""{name=6, anchor=center, inner sep=0}, "\omega"'{inner sep=.8ex}, "\shortmid"{marking}, from=2-8, to=2-10]
	\arrow[equals, from=2-10, to=3-10]
	\arrow[""{name=7, anchor=center, inner sep=0}, "\omega"'{inner sep=.8ex}, "\shortmid"{marking}, from=3-1, to=3-5]
	\arrow[""{name=8, anchor=center, inner sep=0}, "\omega"'{inner sep=.8ex}, "\shortmid"{marking}, from=3-7, to=3-10]
	\arrow["{\chi_{1,L}}"{description}, draw=none, from=1-2, to=3]
	\arrow["{\text{id}}"{description}, draw=none, from=0, to=4]
	\arrow["\alpha"{description}, draw=none, from=1, to=5]
	\arrow["\Theta"{description}, draw=none, from=2, to=6]
	\arrow["\Theta"{description}, draw=none, from=2-3, to=7]
	\arrow["{\omega_L}"{description}, draw=none, from=2-8, to=8]
\end{tikzcd}
			\end{adjustbox}
			\]
            \item (\textbf{Right Outer Equivariance}) 
			\[
			\begin{adjustbox}{}
				\begin{tikzcd}[column sep=small]
	{x_0} && {x_{n-1}} & {x_n} & {x_n} && {x_0} && {x_n} & {x_n} \\
	{x_0} && {x_{n-1}} && {x_n} & {=} & {y_0} && {y_1} & {y_1} \\
	{y_0} &&&& {y_1} && {y_0} &&& {y_1}
	\arrow[""{name=0, anchor=center, inner sep=0}, "{{\chi_1...\chi_{n-1}}}"{inner sep=.8ex}, "\shortmid"{marking}, from=1-1, to=1-3]
	\arrow[equals, from=1-1, to=2-1]
	\arrow["{{\chi_n}}"{inner sep=.8ex}, "\shortmid"{marking}, from=1-3, to=1-4]
	\arrow[equals, from=1-3, to=2-3]
	\arrow["{{\varphi_n}}", from=1-4, to=1-5]
	\arrow[equals, from=1-5, to=2-5]
	\arrow[""{name=1, anchor=center, inner sep=0}, "{{\chi_1...\chi_n}}"{inner sep=.8ex}, "\shortmid"{marking}, from=1-7, to=1-9]
	\arrow["a"', from=1-7, to=2-7]
	\arrow[""{name=2, anchor=center, inner sep=0}, "{{\varphi_n}}"{inner sep=.8ex}, "\shortmid"{marking}, from=1-9, to=1-10]
	\arrow["b"{description}, from=1-9, to=2-9]
	\arrow["b", from=1-10, to=2-10]
	\arrow[""{name=3, anchor=center, inner sep=0}, "{{\chi_1...\chi_{n-1}}}"'{inner sep=.8ex}, "\shortmid"{marking}, from=2-1, to=2-3]
	\arrow["a"', from=2-1, to=3-1]
	\arrow[""{name=4, anchor=center, inner sep=0}, "{{\chi_n}}"'{inner sep=.8ex}, "\shortmid"{marking}, from=2-3, to=2-5]
	\arrow["b", from=2-5, to=3-5]
	\arrow[""{name=5, anchor=center, inner sep=0}, "\omega"'{inner sep=.8ex}, "\shortmid"{marking}, from=2-7, to=2-9]
	\arrow[equals, from=2-7, to=3-7]
	\arrow[""{name=6, anchor=center, inner sep=0}, "{{\psi_1}}"'{inner sep=.8ex}, "\shortmid"{marking}, from=2-9, to=2-10]
	\arrow[equals, from=2-10, to=3-10]
	\arrow[""{name=7, anchor=center, inner sep=0}, "\omega"'{inner sep=.8ex}, "\shortmid"{marking}, from=3-1, to=3-5]
	\arrow[""{name=8, anchor=center, inner sep=0}, "\omega"'{inner sep=.8ex}, "\shortmid"{marking}, from=3-7, to=3-10]
	\arrow["{\text{id}}"{description}, draw=none, from=0, to=3]
	\arrow["{\chi_{n,R}}"{description}, draw=none, from=1-4, to=4]
	\arrow["\Theta"{description}, draw=none, from=1, to=5]
	\arrow["\beta"{description}, draw=none, from=2, to=6]
	\arrow["\Theta"{description}, draw=none, from=2-3, to=7]
	\arrow["{\omega_R}"{description}, draw=none, from=2-9, to=8]
\end{tikzcd}
			\end{adjustbox}
			\]
            \item (\textbf{Inner Equivariance}) For each $2\leq i \leq n$, the pastings 
			\[
			\begin{tikzcd}
	{x_0} && {x_{i-2}} & {x_{i-1}} & {x_{i-1}} && {x_n} \\
	{x_0} && {x_{i-2}} && {x_{i-1}} && {x_n} \\
	{y_0} &&&&&& {y_1}
	\arrow[""{name=0, anchor=center, inner sep=0}, "{{{\chi_1...\chi_{i-2}}}}"{inner sep=.8ex}, "\shortmid"{marking}, from=1-1, to=1-3]
	\arrow[equals, from=1-1, to=2-1]
	\arrow["{{{\chi_{i-1}}}}"{inner sep=.8ex}, "\shortmid"{marking}, from=1-3, to=1-4]
	\arrow[equals, from=1-3, to=2-3]
	\arrow["{{{\varphi_{i-1}}}}"{inner sep=.8ex}, "\shortmid"{marking}, from=1-4, to=1-5]
	\arrow[""{name=1, anchor=center, inner sep=0}, "{{{\chi_i...\chi_n}}}"{inner sep=.8ex}, "\shortmid"{marking}, from=1-5, to=1-7]
	\arrow[equals, from=1-5, to=2-5]
	\arrow[equals, from=1-7, to=2-7]
	\arrow[""{name=2, anchor=center, inner sep=0}, "{{{\chi_1...\chi_{i-2}}}}"'{inner sep=.8ex}, "\shortmid"{marking}, from=2-1, to=2-3]
	\arrow["a"', from=2-1, to=3-1]
	\arrow[""{name=3, anchor=center, inner sep=0}, "{{{\chi_{i-1}}}}"'{inner sep=.8ex}, "\shortmid"{marking}, from=2-3, to=2-5]
	\arrow[""{name=4, anchor=center, inner sep=0}, "{{{\chi_i...\chi_n}}}"'{inner sep=.8ex}, "\shortmid"{marking}, from=2-5, to=2-7]
	\arrow["b", from=2-7, to=3-7]
	\arrow[""{name=5, anchor=center, inner sep=0}, "\omega"'{inner sep=.8ex}, "\shortmid"{marking}, from=3-1, to=3-7]
	\arrow["{\text{id}}"{description}, draw=none, from=0, to=2]
	\arrow["{\chi_{i-1,R}}"{description}, draw=none, from=1-4, to=3]
	\arrow["{\text{id}}"{description}, draw=none, from=1, to=4]
	\arrow["\Theta"{description}, draw=none, from=3, to=5]
\end{tikzcd}
			\]
			and
			\[
			\begin{tikzcd}
	{x_0} && {x_{i-1}} & {x_{i-1}} & {x_{i}} && {x_n} \\
	{x_0} && {x_{i-1}} && {x_i} && {x_n} \\
	{y_0} &&&&&& {y_1}
	\arrow[""{name=0, anchor=center, inner sep=0}, "{{\chi_1...\chi_{i-1}}}"{inner sep=.8ex}, "\shortmid"{marking}, from=1-1, to=1-3]
	\arrow[equals, from=1-1, to=2-1]
	\arrow["{{\varphi_{i-1}}}"{inner sep=.8ex}, "\shortmid"{marking}, from=1-3, to=1-4]
	\arrow[equals, from=1-3, to=2-3]
	\arrow["{{\chi_i}}"{inner sep=.8ex}, "\shortmid"{marking}, from=1-4, to=1-5]
	\arrow[""{name=1, anchor=center, inner sep=0}, "{{\chi_{i+1}...\chi_n}}"{inner sep=.8ex}, "\shortmid"{marking}, from=1-5, to=1-7]
	\arrow[equals, from=1-5, to=2-5]
	\arrow[equals, from=1-7, to=2-7]
	\arrow[""{name=2, anchor=center, inner sep=0}, "{{\chi_1...\chi_{i-1}}}"'{inner sep=.8ex}, "\shortmid"{marking}, from=2-1, to=2-3]
	\arrow["a"', from=2-1, to=3-1]
	\arrow[""{name=3, anchor=center, inner sep=0}, "{{\chi_{i}}}"'{inner sep=.8ex}, "\shortmid"{marking}, from=2-3, to=2-5]
	\arrow[""{name=4, anchor=center, inner sep=0}, "{{\chi_{i+1}...\chi_n}}"'{inner sep=.8ex}, "\shortmid"{marking}, from=2-5, to=2-7]
	\arrow["b", from=2-7, to=3-7]
	\arrow[""{name=5, anchor=center, inner sep=0}, "\omega"'{inner sep=.8ex}, "\shortmid"{marking}, from=3-1, to=3-7]
	\arrow["{\text{id}}"{description}, draw=none, from=0, to=2]
	\arrow["{\chi_{i,L}}"{description}, draw=none, from=1-4, to=3]
	\arrow["{\text{id}}"{description}, draw=none, from=1, to=4]
	\arrow["\Theta"{description}, draw=none, from=3, to=5]
\end{tikzcd}
			\]
			are equal.
        \end{enumerate}
    \end{enumerate}
\qed
\end{defn}

\begin{exmp}[Profunctors via Modules]
    For any category $\Cat{E}$ with pullbacks, we have the VDC $\Prof(\Cat{E}):=\Mod(\Span(\Cat{E}))$ whose underlying tight category is the category $\Cat{Cat}(\Cat{E})$ of categories and functors internal to $\Cat{E}$, whose loose arrows are internal profunctors, and whose multicells are internal natural transformations. Explicitly, an internal profunctor between internal categories $(\Cat{A}_1,\Cat{A}_0)$ and $(\Cat{B}_1,\Cat{B}_0)$ is a span in $\Cat{E}$, $\Cat{A}_0\xleftarrow{s}P\xrightarrow{t}\Cat{B}_0$, together with action maps $\Cat{A}_1\times_{\Cat{A}_0}P\to P$ and $P\times_{\Cat{B}_0}\Cat{B}_1\to P$, satisfying associativity and unitality axioms. An internal natural transformation $\alpha:\cell{(F_{1,0},F_{0,0})}{(P_1 \cdots P_n)}{Q}{(F_{1,1},F_{0,1})}$
	consists of a map of spans
	\[
	\begin{tikzcd}
	{\mathsf{A}_{0,0}} & {P_1\times_{\mathsf{A}_{0,1}}\cdots\times_{\mathsf{A}_{0,n-1}}P_n} & {\mathsf{A}_{0,n}} \\
	{\mathsf{B}_{0,0}} & Q & {\mathsf{B}_{0,1}}
	\arrow["{{F_{0,0}}}"', from=1-1, to=2-1]
	\arrow[from=1-2, to=1-1]
	\arrow[from=1-2, to=1-3]
	\arrow["\alpha"{description}, from=1-2, to=2-2]
	\arrow["{{F_{0,1}}}", from=1-3, to=2-3]
	\arrow[from=2-2, to=2-1]
	\arrow[from=2-2, to=2-3]
\end{tikzcd}
	\]
	compatible with the internal and external actions. 
  
  As an application of the adjunction $U\dashv \Mod$, if $\Dbl{D}\in \Vdc_n$ is any unital VDC with a choice of loose units, we have a natural bijection 
	\begin{equation}\label{eq:SpanProf}
		\Vdc_n(\Dbl{D},\Prof(\Cat{E}))\cong \Vdc(\Dbl{D},\Span(\Cat{E}))
	\end{equation}
	between normal VDFs into $\Prof(\Cat{E})$ and arbitrary VDFs into $\Span(\Cat{E})$.
\end{exmp}

The same kind of construction produces the VDC of profunctors
between categories enriched over an arbitrary monoidal category $\Cat{V}$
(with no cocompleteness assumptions) 
as the modules in the VDC of \emph{matrices} over $\Cat{V},$
and provides an analogous universal property.

For a category $\Cat{C}$ we will write 
$\T_u(\Cat{C})=\Dbl{F}_u(\T(\Cat{C}))$ for the unital tight embedding,
which adds loose units and unique multicells for the loose units to $\T(\Cat{C})$. 
These constructions preserve non-empty products, though $\T(-)$ does not preserve the terminal object. 
This leads to the construction of a number of important little VDCs.

\begin{defn}[Walking multicell VDCs]
  We now give definitions of the VDCs freely representing
  the various shapes of multicells. 
  
  \begin{itemize}
\item $\Ob:=\T([0])$ is the walking object VDC, while $\Ob_u:=\T_u([0])=\L([0])$ is the terminal VDC.\footnote{Note that $\Ob$ has no loose arrows, so no VDC with loose arrows can map into it.}

\item $\Tight:=\T([1])$ is the walking tight arrow VDC which has two objects, $0$ and $1$, and a single tight arrow $t:0\to 1$ between them, along with identity tight arrows.

\item The walking loose arrow, which we denote by $\Loose$, is the VDC with two objects $0$ and $1$, and a single loose arrow $\ell:0\proto 1$. Note that $\Loose=\Susp([0])$ is
the suspension of the terminal category.

\item The walking $n$-ary multicell for an integer $n\geq 0$, denoted $\Sq{n}$, is freely generated by the multicell
$\square_n$ as below:
\[\begin{tikzcd}
	{(0,0)} & {(0,1)} & \cdots & {(0,n)} \\
	{(1,0)} &&& {(1,1)}
	\arrow["{{\ell_1}}"{inner sep=.8ex}, "\shortmid"{marking}, from=1-1, to=1-2]
	\arrow["{{t_0}}"', from=1-1, to=2-1]
	\arrow[""{name=0, anchor=center, inner sep=0}, "{{\ell_2}}"{inner sep=.8ex}, "\shortmid"{marking}, from=1-2, to=1-3]
	\arrow["{{\ell_n}}"{inner sep=.8ex}, "\shortmid"{marking}, from=1-3, to=1-4]
	\arrow["{{t_1}}", from=1-4, to=2-4]
	\arrow[""{name=1, anchor=center, inner sep=0}, "\ell"'{inner sep=.8ex}, "\shortmid"{marking}, from=2-1, to=2-4]
	\arrow["{\square_n}"{description}, draw=none, from=0, to=1]
\end{tikzcd}\]
Explicitly, the walking $n$-ary multicell $\Sq{n}$ has, other that the multicells represented above, 
nothing but its $n+3$ identity tight arrows and its $n+1$ identity unary multicells.
\end{itemize}
\qed
\end{defn}

It follows that the functors 
$$
\uOb,\uTight,\uLoose,\uSq{n}:\Vdc\to \Cat{Set},\;\;n\geq 0
$$
sending a VDC to its set of objects, its set of tight arrows, its set of loose arrows, and its set of $n$-ary multicells, respectively, are all corepresented, by the corresponding double categories defined above.

The final notion we will use from basic VDC theory is the dual to opcartesian multicells, known as cartesian multicells, which allow for the sliding of tight morphisms up to loose morphisms.

\begin{defn}[Cartesian Multicells]\label{defn:Cart}
    For a VDC $\Dbl{D}$, a multicell $\alpha:\cell{a}{\varphi}{\psi}{b}$ is said to be \emph{cartesian} if any multicell in $\Dbl{D}$ of the form below left, factors uniquely through $\alpha$ as a composite below right:
    \[
    \begin{adjustbox}{}\begin{tikzcd}
	{x'} & {y'} && {x'} & {y'} \\
	x & y & {=} & x & y \\
	z & w && z & w
	\arrow[""{name=0, anchor=center, inner sep=0}, "{{{\prolist{\chi}}}}"{inner sep=.8ex}, "\shortmid"{marking}, from=1-1, to=1-2]
	\arrow["c"', from=1-1, to=2-1]
	\arrow["d", from=1-2, to=2-2]
	\arrow[""{name=1, anchor=center, inner sep=0}, "{{{\prolist{\chi}}}}"{inner sep=.8ex}, "\shortmid"{marking}, from=1-4, to=1-5]
	\arrow["c"', from=1-4, to=2-4]
	\arrow["d", from=1-5, to=2-5]
	\arrow["a"', from=2-1, to=3-1]
	\arrow["b", from=2-2, to=3-2]
	\arrow[""{name=2, anchor=center, inner sep=0}, "\varphi"'{inner sep=.8ex}, "\shortmid"{marking}, from=2-4, to=2-5]
	\arrow["a"', from=2-4, to=3-4]
	\arrow["b", from=2-5, to=3-5]
	\arrow[""{name=3, anchor=center, inner sep=0}, "\psi"'{inner sep=.8ex}, "\shortmid"{marking}, from=3-1, to=3-2]
	\arrow[""{name=4, anchor=center, inner sep=0}, "\psi"'{inner sep=.8ex}, "\shortmid"{marking}, from=3-4, to=3-5]
	\arrow["\beta"{description}, draw=none, from=0, to=3]
	\arrow["{\exists!\beta^\#}"{description}, draw=none, from=1, to=2]
	\arrow["\alpha"{description}, draw=none, from=2, to=4]
\end{tikzcd}\end{adjustbox}
    \]
    In this case, we say that $\alpha$ witnesses $\varphi$ as the \emph{restriction} of $\psi$ along $a$ and $b$. When such a restriction exists we denote it by $\psi(a,b)$. We will write a general cartesian multicell as $\Cat{cart}$. If $\Dbl{D}$ has loose units, for tight morphisms $a:x\to y\leftarrow z:b$ we write $y(a,b)$ for the restriction of the loose unit $I_y:y\proto y$ along $a$ and $b.$ In particular, $y(a,\id_y)$ is called the \emph{companion} of $a$ and $y(\id_y,b)$ is referred to as the \emph{conjoint} of $b.$
\qed
\end{defn}

%% file: Sections/3_ApproachExponentiability.tex
In this section, we introduce our general approach to establishing exponentiability in locally 
presentable categories, using sketches. The key result we use, which is novel 
to us at most in the details of its formulation, is Corollary \ref{cor:exponentiable-sketch}, 
which shows that exponentiable objects can 
be detected in such categories by preservation of certain canonical colimits 
determined by the given cones of the sketch. We then move on to apply this result to several examples building towards
our main case, starting with semicategories, and then revising known results on categories
and multicategories. We find that in all cases, the exponentiable objects are those admitting 
``essentially unique decompositions'' in an appropriate sense.

\subsection{Exponentiability via sketched colimits}
\begin{defn}[Limit Sketches and Models]
A \emph{limit sketch} $(\Cat{C},\Cat{S})$ consists of a category $\Cat{C}$ together with a class 
$\Cat{S}$ of cones in $\Cat{C}.$ We denote
the sketch simply by $\Cat{C}$ by abuse of notation. 

A \emph{model} is a functor out of $\Cat{C}$ sending each cone in $\Cat{S}$ to a genuine limit cone. We
denote by $\Cat{Mod}(\Cat{C})$ the category of $\Cat{Set}$-valued models of $\Cat{C}.$
\qed
\end{defn}

If $\Cat{C}$ is small, then it is a fundamental fact (\cite[Corollary 1.52]{AdamekRosicky1994}) that 
$\Cat{Mod}(\Cat{C})$ is locally presentable and accessible reflective in $\Cat{Set}^{\Cat{C}},$ say via the 
reflector $L.$ We can construct a canonical co-model of $\Cat{C},$ that is, a functor 
on $\Cat{C}^{\op}$ sending the cocones in $\Cat{S}^{\op}$ to colimit cocones, 
by composing the Yoneda embedding with the reflection:
$\Cat{C}^{\op}\xrightarrow{\yo} \Cat{Set}^{\Cat{C}}\xrightarrow{L} \Cat{Mod}(\Cat{C}).$ In fact we 
shall see this is the generic co-model of $\Cat{C}$ in a cocomplete category. If the sketched cones are actual 
limit cones in $\Cat{C},$ i.e. the sketch is \emph{realized}, then the co-model just mentioned 
is simply the corestriction of the Yoneda embedding to $\Cat{Mod}(\Cat{C}).$

The following theorem is essentially due to \cite{pultr1970},
but for a more accessible source with discussion of the non-realized case, 
one can see \cite{brandenburg2021}, Theorem 5.6 together with Remark 5.12. 

\begin{thm}[Pultr] \textrm{(Universal property of models of a limit sketch)}\label{thm:univSketchProp}
Let $\Cat{C}:=(\Cat{C},\Cat{S})$ be a small limit sketch. 
Then the composite $\Cat{C}^{\op}\xrightarrow{\yo} \Cat{Set}^{\Cat{C}}\xrightarrow{L} \Cat{Mod}(\Cat{C})$ 
exhibits $\Cat{Mod}(\Cat{C})$ as the initial cocomplete category equipped with a  
co-model of $\Cat{C}.$
\end{thm}

This result characterizes the category of cocontinuous functors $\Cat{Mod}(\Cat{C})\to\Cat{D}$ as 
equivalent to the category of co-models of $\Cat{C}$ in $\Cat{D}.$ For us, this enables the 
most important branch of the following corollary. Condition (4) is a convenient ``if it quacks like
an exponential, it's an exponential'' criterion.

\begin{cor}[Conditions for exponentiability]\label{cor:exponentiable-sketch}
Consider a small limit sketch $\Cat{C}=(\Cat{C},\Cat{S})$ and a model $A$ of $\Cat{C}.$
Then the following are equivalent: 
\begin{enumerate}
\item $A$ is exponentiable in $\Cat{Mod}(\Cat{C}).$
\item For each cone $\gamma:\Delta c \Rightarrow D:\Cat{J}\to \Cat{C}$ in $\Cat{S},$ 
the functor $(-)\times A:\Cat{Mod}(\Cat{C})\to \Cat{Mod}(\Cat{C})$ preserves the colimit 
$\gamma^{\op}\cdot\yo\cdot L,$ where $\yo$ is the Yoneda embedding of 
$\Cat{C}^{\op}$ and $L$ is the reflection of $C$-sets into models; 
that is, the restriction $L(\yo(-))\times A:\Cat{C}^{\op}\to \Cat{Mod}(\Cat{C})$ is a co-model of $\Cat{C}$
in $\Cat{C}$-models.
\item For every $\Cat{C}$-model $B,$ there exists a model $E_{B,A}$ together with isomorphisms
$\Cat{Mod}(\Cat{C})(L(\yo(C))\times A, B) \cong \Cat{Mod}(\Cat{C})(L(\yo(C)), E_{B,A})$ natural in 
$C\in \Cat{C}.$
\item Assuming $\Cat{C}$ is a realized sketch: the exponential $B^A$ as calculated in $\Cat{Set}^{\Cat{C}}$
is a model of $\Cat{C}$ for every $\Cat{C}$-model $B.$
\end{enumerate} 
\end{cor}

\begin{proof}
$(1)\implies(2)$ is a special case of $(-)\times A$ preserving all colimits for $A$ exponentiable.
$(1)\implies (3)$ is a special case of the definition of exponentiability itself.

For $(2)\implies (1),$ 
suppose $(-)\times A$ preserves the sketched cocones.
Then by the theorem of Pultr, the functor $(-)\times A:\Cat{C}^{\op}\to \Cat{Mod}(\Cat{C})$ 
extends essentially uniquely to a cocontinuous functor $\bar A:\Cat{Mod}(\Cat{C})\to \Cat{Mod}(\Cat{C}).$
Let us check that this extension is indeed $(-)\times A.$ One checks using the universal property that
the composite $L\cdot \bar A: \Cat{Set}^{\Cat{C}}\to \Cat{Mod}(\Cat{C})$ must coincide with 
$\Cat{Set}^{\Cat{C}}\xrightarrow{(-)\times A} \Cat{Set}^{\Cat{C}}\xrightarrow{L} \Cat{Mod}(\Cat{C}),$
and thus the result follows from the fact that the inclusion $i:\Cat{Mod}(\Cat{C})\to \Cat{Set}^{\Cat{C}}$ 
preserves products.

For $(3)\implies (2),$ consider a cone $\gamma: \Delta c\Rightarrow D: \Cat{J}\to \Cat{C}$ in 
$\Cat{S}.$ Then we have the following isomorphisms, natural in $B$:
\begin{align*}
\Cat{Mod}(\Cat{C})(L(\yo(x))\times A, B) &\cong \Cat{Mod}(\Cat{C})(L(\yo(x)), E_{B,A}) \\
&\cong \underset{j\in \Cat{J}}{\lim} \Cat{Mod}(\Cat{C})(L(\yo(Dj)), E_{B,A}) \\
&\cong \underset{j\in \Cat{J}}{\lim} \Cat{Mod}(\Cat{C})(L(\yo(Dj))\times A, B)
\end{align*}
and thus by the Yoneda lemma, the canonical map $L(\yo(x)\times A)\to \lim_{j\in \Cat{J}} L(\yo(Dj))\times A$ is an isomorphism. 
(Here we used (3) in the first and last steps, and the fact that $c^{\op}\cdot\yo\cdot L$ is a colimit cone in the second step.)

For $(3)\iff (4)$, in the realized sketch case, since $\yo(C)$ is a model, the unit map $\yo(C)\to iL(\yo(C))$ is an isomorphism,
which implies $\Cat{Mod}(\Cat{C})(L(\yo(C))\times A, B)\cong \Cat{Set}^{\Cat{C}}(\yo(C)\times A, B)$ and 
$\Cat{Mod}(\Cat{C})(L(\yo(C)), B^A)\cong \Cat{Set}^{\Cat{C}}(\yo(C), B^A),$ which shows that $E_{B,A}$, if it exists, 
and $B^A$ are the same $\Cat{C}$-set.
\end{proof}

\subsubsection*{Slicing sketches} 

In exponentiability questions, it is often of interest to find the exponentiable 
\emph{morphisms} in a finitely complete category $\Cat{E},$ i.e.~the exponentiable objects in a slice 
category $\Cat{E}/X,$ or equivalently those morphisms, pullback along which admits a right adjoint. 
We thus extend the results above to see how to handle exponentiability in slices of categories of models of sketches.

\begin{prop}\label{prop:sliced-sketch}
Given a limit sketch $\Cat{C}=(\Cat{C},\Cat{S})$ and a model $X:\Cat{C}\to \Cat{Set}$ of $\Cat{C},$
the slice category $\Cat{Mod}(\Cat{C})/X$ is equivalent to the category of models of the limit sketch
$(\int X,\bar{S}),$ where $\int X$ is the category of elements of $X$ and $\bar{S}$ consists of those 
cones in $\int X$ whose composite with the projection $\pi:\int X\to \Cat{C}$ lies in $\Cat{S}.$
\end{prop}
\begin{proof}
It is well-known that the category of $\int X$-sets is equivalent to the slice $\Cat{Set}^{\Cat{C}}/X,$
with one direction of the equivalence given by the action of the Grothendieck construction on morphisms 
in $\Cat{Set}^{\Cat{C}},$ and the other by the left Kan extension $\pi_!:\Cat{Set}^{\int X}\to \Cat{Set}^
{\Cat{C}}.$ Thus it suffices to show that some $\int X$-set $Y:\int X\to \Cat{Set}$ is a model of $(\int X,\bar{S})$ if and only if its pushforward $\pi_! Y: \Cat{C}\to \Cat{Set}$ is a model of $\Cat{C}.$

Now, by definition $\pi_! Y(c)=\sum_{x:X(c)} Y(x).$ If $\gamma:\Delta c\Rightarrow D:\Cat{J}\to\Cat{C}$ is
a cone in $\Cat{S},$ then $\pi_! Y$ is a model at $c$ if and only if the canonical map is
an isomorphism $\pi_! Y(c)\to \lim_{j\in \Cat{J}} \pi_! Y(Dj),$ i.e. $\sum_{x:X(c)} Y(x) \to \lim_{j\in \Cat{J}} \sum_{x_j:X(Dj)} Y(x_j).$ Meanwhile $Y$ is a model of $(\int X,\bar{S})$ at all the lifts 
of $\gamma$ along $\pi$ if and only if, for each $x:X(c),$ the comparison map 
$Y(x)\to \lim_{j\in \Cat{J}} Y(X(\gamma_j)(x))$ is an isomorphism. 

Setting 
$Q=\lim_{j\in \Cat{J}} \sum_{x_j:X(Dj)} Y(x_j)$ and $Q'=\sum_{x:X(c)} \lim_{j\in \Cat{J}} Y(X(\gamma_j)(x))$,
it is thus enough to show that $Q$ and $Q'$ are isomorphic via an isomorphism commuting 
with the canonical maps. Indeed,$Q$ has elements given by, for each $j,$ choosing 
an $x_j$ and a $y_j\in Y(x_j),$ such that if $f:j\to j'$ then $Y(f)(y_j)=y_{j'}.$ In particular 
this implies that $X(f)(x_j)=x_{j'},$ so $(x_j)$ is in $\lim_j X(Dj),$ which gives the desired 
isomorphism $Q\to Q',$ using the model condition $X(c)\cong \lim_{j\in \Cat{J}} X(Dj)$ for $X.$ 
\end{proof}

This allows for proofs of exponentiability of morphisms in slices of categories of models of sketches
not much harder than in the absolute case. 
\begin{cor}\label{cor:exponentiable-morphisms-via-slice}
Let $\Cat{C}=(\Cat{C},\Cat{S})$ be a limit sketch and $X:\Cat{C}\to \Cat{Set}$ a model of $\Cat{C}.$ Then a morphism $f:A\to X$ is exponentiable in $\Cat{Mod}(\Cat{C})/X$ if and only if, for every 
cone $\lambda:\Delta c\to (D:\Cat{J}\to \Cat{C})$ in $\Cat{S}$ and every morphism $\bar c:L\yo c\to X,$ the functor $(-)\times_X A:\Cat{Mod}(\Cat{C})/X\to \Cat{Mod}(\Cat{C})/X$ preserves the colimiting 
cocone $\bar\lambda:L\circ \yo\circ D\Rightarrow \Delta \bar c.$
\end{cor}
Since colimits in slices are computed as in the base category, this condition is merely a fibered
version of exponentiability in $\Cat{Mod}(\Cat{C})$ itself, and thus frequently no harder to check.

\subsection{A case study: exponentiable semicategories}
We warm up by investigating the case of exponentiable 
semicategories, which is a bit too trivial to allow for many 
exponentiable objects in the end, but simple enough to clearly exhibit the key ideas.

\begin{defn}[Category of semicategories]\label{defn:semicategory}
A \emph{semicategory} $A$ consists of a graph $A_1\rightrightarrows A_0$ equipped 
with an associative composition operation $A_1\times_{A_0} A_1\to A_1.$ A \emph{semifunctor}
$A\to B$ is a graph morphism respecting composition. We write $\Cat{Semicat}$ for the
category of small semicategories and semifunctors.

Alternatively, $\Cat{Semicat}$ is the category of models of a limit sketch 
whose underlying category $\Cat{C}$ is the opposite of the category of the 
first for nonempty finite ordinals $\bullet,\downarrow,\triangle,\tet$ and order-preserving morphisms. The sketched cones $\Cat{S}=\{S_{\triangle},S_{\tet}\}$ are those making 
$\triangle={\downarrow}+_\bullet {\downarrow}$ a pushout in the usual way and similarly for $\tet.$ 
\qed
\end{defn}

A semicategory is thus like a category possibly lacking identity morphisms, though these are unique
when they exist, by the usual argument. 
As defined, $\Cat{C}$ is a full subcategory of $\Cat{Semicat},$ which implies this 
sketch is realized. We may write $\bullet,\downarrow,\triangle,\tet$ also for the semicategories 
represented by the objects of the sketch above. 
Note that, given the absence of identities, the free semicategory $\bullet$ on one object actually contains \emph{no} morphisms at all--in particular, $\bullet$ is not the terminal semicategory.

Thus we have the colimit decompositions $\triangle \cong {\downarrow}+_\bullet{\downarrow}$ and
$\tet \cong {\downarrow} +_\bullet {\downarrow} +_\bullet {\downarrow},$ and by 
Corollary \ref{cor:exponentiable-sketch}, a semicategory 
$A$ is exponentiable if and only if the product functor with $A$ preserves those two colimits. 
Let us now flesh out what this preservation means concretely.

\begin{prop}[Explicit conditions for exponentiability]\label{prop:necessary-semicat}
A semicategory $A$ is exponentiable if and only if every morphism $f:x\to y$ in $A$ admits a unique
binary factorization $f=g\cdot h.$ 
\end{prop}

\begin{proof}
We can write the objects of $A\times \triangle$ as $(a,i)$ for $a\in A$ and $i=0,1,2,$ and similarly 
the morphisms as $(f,i,i'):(a,i)\to (a',i')$ for $f:a\to a'$ in $A$ and $i< i'.$ 
The objects of $A\times {\downarrow} \underset{A\times \bullet}{+} A\times {\downarrow}$ are the same,
as are morphisms of the form $(f,0,1)$ and $(f,1,2),$ but the morphisms $(a_0,0)\to (a_2,2)$ are 
now given precisely by composable pairs $a_0\to a_1\to a_2$ in $A.$
Therefore, the canonical semifunctor 
$c:A\times {\downarrow} \underset{A\times \bullet}{+} A\times {\downarrow}\to A\times\triangle$
is always bijective on objects and on the two ``short'' classes of morphisms, while 
for the ``long'' morphisms, $c$ is full at $(a_0,0),(a_2,2)$ if and only if every morphism
$a_0\to a_2$ has some factorization $a_0\to a_1\to a_2$ in $A,$ and faithful if and only if 
there is at most one such factorization. Thus, $c$ is an isomorphism if and only if
every morphism in $A$ admits a unique binary factorization. Analogously, 
the functor $A\times (-)$ preserves the distinguished colimit for $\tet$ if and only if every morphism admits a unique ternary factorization.

To reduce this pair of conditions to the single condition given in the proposition 
statement, consider that uniqueness of ternary factorizations reduces to uniqueness of 
binary factorizations. Indeed, given binary decompositions, 
we can give a ternary decomposition of any morphism by repeated binary decompositions, 
while for uniqueness, if we have $f_1f_2f_3=f_1'f_2'f_3'$ in $A,$ then by uniqueness of binary 
factorizations we have $f_1f_2=f_1'f_2'$ and $f_3=f_3'$, and then another application 
of uniqueness of binary factorizations gives $f_1=f_1'$ and $f_2=f_2'.$ 
\end{proof} 

Thus we have a single, simple necessary and sufficient condition for exponentiability 
of a semicategory: the existence of unique binary factorizations.
Note that this condition is enormously strong. For instance, it \emph{never} holds for a 
non-discrete semicategory which happens to admit identity morphisms, since if $f:x\to y$ then we 
have $f=\id_x\cdot f=f\cdot \id_y$ as two distinct factorizations of $f.$ Thus no non-discrete 
category is exponentiable when viewed as a semicategory.
The most complex connected exponentiable semicategory, which isn't saying much,
is given by augmenting $\triangle$ with an identity morphism at the middle 
object $1.$ 

Having rather thoroughly understood the exponentiable objects of $\Cat{Semicat},$ we now
consider the exponentials themselves, via techniques that will serve us well in the more 
complicated cases below.
\begin{prop}[Structure of exponential semicategories]
Suppose $A$ and $B$ are semicategories such that the exponential $B^A$ exists. Then: 
\begin{itemize}
\item The objects of $B^A$ are the functions $\operatorname{Ob}(A)\to \operatorname{Ob}(B).$
\item The morphisms $\alpha:F_1\to F_2$ in $B^A$ are the dependent functions\footnote{This is a choice of morphism in $B$ for each morphism in $A,$ over the $F_i,$ satisfying no naturality conditions.}
\[\alpha~:~\underset{f:a_1\to a_2\in A}{\prod} B(F_1(a_1),F_2(a_2))\]
\item Given $\alpha:F_1\to F_2,$ $\beta:F_2\to F_3,$ and $\gamma:F_1\to F_3$ in $B^A$, 
we have the equation $\gamma=\alpha\cdot \beta$ just if for every composable pair $f,g$ in $A,$ 
we have $\gamma_{f\cdot g}=\alpha_f\cdot \beta_g$ in $B.$
\end{itemize}
\end{prop}
\begin{proof}
The three claims follow easily from the representability of objects and arrows and the universal 
property of exponentials. Indeed, the objects of any semicategory $C$ are the semifunctors 
$\bullet\to C,$ and so the objects of $B^A$ are the semifunctors $\bullet\times A\to B.$ Since 
$\bullet$ contains no arrows, the domain $\bullet\times A$ is the discrete semicategory on the 
objects of $A,$ which computes the objects as desired. 

Similarly, a morphism in $B^A$ is given by a semifunctor ${\downarrow}\times A\to B,$ 
where the domain consists of two copies of $\bullet\times A,$ with the only arrows 
of the form $f:(a,0)\to (a',1)$ for $f:a\to a'$ in $A.$ This gives the desired characterization.

Commutative triangles in $B^A$ correspond to semifunctors $T:\triangle\times A\to B.$ As we have 
seen in deriving the unique binary decomposition condition above, the commutative triangles 
in $\triangle\times A$ have as their edges maps $(f,0,1),(g,1,2),(fg,0,2)$ for composable $f,g$ in 
$A.$ Thus such a $T$ is given by a triple $\alpha,\beta,\gamma$ of dependent functions as above such 
that $\gamma_{f\cdot g}=\alpha_f\cdot \beta_g$ for composable $f,g$ in $A,$ as was to be shown.
\end{proof}

Another angle on exponentiability which inspired Arkor's conjecture in \cite
{arkor2025exponentiablevirtualdoublecategories} is via Grothendieck constructions. 
This angle is made clearest when we consider exponentiability of a semifunctor $\pi:E\to B$ in $\Cat
{Semicat}/B,$ generalizing from the case $\Cat{Semicat}=\Cat{Semicat}/1.$ In the nicest case,
such a $\pi$ can be disintegrated into a semifunctor $\nabla \pi:B\to \Cat{Set},$ sending $b\mapsto 
E_b.$ But semifunctoriality of this operation requires $\pi$ to be a \emph{discrete opfibration}\footnote{The only natural notion of opfibration for semicategories is the discrete, since there is no way to arrange a whole semicategory in a ``fiber'' over an object of a semicategory, thus no Grothendieck construction for (pseudo)semifunctors into $\Cat{Semicat}.$}, a strong condition. 

When $\pi$ is not necessarily a discrete opfibration, we may instead disintegrate $\pi$ into a 
virtual double functor $\nabla \pi: \L(B)\to \Span,$ where $\L$ extends the embedding 
of bicategories to VDCs discussed in the previous section to semi(bi)categories; thus $\L(B)$ has
discrete tight category on the objects of $B,$ loose arrows given by arrows of $B,$ and 
cells given by composites in $B.$
Then we define $\nabla\pi$ on objects by $b\mapsto \pi^{-1}(b)$ and on a morphism 
$f:b_1\to b_2$ by the span $\pi^{-1}(b_1) \xleftarrow{\dom} \pi^{-1}(f) \xrightarrow{\cod} \pi^{-1}(b_2),$ 
with action on multicells given by composition in $E.$ 

It is then natural to consider the possibility that $\nabla\pi$ preserve opcartesian multicells, which
is the virtual analogue of the condition that a lax double functor be pseudo. This condition means
precisely that every factorization of a morphism $\pi(f)$ in $B$ lifts uniquely to a factorization
of $f$ in $E.$ Since semicategories over $B$ can be sketched by coloring the sketch for plain 
semicategories by objects and arrows of $B,$ (cf. Proposition \ref{prop:sliced-sketch}) we can repeat 
the argument for the absolute case above
to see that this condition is equivalent to exponentiability of $\pi$ in $\Cat{Semicat}/B,$ 
and furthermore that it suffices to consider binary factorizations alone. We have finished 
the proof of our general result for semicategories:

\begin{prop}[Exponentiable semifunctors]
For a semifunctor $\pi:E\to B,$ the following are equivalent: 
\begin{enumerate}
\item $\pi$ is exponentiable in $\Cat{Semicat}/B.$
\item The virtual double functor $\Dbl{L}B\to \Span$ corresponding to $\pi$ preserves opcartesian multicells.
\item Given a morphism $f$ in $E,$ every factorization of $\pi(f)$ in $B$ lifts uniquely to a factorization in $E.$
\item As in (3), but for binary factorizations alone. 
\end{enumerate}
\end{prop}

\subsection{Exponentiable categories and essentially unique factorizations}

We now move to more general category structures where, once some identity maps come into play, 
exponentiability becomes more complicated, while still following the general shape of the story outlined
for mere semicategories above. 
The exponentiable objects in (slices of) $\Cat{Cat}$ and of $\Cat{Multicat}$ are known. In the first 
case, it is an old result that the exponentiable functors $\pi:\Cat{E}\to \Cat{B}$ are precisely the Conduch\'{e} fibrations (\cite{conduche1972,Giraud1964})
\begin{defn}[Conduch\'{e} fibration]
A functor $\pi:\Cat{E}\to \Cat{B}$ is a \emph{Conduch\'{e} fibration} if for every morphism $f:e_1\to e_3$ in $\Cat{E}$
and for every factorization $\pi(f)=g_1\cdot g_2$ in $B,$ the category $\Cat{Fact}_{g_1,g_2}$ of factorizations of $f$ lifting
$(g_1, g_2)$ is connected.
\qed
\end{defn}
To be sure, we spell out that:
\begin{itemize}
\item A connected category has one connected component; in particular, it is nonempty.
\item For the category $\Cat{Fact}_{g_1,g_2}:$ If $e_1\xrightarrow{f_1}e_2\xrightarrow{f_2} e_3$ and 
$e_1\xrightarrow{f_1'} e_2'\xrightarrow{f_2'} e_3$ 
are two factorizations lifting $(g_1,g_2),$ then a morphism $(f_1,f_2)\to (f_1',f_2')$ in $\Cat{Fact}_{g_1,g_2}$ is given by a morphism $h:e_2\to e_2'$ such that 
$\pi(h)=\id_{b_2}$ and $f_1'=f_1\cdot h,$ $f_2=h\cdot f_2',$ as below:
\[\begin{tikzcd}
	{e_1} && {e_3} \\
	& {e_2} \\
	& {e_2'}
	\arrow["f"{description}, from=1-1, to=1-3]
	\arrow["{f_1}"{description}, from=1-1, to=2-2]
	\arrow["{f_1'}"{description}, from=1-1, to=3-2]
	\arrow["{f_2}"{description}, from=2-2, to=1-3]
	\arrow["h"{description}, from=2-2, to=3-2]
	\arrow["{f_2'}"{description}, from=3-2, to=1-3]
\end{tikzcd}\]
\end{itemize} 

Mnemonically, we can say that exponentiable functors admit \emph{essentially} unique liftings of
factorizations, whereas exponentiable semifunctors had to be provide the much stronger 
\emph{strictly} unique 
liftings. As we alluded to above, essential uniqueness does not even make sense in the world of semicategories, because in the absence of identity morphisms one has no notion of fiber.
On the other hand, we \emph{can} give the definition with unique liftings in the case of categories, 
and then we will have recovered the notion of \emph{discrete} Conduch\'{e} fibration. \footnote{Conduch\'{e} 
fibrations jointly generalize Grothendieck fibrations and opfibrations while discrete Conduch\'{e} 
fibrations jointly generalize discrete fibrations and discrete opfibrations.}

Let us put our sketch machinery to work to re-derive this characterization of exponentiable functors.
Though the theorem is old, we suggest that our approach allows for a more complete proof than 
most authors have had the patience to give without the sketch machinery.

For a functor $\pi:\Cat{E}\to \Cat{B},$ we can treat it as a semifunctor and construct the corresponding virtual double functor $\nabla\pi:\L(\Cat{B})\to \Span$, 
but this $\nabla\pi$ will preserve opcartesian multicells only when $\pi$ is a \emph{discrete} 
Conduch\'{e} fibration. 
To capture the general case, we recall (Eq.~\ref{eq:SpanProf}) that virtual double functors 
into $\Span$ from unital VDCs correspond to normal
virtual double functors into $\Prof$. 
What is crucial is that this correspondence does \emph{not} reflect which virtual double
functors preserve opcartesian multicells, since the composition in $\Prof$ quotients 
that in $\Span$ by the inner actions of the profunctors being composed. Thus it is a weaker 
condition that $\nabla\pi$ preserves opcartesian multicells when viewed as a virtual double functor into 
$\Prof,$ and that weaker condition is the crucial one:

\begin{thm}[Conduch\'{e}, Giraud]\textrm{(Exponentiable functors)} For a functor $\pi:\Cat{E}\to \Cat{B},$ the following are equivalent:
  \begin{enumerate}
  \item $\pi$ is exponentiable in $\Cat{Cat}/\Cat{B}.$
  \item $\pi$ admits essentially unique liftings of factorizations, i.e. is a Conduch\'{e} fibration.
  \item The double functor $\nabla\pi: \L(\Cat{B})\to \Prof$ preserves opcartesian multicells. 
  \end{enumerate}
\end{thm}
\begin{proof}
Since $\Cat{Cat}$ is locally finite presentable, (1) is equivalent to the preservation of 
sketched colimits in a sketch for categories over $\Cat{B}$ by the product functor with $\pi$ 
(Corollary \ref{cor:exponentiable-sketch}).
We can give such a sketch in terms of object, arrow, triangle, and tetrahedron types colored by 
the objects and arrows of $\Cat{B},$ by extending the semicategory sketch above 
with structure for identity morphisms and then applying Proposition \ref{prop:sliced-sketch}. 
Then there is one sketched colimit for each commutative triangle and for each commutative 
tetrahedron in $\Cat{B}.$ 

Given a commutative triangle $g=(b_0\xrightarrow{g_1} b_1 \xrightarrow{g_2} b_2)$ in $\Cat{B},$ the corresponding colimit 
is 
\[g:\triangle\to \Cat{B}=(g_1:{\downarrow}\to \Cat{B})\underset{b_1:\bullet\to \Cat{B}}{+} (g_2:{\downarrow}\to \Cat{B}).\]
Now, the product $\Cat{E}\times_{\Cat{B}} g$ has as arrows over $g$ precisely the arrows of 
$\Cat{E}$ over $g,$
while the pushout $(\Cat{E}\times_\Cat{B} g_1) \underset{\Cat{E}\times_\Cat{B} \{b_1\}}{+} (\Cat{E}\times_\Cat{B} g_2)$ has as arrows over $g$ the composable pairs of lifts of $g_1$ and $g_2,$ 
respectively, up to arrows over $\id_{b_1},$ that is, 
the set of connected components $\pi_0\Cat{Fact}_{g_1,g_2}.$ This shows that 
$(1)\implies (2).$

Now, $(2)$ is equivalent to the binary case of $(3)$ by unfolding the definitions: if 
$g=(b_0\stackrel{g_1}{\to} b_1\stackrel{g_2}{\to} b_2)$ in $B$ and we have the corresponding
opcartesian cell $\alpha:g_1,g_2\Rightarrow g$ in $\L(\Cat{B}),$ then by definition of
the multicells in $\Prof,$ the image $\nabla\pi(\alpha)$ is the map $\int^{\pi_{b_1}}\pi_{g_1}(e_0,e_1)\pi_{g_2}(e_1,e_2)\to \pi_g(e_0,e_2),$
natural in $e_0\in \pi_{b_0}$ and $e_2\in \pi_{b_2},$ and
induced by the composition in $E.$ The equivalence relation induced on paths from $e_0$ to $e_2$ 
over $(g_1,g_2)$ in the coend is precisely that of lying in the same connected component in the
category of factorizations lifting $(g_1,g_2)$ defined above. Thus $\nabla\pi(\alpha)$ is an isomorphism if and only if $\pi$ admits essentially unique liftings of the factorization 
$(g_1,g_2).$

Next, the binary case of $(3)$ implies the general case because every opcartesian multicell 
in $\L(\Cat{B})$ may be decomposed in terms of binary and unary ones, where a unary opcartesian
cell is just an identity multicell, thus certainly preserved. In $\Prof,$ this becomes the observation 
that, for instance, we have a natural isomorphism 
\[\begin{tikzcd}
	{\int^{\pi_{b_1},\pi_{b_2}} \pi_{g_1}(e_0,e_1)\pi_{g_2}(e_1,e_2)\pi_{g_3}(e_2,e_3)} \\
	{\int^{\pi_{b_2}} \left(\int^{\pi_{b_1}} \pi_{g_1}(e_0,e_1)\pi_{g_2}(e_1,e_2)\right)\pi_{g_3}(e_2,e_3)}
	\arrow["\cong"{marking, allow upside down}, draw=none, from=1-1, to=2-1],
\end{tikzcd}\]
which is a case of the Fubini theorem for coends. \cite[2.8]{Kelly1982} Thus we have $(2) \iff (3).$

Finally, for $(3)$ implies $(1),$ it remains only to show that if $\pi$ is a Conduch\'{e} fibration, 
then product with $\pi$ preserves the sketched colimits for tetrahedra as well. Preserving the 
colimit in $\Cat{Cat}/\Cat{B}$ for a commutative tetrahedron $h=(b_0\xrightarrow{h_1} b_1 
\xrightarrow{h_2} b_2 \xrightarrow{h_3} b_3)$ means that if $\pi(f)=h,$ then there are essentially 
unique factorizations of $f$ lifting $(h_1,h_2,h_3),$ where the category of such factorizations has 
morphisms of this form: 
\[\begin{tikzcd}
	& e_1 & e_2 \\
	e_0 &&& e_3 \\
	& e_1' & e_2' \\
	\\
	{b_0} & {b_1} & {b_2} & {b_3}
	\arrow["{f_2}"{description}, from=1-2, to=1-3]
	\arrow["k_1"{description}, from=1-2, to=3-2]
	\arrow["{f_3}"{description}, from=1-3, to=2-4]
	\arrow["k_2"{description}, from=1-3, to=3-3]
	\arrow["{f_1}"{description}, from=2-1, to=1-2]
	\arrow["{f_1'}"{description}, from=2-1, to=3-2]
	\arrow["{f_2'}"{description}, from=3-2, to=3-3]
	\arrow["{f_3'}"{description}, from=3-3, to=2-4]
	\arrow["{h_1}", from=5-1, to=5-2]
	\arrow["{h_2}", from=5-2, to=5-3]
	\arrow["{h_3}", from=5-3, to=5-4]
\end{tikzcd}\]
But the set of such factorizations is then precisely the coend 
\[\int^{\pi_{b_1},\pi_{b_2}} \pi_{g_1}(e_0,e_1)\pi_{g_2}(e_1,e_2)\pi_{g_3}(e_2,e_3),\]
which gives the result.
\end{proof}

We thus recover the very well-known fact that all small categories are exponentiable. In terms
of the explicit structure of the exponential $\Cat{A}^\Cat{B},$ the key point is that, rather than 
the components $\alpha_f$ for morphisms $f:x\to y$ in $\Cat{B}$ from the semicategory case, which 
are generally not composable, we can use the naturality condition 
$\alpha_f=\alpha_x\cdot G(f)=F(f)\cdot \alpha_y$ to reduce to the components at identities, which 
one knows how to compose. In $\Cat{Cat}/\Cat{B},$ one must use the decomposition condition
more explicitly to compose fibered natural transformations. 

\subsection{Exponentiable multicategories}

Pisani~\cite{Pisani2014} characterized the exponentiable multicategories as the promonoidal categories. 
This result is correct, but he relied on an inaccurate description of promonoidal categories in his 
reference~\cite{DayPanchadcharamStreet2005}, so we take this opportunity to clarify the situation. 

A multicategory $\MultCat{M}$ is a model of a limit sketch including a type $O$ for objects and types $A_n$ for 
$n$-ary arrows for all $n\ge 0,$ as well as for the gluings $A_n \star_i A_m$ of an 
$n$-ary morphism to an $m$-ary morphism at position $i\in \{1,\ldots,m\},$ and finally for the 
threefold gluings $A_n\star_i A_m\star_j A_\ell.$ There are sketched cones for the twofold gluings
for every $n,i,m$ and the threefold gluings for every $n,i,m,j,\ell.$ 

By Corollary \ref{cor:exponentiable-sketch}, a multicategory $\MultCat{M}$ is exponentiable if and only if the functor
$(-)\times \MultCat{M}:\Cat{Multicat}\to \Cat{Multicat}$ preserves the colimits gluing two, respectively 
three, morphisms at all arities at positions. Let us write $O,A_n,$ and so on also for the 
multicategories freely generated by an element of the corresponding type in the sketch. 
Then for exponentiability 
we must have, first of all, isomorphisms 
$\MultCat{M}\times A_n +_{(\MultCat{M}\times O)_i} \MultCat{M}\times A_m\cong \MultCat{M}\times (A_n\star_i A_m).$ 
Here on the left, we have $n+m-1$-ary morphisms given by formal composites 
of every composable pair of the form below in $\MultCat{M},$ taken up to unary morphisms on either 
side at the $y_i$ position:
\[\begin{tikzcd}
	{x_1} && {y_1} \\
	& f & {y_i} & g & z \\
	{x_m} && {y_n}
	\arrow[from=1-1, to=2-2]
	\arrow["\vdots"{description}, draw=none, from=1-1, to=3-1]
	\arrow["\vdots"{description}, draw=none, from=1-3, to=2-3]
	\arrow[from=1-3, to=2-4]
	\arrow[from=2-2, to=2-3]
	\arrow[from=2-3, to=2-4]
	\arrow["\vdots"{description}, draw=none, from=2-3, to=3-3]
	\arrow[from=2-4, to=2-5]
	\arrow[from=3-1, to=2-2]
	\arrow[from=3-3, to=2-4]
\end{tikzcd}\]
Such formal composites on the left map to their genuine composite on the right. We thus have as a 
necessary condition for exponentiability that each map from a coend, taken over unary morphisms:
\[\begin{tikzcd}
	{\int^{y_i} \MultCat{M}(x_1,\ldots,x_m; y_i)\times \MultCat{M}(y_1,\ldots,y_i,\ldots, y_n;z)} \\
	{\MultCat{M}(y_1,\ldots,y_{i-1},x_1,\ldots,x_m,y_{i+1},\ldots,y_n;z)}
	\arrow[from=1-1, to=2-1]
\end{tikzcd}\]
is bijective. This is precisely what it means for $\MultCat{M}$ to be a promonoidal category. 
\footnote{The 
issue in \cite{DayPanchadcharamStreet2005} is that the authors restrict to the cases 
$m=0,n=2$ and $m=n=2,$ which is insufficient. This corrected definition is the same as the notion 
recently given by Rom\'an \cite{Román2024} and called there \emph{malleable} multicategories, but 
Rom\'an does not recall the connection to exponentiability.}

To finish applying our corollary, we have only to get past the sketched cones for gluings 
of threefold composites. This can be reduced to the binary case much as in the situation of semicategories
treated above, though one has to track the coend equivalence relation rather than working
with equalities directly. Thus it's more expedient to use the Fubini theorem on iterated coends,
which again gets the desired result: 
\begin{prop}[Pisani] \textrm{(Exponentiable multicategories)}
A multicategory $\MultCat{M}$ is exponentiable if and only if it is the multicategory associated to a 
promonoidal category. 
\end{prop}

With these warmups complete, we are ready to turn to our subject of primary interest.

%% file: Sections/4_ExponentiabilityViaDecompositions.tex
We will characterize the exponentiable VDCs through a collection of equivalent formulations in terms of decomposition properties for the multicells in the VDC $\Dbl{D}$. 
However, before characterizing the exponentiable VDCs, we can use their universal property to describe the objects, tight arrows, loose arrows, and multicells 
in the exponential when it exists. 
Here we will use the observation from Section~\ref{sec:VDCs} that the functors $\uOb , \uTight,\uLoose,\uSq{n}:\Vdc\to \Cat{Set}$, $n\geq 0$, are all corepresented. 
Additionally, we will need to understand products in $\Vdc$ with the corepresenting VDCs, along with functors out of such products. 

\subsection{The structure of the exponential, if it exists}

\begin{prop}[Elements of the exponential VDC]\label{prop:elementsOfExpVDC}
  If $\Dbl{D}$ and $\Dbl{E}$ are VDCs such that the exponential $\Dbl{E}^\Dbl{D}$ exists, then the VDC $\Dbl{E}^\Dbl{D}$ consists of the following data:
\begin{enumerate}
    \item Objects $F \in \uOb (\Dbl{E}^\Dbl{D})\cong \Vdc(\Ob,\Dbl{E}^\Dbl{D})\cong \Vdc(\Ob\times \Dbl{D},\Dbl{E})$ correspond to functors $F:\Dbl{D}_0\to \Dbl{E}_0$ on underlying tight categories.\footnote{Note that, since $\uOb$, lacking loose arrows and multicells, is not terminal, the objects of an exponential are not simply the virtual double functors. Exponentiability will nonetheless be seen below to imply the existence of the virtual double category of virtual double functors.}
    \item Tight arrows $f \in \mathsf{Arr_T}(\Dbl{E}^\Dbl{D})\cong \Vdc(\Tight,\Dbl{E}^\Dbl{D})\cong \Vdc(\Tight\times \Dbl{D},\Dbl{E})$ correspond to natural transformations $f:F\Rightarrow G$ between functors $F,G:\Dbl{D}_0\to \Dbl{E}_0$.
    \item Loose arrows $\Phi \in \uLoose(\Dbl{E}^\Dbl{D})\cong \Vdc(\Loose ,\Dbl{E}^\Dbl{D})\cong \Vdc(\Loose \times \Dbl{D},\Dbl{E})$ correspond to maps of spans of categories 
    \[\begin{tikzcd}
      {\Dbl{D}_0} & {\Dbl{D}_1} & {\Dbl{D}_0} \\
      {\Dbl{E}_0} & {\Dbl{E}_1} & {\Dbl{D}_0}
      \arrow[from=1-1, to=2-1]
      \arrow[from=1-2, to=1-1]
      \arrow[from=1-2, to=1-3]
      \arrow[from=1-2, to=2-2]
      \arrow[from=1-3, to=2-3]
      \arrow[from=2-2, to=2-1]
      \arrow[from=2-2, to=2-3]
    \end{tikzcd}\]
    \item Unary multicells in $\uSq{1}(\Dbl{E}^\Dbl{D})\cong \Vdc(\Sq{1},\Dbl{E}^\Dbl{D})\cong \Vdc(\Sq{1}\times \Dbl{D},\Dbl{E})$ correspond to 2-cells in the 2-category of spans of categories between maps of spans as above.
    \item Multicells $\Gamma \in \uSq{n}(\Dbl{E}^\Dbl{D})\cong \Vdc(\Sq{n},\Dbl{E}^\Dbl{D})\cong \Vdc(\Sq{n}\times \Dbl{D},\Dbl{E}),$ for $n\ne 1$ are given by a map of profunctors $\MCelln{n}(\Dbl{D})\to \partial^*\MCelln{n}(\Dbl{E})$ over $\Dbl{E}_0^2,$ where 
    $\partial:\Cat{fc}_n(\Dbl{D}_1)^{op}\times \Dbl{D}_1\to \Cat{fc}_n(\Dbl{E}_1)^{op}\times \Dbl{E}_1$ is the functor which together with the natural transformations putting $\MCelln{n}(\Dbl{D})$ over $\Dbl{E}_0^2$ determines the  boundary data of $\Gamma.$
\end{enumerate}
\end{prop}

\begin{proof}[Proof of Proposition~\ref{prop:elementsOfExpVDC}.]

Since $\Ob$ contains nothing but a single object with its tight identity arrow, 
we see crossing with $\Ob$ eliminates all loose arrows and multicells, that is, 
$\Ob\times \Dbl{D}$ is the virtual double category with tight category $\Dbl{D}_0$ 
and no loose arrows or multicells. This and the analogous computation for 
$\Tight\times \Dbl{D}$ provide the proofs of claims 1. and 2.

Next, observe that $\Loose \times \Dbl{D}$ consists of two copies of the underlying tight category of 
$\Dbl{D}$, corresponding to the two objects $0$ and $1$ in $\Loose$, along with a loose arrow 
$(\ell,\varphi):(0,x)\proto (1,y)$ between these two copies for each loose arrow $\varphi:x\proto y$ in $\Dbl{D}$. 
Since $\Loose$ contains one unary multicell, namely $\text{id}_{\ell}$, the product $\Loose\times \Dbl{D}$ contains, 
for each unary multicell $\alpha:\cell{a}{\varphi}{\psi}{b}$ in $\Dbl{D}$, 
a unary multicell $(\text{id}_{\ell},\alpha):\cell{(\text{id}_0,a)}{(\ell,\varphi)}{(\ell,\psi)}{(\text{id}_1,b)}$ in 
$\Loose \times \Dbl{D}$. 

Thus, a VDF $\Phi:\Loose\times \Dbl{D}\to \Dbl{E}$ consists of the data of two functors 
on underlying tight categories, $F_0,F_1:\Dbl{D}_0\to \Dbl{E}_0$, along with for each loose arrow 
$\varphi:x\proto y$ in $\Dbl{D}$, a loose arrow $\Phi(\ell,\varphi):F_0(x)\proto F_1(y)$ in $\Dbl{E}$, and for each 
unary multicell $\alpha$ in $\Dbl{D}$, a unary multicell
\[
\begin{adjustbox}{}
\begin{tikzcd}
	{F_0(x)} & {F_1(y)} \\
	{F_0(z)} & {F_1(w)}
	\arrow[""{name=0, anchor=center, inner sep=0}, "{{\Phi(\ell,\varphi)}}"{inner sep=.8ex}, "\shortmid"{marking}, from=1-1, to=1-2]
	\arrow["{{F_0(a)}}"', from=1-1, to=2-1]
	\arrow["{{F_1(b)}}", from=1-2, to=2-2]
	\arrow[""{name=1, anchor=center, inner sep=0}, "{{\Phi(\ell,\psi)}}"'{inner sep=.8ex}, "\shortmid"{marking}, from=2-1, to=2-2]
	\arrow["{\Phi(\text{id}_{\ell},\alpha)}"{description}, draw=none, from=0, to=1]
\end{tikzcd}
\end{adjustbox}
\]
in $\Dbl{E}$. Further, this assignment is functorial under vertical composition of unary multicells, and hence amounts to a map 
of spans of categories as in claim 3.

From our description of $\Sq{n}$ in Section~\ref{sec:VDCs}, $\Sq{n}\times \Dbl{D}$ has $n+3$ copies of $\Dbl{D}_0$, 
loose arrows and unary multicells as in $\Loose \times \Dbl{D}$ between these copies for each loose arrow in $\Dbl{D}$ and 
loose arrow in $\Sq{n}$, tight arrows between the copies on the edges of $\Sq{n}$, and for each $n$-ary multicell 
$\alpha$ in $\Dbl{D}$, an $n$-ary multicell $(\square_n,\alpha)$.
A VDF $\Gamma:\Sq{n}\times \Dbl{D}\to \Dbl{E}$ thus consists of the data of $n+3$ functors on underlying tight categories, 
$F_0,...,F_n,G_0,G_1:\Dbl{D}_0\to \Dbl{E}_0$, together with two natural transformations 
$f_0:F_0\Rightarrow G_0$ and $f_1:F_n\Rightarrow G_1$, $n+1$ functors 
$\Phi_1,\ldots,\Phi_n,\Psi:\Dbl{D}_1\to \Dbl{E}_1$, as above, and for each 
$n$-ary multicell $\beta:\cell{a}{\prolist{\phi}}{\psi}{b}$ in $\Dbl{D}$, 
an $n$-ary multicell
\[\begin{tikzcd}
	{F_0x_0} && \cdots && {F_nx_n} \\
	{G_0y_0} &&&& {G_1y_1}
	\arrow["{{{{\Phi_1(\varphi_1)}}}}"{inner sep=.8ex}, "\shortmid"{marking}, from=1-1, to=1-3]
	\arrow["{{{{f_{0,a}}}}}"', from=1-1, to=2-1]
	\arrow["{{{{\Phi_n(\varphi_n)}}}}"{inner sep=.8ex}, "\shortmid"{marking}, from=1-3, to=1-5]
	\arrow["{{{{f_{1,b}}}}}", from=1-5, to=2-5]
	\arrow[""{name=0, anchor=center, inner sep=0}, "{{{\Psi(\psi)}}}"'{inner sep=.8ex}, "\shortmid"{marking}, from=2-1, to=2-5]
	\arrow["{{\Gamma(\beta)}}"{description}, draw=none, from=1-3, to=0]
\end{tikzcd}\]
where for instance $f_{0,a}=f_{0,x_0}\cdot G_0(a)=F_0(a)\cdot f_{0,y_0}.$ 

This data is functorial in the sense that the composite of the pasting diagram
\[\begin{tikzcd}
	{F_0w_0} && {F_1w_1} && \cdots && {F_nw_n} \\
	{F_0x_0} && {F_1x_1} && \cdots && {F_nx_n} \\
	{G_0y_0} &&&&&& {G_1y_1} \\
	{G_0z_0} &&&&&& {G_1z_1}
	\arrow[""{name=0, anchor=center, inner sep=0}, "{{{{\Phi_1(\varphi_1')}}}}"{inner sep=.8ex}, "\shortmid"{marking}, from=1-1, to=1-3]
	\arrow["{{{{{{F_0c_0}}}}}}"', from=1-1, to=2-1]
	\arrow[""{name=1, anchor=center, inner sep=0}, "{{{{\Phi_2(\varphi_2')}}}}"{inner sep=.8ex}, "\shortmid"{marking}, from=1-3, to=1-5]
	\arrow["{{{{{{F_1c_1}}}}}}"{description}, from=1-3, to=2-3]
	\arrow[""{name=2, anchor=center, inner sep=0}, "{{{{\Phi_n(\varphi_n')}}}}"{inner sep=.8ex}, "\shortmid"{marking}, from=1-5, to=1-7]
	\arrow["{{{{{{F_nc_n}}}}}}", from=1-7, to=2-7]
	\arrow[""{name=3, anchor=center, inner sep=0}, "{{{{{{\Phi_1(\varphi_1)}}}}}}"'{inner sep=.8ex}, "\shortmid"{marking}, from=2-1, to=2-3]
	\arrow["{{{{f_{0,a}}}}}"', from=2-1, to=3-1]
	\arrow[""{name=4, anchor=center, inner sep=0}, "{{{{\Phi_2(\varphi_2)}}}}"'{inner sep=.8ex}, "\shortmid"{marking}, from=2-3, to=2-5]
	\arrow[""{name=5, anchor=center, inner sep=0}, "{{{{\Phi_n(\varphi_n)}}}}"'{inner sep=.8ex}, "\shortmid"{marking}, from=2-5, to=2-7]
	\arrow["{{{{f_{1,b}}}}}", from=2-7, to=3-7]
	\arrow[""{name=6, anchor=center, inner sep=0}, "{{{{\Psi(\psi)}}}}"'{inner sep=.8ex}, "\shortmid"{marking}, from=3-1, to=3-7]
	\arrow["{{{{{G_0d_0}}}}}"', from=3-1, to=4-1]
	\arrow["{{{{{{G_1d_1}}}}}}", from=3-7, to=4-7]
	\arrow[""{name=7, anchor=center, inner sep=0}, "{{{{{\Psi(\psi')}}}}}"'{inner sep=.8ex}, "\shortmid"{marking}, from=4-1, to=4-7]
	\arrow["{{\Phi_1(\delta_1)}}"{description}, draw=none, from=0, to=3]
	\arrow["{{\Phi_2(\delta_2)}}"{description}, draw=none, from=1, to=4]
	\arrow["{{\Phi_n(\delta_n)}}"{description}, draw=none, from=2, to=5]
	\arrow["{{\Gamma(\beta)}}"{description}, draw=none, from=4, to=6]
	\arrow["{{\Psi(\gamma)}}"{description}, draw=none, from=6, to=7]
\end{tikzcd}\]
is equal to $\Gamma\left(\threefrac{\delta_1 \cdots \delta_n}{\beta}{\gamma}\right)$.

To establish the lighter axiomatization for the unary case as in 4, 
consider that in the product $\Sq{1}\times\Dbl{D},$ 
for any $\beta:\cell{a}{\chi}{\omega}{b}$ in $\Dbl{D},$ we have 
$(\square_1,\beta)=\frac{(\square_1,\text{id}_\chi)}{(\text{id}_\ell,\beta)}=\frac{(\text{id}_{\ell_1},\beta)}{(\square_1,\text{id}_\omega)}.$ 
This shows that the unary multicells in $\Sq{1}\times\Dbl{D}$ are generated by those coming from 
the boundary together with those of the form $(\square_1,\text{id}_\chi)$, which is the fundamental 
reason why 
a modulation may be defined with components indexed by proarrows of $\Dbl{D}$, rather than unary multicells.

Under this presentation, the only relations in the category of loose arrows of $\Sq{1}\times \Dbl{D}$ are of the form 
$\frac{(\square_1,\text{id}_\chi)}{(\text{id}_\ell,\beta)}=\frac{(\text{id}_{\ell_1},\beta)}{(\square_1,\text{id}_\omega)}.$
Thus we can give a unary cell in $\Dbl{E}^\Dbl{D}$ by giving 
$\Gamma_\varphi:=\Gamma(\square_1,\text{id}_{\varphi}):\cell{f_{0,x}}{\Phi_\varphi}{\Psi_\varphi}{f_{1,y}}$
for every $\varphi:x\proto y$ in $\Dbl{D},$ subject to the naturality condition, for every 
$\alpha:\cell{}{\varphi}{\psi}{}$ in $\Dbl{D}.$ 
\[\begin{tikzcd}
	{F_0x_0} & {F_1x_1} && {F_0x_0} & {F_1x_1} \\
	{G_0x_0} & {G_1x_1} & {=} & {F_0y_0} & {F_1y_1} \\
	{G_0y_0} & {G_1y_1} && {G_0y_0} & {G_1y_1}
	\arrow[""{name=0, anchor=center, inner sep=0}, "{{\Phi_{\varphi}}}"{inner sep=.8ex}, "\shortmid"{marking}, from=1-1, to=1-2]
	\arrow["{{f_{0,x_0}}}"', from=1-1, to=2-1]
	\arrow["{{f_{1,x_1}}}", from=1-2, to=2-2]
	\arrow[""{name=1, anchor=center, inner sep=0}, "{{\Phi_\varphi}}"{inner sep=.8ex}, "\shortmid"{marking}, from=1-4, to=1-5]
	\arrow[from=1-4, to=2-4]
	\arrow[from=1-5, to=2-5]
	\arrow[""{name=2, anchor=center, inner sep=0}, "{{\Psi_\varphi}}"'{inner sep=.8ex}, "\shortmid"{marking}, from=2-1, to=2-2]
	\arrow[from=2-1, to=3-1]
	\arrow[from=2-2, to=3-2]
	\arrow[""{name=3, anchor=center, inner sep=0}, "{{\Phi_\psi}}"'{inner sep=.8ex}, "\shortmid"{marking}, from=2-4, to=2-5]
	\arrow["{{f_{0,y_0}}}"', from=2-4, to=3-4]
	\arrow["{{f_{1,y_1}}}", from=2-5, to=3-5]
	\arrow[""{name=4, anchor=center, inner sep=0}, "{{\Psi_\psi}}"'{inner sep=.8ex}, "\shortmid"{marking}, from=3-1, to=3-2]
	\arrow[""{name=5, anchor=center, inner sep=0}, "{{\Psi_\psi}}"'{inner sep=.8ex}, "\shortmid"{marking}, from=3-4, to=3-5]
	\arrow["{{\Gamma_\varphi}}"{description}, draw=none, from=0, to=2]
	\arrow["{{\Phi_\alpha}}"{description}, draw=none, from=1, to=3]
	\arrow["{{\Psi_\alpha}}"{description}, draw=none, from=2, to=4]
	\arrow["{{\Gamma_\psi}}"{description}, draw=none, from=3, to=5]
\end{tikzcd}\]
Indeed, such a $\Gamma$ amounts to a 2-cell in the 2-category of spans of categories. 
\end{proof}

\subsubsection*{The exponential versus the VDC of virtual double functors}
The exponential $\Dbl{E}^\Dbl{D}$ is often not the VDC we are interested in, since its objects 
are functors between underlying tight categories rather than VDFs between the VDCs.
However, once we apply the $\Mod(-)$ construction (Definition \ref{defn:Modules}), the resulting 
VDC $\Dbl{V}\mathsf{df}(\Dbl{D},\Dbl{E}):=\Mod(\Dbl{E}^\Dbl{D})$ does have VDFs as objects. 
Note that if $\Dbl{D}={\Ob_u}$ is the terminal VDC, which is representable, as is the 
terminal object in any cartesian category.
That is, $\Dbl{E}^{{\Ob_u}}\cong \Dbl{E}$ for any VDC $\Dbl{E}$, so that 
$\Dbl{V}\mathsf{df}({\Ob_u},\Dbl{E})\cong \Mod(\Dbl{E}),$ which essentially reproduces
Benabou's old observation that lax functors from the point are monads (in a bicategory, 
for Benabou.~\cite{Benabou1967})

\begin{lem}\label{lem:VdfVDC}
	If $\Dbl{D}$ is an exponentiable VDC, then the VDC $\Dbl{V}\mathsf{df}(\Dbl{D},\Dbl{E}):=\Mod(\Dbl{E}^\Dbl{D})$ exists for any VDC $\Dbl{E}$, and it has VDFs as objects with tight transformations as tight arrows.
\end{lem}
\begin{proof}
	The existence of the VDC $\Dbl{V}\mathsf{df}(\Dbl{D},\Dbl{E}):=\Mod(\Dbl{E}^\Dbl{D})$ follows immediately from the existence of $\Dbl{E}^\Dbl{D}$ since $\Dbl{D}$ is assumed to be exponentiable. We can give an elegant description of the data in $\Dbl{V}\mathsf{df}(\Dbl{D},\Dbl{E})$ using the 2-adjoint triple 
\[
\begin{tikzcd}
	{\BiCat{Vdc}_n} && {\BiCat{Vdc}}
	\arrow[""{name=0, anchor=center, inner sep=0}, "U"{description}, from=1-1, to=1-3]
	\arrow[""{name=1, anchor=center, inner sep=0}, "{\Dbl{F}_u}"', curve={height=18pt}, from=1-3, to=1-1]
	\arrow[""{name=2, anchor=center, inner sep=0}, "{\Mod}", curve={height=-18pt}, from=1-3, to=1-1]
	\arrow["\dashv"{anchor=center, rotate=-92}, draw=none, from=0, to=2]
	\arrow["\dashv"{anchor=center, rotate=-88}, draw=none, from=1, to=0]
\end{tikzcd}
\]
described in Section~\ref{sec:VDCs}. Explicitly, this adjoint triple induces a pair of 2-adjunctions
\[\begin{tikzcd}
	{\BiCat{Vdc}} && {\BiCat{Vdc}} && {\BiCat{Vdc}_n} && {\BiCat{Vdc}_n}
	\arrow[""{name=0, anchor=center, inner sep=0}, "{{U\circ \Mod}}"', curve={height=18pt}, from=1-1, to=1-3]
	\arrow[""{name=1, anchor=center, inner sep=0}, "{{U\circ \Dbl{F}_u}}"', curve={height=18pt}, from=1-3, to=1-1]
	\arrow[""{name=2, anchor=center, inner sep=0}, "{{U\circ \Mod}}"', curve={height=18pt}, from=1-5, to=1-7]
	\arrow[""{name=3, anchor=center, inner sep=0}, "{{U\circ \Dbl{F}_u}}"', curve={height=18pt}, from=1-7, to=1-5]
	\arrow["\dashv"{anchor=center, rotate=-90}, draw=none, from=1, to=0]
	\arrow["\dashv"{anchor=center, rotate=-90}, draw=none, from=3, to=2]
\end{tikzcd}\]

From these adjunctions we obtain the natural isomorphisms
\begin{align*}
\BiCat{Vdc}(\Dbl{A},\Dbl{V}\mathsf{df}(\Dbl{D},\Dbl{E}))&\cong \BiCat{Vdc}(\Dbl{A},\Mod(\Dbl{E}^\Dbl{D}))\cong \BiCat{Vdc}_n(\Dbl{F}_u(\Dbl{A}),\Mod(\Dbl{E}^\Dbl{D}))\\
&\cong \BiCat{Vdc}(\Dbl{F}_u(\Dbl{A}),\Dbl{E}^\Dbl{D})\cong \BiCat{Vdc}(\Dbl{F}_u(\Dbl{A})\times \Dbl{D},\Dbl{E})
\end{align*}
Thus for $\Dbl{D}$ exponentiable, we have the adjunction

\[\begin{adjustbox}{}
\begin{tikzcd}
	{\BiCat{Vdc}} && {\BiCat{Vdc}}
	\arrow[""{name=0, anchor=center, inner sep=0}, "{\Dbl{V}\mathsf{df}(\Dbl{D},-)}"', curve={height=18pt}, from=1-1, to=1-3]
	\arrow[""{name=1, anchor=center, inner sep=0}, "{\Dbl{F}_u(-)\times \Dbl{D}}"', curve={height=18pt}, from=1-3, to=1-1]
	\arrow["\dashv"{anchor=center, rotate=-90}, draw=none, from=1, to=0]
\end{tikzcd}
\end{adjustbox}\]
(which also appears as Theorem 6.10 in~\cite{arkor2025exponentiablevirtualdoublecategories}) 

Using this adjunction we can describe the data for the VDC $\Dbl{V}\mathsf{df}(\Dbl{D},\Dbl{E})$ as follows:
\begin{enumerate}
    \item Objects are VDFs $\Dbl{D}\cong {\Ob_u}\times \Dbl{D}\to \Dbl{E}$.
    \item Tight arrows are VDFs $\Dbl{F}_u(\Tight)\times \Dbl{D}\to \Dbl{E}$. 
	Such VDFs $F:\Dbl{F}_u(\Tight)\times \Dbl{D}\to \Dbl{E}$ correspond to tight transformations
	between the VDFs $F_0,F_1:\Dbl{D}\to \Dbl{E}$ obtained by restricting $F$
	along $s\times \text{id},t\times \text{id}:\Ob_u\times \Dbl{D}\hookrightarrow \Dbl{F}_u(\Tight)\times \Dbl{D}$.
	The component of the tight transformation $f:F_0\to F_1$ associated to $F$
	at an object $x\in \Dbl{D}$ is given by $f_x:= F(t,\text{id}_x)$, 
	where $t$ is the unique non-identity tight arrow in $\Dbl{F}_u(\Tight)$,
	while the component at a loose arrow $\varphi:x\proto y$ in $\Dbl{D}$ 
	is given by $f_\varphi:= F(I_t,\text{id}_{\varphi})$.
	Naturality of $f$ follows from functoriality of $F$ as follows
	for $a:x\to y$ and $\alpha:\prolist{\varphi}\to \psi$:
	\begin{equation*}
		\frac{f_x}{F_1(a)}=\frac{F(t,\text{id}_x)}{F(\text{id}_1,a)}=F(t,a)=\frac{F(\text{id}_0,a)}{F(t,\text{id}_y)}=\frac{F_0(a)}{f_y}
	\end{equation*}
	and
	\begin{equation*}
		\frac{(f_{\varphi_1}\cdots f_{\varphi_n})}{F_1(\alpha)}=\frac{(F(I_t,\text{id}_{\varphi_1})\cdots F(I_t,\text{id}_{\varphi_n}))}{F(\mathsf{opcart},\alpha)}=\frac{F(\mathsf{opcart},\alpha)}{F(I_t,\text{id}_\psi)}=\frac{F_0(\alpha)}{f_{\psi}}
	\end{equation*}
	If we instead started with a tight transformation $f:F_0\to F_1$, we could use these naturality 
	equalities to define the VDF $F$.
    \item Loose arrows are VDFs $\Dbl{F}_u(\Loose )\times \Dbl{D}\to \Dbl{E}$.
    \item $n$-ary multicells are VDFs $\Dbl{F}_u(\Sq{n})\times \Dbl{D}\to \Dbl{E}$.
\end{enumerate}
	This completes the proof.
\end{proof}

\begin{expl}[Modules]\label{expl:modulesAndModulations}
	When $\Dbl{D}$ and $\Dbl{E}$ are pseudo-double categories, lax double functors of the form 
	$\Dbl{F}_u(\Loose )\times \Dbl{D}\to \Dbl{E}$ correspond to Par\'{e}'s modules,
	while lax double functors of the form $\Dbl{F}_u(\Sq{n})\times \Dbl{D}\to \Dbl{E}$
	are what Par\'{e} calls multimodulations (c.f.~\cite{Pare2011}). 
	In general for VDCs, 
	we can unpack the data of a VDF $H:\Dbl{F}_u(\Loose )\times \Dbl{D}\to \Dbl{E}$ 
	as follows:
	\begin{enumerate}
		\item[(M0)] VDFs $F,G:\Dbl{D}\to \Dbl{E}$ obtained by restricting along the inclusions
		$s\times \text{id},t\times \text{id}:\Ob_u\times \Dbl{D}\hookrightarrow \Dbl{F}_u(\Loose)\times \Dbl{D}$;
		\item[(M1)] For each loose arrow $\varphi:x\proto y$ in $\Dbl{D}$ a loose arrow
		$\Phi(\varphi):=H(\ell,\varphi):F(x)\proto G(y)$;
		\item[(M2)] For each $n$-ary multicell $\alpha:\cell{a}{\prolist{\varphi}}{\psi}{b}$ in $\Dbl{D}$ 
		and each loose arrow $\varphi_i$ in the loose source of $\alpha$,
		an $n$-ary multicell $\Phi_i(\alpha):=H(\mathsf{opcart}_i,\alpha)$
		\[
		\begin{tikzcd}
			{Fx_0} & {Fx_{i-1}} & {Gx_i} & {G(x_n)} \\
			\\
			{F(y_0)} &&& {G(y_1)}
			\arrow["{F({{{\prolist{\varphi_1}}}})}"{inner sep=.8ex}, "\shortmid"{marking}, from=1-1, to=1-2]
			\arrow["{F(a)}"', from=1-1, to=3-1]
			\arrow[""{name=0, anchor=center, inner sep=0}, "{\Phi(\varphi_i)}", from=1-2, to=1-3]
			\arrow["{G({{{\prolist{\varphi_2}}}})}"{inner sep=.8ex}, "\shortmid"{marking}, from=1-3, to=1-4]
			\arrow["{G(b)}", from=1-4, to=3-4]
			\arrow[""{name=1, anchor=center, inner sep=0}, "{\Phi(\psi)}"'{inner sep=.8ex}, "\shortmid"{marking}, from=3-1, to=3-4]
			\arrow["{\Phi_i(\alpha)}"{description}, draw=none, from=0, to=1]
		\end{tikzcd}
		\]
		where $\mathsf{opcart}_i$ is the $n$-ary opcartesian multicell in $\Dbl{F}_u(\Loose)$ with target
		the unique non-identity loose arrow $\ell$ and with loose source having $\ell$ 
		in the $i$th position.
	\end{enumerate}
	The multicell assignment is required to satisfy functoriality in the following sense:
	\begin{enumerate}
		\item[(M3)] If $\alpha=\text{id}_\varphi$ is the unit multicell for a loose arrow $\varphi$,
		then $\Phi_1(\text{id}_\varphi)=\text{id}_{\Phi(\varphi)}$;
		\item[(M4)] If $\tfrac{\prolist{\alpha}}{\beta}$ is the decomposition of a multicell in $\Dbl{D}$,
		and if $\prolist{\alpha}=(\alpha_1\cdots\alpha_n)$ where $\alpha_i$ has arity $k_i$, 
		then for each $1\leq i\leq n$ and $1\leq j\leq k_i$,
		\begin{equation*}
			\Phi_{k_1+\cdots+k_{i-1}+j}\left(\frac{\prolist{\alpha}}{\beta}\right)=\frac{(F(\alpha_1)\cdots F(\alpha_{i-1}),\Phi_j(\alpha_i),G(\alpha_{i+1})\cdots G(\alpha_n))}{\Phi_i(\beta)}
		\end{equation*}
	\end{enumerate}
	Note that in Definition 3.2 of~\cite{Pare2011} this data is unpacked further using the existence 
	of loose units and composites in $\Dbl{D}$. 
	When $\Dbl{D}=\Ob_u$, this data coincides with the definition of a module in $\Dbl{E}$,
	as unpacked in Definition~\ref{defn:Modules}. 
  Note that there are many more action multicells in (M2) than the mere compatible left and right actions of 
  a loose arrow in a module out of a pseudo double category; in the intermediate case of 
  a VDC with decomposable multicells, we expect a module can have its actions indexed only by binary 
  multicells of $\Dbl{D},$ but we do not pursue the details here.
\end{expl}

\subsection{Essentially unique decomposition of multicells in VDCs}\label{subsec:decompositions}

As in the cases of multicategories and semicategories discussed above, we shall characterize
the exponentiable VDCs as those for which the exponential admits composites, which will reduce 
to existence of essentially unique decompositions of multicells of various kinds. For motivation, we
next discuss a particular example of such decompositions in some detail. 
\begin{exmp}[A necessary kind of decomposition]\label{exmp:decomp}
We give an example of one of the colimits, preservation of which by products will characterize
exponentiable VDC's. 

Let $\Dbl{J}$ be the VDC freely generated by a unary, a ternary, and a binary multicell, juxtaposed as shown 
below: 
\[\begin{tikzcd}
	\bullet & \bullet & \bullet & \bullet & \bullet \\
	\bullet & \bullet &&& \bullet \\
	\bullet &&&& \bullet
	\arrow[""{name=0, anchor=center, inner sep=0}, "\shortmid"{marking}, from=1-1, to=1-2]
	\arrow[from=1-1, to=2-1]
	\arrow["\shortmid"{marking}, from=1-2, to=1-3]
	\arrow[from=1-2, to=2-2]
	\arrow[""{name=1, anchor=center, inner sep=0}, "\shortmid"{marking}, from=1-3, to=1-4]
	\arrow["\shortmid"{marking}, from=1-4, to=1-5]
	\arrow[from=1-5, to=2-5]
	\arrow[""{name=2, anchor=center, inner sep=0}, "\shortmid"{marking}, from=2-1, to=2-2]
	\arrow[from=2-1, to=3-1]
	\arrow[""{name=3, anchor=center, inner sep=0}, "\shortmid"{marking}, from=2-2, to=2-5]
	\arrow[from=2-5, to=3-5]
	\arrow[""{name=4, anchor=center, inner sep=0}, "\shortmid"{marking}, from=3-1, to=3-5]
	\arrow["{\alpha_1}"{description}, draw=none, from=0, to=2]
	\arrow["{\alpha_2}"{description}, draw=none, from=1, to=3]
	\arrow["\beta"{description}, draw=none, from=3, to=4]
\end{tikzcd}\]
In other words, $\Dbl{J}$ is the colimit of a diagram $D:J\to \Vdc$ shown below:
\[\begin{tikzcd}
	{\Sq{1}} && {\Sq{3}} \\
	\Loose & \Tight & \Loose \\
	& {\Sq{2}}
	\arrow[from=2-1, to=1-1]
	\arrow[from=2-1, to=3-2]
	\arrow[from=2-2, to=1-1]
	\arrow[from=2-2, to=1-3]
	\arrow[from=2-3, to=1-3]
	\arrow[from=2-3, to=3-2]
\end{tikzcd}\]
For any VDC $\Dbl{D},$ we have the canonical comparison morphism 
$c:\Dbl{D}':=\colim_{j\in J} \Dbl{D}\times D(j)\to \Dbl{D}\times \Dbl{J}$, whose invertibility
we aim to investigate. 

To compute the colimit defining $\Dbl{D}'$, we first take the 
levelwise colimit of the $\Dbl{D}\times D(j)$ on objects, tight and loose arrows, and multicells of each arity.
This already produces the correct value of $\Dbl{D}'$ except on quaternary multicells, and indeed $c$ 
is always an isomorphism away from the multicells. What is nontrivial is, thus, what happens 
when we compose multicells of $\Dbl{D}$ according to the single interesting multicell composite in $\Dbl{J}.$ 
Consider, then, a composable triple $(\gamma_1,\gamma_2,\delta)$ in $\Dbl{D}$ like so:
\[\begin{tikzcd}
	\bullet & \bullet & \bullet & \bullet & \bullet \\
	\bullet & \bullet &&& \bullet \\
	\bullet &&&& \bullet
	\arrow[""{name=0, anchor=center, inner sep=0}, "\shortmid"{marking}, from=1-1, to=1-2]
	\arrow[from=1-1, to=2-1]
	\arrow["\shortmid"{marking}, from=1-2, to=1-3]
	\arrow[from=1-2, to=2-2]
	\arrow[""{name=1, anchor=center, inner sep=0}, "\shortmid"{marking}, from=1-3, to=1-4]
	\arrow["\shortmid"{marking}, from=1-4, to=1-5]
	\arrow[from=1-5, to=2-5]
	\arrow[""{name=2, anchor=center, inner sep=0}, "\shortmid"{marking}, from=2-1, to=2-2]
	\arrow[from=2-1, to=3-1]
	\arrow[""{name=3, anchor=center, inner sep=0}, "\shortmid"{marking}, from=2-2, to=2-5]
	\arrow[from=2-5, to=3-5]
	\arrow[""{name=4, anchor=center, inner sep=0}, "\shortmid"{marking}, from=3-1, to=3-5]
	\arrow["{\gamma_1}"{description}, draw=none, from=0, to=2]
	\arrow["{\gamma_2}"{description}, draw=none, from=1, to=3]
	\arrow["\delta"{description}, draw=none, from=3, to=4]
\end{tikzcd}\]
The set of quaternary multicells in the colimit $\Dbl{D}'$ is a certain quotient of the set of such triples,
each of which provides a formal composite $\tfrac{(\gamma_1 ,\gamma_2)}{\delta}.$ The quotient 
relation arises from the identity multicells of the loose arrows along the 
middle row of $\Dbl{J}.$ Specifically, we have a natural isomorphism of the quaternary multicells of $\Dbl{D}'$ 
with the quotient of the set $\Sigma$ of composable triples $(\gamma_1,\gamma_2,\delta)$ as above by the relation that,
given a situation of the form below, we must have 
\[\left(\frac{\gamma_1}{\varepsilon_1},\frac{\gamma_2}{\varepsilon_2},\delta\right)\sim\left(\gamma_1,\gamma_2,\frac{(\varepsilon_1,\varepsilon_2)}{\delta}\right):\]
\[\begin{tikzcd}
	\bullet & \bullet & \bullet & \bullet & \bullet \\
	\bullet & \bullet &&& \bullet \\
	\bullet & \bullet &&& \bullet \\
	\bullet &&&& \bullet
	\arrow[""{name=0, anchor=center, inner sep=0}, "\shortmid"{marking}, from=1-1, to=1-2]
	\arrow[from=1-1, to=2-1]
	\arrow["\shortmid"{marking}, from=1-2, to=1-3]
	\arrow[from=1-2, to=2-2]
	\arrow[""{name=1, anchor=center, inner sep=0}, "\shortmid"{marking}, from=1-3, to=1-4]
	\arrow["\shortmid"{marking}, from=1-4, to=1-5]
	\arrow[from=1-5, to=2-5]
	\arrow[""{name=2, anchor=center, inner sep=0}, "\shortmid"{marking}, from=2-1, to=2-2]
	\arrow[from=2-1, to=3-1]
	\arrow[""{name=3, anchor=center, inner sep=0}, "\shortmid"{marking}, from=2-2, to=2-5]
	\arrow[from=2-2, to=3-2]
	\arrow[from=2-5, to=3-5]
	\arrow[""{name=4, anchor=center, inner sep=0}, "\shortmid"{marking}, from=3-1, to=3-2]
	\arrow[from=3-1, to=4-1]
	\arrow[""{name=5, anchor=center, inner sep=0}, "\shortmid"{marking}, from=3-2, to=3-5]
	\arrow[from=3-5, to=4-5]
	\arrow[""{name=6, anchor=center, inner sep=0}, "\shortmid"{marking}, from=4-1, to=4-5]
	\arrow["{\gamma_1}"{description}, draw=none, from=0, to=2]
	\arrow["{\gamma_2}"{description}, draw=none, from=1, to=3]
	\arrow["{\varepsilon_1}"{description}, draw=none, from=2, to=4]
	\arrow["{\varepsilon_2}"{description}, draw=none, from=3, to=5]
	\arrow["\delta"{description}, draw=none, from=5, to=6]
\end{tikzcd}\]

Let us rephrase. Recall that $\Cat{fc}_n(\Dbl{D}_1)$ denotes the category of paths of loose arrows of length $n$ 
in $\Dbl{D}.$ Let $\Dbl{D}_{1,3}$ denote the profunctor 
\[\MCelln{1}\Dbl{D}'(\varphi_1,\psi_1)\times_{\Tight \Dbl{D}'}
\MCelln{3}\Dbl{D}'(\varphi_2,\varphi_3,\varphi_4;\psi_2):
\Cat{fc}_4(\Dbl{D}_1)\proto \Cat{fc}_2(\Dbl{D}_1)\]

Then the argument above demonstrates the coend formula
\[\MCelln{4}\Dbl{D}'(\prolist{\varphi};\chi) = 
\int^{\prolist{\psi}\in\Cat{fc}_2(\Dbl{D}_1)}
\Dbl{D}_{1,3}(\prolist{\varphi};\prolist{\psi})
\times \MCelln{2}\Dbl{D}'(\prolist{\psi};\chi).\]
In short, the invertibility of $c$ in this case reduces to checking the 
coend formula just above for $\Dbl{D}\times\Dbl{J}.$  

Now, $\Dbl{J}$ has exactly one quaternary multicell $\rho:=\frac{(\alpha_1,\alpha_2)}{\beta},$ so 
$\MCelln{4}(\Dbl{D}\times \Dbl{J})\cong \MCelln{4}(\Dbl{D}).$ By all means we have a map 
\[\pi:\int^{\prolist{\psi}\in\Cat{fc}_2(\Dbl{D}_1)}
\Dbl{D}_{1,3}(\prolist{\varphi};\prolist{\psi})
\times \MCelln{2}\Dbl{D}'(\prolist{\psi};\chi)
\to \MCelln{4}(\Dbl{D}\times \Dbl{J})\cong \MCelln{4}(\Dbl{D})\]
given as for $\Dbl{D}'$ by a composition like $\tfrac{(\gamma_1 ,\gamma_2)}{\delta}.$
The question is when this map $\pi$ is an isomorphism. The surjectivity of $\pi$ amounts to the condition 
that every quaternary multicell in $\Dbl{D}$ may be decomposed as a composite $\frac{(\gamma_1,\gamma_2)}{\delta}$ as above. 
Injectivity says that such a decomposition is essentially unique, up to ``interchanging'' 
unary multicells $\varepsilon_1,\varepsilon_2$ as above.

Since the invertibility of $\pi$ is necessary for exponentiability of $\Dbl{D},$ this already provides two 
concrete ways in which $\Dbl{D}$ can fail to be exponentiable. For instance, if $\Dbl{D}=\Sq{4}$ 
is \emph{freely} generated by a quaternary multicell $\mu$, then $\Dbl{D}$ has no binary or ternary multicells at all, 
so $\mu$ has no decomposition, making $\pi$ non-surjective. For a non-injective example, we could take 
a pushout $\Dbl{J}+_{\partial \Dbl{J}} \Dbl{J}$, modulo the relation identifying the two quaternary multicells.
\end{exmp}

\subsubsection*{A Sketch For VDCs}

We now aim toward applying Corollary \ref{cor:exponentiable-sketch} to 
fully classify the exponentiable
objects in $\Vdc.$ This category can be sketched via a sketch containing objects for the 
objects, tight arrows, loose arrows, and unary multicells, strings of tight arrows and unary multicells 
of lengths up to 3, multicells of each arity, as well as for gluings
of multicells of heights $2$ and $3.$ Let us be more precise.

\begin{defn}[Tree]\label{defn:tree}
By a \emph{tree}, we shall mean a rooted finite planar tree in the sense of Kock (\cite{Kock2011},) 
of uniform height. 

In particular, there is a unique root edge. We think of the root as the output, which determines an interpretation
of words including ``input'', ``output'', ``source'', ``target'', ``parent'', and ``child'' for edges and vertices.
The root edge uniquely has no output vertex, while leaf edges are characterized as having no input vertex. 
The input edges of a vertex are equipped with a total order, and all leaf edges are the same distance from the root.
\qed
\end{defn}
In particular, the smallest tree, $\mid$, is the one with a single edge and no vertices. 

Every height-$n$ tree $T$ determines a VDC $\Dbl{J}(T)$ given by 
gluing $n$ tight layers of multicells, one multicell per node, with the arity of the multicell
corresponding to each node given by the node's number of children, and the multicells corresponding 
to adjacent sibling nodes glued along their tight edges, while parents' multicells are glued to their children
along a loose edge. The resulting gluing 
has a single maximal multicell of arity the total number of leaves in the tree. An 
example with $n=3$ and $10,8,$ and $4$ edges in the respective layers
is depicted below. 

\[\begin{tikzcd}
	{\color{white}\bullet} & {\color{white}\bullet} & {\color{white}\bullet} & {\color{white}\bullet} & {\color{white}\bullet} & {\color{white}\bullet} & {\color{white}\bullet} & {\color{white}\bullet} & {\color{white}\bullet} & {\color{white}\bullet} \\
	\bullet &&& \bullet & \bullet & \bullet & \bullet & \bullet & \bullet & \bullet \\
	& \bullet && \bullet && \bullet && \bullet \\
	&&&& \bullet \\
	&&&& {\color{white}\square}
	\arrow[no head, from=2-1, to=1-1]
	\arrow[no head, from=2-1, to=1-2]
	\arrow[no head, from=2-4, to=1-3]
	\arrow[no head, from=2-6, to=1-4]
	\arrow[no head, from=2-6, to=1-5]
	\arrow[no head, from=2-8, to=1-6]
	\arrow[no head, from=2-9, to=1-7]
	\arrow[no head, from=2-9, to=1-8]
	\arrow[no head, from=2-9, to=1-9]
	\arrow[no head, from=2-10, to=1-10]
	\arrow[no head, from=3-2, to=2-1]
	\arrow[no head, from=3-6, to=2-4]
	\arrow[no head, from=3-6, to=2-5]
	\arrow[no head, from=3-6, to=2-6]
	\arrow[no head, from=3-6, to=2-7]
	\arrow[no head, from=3-6, to=2-8]
	\arrow[no head, from=3-8, to=2-9]
	\arrow[no head, from=3-8, to=2-10]
	\arrow[no head, from=4-5, to=3-2]
	\arrow[no head, from=4-5, to=3-4]
	\arrow[no head, from=4-5, to=3-6]
	\arrow[no head, from=4-5, to=3-8]
	\arrow[no head, from=4-5, to=5-5]
\end{tikzcd}\]
\[\begin{tikzcd}[column sep=small]
	\bullet & \bullet & \bullet & \bullet & \bullet & \bullet & \bullet & \bullet & \bullet & \bullet & \bullet \\
	\bullet & \bullet & \bullet & \bullet & \bullet & \bullet & \bullet &&& \bullet & \bullet \\
	\bullet & \bullet & \bullet &&&& \bullet &&&& \bullet \\
	\bullet &&&&&&&&&& \bullet
	\arrow["\shortmid"{marking}, from=1-1, to=1-2]
	\arrow[from=1-1, to=2-1]
	\arrow["\shortmid"{marking}, from=1-2, to=1-3]
	\arrow["\shortmid"{marking}, from=1-3, to=1-4]
	\arrow[from=1-3, to=2-2]
	\arrow["\shortmid"{marking}, from=1-4, to=1-5]
	\arrow[from=1-4, to=2-3]
	\arrow[from=1-4, to=2-4]
	\arrow["\shortmid"{marking}, from=1-5, to=1-6]
	\arrow["\shortmid"{marking}, from=1-6, to=1-7]
	\arrow[from=1-6, to=2-5]
	\arrow[from=1-6, to=2-6]
	\arrow["\shortmid"{marking}, from=1-7, to=1-8]
	\arrow[from=1-7, to=2-7]
	\arrow["\shortmid"{marking}, from=1-8, to=1-9]
	\arrow["\shortmid"{marking}, from=1-9, to=1-10]
	\arrow["\shortmid"{marking}, from=1-10, to=1-11]
	\arrow[from=1-10, to=2-10]
	\arrow[from=1-11, to=2-11]
	\arrow["\shortmid"{marking}, from=2-1, to=2-2]
	\arrow[from=2-1, to=3-1]
	\arrow["\shortmid"{marking}, from=2-2, to=2-3]
	\arrow[from=2-2, to=3-2]
	\arrow[from=2-2, to=3-3]
	\arrow["\shortmid"{marking}, from=2-3, to=2-4]
	\arrow["\shortmid"{marking}, from=2-4, to=2-5]
	\arrow["\shortmid"{marking}, from=2-5, to=2-6]
	\arrow["\shortmid"{marking}, from=2-6, to=2-7]
	\arrow["\shortmid"{marking}, from=2-7, to=2-10]
	\arrow[from=2-7, to=3-7]
	\arrow["\shortmid"{marking}, from=2-10, to=2-11]
	\arrow[from=2-11, to=3-11]
	\arrow["\shortmid"{marking}, from=3-1, to=3-2]
	\arrow[from=3-1, to=4-1]
	\arrow["\shortmid"{marking}, from=3-2, to=3-3]
	\arrow["\shortmid"{marking}, from=3-3, to=3-7]
	\arrow["\shortmid"{marking}, from=3-7, to=3-11]
	\arrow[from=3-11, to=4-11]
	\arrow["\shortmid"{marking}, from=4-1, to=4-11]
\end{tikzcd}\]

Precisely:
\begin{defn}\label{defn:J}
  Let $\Cat{PTEmb}$ be the category of  trees and embeddings preserving arities 
  and planar orderings (i.e. $\Cat{TEmb}/\mathrm{List}$ in the notation of \cite{Kock2011}.) 

  If $T$ is a tree of height $n,$ we define a diagram $C(T):J(T)\to \Vdc$
  as follows: 
  \begin{itemize}
  \item For each edge $e$ of $T,$ $J(T)$ has an object $e$ for which $C(T)(e)=\Loose$ 
  is the walking loose arrow.
  \item For each node $v$ of $T$ with $k(v)$ children (note that $k(v)$ may be $0$), $J(T)$ has an object $v$ 
  for which $C(T)(v)=\uSq{k(v)}$ is the walking $k$-ary multicell.
  \item If an edge $e$ has a parent vertex $s(e),$ then $J(T)$ contains a morphism 
  $e\to s(e)$ sent by $C(T)$ to the inclusion $\Loose \to \uSq{k(s(e))}$ of the loose target, 
  while if $e$ is itself the $i$th parent of the vertex $t(e)$, 
  then $J(T)$ contains a morphism $e\to t(e)$ sent by $C(T)$ to the inclusion 
  $\Loose \to \uSq{k(t(e))}$ of the $i$th loose source.
  \item The total order on each node's input edges induces a total order on 
  all the nodes of each height. If a node $v$ is the successor of $w$ in this order,
  then $J(T)$ contains a span $w\leftarrow wv \to v$ sent by $C(T)$ to the 
  span $\uSq{k(w)}\leftarrow \Tight \to \uSq{k(v)}$ including $\Tight$ as the 
  tight target at $w$ and the tight source at $v.$
  \end{itemize}
  Then we define the VDC $\Dbl{J}(T):=\colim C(T)$ in $\Vdc.$ 
  Furthermore, $C(T)$ is evidently natural in $T,$ so that $\Dbl{J}$ becomes a functor $\Cat{PTEmb}\to \Vdc.$
\qed
\end{defn}

Note in particular that for $\mid$ the tree with no nodes, then $\Dbl{J}(\mid)=\Loose.$
These virtual double categories include almost all those needed to define our desired sketch
for $\Vdc.$

\begin{defn}\label{defn:virtCubeSite}
	We define the \emph{virtual cubical site} $\square_{v\partial}$ (in analogy with the cubical site in~\cite{GrandisLuca2003}) as the full subcategory of $\Vdc$ 
  spanned by the walking object $\Ob$, the walking tight arrow $\Tight$, the pushouts 
  $\Tight[n]:=\underbrace{\Tight\cup_{\Ob}\cdots\cup_{\Ob} \Tight}_n$, and $\Dbl{J}(T)$ 
  for each tree $T.$
\qed
\end{defn}

Let $\Cat{inrt}\subseteq \square_{v\partial}$ denote the wide subcategory generated by 
morphisms that in the image of $\Dbl{J}$ together with morphisms among the $\Tight[n]$ 
which are inert in the sense of the simplicial category, that is, 
sending generating arrows to generating arrows. Thus all the inert morphisms are characterized 
by refusing to map a generating cell to a nontrivial composite.

If $\Cat{el}\subseteq \Cat{inrt}$ is the full subcategory spanned by $\Ob$, $\Loose$, $\Tight$, and 
$\Sq{n}$, then every object of the virtual cubical site is canonically glued from elements of 
$\Cat{el},$ by definition. More explicitly, for each 
$\Dbl{D}\in \square_{v\partial}$ we have a natural colimit cocone $C(\Dbl{D})^\rhd$ under the diagram 
of all inert maps from elementaries to $\Dbl{D}$, i.e. the diagram
\begin{equation*}
	C(\Dbl{D}):\Cat{el}\downarrow \Dbl{D} \to \Cat{el}\hookrightarrow \square_{v\partial}
\end{equation*}
with colimit $\Dbl{D}$. This recovers the defining diagram $C(T)=C(\Dbl{J}(T))$ for trees $T$ while also covering 
the cases of the $\Tight[n].$
(Here $\Cat{el}\downarrow \Dbl{D}$ is the comma category for the inclusions $\Cat{el}\hookrightarrow \Cat{inrt}\hookleftarrow \{\Dbl{D}\}$.)

\begin{defn}[Sketch for VDCs]\label{defn:sketchVDC}
The sketch $(\Cat{C},\Cat{S})$ for VDCs has 
$\Cat{C}^{\op}\subseteq \square_{v\partial}$ as the full subcategory 
spanned by the following objects: 
\begin{enumerate}
	\item $O:=\Ob,\downarrow:=\Tight, \triangle:=\downarrow\cup_O\downarrow$, and 
  $\tet:=\downarrow\cup_O\downarrow\cup_O\downarrow$, for the tight category;
	\item $A_T:=\Dbl{J}(T)$ for every  tree $T$ of height $\leq 3.$\footnote{If $T$ is a tree of height $1$ it is uniquely determined by its single branching factor $n,$ in which case 
  $A_n:=A_T$ will represent the $n$-ary multicells of a VDC.}
\end{enumerate}
The sketched cones in $\Cat{S}$ are precisely the opposites of the cocones 
$C(\Dbl{D})^\rhd$ in $\Cat{C}^{op}$ for each $\Dbl{D}\in \Cat{C}^{op}\backslash \Cat{el}$. 
Informally, these are the cones making the objects and tight arrows, as well as the loose 
arrows and unary multicells, form a category, 
along with those forcing a model to admit a composition of layers of multicells of height 2
and those forcing that composition to be associative.
\end{defn}

Now consider exponentiability. Taking products with any virtual double
category preserves the colimits for paths of tight arrows, since every category is exponentiable. 
Thus the exponentiable virtual double categories will be precisely those that preserve all the colimits
freely forming the pasting virtual double categories $\Dbl{J}(T)$ for every  tree $T$
of height $2$ or $3$. As illustrated in Example \ref{exmp:decomp}, preservation 
of such colimits is equivalent to existence of essentially unique decompositions of multicells,
as we now turn to show.

\subsection{Reduction of height of decompositions}

In order to phrase our characterizations of exponentiable VDCs we will need some notation and terminology 
for decompositions of multicells in VDCs. Recall from Section~\ref{sec:VDCs} we have the profunctor 
$\MCell (\Dbl{D}):\Cat{fc}(\Dbl{D}_1)^{op}\times \Dbl{D}_1\to \Cat{Set}$ of multicells for a VDC $\Dbl{D}$, 
with associated category obtained by the Grothendieck construction. 
Viewing this profunctor as a category via the collage, we can define the category of height $n$ pasting diagrams 
in $\Dbl{D}$ as the iterated pullback
\begin{equation*}
	\Cat{Paste}_n(\Dbl{D}):=\Cat{fc}^{n-1}(\MCell (\Dbl{D}))\times_{\Cat{fc}^{n-1}(\Dbl{D}_1)}\cdots \times_{\Cat{fc}(\Dbl{D}_1)}\MCell (\Dbl{D})
\end{equation*}
We will notate a pasting diagram, viewed as an object in $\Cat{Paste}_n(\Dbl{D})$, by 
${\threefrac{\prolist{\alpha_n}}{\vdots}{\alpha_1}}$, and for each $0\leq i \leq n$ we refer to the 
sequence of sequences of loose arrows $\prolist{\varphi_{i,1}} \cdots \prolist{\varphi_{i,m}}$ which are the 
sources of the $\prolist{\alpha_i}$ as the $i$th \emph{interface} of the pasting diagram. 

The category $\Cat{Paste}_n(\Dbl{D})$ is the generalization of the category $\Cat{Fun}([n],\Cat{C})$ of 
length $n$-sequences of composable arrows in a category $\Cat{C}$ to the case of VDCs. 
Note that for each $1\leq i \leq n-1$, we have a natural functor 
$c_i:\Cat{Paste}_n(\Dbl{D})\to \Cat{Paste}_{n-1}(\Dbl{D})$ given by composing $i$th and $(i+1)$st layers of the 
pasting diagram.
We will write $c:\Cat{Paste}_n(\Dbl{D})\to \Cat{Paste}_1(\Dbl{D})=\MCell (\Dbl{D})$ for the functor given by composing the whole pasting diagram. For a tree $T$ of height $n$, indexing a shape of pasting diagrams in $\Cat{Paste}_n(\Dbl{D})$, we will write $c_i(T)$ for the height $(n-1)$ tree obtained by composing along the $i$th layer.

For an $m$-ary multicell $\alpha$ in $\Dbl{D}$, we write 
$\Cat{Paste}_n(\Dbl{D};\alpha):=c^{-1}(\{\alpha\})\subseteq \Cat{Paste}_n(\Dbl{D})$ for the 
(non-full) subcategory of height-$n$ pasting diagrams which compose to $\alpha$, with morphisms those which are 
identities on the loose source and target of $\alpha$. We will say that $\alpha$ 
\emph{admits essentially unique decompositions} if for each integer $n\geq 2$, the set of connected 
components $\pi_0\Cat{Paste}_n(\Dbl{D};\alpha)$ is a singleton for each height $n$ tree $T$ with $m$ 
leaves, corresponding to a pasting diagram of shape $T$ lifting $\alpha$. 
Here two decompositions of $\alpha$ in $\Cat{Paste}_n(\Dbl{D};\alpha)$ will be said to be 
\emph{equivalent up to associativity} if they lie in the same connected component. 

Our first characterization of exponentiable VDCs will be those admitting 
essentially unique decompositions for all multicells. 

\begin{defn}[VDCs with Decomposable multicells]\label{defn:DecompCells}
	We will say a VDC $\Dbl{D}$ has \emph{decomposable multicells} if every multicell $\alpha$ in $\Dbl{D}$ 
  admits essentially unique decompositions.
\qed
\end{defn}

In practice showing that a VDC has decomposable multicells directly from the definition is unreasonable. 
Thus, in the process of characterizing exponentiability, we will aim to discover apparently much weaker 
conditions which are in fact sufficient for existence and essential uniqueness of decompositions. 
Let's first show we can reduce from considering decompositions into arbitrarily many layers to the case of 
decompositions into two layers. Explication \ref{expl:binaryToTernary} may help fix 
the basic idea in the reader's mind before reading the general proof of the next two
results.

\begin{lem}\label{lem:components}
	Let $\Dbl{D}$ be a VDC whose multicells admit essentially unique decompositions into $2$ layers, 
  let $n\geq 2$ be an integer, and let $\alpha$ be a multicell in $\Dbl{D}$. 
  If two $(n+1)$-layer decompositions of $\alpha$ over the same tree map to the same connected component 
  under the functor $c_1:\Cat{Paste}_{n+1}(\Dbl{D};\alpha)\to \Cat{Paste}_n(\Dbl{D};\alpha)$, 
  then they are in the same connected component in $\Cat{Paste}_{n+1}(\Dbl{D};\alpha)$.
\end{lem}
\begin{proof}
	Let $\beta={\threefrac{\prolist{\beta_{n+1}}}{\vdots}{\beta_1}}$ and 
  $\gamma={\threefrac{\prolist{\gamma_{n+1}}}{\vdots}{\gamma_1}}$ be $(n+1)$-layer decompositions 
  of $\alpha$ with the same arities that map to the same connected component under $c_1$. 
  We wish to show that the zig-zag of morphisms witnessing that $c_1(\beta)$ and $c_1(\gamma)$ are in the same 
  connected component lifts to a possibly differently shaped zig-zag of morphisms witnessing 
  that $\beta$ and $\gamma$ are in the same connected component. 
  
  Note that we can consider a zig-zag such that 
  each morphism consists of a sequence of identity multicells on all interfaces but one. 
  Then all morphisms which are non-trivial on an interface other than the $1$st interface lift to 
  $\Cat{Paste}_{n+1}(\Dbl{D};\alpha)$. Thus, it suffices to consider the case where the morphisms in the 
  zig-zag consist of identity multicells on all interfaces but the first interface. In particular, this 
  observation allows us to reduce to the case where $n=2$. 
  
  Proceeding inductively on the length of the chain 
  of zig-zags, we can further reduce to the case where the zig-zag consists of a single morphism. 
  Without loss of generality suppose the morphism is from $c_1(\beta)$ to $c_1(\gamma)$. Then we have the map 
  of pasting diagrams of multicells
	\[
	\begin{tikzcd}
	& {x_{3,k_{3,m}}} && {x_{3,k_{3,m}}} \\
	{x_{3,0}} && {x_{3,0}} \\
	& {x_{2,k_{2,m}}} && {y_{2,k_{2,m}}} \\
	{x_{2,0}} && {y_{2,0}} \\
	& {x_{1,m}} && {y_{1,m}} \\
	{x_{1,0}} && {y_{1,0}} \\
	& {x_{0,1}} && {x_{0,1}} \\
	{x_{0,0}} && {x_{0,0}}
	\arrow[""{name=0, anchor=center, inner sep=0}, equals, from=1-2, to=1-4]
	\arrow[dashed, from=1-2, to=3-2]
	\arrow[from=1-4, to=3-4]
	\arrow[""{name=1, anchor=center, inner sep=0}, "\shortmid"{marking}, from=2-1, to=1-2]
	\arrow[""{name=2, anchor=center, inner sep=0}, equals, from=2-1, to=2-3]
	\arrow[from=2-1, to=4-1]
	\arrow[""{name=3, anchor=center, inner sep=0}, "\shortmid"{marking}, from=2-3, to=1-4]
	\arrow[from=2-3, to=4-3]
	\arrow[""{name=4, anchor=center, inner sep=0}, from=3-2, to=3-4]
	\arrow[dashed, from=3-2, to=5-2]
	\arrow[from=3-4, to=5-4]
	\arrow[""{name=5, anchor=center, inner sep=0}, "\shortmid"{marking}, from=4-1, to=3-2]
	\arrow[""{name=6, anchor=center, inner sep=0}, from=4-1, to=4-3]
	\arrow[from=4-1, to=6-1]
	\arrow[""{name=7, anchor=center, inner sep=0}, "\shortmid"{marking}, from=4-3, to=3-4]
	\arrow[from=4-3, to=6-3]
	\arrow[from=5-2, to=7-2]
	\arrow[from=5-4, to=7-4]
	\arrow[""{name=8, anchor=center, inner sep=0}, "\shortmid"{marking}, from=6-1, to=5-2]
	\arrow[from=6-1, to=8-1]
	\arrow[""{name=9, anchor=center, inner sep=0}, "\shortmid"{marking}, from=6-3, to=5-4]
	\arrow[from=6-3, to=8-3]
	\arrow[""{name=10, anchor=center, inner sep=0}, equals, from=7-2, to=7-4]
	\arrow[""{name=11, anchor=center, inner sep=0}, "\shortmid"{marking}, from=8-1, to=7-2]
	\arrow[""{name=12, anchor=center, inner sep=0}, equals, from=8-1, to=8-3]
	\arrow[""{name=13, anchor=center, inner sep=0}, "\shortmid"{marking}, from=8-3, to=7-4]
	\arrow["{{\text{id}}}"{description}, draw=none, from=0, to=2]
	\arrow["{\prolist{{\beta_3}}}"{description}, draw=none, from=1, to=5]
	\arrow["{\prolist{{\gamma_3}}}"{description}, draw=none, from=3, to=7]
	\arrow["{\prolist{{\sigma}}}"{description}, draw=none, from=4, to=6]
	\arrow["{\prolist{\beta_2}}"{description}, draw=none, from=5, to=8]
	\arrow["{\prolist{\gamma_2}}"{description}, draw=none, from=7, to=9]
	\arrow["{{\beta_1}}"{description}, draw=none, from=8, to=11]
	\arrow["{{\gamma_1}}"{description}, draw=none, from=9, to=13]
	\arrow["{{\text{id}}}"{description}, draw=none, from=10, to=12]
\end{tikzcd}
	\]
	where the bottom rectangular box only commutes after composing the multicells. But then if $\prolist{\rho}:=\frac{\prolist{{\sigma}}}{\prolist{\gamma_2}}$, it follows that $\frac{\prolist{\beta_2}}{\beta_1}$ and $\frac{\prolist{\rho}}{\gamma_1}$ are 2-layer decompositions of the same multicell $\chi$ in $\Dbl{D}$, and hence by assumption lie in the same connected component of $\Cat{Paste}_2(\Dbl{D};\chi)$. Once again, we can reduce to the case where the decompositions being in the same connected component is witnessed by a single morphism from one to the other. If we have such a morphism $\prolist{\sigma'}$ from $\frac{\prolist{\beta_2}}{\beta_1}$ to $\frac{\prolist{\rho}}{\gamma_1}$, then our previous map of pasting diagrams lifts to an actual map in $\Cat{Paste}_3(\Dbl{D};\alpha)$:
	\[
	\begin{tikzcd}
	& {x_{3,k_{3,m}}} && {x_{3,k_{3,m}}} \\
	{x_{3,0}} && {x_{3,0}} \\
	& {x_{2,k_{2,m}}} && {y_{2,k_{2,m}}} \\
	{x_{2,0}} && {y_{2,0}} \\
	& {x_{1,m}} && {y_{1,m}} \\
	{x_{1,0}} && {y_{1,0}} \\
	& {x_{0,1}} && {x_{0,1}} \\
	{x_{0,0}} && {x_{0,0}}
	\arrow[""{name=0, anchor=center, inner sep=0}, equals, from=1-2, to=1-4]
	\arrow[dashed, from=1-2, to=3-2]
	\arrow[from=1-4, to=3-4]
	\arrow[""{name=1, anchor=center, inner sep=0}, "\shortmid"{marking}, from=2-1, to=1-2]
	\arrow[""{name=2, anchor=center, inner sep=0}, equals, from=2-1, to=2-3]
	\arrow[from=2-1, to=4-1]
	\arrow[""{name=3, anchor=center, inner sep=0}, "\shortmid"{marking}, from=2-3, to=1-4]
	\arrow[from=2-3, to=4-3]
	\arrow[""{name=4, anchor=center, inner sep=0}, from=3-2, to=3-4]
	\arrow[dashed, from=3-2, to=5-2]
	\arrow[from=3-4, to=5-4]
	\arrow[""{name=5, anchor=center, inner sep=0}, "\shortmid"{marking}, from=4-1, to=3-2]
	\arrow[""{name=6, anchor=center, inner sep=0}, from=4-1, to=4-3]
	\arrow[from=4-1, to=6-1]
	\arrow[""{name=7, anchor=center, inner sep=0}, "\shortmid"{marking}, from=4-3, to=3-4]
	\arrow[from=4-3, to=6-3]
	\arrow[""{name=8, anchor=center, inner sep=0}, from=5-2, to=5-4]
	\arrow[from=5-2, to=7-2]
	\arrow[from=5-4, to=7-4]
	\arrow[""{name=9, anchor=center, inner sep=0}, "\shortmid"{marking}, from=6-1, to=5-2]
	\arrow[""{name=10, anchor=center, inner sep=0}, from=6-1, to=6-3]
	\arrow[from=6-1, to=8-1]
	\arrow[""{name=11, anchor=center, inner sep=0}, "\shortmid"{marking}, from=6-3, to=5-4]
	\arrow[from=6-3, to=8-3]
	\arrow[""{name=12, anchor=center, inner sep=0}, equals, from=7-2, to=7-4]
	\arrow[""{name=13, anchor=center, inner sep=0}, "\shortmid"{marking}, from=8-1, to=7-2]
	\arrow[""{name=14, anchor=center, inner sep=0}, equals, from=8-1, to=8-3]
	\arrow[""{name=15, anchor=center, inner sep=0}, "\shortmid"{marking}, from=8-3, to=7-4]
	\arrow["{{\text{id}}}"{description}, draw=none, from=0, to=2]
	\arrow["{\prolist{{\beta_3}}}"{description}, draw=none, from=1, to=5]
	\arrow["{\prolist{{\gamma_3}}}"{description}, draw=none, from=3, to=7]
	\arrow["{\prolist{{\sigma}}}"{description}, draw=none, from=4, to=6]
	\arrow["{\prolist{\beta_2}}"{description}, draw=none, from=5, to=9]
	\arrow["{\prolist{\gamma_2}}"{description}, draw=none, from=7, to=11]
	\arrow["{\prolist{\sigma'}}"{description}, draw=none, from=8, to=10]
	\arrow["{{\beta_1}}"{description}, draw=none, from=9, to=13]
	\arrow["{{\gamma_1}}"{description}, draw=none, from=11, to=15]
	\arrow["{{\text{id}}}"{description}, draw=none, from=12, to=14]
\end{tikzcd}
	\]
	Conversely, if the morphism $\prolist{\sigma'}$ goes from $\frac{\prolist{\rho}}{\gamma_1}$ to $\frac{\prolist{\beta_2}}{\beta_1}$, then we have the following zig-zag of maps in $\Cat{Paste}_3(\Dbl{D};\alpha)$:
	\[
	\begin{tikzcd}
	& {x_{3,k_{3,m}}} && {x_{3,k_{3,m}}} && {x_{3,k_{3,m}}} \\
	{x_{3,0}} && {x_{3,0}} && {x_{3,0}} \\
	& {x_{2,k_{2,m}}} && {x_{2,k_{2,m}}} && {y_{2,k_{2,m}}} \\
	{x_{2,0}} && {x_{2,0}} && {y_{2,0}} \\
	& {x_{1,m}} && {y_{1,m}} && {y_{1,m}} \\
	{x_{1,0}} && {y_{1,0}} && {y_{1,0}} \\
	& {x_{0,1}} && {x_{0,1}} && {x_{0,1}} \\
	{x_{0,0}} && {x_{0,0}} && {x_{0,0}}
	\arrow[""{name=0, anchor=center, inner sep=0}, equals, from=1-2, to=1-4]
	\arrow[dashed, from=1-2, to=3-2]
	\arrow[""{name=1, anchor=center, inner sep=0}, equals, from=1-4, to=1-6]
	\arrow[from=1-4, to=3-4]
	\arrow[from=1-6, to=3-6]
	\arrow[""{name=2, anchor=center, inner sep=0}, "\shortmid"{marking}, from=2-1, to=1-2]
	\arrow[""{name=3, anchor=center, inner sep=0}, equals, from=2-1, to=2-3]
	\arrow[from=2-1, to=4-1]
	\arrow[""{name=4, anchor=center, inner sep=0}, "\shortmid"{marking}, from=2-3, to=1-4]
	\arrow[""{name=5, anchor=center, inner sep=0}, equals, from=2-3, to=2-5]
	\arrow[from=2-3, to=4-3]
	\arrow[""{name=6, anchor=center, inner sep=0}, "\shortmid"{marking}, from=2-5, to=1-6]
	\arrow[from=2-5, to=4-5]
	\arrow[""{name=7, anchor=center, inner sep=0}, equals, from=3-2, to=3-4]
	\arrow[dashed, from=3-2, to=5-2]
	\arrow[""{name=8, anchor=center, inner sep=0}, from=3-4, to=3-6]
	\arrow[from=3-4, to=5-4]
	\arrow[from=3-6, to=5-6]
	\arrow[""{name=9, anchor=center, inner sep=0}, "\shortmid"{marking}, from=4-1, to=3-2]
	\arrow[""{name=10, anchor=center, inner sep=0}, equals, from=4-1, to=4-3]
	\arrow[from=4-1, to=6-1]
	\arrow[""{name=11, anchor=center, inner sep=0}, "\shortmid"{marking}, from=4-3, to=3-4]
	\arrow[""{name=12, anchor=center, inner sep=0}, from=4-3, to=4-5]
	\arrow[from=4-3, to=6-3]
	\arrow[""{name=13, anchor=center, inner sep=0}, "\shortmid"{marking}, from=4-5, to=3-6]
	\arrow[from=4-5, to=6-5]
	\arrow[from=5-2, to=7-2]
	\arrow[""{name=14, anchor=center, inner sep=0}, from=5-4, to=5-2]
	\arrow[equals, from=5-4, to=5-6]
	\arrow[from=5-4, to=7-4]
	\arrow[from=5-6, to=7-6]
	\arrow[""{name=15, anchor=center, inner sep=0}, "\shortmid"{marking}, from=6-1, to=5-2]
	\arrow[from=6-1, to=8-1]
	\arrow[""{name=16, anchor=center, inner sep=0}, "\shortmid"{marking}, from=6-3, to=5-4]
	\arrow[""{name=17, anchor=center, inner sep=0}, from=6-3, to=6-1]
	\arrow[equals, from=6-3, to=6-5]
	\arrow[from=6-3, to=8-3]
	\arrow[""{name=18, anchor=center, inner sep=0}, "\shortmid"{marking}, from=6-5, to=5-6]
	\arrow[from=6-5, to=8-5]
	\arrow[""{name=19, anchor=center, inner sep=0}, equals, from=7-2, to=7-4]
	\arrow[""{name=20, anchor=center, inner sep=0}, equals, from=7-4, to=7-6]
	\arrow[""{name=21, anchor=center, inner sep=0}, "\shortmid"{marking}, from=8-1, to=7-2]
	\arrow[""{name=22, anchor=center, inner sep=0}, equals, from=8-1, to=8-3]
	\arrow[""{name=23, anchor=center, inner sep=0}, "\shortmid"{marking}, from=8-3, to=7-4]
	\arrow[""{name=24, anchor=center, inner sep=0}, equals, from=8-3, to=8-5]
	\arrow[""{name=25, anchor=center, inner sep=0}, "\shortmid"{marking}, from=8-5, to=7-6]
	\arrow["{{\text{id}}}"{description}, draw=none, from=0, to=3]
	\arrow["{{\text{id}}}"{description}, draw=none, from=1, to=5]
	\arrow["{\prolist{{\beta_3}}}"{description}, draw=none, from=2, to=9]
	\arrow["{\prolist{{\beta_3}}}"{description}, draw=none, from=4, to=11]
	\arrow["{\prolist{{\gamma_3}}}"{description}, draw=none, from=6, to=13]
	\arrow["{{\text{id}}}"{description}, draw=none, from=7, to=10]
	\arrow["{\prolist{{\sigma}}}"{description}, draw=none, from=8, to=12]
	\arrow["{\prolist{\beta_2}}"{description}, draw=none, from=9, to=15]
	\arrow["{{\dfrac{\prolist{{\sigma}}}{\prolist{\gamma_2}}}}"{description}, draw=none, from=11, to=16]
	\arrow["{\prolist{\gamma_2}}"{description}, draw=none, from=13, to=18]
	\arrow["{\prolist{\sigma'}}"{description}, draw=none, from=14, to=17]
	\arrow["{{\beta_1}}"{description}, draw=none, from=15, to=21]
	\arrow["{\text{id}}"{description}, draw=none, from=16, to=18]
	\arrow["{{\gamma_1}}"{description}, draw=none, from=16, to=23]
	\arrow["{{\gamma_1}}"{description}, draw=none, from=18, to=25]
	\arrow["{{\text{id}}}"{description}, draw=none, from=19, to=22]
	\arrow["{{\text{id}}}"{description}, draw=none, from=20, to=24]
\end{tikzcd}
	\]
	In either case, we see that the decompositions are in the same connected component, completing the proof.
\end{proof}

\begin{cor}\label{cor:decompSimp}
	Let $\Dbl{D}$ be a VDC. Then $\Dbl{D}$ has decomposable multicells if and only if all multicells in 
  $\Dbl{D}$ admit essentially unique decompositions into $2$ layers.
\end{cor}
\begin{proof}
	We will prove that a multicell $\alpha$, of arity $m$, admits essentially unique decompositions into $n$-layers for each $n\geq 2$ by induction on $n$, where the base case $n=2$ is by assumption. Thus, suppose $\alpha$ admits essentially unique decompositions into $n\geq 2$ layers, and we aim to show it also admits essentially unique decompositions into $(n+1)$-layers. First, to show existence of such decompositions, let $T$ be a tree of height $(n+1)$ with $m$ leaves. Then $c_1(T)$ is a height $n$ tree with $|\alpha|$ leaves, so by the inductive hypothesis we have an essentially unique decomposition

	\[
	\begin{tikzcd}
	{x_{n+1,0}} && {x_{n+1,m}} \\
	{x_{n,0}} && {x_{n,k_n}} \\
	\vdots & \vdots & \vdots \\
	{x_{2,0}} && {x_{2,k_2}} \\
	{x_{0,0}} && {x_{0,1}}
	\arrow[""{name=0, anchor=center, inner sep=0}, "{{\prolist{{\varphi}}}}"{inner sep=.8ex}, "\shortmid"{marking}, from=1-1, to=1-3]
	\arrow[from=1-1, to=2-1]
	\arrow[from=1-3, to=2-3]
	\arrow[""{name=1, anchor=center, inner sep=0}, "{{\prolist{\lambda_n}}}"'{inner sep=.8ex}, "\shortmid"{marking}, from=2-1, to=2-3]
	\arrow[from=2-1, to=3-1]
	\arrow[from=2-3, to=3-3]
	\arrow[from=3-1, to=4-1]
	\arrow[from=3-3, to=4-3]
	\arrow[""{name=2, anchor=center, inner sep=0}, "{{\prolist{\lambda_2}}}"{inner sep=.8ex}, "\shortmid"{marking}, from=4-1, to=4-3]
	\arrow[from=4-1, to=5-1]
	\arrow[from=4-3, to=5-3]
	\arrow[""{name=3, anchor=center, inner sep=0}, "\psi"'{inner sep=.8ex}, "\shortmid"{marking}, from=5-1, to=5-3]
	\arrow["{{\prolist{\beta_{n+1}}}}"{description}, draw=none, from=0, to=1]
	\arrow["{{\beta_2}}"{description}, draw=none, from=2, to=3]
\end{tikzcd}
	\]
	By the base case of $n=2$ we have an essentially unique decomposition of $\beta_2=\frac{\prolist{\alpha_2}}{\alpha_1}$ using the arities from the root and first layer of the tree $T$, which gives the decomposition
	\[
	\begin{tikzcd}
	{x_{n+1,0}} && {x_{n+1,m}} \\
	{x_{n,0}} && {x_{n,k_n}} \\
	\vdots & \vdots & \vdots \\
	{x_{2,0}} && {x_{2,k_2}} \\
	{x_{1,0}} && {x_{1,k_1}} \\
	{x_{0,0}} && {x_{0,1}}
	\arrow[""{name=0, anchor=center, inner sep=0}, "{{\prolist{{\varphi}}}}"{inner sep=.8ex}, "\shortmid"{marking}, from=1-1, to=1-3]
	\arrow[from=1-1, to=2-1]
	\arrow[from=1-3, to=2-3]
	\arrow[""{name=1, anchor=center, inner sep=0}, "{{\prolist{\lambda_n}}}"'{inner sep=.8ex}, "\shortmid"{marking}, from=2-1, to=2-3]
	\arrow[from=2-1, to=3-1]
	\arrow[from=2-3, to=3-3]
	\arrow[from=3-1, to=4-1]
	\arrow[from=3-3, to=4-3]
	\arrow[""{name=2, anchor=center, inner sep=0}, "{{\prolist{\lambda_2}}}"{inner sep=.8ex}, "\shortmid"{marking}, from=4-1, to=4-3]
	\arrow[from=4-1, to=5-1]
	\arrow[from=4-3, to=5-3]
	\arrow[""{name=3, anchor=center, inner sep=0}, "{\lambda_1}"{description}, "\shortmid"{marking}, from=5-1, to=5-3]
	\arrow[from=5-1, to=6-1]
	\arrow[from=5-3, to=6-3]
	\arrow[""{name=4, anchor=center, inner sep=0}, "\psi"'{inner sep=.8ex}, "\shortmid"{marking}, from=6-1, to=6-3]
	\arrow["{{\prolist{\beta_{n+1}}}}"{description}, draw=none, from=0, to=1]
	\arrow["{\prolist{\alpha_2}}"{description}, draw=none, from=2, to=3]
	\arrow["{\alpha_1}"{description}, draw=none, from=3, to=4]
\end{tikzcd}
	\]
	Now, to prove essential uniqueness suppose we had another decomposition of shape $T$, ${\threefrac{\prolist{\gamma_{n+1}}}{\vdots}{\gamma_1}}$.
	Composing the bottom two layers $\delta_1=\frac{\prolist{\gamma_2}}{\gamma_1}$, we obtain an $n$-layer decomposition of $\alpha$ of shape $c_1(T)$, so by the inductive hypothesis it lies in the same connected component as the decomposition ${\threefrac{\prolist{\beta_{n+1}}}{\vdots}{\beta_2}}$. Then by Lemma~\ref{lem:components} these two decompositions lie in the same connected component of $\Cat{Paste}_{n+1}(\Dbl{D};\alpha)$, completing the proof.
\end{proof}

\begin{expl}[Binary-to-Ternary Case]\label{expl:binaryToTernary}
	To explicate the central idea of the above lemma-corollary pair, let's consider the inductive step going from 2-layer decompositions to 3-layer decompositions in the case of a pair of binary decompositions. Suppose we had a quaternary multicell that we can essentially uniquely decompose in 2-layers, and we wanted to decompose into three layers using binary multicells as in the pasting diagram:
	\begin{equation}\label{eq:3LayerDecomp}
	\begin{adjustbox}{}
		\begin{tikzcd}
	\bullet & \bullet & \bullet & \bullet & \bullet \\
	\bullet && \bullet & \bullet & \bullet \\
	\bullet &&& \bullet & \bullet \\
	\bullet &&&& \bullet
	\arrow["\shortmid"{marking}, from=1-1, to=1-2]
	\arrow[from=1-1, to=2-1]
	\arrow["\shortmid"{marking}, from=1-2, to=1-3]
	\arrow["\shortmid"{marking}, from=1-3, to=1-4]
	\arrow[from=1-3, to=2-3]
	\arrow["\shortmid"{marking}, from=1-4, to=1-5]
	\arrow[from=1-4, to=2-4]
	\arrow[from=1-5, to=2-5]
	\arrow["\shortmid"{marking}, from=2-1, to=2-3]
	\arrow[from=2-1, to=3-1]
	\arrow["\shortmid"{marking}, from=2-3, to=2-4]
	\arrow["\shortmid"{marking}, from=2-4, to=2-5]
	\arrow[from=2-4, to=3-4]
	\arrow[from=2-5, to=3-5]
	\arrow["\shortmid"{marking}, from=3-1, to=3-4]
	\arrow[from=3-1, to=4-1]
	\arrow["\shortmid"{marking}, from=3-4, to=3-5]
	\arrow[from=3-5, to=4-5]
	\arrow["\shortmid"{marking}, from=4-1, to=4-5]
\end{tikzcd}
	\end{adjustbox}
	\end{equation}
	To construct such a decomposition we first construct a 2-layer decomposition of the shape obtained by composing the bottom two layers in our desired decomposition:
	\[
	\begin{adjustbox}{}
		\begin{tikzcd}
	\bullet & \bullet & \bullet & \bullet & \bullet \\
	\bullet && \bullet & \bullet & \bullet \\
	\bullet &&&& \bullet
	\arrow["\shortmid"{marking}, from=1-1, to=1-2]
	\arrow[from=1-1, to=2-1]
	\arrow["\shortmid"{marking}, from=1-2, to=1-3]
	\arrow["\shortmid"{marking}, from=1-3, to=1-4]
	\arrow[from=1-3, to=2-3]
	\arrow["\shortmid"{marking}, from=1-4, to=1-5]
	\arrow[from=1-4, to=2-4]
	\arrow[from=1-5, to=2-5]
	\arrow["\shortmid"{marking}, from=2-1, to=2-3]
	\arrow[from=2-1, to=3-1]
	\arrow["\shortmid"{marking}, from=2-3, to=2-4]
	\arrow["\shortmid"{marking}, from=2-4, to=2-5]
	\arrow[from=2-5, to=3-5]
	\arrow["\shortmid"{marking}, from=3-1, to=3-5]
\end{tikzcd}
	\end{adjustbox}
	\]
	Then, to obtain a 3-layer decomposition we use the fact that we can decompose the bottom trinary multicell in terms of a 2-layer decomposition with a binary multicell, giving the originally desired decomposition in Eq.~\ref{eq:3LayerDecomp}. Lemma~\ref{lem:components} then guarantees essential uniqueness of this 3-layer decomposition since we have essential uniqueness on the level of 2-layers.
\end{expl}

The 2-layer formulation of having decomposable multicells has an elegant description in terms of coends. 
Namely, a VDC $\Dbl{D}$ has decomposable multicells if and only if the natural transformation 
\begin{equation*}
	\mathsf{fc}(\mathsf{MCell}(\mathbb{D}))\odot \mathsf{MCell}(\mathbb{D})\stackrel{\mathsf{comp}}{\Rightarrow}\mathsf{MCell}(\mathbb{D})(\mu_\mathsf{fc},1):\mathsf{fc}(\mathsf{fc}(\Dbl{D}_1))^{op}\times \Dbl{D}_1\to \mathsf{Set}
\end{equation*}
is an isomorphism, where $\mu_\mathsf{fc}:\mathsf{fc}^2\to \mathsf{fc}$ 
is the multiplication of the free category monad. Thus the codomain profunctor is that of multicells with domain 
a list of lists of loose arrows, rather than a flat list thereof.

In terms of coends, this says that the natural composition map
\begin{equation}\label{eq:CompChar}
	\int^{\prolist{\varphi}\in \Cat{fc}(\Dbl{D}_1)}\mathsf{fc}(\mathsf{MCell}(\mathbb{D}))(-,\prolist{\varphi})\times \mathsf{MCell}(\mathbb{D})(\prolist{\varphi},-)\to 
  \mathsf{MCell}(\mathbb{D})(\mu_\mathsf{fc}(-),-)
\end{equation}
is an isomorphism. This condition is also the cellular analogue of the classical Conduch\'{e} condition for exponentiable functors between categories. We will explore this point more in Section~\ref{sec:ExponentiabilityOfMorphisms}.

\subsection{Exponentiability of VDCs is equivalent to decomposability of multicells}

We can now show that having decomposable multicells is equivalent to being exponentiable by showing 
that this is precisely the condition induced by requiring that the product $\Dbl{D}\times -$ 
preserves the sketched cones in our sketch for virtual double categories. 

\begin{rmk}[Colimits for Trees]\label{rmk:cellInColimit}
	Before stating and proving the main theorem, it will be helpful to understand the colimits of the diagrams
	\begin{equation*}
		\Cat{el}\downarrow \Dbl{J}(T)\to \Vdc\xrightarrow{\Dbl{D}\times -}\Vdc
	\end{equation*}
	for $T$ a  tree. 
	We can express $\colim_{i:\Dbl{c}\hookrightarrow \Dbl{J}(T)}(\Dbl{D}\times \Dbl{c})$ as the colimit of a diagram in $\Vdc$ obtained from the tree $T$ as follows:
	\begin{enumerate}
    \item Replace each edge in $T$ by a copy of $\Dbl{D}\times \Loose$.
		\item Replace each node in $T$ with $r$ children by $\Dbl{D}\times \Sq{r}$.
		\item Glue the VDC's corresponding to an edge appropriately into those corresponding to its
		parent and child nodes, if they exist. 
    \item Glue the VDC's corresponding to adjacent nodes in $T$ appropriately.
	\end{enumerate}
	Objects and loose arrows in the colimit are then precisely pairs of such in $\Dbl{D}$ and $\Dbl{J}(T)$, 
  with no relations. Similarly, since the tight category functor $(-)_0:\Vdc\to \Cat{Cat}$ preserves colimits
  and finite products, the tight arrows in the colimit are precisely pairs of tight arrows in 
  $\Dbl{D}$ and $\Dbl{J}(T)$, with no new relations.	

  Finally, a multicell in the colimit corresponds to an equivalence class of tuples of a subtree $T'\subseteq T$ 
  of $T$, together with for each node in $T'$ a multicell in $\Dbl{D}$ of arity equal to that of the 
  node, such that the multicells have compatible boundaries. 
  The equivalence relation on these tuples comes from associativity of composition along with the spans
	\begin{equation*}
		\Dbl{D}\times \Sq{r}\xleftarrow{\Dbl{D}\times s_i}\Dbl{D}\times \Loose \xrightarrow{\Dbl{D}\times t}\Dbl{D}\times \Sq{k_i}
	\end{equation*}
	appearing in the diagram. Namely, if $T'$ is a height $n$ tree, then the data of the multicells in 
  $\Dbl{D}$ gives a pasting diagram of height $n$ (i.e.~an object in $\Cat{Paste}_n(\Dbl{D})$). 
  Then the spans above imply that we can move sequences of unary multicells between rows without changing the 
  equivalence class of the resulting tuple. For example, if $T'$ was a subtree of height $2$ whose root had 
  $r$ children and whose $i$th node in the second layer had $k_i$ children, then the two pasting diagrams 
  below would be identified in the colimit:
    \[\begin{tikzcd}
      {x_0} & {x_{m_1}} & \cdots & {x_{m_r}} & {x_0} & {x_{m_1}} & \cdots & {x_{m_r}} \\
      {z_0} & {z_1} & \cdots & {z_r} & {z_0} & {z_1} & \cdots & {z_r} \\
      \\
      {y_0} &&& {y_1} & {y_0} &&& {y_1}
      \arrow[""{name=0, anchor=center, inner sep=0}, "{{{{{{{\prolist{\varphi}_1}}}}}}}"{inner sep=.8ex}, "\shortmid"{marking}, from=1-1, to=1-2]
      \arrow[from=1-1, to=2-1]
      \arrow[""{name=1, anchor=center, inner sep=0}, "{{{{{{{\prolist{\varphi}_2}}}}}}}"{inner sep=.8ex}, "\shortmid"{marking}, from=1-2, to=1-3]
      \arrow[from=1-2, to=2-2]
      \arrow[""{name=2, anchor=center, inner sep=0}, "{{{{{{{\prolist{\varphi}_r}}}}}}}"{inner sep=.8ex}, "\shortmid"{marking}, from=1-3, to=1-4]
      \arrow["\cdots"{description}, draw=none, from=1-3, to=2-3]
      \arrow[from=1-4, to=2-4]
      \arrow[""{name=3, anchor=center, inner sep=0}, "{{{{{{{\prolist{\varphi}_1}}}}}}}"{inner sep=.8ex}, "\shortmid"{marking}, from=1-5, to=1-6]
      \arrow[from=1-5, to=2-5]
      \arrow[""{name=4, anchor=center, inner sep=0}, "{{{{{{{\prolist{\varphi}_2}}}}}}}"{inner sep=.8ex}, "\shortmid"{marking}, from=1-6, to=1-7]
      \arrow[from=1-6, to=2-6]
      \arrow[""{name=5, anchor=center, inner sep=0}, "{{{{{{{\prolist{\varphi}_r}}}}}}}"{inner sep=.8ex}, "\shortmid"{marking}, from=1-7, to=1-8]
      \arrow["\cdots"{description}, draw=none, from=1-7, to=2-7]
      \arrow[from=1-8, to=2-8]
      \arrow[""{name=6, anchor=center, inner sep=0}, "{{{{{{{\chi_1}'}}}}}}"'{inner sep=.8ex}, "\shortmid"{marking}, from=2-1, to=2-2]
      \arrow[from=2-1, to=4-1]
      \arrow[""{name=7, anchor=center, inner sep=0}, "{{{{{{{\chi_2}'}}}}}}"'{inner sep=.8ex}, "\shortmid"{marking}, from=2-2, to=2-3]
      \arrow[""{name=8, anchor=center, inner sep=0}, "{{{{{{{\chi_r}'}}}}}}"'{inner sep=.8ex}, "\shortmid"{marking}, from=2-3, to=2-4]
      \arrow["\sim"{description}, draw=none, from=2-4, to=2-5]
      \arrow[from=2-4, to=4-4]
      \arrow[""{name=9, anchor=center, inner sep=0}, "{{{{{{{\chi_1}}}}}}}"'{inner sep=.8ex}, "\shortmid"{marking}, from=2-5, to=2-6]
      \arrow[from=2-5, to=4-5]
      \arrow[""{name=10, anchor=center, inner sep=0}, "{{{{{{{\chi_2}}}}}}}"'{inner sep=.8ex}, "\shortmid"{marking}, from=2-6, to=2-7]
      \arrow[""{name=11, anchor=center, inner sep=0}, "{{{{{{{\chi_r}}}}}}}"'{inner sep=.8ex}, "\shortmid"{marking}, from=2-7, to=2-8]
      \arrow[from=2-8, to=4-8]
      \arrow[""{name=12, anchor=center, inner sep=0}, "\psi"'{inner sep=.8ex}, "\shortmid"{marking}, from=4-1, to=4-4]
      \arrow[""{name=13, anchor=center, inner sep=0}, "\psi"'{inner sep=.8ex}, "\shortmid"{marking}, from=4-5, to=4-8]
      \arrow["{{{\left(\frac{\alpha_1}{\sigma_1},k_1\right)}}}"{description}, draw=none, from=0, to=6]
      \arrow["{{{\left(\frac{\alpha_2}{\sigma_2},k_2\right)}}}"{description}, draw=none, from=1, to=7]
      \arrow["{{{\left(\frac{\alpha_r}{\sigma_r},k_r\right)}}}"{description}, draw=none, from=2, to=8]
      \arrow["{{{(\alpha_1,k_1)}}}"{description}, draw=none, from=3, to=9]
      \arrow["{{{(\alpha_2,k_2)}}}"{description}, draw=none, from=4, to=10]
      \arrow["{{{(\alpha_r,k_r)}}}"{description}, draw=none, from=5, to=11]
      \arrow["{{{(\beta',r)}}}"{description}, draw=none, from=7, to=12]
      \arrow["{{{\left(\frac{\sigma_1 \cdots \sigma_r}{\beta},r\right)}}}"{description}, draw=none, from=10, to=13]
    \end{tikzcd}\]
\end{rmk}

\begin{thm}[Exponentiable VDCs Have Decomposable Multicells]\label{thm:ExpViaDecomp}
	A VDC $\Dbl{D}$ is exponentiable if and only if it has decomposable multicells.
\end{thm}
\begin{proof}
	First, fix a tree $T$, and consider the natural map
	\begin{equation}\label{eq:compMap}
		c:\colim_{i:\Dbl{c}\hookrightarrow \Dbl{J}(T)}(\Dbl{D}\times \Dbl{c})\to \Dbl{D}\times \Dbl{J}(T)
	\end{equation}
	induced by composition in $\Dbl{D}$ and $\Dbl{J}(T)$. 
  For a subtree $T'\subseteq T$ of height $n$ with $m$ leaves, let $\square(T')$ denote the associated 
  glued multicell in $\Dbl{J}(T)$, and fix an $m$-ary multicell $\alpha \in \Dbl{D}$. 
  Then from the description of the above colimit in Remark~\ref{rmk:cellInColimit}, the map in 
  Eq.~\ref{eq:compMap} having $(\alpha,\square(T'))$ in its image is equivalent to $\alpha$ admitting 
  a decomposition into a pasting diagram with arities specified by the subtree $T'$. 
  Further, that $c$ in Eq.~\ref{eq:compMap} should have at most one multicell mapped to $(\alpha,\square(T'))$ is 
  equivalent to $\alpha$ either not admitting a decomposition of the appropriate shape, or any two 
  decompositions being related by passing sequences of unary multicells between layers. But this is precisely 
  the relation that says that two pasting diagrams lying over the same multicell $\alpha$ in 
  $\Dbl{D}$ are in the same connected component of $\Cat{Paste}_n(\Dbl{D};\alpha)$. 

	Next, recall that by Corollary~\ref{cor:exponentiable-sketch} and the definition of the sketch for 
  VDCs (Definition \ref{defn:sketchVDC}), to show $\Dbl{D}$ is exponentiable it is 
  equivalent to show that $\Dbl{D}\times -$ preserves the gluings $A_T$ where $T$ is a  tree of height $2$ or $3$. 
  Thus, $\Dbl{D}$ is exponentiable if and only if Eq.~\ref{eq:compMap} is an isomorphism for all height $2$ and 
  $3$ trees $T$. Additionally, from the above argument, this is equivalent to $\Dbl{D}$ having essentially unique 
  decompositions into $2$ and $3$ layers for its multicells. But by Corollary~\ref{cor:decompSimp} $\Dbl{D}$ has 
  decomposable multicells if and only if it has essentially unique decompositions into $2$ layers, so the claim 
  follows.
\end{proof}

The description of exponentiable VDCs in Theorem~\ref{thm:ExpViaDecomp} provides a usable
technique for testing whether a VDC is exponentiable, which we can immediately apply 
to provide a minimal example of a VDC that is not exponentiable, namely the walking loose arrow $\Loose $.

\begin{exmp}[Counter-example to Exponentiability]\label{eg:looseCounter}
	The walking loose arrow $\Loose$ is not an exponentiable VDC because there are no nullary or binary multicells to decompose the unary multicell $\text{id}_\ell$.
\end{exmp}

We refer the reader to Section~\ref{sec:Examples} for further examples and non-examples of exponentiable VDCs.

%% file: Sections/5_OtherCharacterizations.tex
In this section, we provide three alternative characterizations of exponentiable virtual double
categories. The first is in terms of \emph{pro-representability}.

\subsection{Exponentiability of VDCs is equivalent to pro-representability}\label{subsec:proRep}

Our next goal is to obtain a further reduction of the characterization of 
exponentiability via decomposable multicells. The choice of decompositions included in this reduction are partially motivated by 
Pisani's characterization of exponentiable multicategories in terms of pro-monoidal 
multicategories~\cite{Pisani2014}. We first pick out a class of trees 
(in the sense of Definition \ref{defn:tree}) which will detect decomposability
of multicells into arbitrary shapes. 

\begin{defn}[Binary-Nullary Tree]\label{defn:binaryNullaryTree}
	A \emph{binary-nullary tree} $T$ is a tree of height at least $2$
  such that, in every layer of nodes 
  except the last, at most one node is non-unary, with 
  arity at most $2.$
  If the only non-unary node has arity $2$ (resp. arity $0$) in each layer except the last,
  then we say that $T$ is a \emph{binary-unary} (resp. \emph{nullary-unary}) tree.
\end{defn}

Binary-nullary trees thus look like the examples below:
\[\begin{tikzcd}
	{\color{white}\bullet} & {\color{white}\bullet} & {\color{white}\bullet} & {\color{white}\bullet} && {\color{white}\bullet} & {\color{white}\bullet} & {\color{white}\bullet} & {\color{white}\bullet} & {\color{white}\bullet} \\
	\bullet & \bullet & \bullet &&& \bullet & \bullet & \bullet & \bullet \\
	& \bullet && {\color{white}\bullet} && \bullet & \bullet & \bullet & \bullet \\
	& {\color{white}\bullet} && \bullet & \bullet && \bullet & \bullet & \bullet \\
	&&&& \bullet &&& \bullet \\
	&&&& {\color{white}\bullet} &&& {\color{white}\bullet}
	\arrow[no head, from=2-1, to=1-1]
	\arrow[no head, from=2-2, to=1-2]
	\arrow[no head, from=2-3, to=1-3]
	\arrow[no head, from=2-3, to=1-4]
	\arrow[no head, from=2-6, to=1-6]
	\arrow[no head, from=2-7, to=1-7]
	\arrow[no head, from=2-9, to=1-9]
	\arrow[no head, from=3-2, to=2-1]
	\arrow[no head, from=3-2, to=2-2]
	\arrow[no head, from=3-2, to=2-3]
	\arrow[no head, from=3-2, to=4-2]
	\arrow[no head, from=3-6, to=2-6]
	\arrow[no head, from=3-7, to=2-7]
	\arrow[no head, from=3-8, to=2-8]
	\arrow[no head, from=3-9, to=2-9]
	\arrow[no head, from=4-4, to=3-4]
	\arrow[no head, from=4-7, to=3-6]
	\arrow[no head, from=4-7, to=3-7]
	\arrow[no head, from=4-8, to=3-8]
	\arrow[no head, from=4-9, to=3-9]
	\arrow[no head, from=5-5, to=4-4]
	\arrow[no head, from=5-5, to=4-5]
	\arrow[no head, from=5-5, to=6-5]
	\arrow[no head, from=5-8, to=4-7]
	\arrow[no head, from=5-8, to=4-8]
	\arrow[no head, from=5-8, to=4-9]
	\arrow[no head, from=5-8, to=6-8]
\end{tikzcd}\]
The left-hand example is, in particular, binary-unary, while the middle example is nullary-unary.

\begin{defn}[Pro-Representable VDC]\label{defn:proRep}
	A VDC $\Dbl{D}$ is said to be \emph{pro-representable} if the following conditions hold:
	\begin{itemize}
		\item[(P1)] For every binary-nullary tree $T$ of height $2,$ $\Dbl{D}$ admits essentially unique decompositions
    of shape $T.$
		\item[(P3)] For every height $2$ tree $T'$ there is some binary-nullary tree $T$ with $T'$ as some partial 
		composite,
    such that every decomposition of a $\Dbl{D}$-cell of shape $T'$ arises from a decomposition of shape $T$ 
    up to equivalence.
	\end{itemize}
\qed
\end{defn}
Note that $T'$ in (P3) may be of any height, and in general will have to be of height greater than $2.$
Note also that by definition all VDCs with decomposable multicells are pro-representable. 
We will now show that the converse is also true, starting with existence of decompositions.

\begin{lem}\label{lem:existProRepDecomps}
	Let $\Dbl{D}$ satisfy (P1) above. Then $\Dbl{D}$ admits decompositions 
  of shape $T$ for every tree $T$ of height $2.$ 
\end{lem}
Here a decomposition of a multicell $\alpha$ in $\Dbl{D}$ of shape $T$ is a morphism 
$\Dbl{J}(T)\to \Dbl{D}$ sending the maximal composite cell of $\Dbl{J}(T)$ to $\alpha.$
\begin{proof}
	Let $\alpha$ be an $n$-ary multicell in $\Dbl{D}$, and let $(k_1,\ldots,k_m)$ 
  be a partition of $n$, so that $n=\sum k_i.$
  Note that such a partition uniquely determines a tree of height $2,$ according to whose shape we must 
  show $\alpha$ admits a decomposition.

  First, we consider the case where $k_1,\ldots,k_m\geq 1$ are all positive integers. 
  To prove this case we proceed by induction on $d=\max\{k_1,\ldots,k_m\}$. 
  If $d=1$, then we have the decomposition $\frac{\text{id}_{\varphi_1} \cdots \text{id}_{\varphi_n}}{\alpha}$  
  where $\varphi_1 \cdots \varphi_n$ is the loose source of $\alpha$. 
  Now, inductively suppose $d\geq 2$ 
  and the claim holds for $d-1$. Let $1\leq i_1<...<i_r\leq m$ be the indices such that $k_{i_t}=d$. 
  Replacing $k_{i_t}$ by $((d-1),1)$ in the partition, we obtain a new partition $n=\ell_1+\cdots+\ell_{m+r}$ 
  where $\max\{\ell_1,\ldots,\ell_{m+r}\}=d-1$. Thus, by the inductive hypothesis, $\alpha$ admits a 
  decomposition $\frac{\beta_1 \cdots \beta_{m+r}}{\beta}$ where $\beta_i$ has arity $\ell_i$. 
	
	Now, since $\Dbl{D}$ satisfies (P1), we can decompose $\beta$ as 
  $\frac{\gamma_1 \cdots \gamma_{m+r-1}}{\gamma},$ where $\gamma_{i_1}$ is a binary multicell
  and all other $\gamma_j$ are unary. Pasting these two decompositions, we obtain a new decomposition 
  $\frac{\delta_1 \cdots \delta_{m+r-1}}{\gamma}$ where 
  $\delta_{i_1}=\frac{(\beta_{i_1},\beta_{i_1+1})}{\gamma_{i_1}}$ has arity $(d-1)+1=d$. 
  By induction on $r,$ we obtain a decomposition of $\alpha$ of the desired shape.

	Now, let's consider the case where the partition $n=k_1+\cdots+k_m$ may contain zeros. Let $1\leq i_1<\cdots<i_r\leq m$ denote the indices such that $k_{i_t}=0$. Then we proceed inductively on $r$. If $r=0$, the claim follows from the first case. Now, suppose $r\geq 1$ such that the claim holds for $r-1$. Then the partition $n=k_1+\cdots+k_{i_r-1}+k_{i_r+1}+\cdots+k_m$ contains $r-1$ zeros. Thus, by the inductive hypothesis $\alpha$ admits a decomposition $\frac{\beta_1 \cdots \beta_{i_r-1},\beta_{i_r+1} \cdots \beta_m}{\beta}$ where $\beta_j$ has arity $k_j$. Since $\Dbl{D}$ satisfies (P1), we can decompose $\beta$ as $\frac{\gamma_1 \cdots \gamma_m}{\gamma}$ where $\gamma_{i_r}$ is nullary and the remaining $\gamma_j$ are unary. Composing these factorizations gives a decomposition $\frac{\delta_1 \cdots \delta_m}{\gamma}$ of the desired shape, proving the inductive step, and the claim.
\end{proof}

\begin{expl}[Inserting Binary and Nullary Decompositions]
	We will explicate the inductive step in Lemma~\ref{lem:existProRepDecomps} in the case of going from a binary decomopsition to a trinary decomposition, and then inserting a nullary multicell. Thus, for a quaternary multicell consider the following binary-nullary decomposition:
	\[
	\begin{adjustbox}{}
		\begin{tikzcd}
	\bullet & \bullet & \bullet & \bullet & \bullet \\
	\bullet && \bullet & \bullet & \bullet \\
	\bullet &&&& \bullet
	\arrow["\shortmid"{marking}, from=1-1, to=1-2]
	\arrow[from=1-1, to=2-1]
	\arrow["\shortmid"{marking}, from=1-2, to=1-3]
	\arrow["\shortmid"{marking}, from=1-3, to=1-4]
	\arrow[from=1-3, to=2-3]
	\arrow["\shortmid"{marking}, from=1-4, to=1-5]
	\arrow[from=1-4, to=2-4]
	\arrow[from=1-5, to=2-5]
	\arrow["\shortmid"{marking}, from=2-1, to=2-3]
	\arrow[from=2-1, to=3-1]
	\arrow["\shortmid"{marking}, from=2-3, to=2-4]
	\arrow["\shortmid"{marking}, from=2-4, to=2-5]
	\arrow[from=2-5, to=3-5]
	\arrow["\shortmid"{marking}, from=3-1, to=3-5]
\end{tikzcd}
	\end{adjustbox}
	\]
	Then we can use (P1) on the bottom layer of the decomposition to obtain a 3-layer decomposition, which composes to a trinary decomposition as below right:
	\[
	\begin{adjustbox}{}
		\begin{tikzcd}
	\bullet & \bullet & \bullet & \bullet & \bullet & \bullet & \bullet & \bullet & \bullet & \bullet \\
	\bullet && \bullet & \bullet & \bullet & \bullet &&& \bullet & \bullet \\
	\bullet &&& \bullet & \bullet & \bullet &&& \bullet & \bullet \\
	\bullet &&&& \bullet & \bullet &&&& \bullet
	\arrow["\shortmid"{marking}, from=1-1, to=1-2]
	\arrow[from=1-1, to=2-1]
	\arrow["\shortmid"{marking}, from=1-2, to=1-3]
	\arrow["\shortmid"{marking}, from=1-3, to=1-4]
	\arrow[from=1-3, to=2-3]
	\arrow["\shortmid"{marking}, from=1-4, to=1-5]
	\arrow[from=1-4, to=2-4]
	\arrow[from=1-5, to=2-5]
	\arrow["\shortmid"{marking}, from=1-6, to=1-7]
	\arrow[from=1-6, to=2-6]
	\arrow["\shortmid"{marking}, from=1-7, to=1-8]
	\arrow["\shortmid"{marking}, from=1-8, to=1-9]
	\arrow["\shortmid"{marking}, from=1-9, to=1-10]
	\arrow[from=1-9, to=2-9]
	\arrow[from=1-10, to=2-10]
	\arrow["\shortmid"{marking}, from=2-1, to=2-3]
	\arrow[from=2-1, to=3-1]
	\arrow["\shortmid"{marking}, from=2-3, to=2-4]
	\arrow["\shortmid"{marking}, from=2-4, to=2-5]
	\arrow[from=2-4, to=3-4]
	\arrow[""{name=0, anchor=center, inner sep=0}, from=2-5, to=3-5]
	\arrow[""{name=1, anchor=center, inner sep=0}, from=2-6, to=3-6]
	\arrow[from=2-9, to=3-9]
	\arrow[from=2-10, to=3-10]
	\arrow["\shortmid"{marking}, from=3-1, to=3-4]
	\arrow[from=3-1, to=4-1]
	\arrow["\shortmid"{marking}, from=3-4, to=3-5]
	\arrow[from=3-5, to=4-5]
	\arrow["\shortmid"{marking}, from=3-6, to=3-9]
	\arrow[from=3-6, to=4-6]
	\arrow["\shortmid"{marking}, from=3-9, to=3-10]
	\arrow[from=3-10, to=4-10]
	\arrow["\shortmid"{marking}, from=4-1, to=4-5]
	\arrow["\shortmid"{marking}, from=4-6, to=4-10]
	\arrow[between={0.2}{0.8}, maps to, from=0, to=1]
\end{tikzcd}
	\end{adjustbox}
	\]
	If we want instead to insert a nullary multicell, we can use (P1) on the bottom layer of the decomposition, and then compose the top to layers of the resulting 3-layer decomposition to obtain the desired 2-layer decomposition below right:
	\[
		\begin{tikzcd}[column sep=small]
	\bullet & \bullet & \bullet & \bullet & \bullet & \bullet & \bullet & \bullet & \bullet & \bullet \\
	\bullet &&& \bullet & \bullet & \bullet \\
	\bullet && \bullet & \bullet & \bullet & \bullet && \bullet && \bullet & \bullet \\
	\bullet &&&& \bullet & \bullet &&&&& \bullet
	\arrow["\shortmid"{marking}, from=1-1, to=1-2]
	\arrow[from=1-1, to=2-1]
	\arrow["\shortmid"{marking}, from=1-2, to=1-3]
	\arrow["\shortmid"{marking}, from=1-3, to=1-4]
	\arrow["\shortmid"{marking}, from=1-4, to=1-5]
	\arrow[from=1-4, to=2-4]
	\arrow[from=1-5, to=2-5]
	\arrow["\shortmid"{marking}, from=1-6, to=1-7]
	\arrow[from=1-6, to=2-6]
	\arrow["\shortmid"{marking}, from=1-7, to=1-8]
	\arrow["\shortmid"{marking}, from=1-8, to=1-9]
	\arrow["\shortmid"{marking}, from=1-9, to=1-10]
	\arrow[from=1-9, to=3-8]
	\arrow[from=1-9, to=3-10]
	\arrow[from=1-10, to=3-11]
	\arrow["\shortmid"{marking}, from=2-1, to=2-4]
	\arrow[from=2-1, to=3-1]
	\arrow["\shortmid"{marking}, from=2-4, to=2-5]
	\arrow[from=2-4, to=3-3]
	\arrow[from=2-4, to=3-4]
	\arrow[""{name=0, anchor=center, inner sep=0}, from=2-5, to=3-5]
	\arrow[""{name=1, anchor=center, inner sep=0}, from=2-6, to=3-6]
	\arrow["\shortmid"{marking}, from=3-1, to=3-3]
	\arrow[from=3-1, to=4-1]
	\arrow["\shortmid"{marking}, from=3-3, to=3-4]
	\arrow["\shortmid"{marking}, from=3-4, to=3-5]
	\arrow[from=3-5, to=4-5]
	\arrow["\shortmid"{marking}, from=3-6, to=3-8]
	\arrow[from=3-6, to=4-6]
	\arrow["\shortmid"{marking}, from=3-8, to=3-10]
	\arrow["\shortmid"{marking}, from=3-10, to=3-11]
	\arrow[from=3-11, to=4-11]
	\arrow["\shortmid"{marking}, from=4-1, to=4-5]
	\arrow["\shortmid"{marking}, from=4-6, to=4-11]
	\arrow[between={0.2}{0.8}, maps to, from=0, to=1]
\end{tikzcd}
	\]
\end{expl}

Property (P1) in the definition of pro-representability is also
sufficient for showing that multicells can be essentially uniquely
decomposed into the shape of arbitrary binary-nullary trees.

\begin{lem}\label{lem:essUniqBinaryyNullaryDecomps}
	Let $\Dbl{D}$ satisfy (P1) above. Then $\Dbl{D}$ admits
	essentiallly unique decompositions of shape $T$ for every binary-nullary tree $T$.
\end{lem}
\begin{proof}
	We proceed by induction on the height, $N$, of the binary-nullary tree $T$. 
	If $N=2$, then the claim is precisely condition (P1), which holds by assumption.
	Now, suppose $N\geq 3$, and that the claim holds for binary-nullary trees of height $N-1$.
	Composing the bottom two layers of such an $N$-layer pasting diagram 
  ${\threefrac{\prolist{\alpha_N}}{\vdots}{\alpha_1}}$
	produces an $(N-1)$-layer pasting diagram of binary-nullary shape. We can thus decrease
  the height of a tree while staying in the binary-nullary space, which is the main point.  
	Indeed, by the inductive hypothesis the resulting pasting diagrams are essentially unique, 
	so any other $N$-layer pasting diagram ${\threefrac{\prolist{\beta_N}}{\vdots}{\beta_1}}$ 
	for $\alpha$ of the same shape gets mapped to the same connected component of 
  $\Cat{Paste}_{N-1}(\Dbl{D};\alpha)$. 
	Then by Lemma~\ref{lem:components} this implies that the original pasting diagrams are equivalent, as desired. 
	Thus, by induction $\Dbl{D}$ admits essentially unique decompositions of shape $T$ for any binary-nullary tree $T$.
\end{proof}

To show that pro-representable VDCs have decomposable multicells, it remains to show essential uniqueness of the decompositions from Lemma~\ref{lem:existProRepDecomps}.

\begin{cor}\label{cor:ProRepImpliesDecomp}
	If $\Dbl{D}$ is a pro-representable VDC, then $\Dbl{D}$ has decomposable multicells, and thus is exponentiable.
\end{cor}
\begin{proof}
	Fix an $n$-ary multicell $\alpha$ in $\Dbl{D}$. 
	By Lemma~\ref{lem:existProRepDecomps}, $\alpha$ admits decompositions of shape $T$ 
	for any tree $T$ of height $2$.
	It remains to show that any two such decompositions are equivalent.
	
	By property (P3), for any height $2$ tree $T$, there exists an $N$-layered binary-nullary tree $T'$
	that partially composes to $T$ such that every decomposition of $\alpha$ of shape $T$ is equivalent 
	to a decomposition which is the partial composite of a decomposition of shape $T'$.
	However, by Lemma~\ref{lem:essUniqBinaryyNullaryDecomps}, all decompositions of $\alpha$ of shape $T'$
	are equivalent. 
	Thus, since equivalent decompositions of shape $T'$ compose to equivalent decompositions of shape $T$,
	and since every decomposition of shape $T$ is equivalent to one coming from a decomposition of shape $T'$,
	it follows that all decompositions of shape $T$ are equivalent.
	By Corollary~\ref{cor:decompSimp} it follows that $\alpha$ admits essentially unique decompositions, as desired.
\end{proof}

\begin{rmk}[Insufficiency of (P1)]
(P1) alone is not sufficient for exponentiability. 
The need for property (P3) comes from the fact that we could have a 2-layer decomposition of the form 
\[
\begin{tikzcd}
	\bullet & \bullet & \bullet \\
	\bullet && \bullet & \bullet \\
	\bullet &&& \bullet
	\arrow["\shortmid"{marking}, from=1-1, to=1-2]
	\arrow[from=1-1, to=2-1]
	\arrow["\shortmid"{marking}, from=1-2, to=1-3]
	\arrow[from=1-3, to=2-3]
	\arrow[from=1-3, to=2-4]
	\arrow["\shortmid"{marking}, from=2-1, to=2-3]
	\arrow[from=2-1, to=3-1]
	\arrow["\shortmid"{marking}, from=2-3, to=2-4]
	\arrow[from=2-4, to=3-4]
	\arrow["\shortmid"{marking}, from=3-1, to=3-4]
\end{tikzcd}
\]
which is not equivalent to a 2-layer decomposition admitting a binary-nullary decomposition, since we may not be able to decompose the binary and nullary multicell in the second layer simultaneously to obtain decompositions of the form bottom left or bottom right:
\[\begin{adjustbox}{}
\begin{tikzcd}
	\bullet & \bullet & \bullet &&& \bullet & \bullet & \bullet \\
	\bullet && \bullet &&& \bullet & \bullet & \bullet & \bullet \\
	\bullet && \bullet & \bullet && \bullet && \bullet & \bullet \\
	\bullet &&& \bullet && \bullet &&& \bullet
	\arrow["\shortmid"{marking}, from=1-1, to=1-2]
	\arrow[from=1-1, to=2-1]
	\arrow["\shortmid"{marking}, from=1-2, to=1-3]
	\arrow[from=1-3, to=2-3]
	\arrow["\shortmid"{marking}, from=1-6, to=1-7]
	\arrow[from=1-6, to=2-6]
	\arrow["\shortmid"{marking}, from=1-7, to=1-8]
	\arrow[from=1-7, to=2-7]
	\arrow[from=1-8, to=2-8]
	\arrow[from=1-8, to=2-9]
	\arrow["\shortmid"{marking}, from=2-1, to=2-3]
	\arrow[from=2-1, to=3-1]
	\arrow[from=2-3, to=3-3]
	\arrow[from=2-3, to=3-4]
	\arrow["\shortmid"{marking}, from=2-6, to=2-7]
	\arrow[from=2-6, to=3-6]
	\arrow["\shortmid"{marking}, from=2-7, to=2-8]
	\arrow["\shortmid"{marking}, from=2-8, to=2-9]
	\arrow[from=2-8, to=3-8]
	\arrow[from=2-9, to=3-9]
	\arrow["\shortmid"{marking}, from=3-1, to=3-3]
	\arrow[from=3-1, to=4-1]
	\arrow["\shortmid"{marking}, from=3-3, to=3-4]
	\arrow[from=3-4, to=4-4]
	\arrow["\shortmid"{marking}, from=3-6, to=3-8]
	\arrow[from=3-6, to=4-6]
	\arrow["\shortmid"{marking}, from=3-8, to=3-9]
	\arrow[from=3-9, to=4-9]
	\arrow["\shortmid"{marking}, from=4-1, to=4-4]
	\arrow["\shortmid"{marking}, from=4-6, to=4-9]
\end{tikzcd}
\end{adjustbox}
\]
due to the tight sources and targets that need to coincide in the decompositions. 

However, in the case where $\Dbl{D}$ has trivial underlying tight category, and hence corresponds to a multicategory, condition (P1) does suffice to guarantee exponentiability, since there are no constraints
arising from tight arrows in that case.
\end{rmk}

\subsection{Exponentiability of VDCs is equivalent to malleability}\label{subsec:malleable}

In addition to the classes of VDCs discussed so far, it is conjectured by Arkor that the exponentiable VDCs 
are the \emph{malleable} VDCs, defined in~\cite{arkor2025exponentiablevirtualdoublecategories}. 
Once the definition is expanded, we will see that malleable VDCs are indeed precisely those having 
decomposable multicells. Before stating the definition we need to recall the construction of the loose 
Kleisli VDC for a monad on a VDC~\cite[Definition 4.1]{Cruttwell2010}. Here a monad on a VDC is a monad 
in the 2-category $\BiCat{Vdc}$.

\begin{defn}[Kleisli VDCs]
	The loose Kleisli VDC associated to a VDC $\Dbl{D}$ equipped with a monad $T:\Dbl{D}\to \Dbl{D}$ is denoted $\Dbl{L}\Cat{Kl}(\Dbl{D},T)$ and defined as follows:
	\begin{enumerate}
		\item Its objects and tight arrows are as in $\Dbl{D}$;
		\item A loose arrow $\varphi:x\proto y$ corresponds to a loose arrow $U(\varphi):x\proto Ty$ in $\Dbl{D}$;
		\item A nullary multicell, as below left, is equivalent to a multicell in $\Dbl{D}$ as below right
		\[
		\begin{tikzcd}
	&&&&& x \\
	& x & {} & {} &&& Tx \\
	y && z & y &&&& Tz
	\arrow["\eta", from=1-6, to=2-7]
	\arrow["a"', from=1-6, to=3-4]
	\arrow["a"', from=2-2, to=3-1]
	\arrow["b", from=2-2, to=3-3]
	\arrow[squiggly, tail reversed, from=2-4, to=2-3]
	\arrow["Tb", from=2-7, to=3-8]
	\arrow[""{name=0, anchor=center, inner sep=0}, "\varphi"'{inner sep=.8ex}, "\shortmid"{marking}, from=3-1, to=3-3]
	\arrow[""{name=1, anchor=center, inner sep=0}, "{U(\varphi)}"'{inner sep=.8ex}, "\shortmid"{marking}, from=3-4, to=3-8]
	\arrow["\alpha"{description}, draw=none, from=1-6, to=1]
	\arrow["\alpha"{description}, draw=none, from=2-2, to=0]
\end{tikzcd}
		\]
		\item An $n$-ary multicell, $n\geq 1$, as below left, is equivalent to an $n$-ary multicell in $\Dbl{D}$ as below right:
\[\begin{tikzcd}
	{x_0} && {x_n} & {x_0} & {Tx_1} & \cdots && {T^nx_n} \\
	&& {} & {} &&&& {Tx_n} \\
	y && z & y &&&& Tz
	\arrow[""{name=0, anchor=center, inner sep=0}, "{{{\prolist{\varphi}}}}"{inner sep=.8ex}, "\shortmid"{marking}, from=1-1, to=1-3]
	\arrow["a"', from=1-1, to=3-1]
	\arrow["b", from=1-3, to=3-3]
	\arrow["{{{U(\varphi_1)}}}"{inner sep=.8ex}, "\shortmid"{marking}, from=1-4, to=1-5]
	\arrow["a"', from=1-4, to=3-4]
	\arrow["{{{T(U\varphi_2)}}}"{inner sep=.8ex}, "\shortmid"{marking}, from=1-5, to=1-6]
	\arrow["{{{T^{n-1}(U\varphi_n)}}}"{inner sep=.8ex}, "\shortmid"{marking}, from=1-6, to=1-8]
	\arrow["{{{\mu\circ \cdots \circ T^{n-2}\mu}}}"', from=1-8, to=2-8]
	\arrow[between={0.2}{0.8}, squiggly, tail reversed, from=2-4, to=2-3]
	\arrow["Tb"', from=2-8, to=3-8]
	\arrow[""{name=1, anchor=center, inner sep=0}, "\psi"'{inner sep=.8ex}, "\shortmid"{marking}, from=3-1, to=3-3]
	\arrow[""{name=2, anchor=center, inner sep=0}, "{{{U(\psi)}}}"'{inner sep=.8ex}, "\shortmid"{marking}, from=3-4, to=3-8]
	\arrow["\alpha"{description}, draw=none, from=0, to=1]
	\arrow["\alpha"{description}, draw=none, from=1-6, to=2]
\end{tikzcd}\]
	\end{enumerate}
\qed
\end{defn}

We write $\Dbl{K}\Mod(\Dbl{D},T):=\Mod(\Dbl{L}\Cat{Kl}(\Dbl{D},T))$ for the VDC of $T$-monoids in $\Dbl{D}$. 
Cruttwell and Shulman have a very general framework in \cite{Cruttwell2010}, 
but for this paper we will only need to consider the $T$-monoids as in~\cite{Burroni1971,Leinster2004}, 
where $T=\Cat{fc}$ is the free category monad on the category of graphs. 
Such $T$-monoids give one way to define virtual double categories. 
Namely, a VDC $\Dbl{D}$ is equivalent to the associated $\Cat{fc}$-monoid
\[\begin{tikzcd}[column sep=small]
	& {G_1} &&& {\Sigma_n \uSq{n}\Dbl{D}} & {\uTight(\Dbl{D})} & \\
	{\Cat{fc}(G_0)} && {G_0} & {\Cat{fc}(\uLoose(\Dbl{D})} & {\uOb (\Dbl{D}))} & {\uLoose(\Dbl{D})} & {\uOb (\Dbl{D})}
	\arrow[from=1-2, to=2-1]
	\arrow[""{name=0, anchor=center, inner sep=0}, from=1-2, to=2-3]
	\arrow[""{name=1, anchor=center, inner sep=0}, shift right, from=1-5, to=1-6]
	\arrow[""{name=2, anchor=center, inner sep=0}, shift left, from=1-5, to=1-6]
	\arrow[""{name=3, anchor=center, inner sep=0}, shift left, from=2-4, to=2-5]
	\arrow[shift right, from=2-4, to=2-5]
	\arrow[shift left, from=2-6, to=2-7]
	\arrow[""{name=4, anchor=center, inner sep=0}, shift right, from=2-6, to=2-7]
	\arrow[""{name=5, anchor=center, inner sep=0}, between={0.4}{0.6}, from=1, to=3]
	\arrow[between={0.4}{0.6}, from=2, to=4]
	\arrow["{:=}"{description, pos=0.2}, draw=none, from=0, to=5]
\end{tikzcd}\]
in $\Cat{Grph}$, the category of graphs. Thus we have the monoid multiplication
$m:\fc(G_1)\times_{\fc(G_0)} G_1\to G_1$ which pastes height-2 arrangements of multicells
as well as the monoid unit $e:G_0\to G_1,$ which picks out identity tight arrows 
and unary multicells.

Using the normalized embedding $\Dbl{K}\Mod(\Dbl{D},T)\hookrightarrow \Dbl{K}\Mod(\Mod(\Dbl{D}),\Mod(T))$ 
constructred in~\cite[Theorem 8.7]{Cruttwell2010}, we can further view the VDC $\Dbl{D}$ as a $\Mod(\mathsf{fc})$-monoid in 
$\Dbl{P}\mathsf{rof}(\Cat{Grph})$. This $\Mod(\mathsf{fc})$-monoid is carried by by a profunctor internal to graphs
whose codomain is the category internal to graphs carried by the span
\[\begin{tikzcd}
	& {\uSq{1}\Dbl{D}} & {\uTight(\Dbl{D})} & \\
	{\uLoose(\Dbl{D})} & {\uOb (\Dbl{D})} & {\uLoose(\Dbl{D})} & {\uOb (\Dbl{D})}
	\arrow[""{name=0, anchor=center, inner sep=0}, shift right, from=1-2, to=1-3]
	\arrow[""{name=1, anchor=center, inner sep=0}, shift left, from=1-2, to=1-3]
	\arrow[""{name=2, anchor=center, inner sep=0}, shift left, from=2-1, to=2-2]
	\arrow[shift right, from=2-1, to=2-2]
	\arrow[shift left, from=2-3, to=2-4]
	\arrow[""{name=3, anchor=center, inner sep=0}, shift right, from=2-3, to=2-4]
	\arrow[between={0.4}{0.6}, from=0, to=2]
	\arrow[between={0.4}{0.6}, from=1, to=3]
\end{tikzcd}.\]
The $\Cat{Grph}$-profunctor $M$ itself carrying the desired $\Mod(\mathsf{fc})$-monoid 
is carried by the span of graphs $\mathsf{fc}(G_0)\leftarrow G_1\rightarrow G_0$ by 
which we already gave a simpler reconstitution of $\Dbl{D},$ and the multiplication of the monoid
is a cell $M,M \Rightarrow M,$ in the loose Kleisli VDC of $\Mod(\fc)$ on $\Dbl{P}\mathsf{rof}(\Cat{Grph})$. 
We can now recall Arkor's definition of malleable VDCs.

\begin{defn}[Malleable VDC (\cite{arkor2025exponentiablevirtualdoublecategories})]
    A VDC $\Dbl{D}$ is said to be \emph{malleable} if when viewed as a $\Mod(\mathsf{fc})$-monoid in 
    $\Dbl{P}\mathsf{rof}(\Cat{Grph})$, its multiplication multicell is cartesian in $\Dbl{P}\mathsf{rof}(\Cat{Grph})$.
\qed
\end{defn}

To work with this definition we can use the fact that $\Cat{Grph}=\Cat{Set}^{\bullet\rightrightarrows\bullet}$
along with the isomorphisms in Proposition~\ref{prop:InternalizeUModAdj} and
Example 7.3 of~\cite{arkor2025exponentiablevirtualdoublecategories},
which imply that $\Dbl{P}\mathsf{rof}(\Cat{Grph})\cong \Dbl{V}\mathsf{df}_n(\T_u(\bullet\rightrightarrows\bullet),\Prof)$. 
Then the category internal to graphs associated to $\Dbl{D}$ corresponds to the graph of categories
$(\Dbl{D}_1\rightrightarrows \Dbl{D}_0)$, while the span carrying the $\Mod(\mathsf{fc})$-monoid structure
corresponds to the graph of profunctors $\MCell (\Dbl{D})\rightrightarrows \Dbl{D}_0(-,-)$.

Then $\Dbl{D}$ being malleable means that the natural transformation 
$\Mod(\mathsf{fc})(\MCell (\Dbl{D}))\odot \MCell (\Dbl{D})\Rightarrow \MCell (\Dbl{D})(\mu_\Cat{fc},1)$ 
induced by composition is an isomorphism, which is precisely Eq.~\ref{eq:CompChar}, 
so by Corollary~\ref{cor:decompSimp} we have the following:

\begin{prop}[Malleable VDCs have Decomposable Multicells]\label{prop:MallDecomp}
    A VDC $\Dbl{D}$ is malleable if and only if it has decomposable multicells.
\end{prop}

We can now state the main theorem of the paper:

\begin{thm}[Characterization of Exponentiable VDCs]\label{thm:ExpChar}
	For a VDC $\Dbl{D}$, the following are equivalent:
	\begin{enumerate}
		\item $\Dbl{D}$ is exponentiable;
		\item $\Dbl{D}$ is pro-representable;
		\item $\Dbl{D}$ has decomposable multicells;
		\item $\Dbl{D}$ is malleable;
		\item The exponential $\Span^{\Dbl{D}}$ exists in the category of large VDCs.
	\end{enumerate}
\end{thm}

So far we've shown the equivalences of conditions (1) through (4). 
We now turn to the last equivalent condition, involving the existence of the exponential $\Span^{\Dbl{D}}.$\footnote{This condition can be understood to be intimately related to the 
characterization of pro-monoidal multicategories in terms of bi-cocontinuous monoidal copresheaf categories (c.f.~\cite{Day1970}). 
We will explicate this relation once we've established conditions for the representability of exponentials in 
Section~\ref{sec:RepresentabilityOfExponentials}.}

\subsection{Exponentiability reduced to the case of $\Span$}
In personal communication, Arkor has indicated that an upcoming draft of \cite{arkor2025exponentiablevirtualdoublecategories}
will contain a proof that $\Span$ detects exponentiability for VDCs along very similar lines to the argument
below.

Before proving $(1)\iff(5)$ we will need the fact that all representable VDCs are exponentiable 
(this is proved directly in~\cite{arkor2025exponentiablevirtualdoublecategories}). 
We will use the equivalence $(1)\iff(3)$ to provide a simple proof of this fact.

\begin{prop}[Representable Implies Exponentiable]\label{prop:RepImpProRep}
    All representable VDCs are exponentiable.
\end{prop}
\begin{proof}
    Let $\Dbl{D}$ be a representable VDC. From the proof that exponentiable VDCs are precisely those which have 
    decomposable multicells, it suffices to show that each $n$-ary multicell $\alpha$ 
    in $\Dbl{D}$ admits essentially unique decompositions of height 2. 
    Let $0\leq m_1\leq \cdots\leq m_r=n$ be a sequence of non-negative integers, and let 
    $\prolist{\varphi}_1 \cdots \prolist{\varphi}_r$ denote the resulting partition of the loose source of $\alpha$ into a 
    sequence of sequences of loose arrows. Then since $\Dbl{D}$ has all opcartesian multicells, we have a canonical decomposition of 
    $\alpha$ relative to this partition:
\[\begin{tikzcd}
	{x_0} & {x_{m_1}} & \cdots & {x_{m_r}} \\
	{x_0} & {x_{m_1}} & \cdots & {x_{m_r}} \\
	{y_0} &&& {y_1}
	\arrow[""{name=0, anchor=center, inner sep=0}, "{{{\prolist{\varphi}_1}}}"{inner sep=.8ex}, "\shortmid"{marking}, from=1-1, to=1-2]
	\arrow[equals, from=1-1, to=2-1]
	\arrow[""{name=1, anchor=center, inner sep=0}, "{{{\prolist{\varphi}_2}}}"{inner sep=.8ex}, "\shortmid"{marking}, from=1-2, to=1-3]
	\arrow[equals, from=1-2, to=2-2]
	\arrow[""{name=2, anchor=center, inner sep=0}, "{{{\prolist{\varphi}_r}}}"{inner sep=.8ex}, "\shortmid"{marking}, from=1-3, to=1-4]
	\arrow["\cdots"{description}, draw=none, from=1-3, to=2-3]
	\arrow[equals, from=1-4, to=2-4]
	\arrow[""{name=3, anchor=center, inner sep=0}, "{{{\odot(\prolist{\varphi}_1)}}}"'{inner sep=.8ex}, "\shortmid"{marking}, from=2-1, to=2-2]
	\arrow["a"', from=2-1, to=3-1]
	\arrow[""{name=4, anchor=center, inner sep=0}, "{{{\odot(\prolist{\varphi}_2)}}}"'{inner sep=.8ex}, "\shortmid"{marking}, from=2-2, to=2-3]
	\arrow[""{name=5, anchor=center, inner sep=0}, "{{{\odot(\prolist{\varphi}_r)}}}"'{inner sep=.8ex}, "\shortmid"{marking}, from=2-3, to=2-4]
	\arrow["b", from=2-4, to=3-4]
	\arrow[""{name=6, anchor=center, inner sep=0}, "\psi"'{inner sep=.8ex}, "\shortmid"{marking}, from=3-1, to=3-4]
	\arrow["{{\mathsf{opcart}}}"{description}, draw=none, from=0, to=3]
	\arrow["{{\mathsf{opcart}}}"{description}, draw=none, from=1, to=4]
	\arrow["{{\mathsf{opcart}}}"{description}, draw=none, from=2, to=5]
	\arrow["{{\exists!\alpha^\flat}}"{description, pos=0.7}, draw=none, from=4, to=6]
\end{tikzcd}\]
	Further, by the universal property of opcartesian multicells, any other decomposition $\frac{\beta_1 \cdots \beta_r}{\beta}$ of the same shape factors uniquely through this universal one as below:
	\[
	\begin{adjustbox}{}
		\begin{tikzcd}
	&&&& {x_0} & \cdots & {x_{m_r}} \\
	{x_0} & \cdots & {x_{m_r}} && {x_0} & \cdots & {x_{m_r}} \\
	{z_0} & \cdots & {z_r} & {=} \\
	{y_0} && {y_1} && {z_0} & \cdots & {z_r} \\
	&&&& {y_0} && {y_1}
	\arrow[""{name=0, anchor=center, inner sep=0}, "{{\prolist{\varphi}_1}}"{inner sep=.8ex}, "\shortmid"{marking}, from=1-5, to=1-6]
	\arrow[equals, from=1-5, to=2-5]
	\arrow[""{name=1, anchor=center, inner sep=0}, "{{\prolist{\varphi}_r}}"{inner sep=.8ex}, "\shortmid"{marking}, from=1-6, to=1-7]
	\arrow["\cdots"{description}, draw=none, from=1-6, to=2-6]
	\arrow[equals, from=1-7, to=2-7]
	\arrow[""{name=2, anchor=center, inner sep=0}, "{{\prolist{\varphi}_1}}"{inner sep=.8ex}, "\shortmid"{marking}, from=2-1, to=2-2]
	\arrow["{{a'}}"', from=2-1, to=3-1]
	\arrow[""{name=3, anchor=center, inner sep=0}, "{{\prolist{\varphi}_r}}"{inner sep=.8ex}, "\shortmid"{marking}, from=2-2, to=2-3]
	\arrow["\cdots"{description}, draw=none, from=2-2, to=3-2]
	\arrow["{{b'}}", from=2-3, to=3-3]
	\arrow[""{name=4, anchor=center, inner sep=0}, "{{\odot(\prolist{\varphi}_1)}}"'{inner sep=.8ex}, "\shortmid"{marking}, from=2-5, to=2-6]
	\arrow["{{a'}}"', from=2-5, to=4-5]
	\arrow[""{name=5, anchor=center, inner sep=0}, "{{\odot(\prolist{\varphi}_r)}}"'{inner sep=.8ex}, "\shortmid"{marking}, from=2-6, to=2-7]
	\arrow["\cdots"{description}, draw=none, from=2-6, to=4-6]
	\arrow["{{b'}}", from=2-7, to=4-7]
	\arrow[""{name=6, anchor=center, inner sep=0}, "{{\chi_1}}"'{inner sep=.8ex}, "\shortmid"{marking}, from=3-1, to=3-2]
	\arrow["{{a''}}"', from=3-1, to=4-1]
	\arrow[""{name=7, anchor=center, inner sep=0}, "{{\chi_r}}"'{inner sep=.8ex}, "\shortmid"{marking}, from=3-2, to=3-3]
	\arrow["{{b''}}", from=3-3, to=4-3]
	\arrow[""{name=8, anchor=center, inner sep=0}, "\psi"'{inner sep=.8ex}, "\shortmid"{marking}, from=4-1, to=4-3]
	\arrow[""{name=9, anchor=center, inner sep=0}, "{{\chi_1}}"'{inner sep=.8ex}, "\shortmid"{marking}, from=4-5, to=4-6]
	\arrow["{{a''}}"', from=4-5, to=5-5]
	\arrow[""{name=10, anchor=center, inner sep=0}, "{{\chi_r}}"'{inner sep=.8ex}, "\shortmid"{marking}, from=4-6, to=4-7]
	\arrow["{{b''}}", from=4-7, to=5-7]
	\arrow[""{name=11, anchor=center, inner sep=0}, "\psi"'{inner sep=.8ex}, "\shortmid"{marking}, from=5-5, to=5-7]
	\arrow["{\mathsf{opcart}}"{description}, draw=none, from=0, to=4]
	\arrow["{\mathsf{opcart}}"{description}, draw=none, from=1, to=5]
	\arrow["{\beta_1}"{description}, draw=none, from=2, to=6]
	\arrow["{\beta_r}"{description}, draw=none, from=3, to=7]
	\arrow["{\beta_1^\flat}"{description}, draw=none, from=4, to=9]
	\arrow["{\beta_r^\flat}"{description}, draw=none, from=5, to=10]
	\arrow["\beta"{description}, draw=none, from=3-2, to=8]
	\arrow["\beta"{description}, draw=none, from=4-6, to=11]
\end{tikzcd}
	\end{adjustbox}
	\]
	Finally, the uniqueness in the factorizations through opcartesian multicells ensures that 
  $\frac{\beta_1^\flat \cdots \beta_r^\flat}{\beta}=\alpha^\flat$, so the two decompositions are equivalent. 
  It follows that $\Dbl{D}$ admits essentially unique decompositions of multicells, and hence is exponentiable by the equivalences 
  (1)$\iff$(3)$\iff$(4) in Theorem~\ref{thm:ExpChar} which we've proved so far. 
  Therefore all representable VDCs are exponentiable, as desired.
\end{proof}

Note that the proof of~\ref{prop:RepImpProRep} shows 
that we could weaken the hypotheses on $\Dbl{D}$ by instead requiring that
it has all loose units and that a sequence of loose arrows $\prolist{\varphi}$
admits a loose composite if it appears as the subsequence of the loose source
for a multicell in $\Dbl{D}$.

The argument for showing $(1)\iff(5)$ in Theorem~\ref{thm:ExpChar} will follow as in the proof 
(ii) $\iff$ (iii) in Theorem 2 in~\cite{Chris2006}. Indeed, although Theorem 2 in~\cite{Chris2006} is stated for locales, 
the proof applies more generally in any category with finite limits.

\begin{prop}[Exponentials via Universal Example (Townsend)]\label{prop:ExpByUnivExample}
	Let $\mathsf{C}$ be a category with finite limits, and suppose there exists an object $u \in \mathsf{C}$ such that for any other 
  object $x \in \mathsf{C}$, there exists an exponentiable object $e(x)\in \mathsf{C}$ and a regular monomorphism 
  $x\hookrightarrow u^{e(x)}$. Then an object $x \in \mathsf{C}$ is exponentiable if and only if the exponential $u^x$ exists.
\end{prop}

From Proposition~\ref{prop:ExpByUnivExample}, to prove (1)$\iff$(5) in Theorem~\ref{thm:ExpChar} it suffices to show that for every VDC $\Dbl{D}$, we have some exponentiable VDC $p(\Dbl{D})$ together with a regular monomorphism $\Dbl{D}\hookrightarrow \Span^{p(\Dbl{D})}$. To motivate the exponentiable VDC $p(\Dbl{D})$ that we will use, we will modify the categorical hom-bifunctor to a hom-bifunctor for virtual double categories. Recall that for a category $\Cat{C}$, its hom-bifunctor $\Cat{C}(-,-):\Cat{C}^{op}\times \Cat{C}\to \Cat{Set}$ is precisely the loose unit for the object $\Cat{C}$ in the profunctor VDC $\Prof=\Mod(\Span)$. 

Our first generalization of the hom-functor comes from looking at $T$-monoids in $\Span$, for a monad $T$ on 
$\Span$, which under the normalized embedding in~\cite[Theorem 8.7]{Cruttwell2010} correspond to $\Mod(T)$-monoids 
in $\Prof$. Using this perspective, a $T$-monoid determines a category $\Cat{C}$ along with a profunctor 
$T(\Cat{C})^{op}\times \Cat{C}\to \Cat{Set}$ with multiplication and unit multicells. In certain cases we can 
internalize this profunctor in $T$-monoids, such as when $T=\Cat{fm}$ is the free monoid monad on $\Cat{Set}$:

\begin{exmp}
	If $T=\Cat{fm}$ is the free monoid monad, then $\Cat{fm}$-monoids are precisely multicategories $\Cat{M}$, and $\Cat{fm}(\Cat{M})$ is the multicategory whose objects are sequences of objects in $\Cat{M}$, and whose multimorphisms are sequences of multimorphisms in $\Cat{M}$. In this case, $\Cat{fm}(\Cat{M})^{op}$ is also a multicategory, since sequences of objects can be identified with objects. Then the profunctor along with its multiplication and unit multicells which encode the multicategory $\Cat{M}$ can be enhanced to a functor of multicategories $\Cat{fm}(\Cat{M})^{op}\times \Cat{M}\to \Cat{Set}_\times$, where $\Cat{Set}_\times$ is the multicategory whose multimorphisms are given by maps out of products.
\end{exmp}

In order to obtain an analogous generalization for VDCs we also need to treat the case where $\Cat{Set}$ is replaced by a pre-sheaf category $\Cat{PSh}(\Cat{C})$, for $\Cat{C}$ some small category. In this case the profunctor appearing in the last examples is replaced by a profunctor valued pre-sheaf $\Cat{C}^{op}\to \Prof_1$, which in the case of VDCs recovers the graph of profunctors discussed in Subsection~\ref{subsec:malleable}. As with multicategories, we can internalize the data of this graph of profunctors with its multiplication and unit multicells in terms of a VDF, which for a VDC $\Dbl{D}$ takes the form:
\begin{equation}\label{eq:homVDF}
	\Dbl{D}(-,-):\Cat{fc}(\Dbl{D})^{op_t}\times \Dbl{D}\to \Span
\end{equation}
Note that $\Cat{fc}(\Dbl{D})$ has loose arrows and multicells given by sequences of loose arrows and multicells in $\Dbl{D}$, respectively. Thus, $\Cat{fc}(\Dbl{D})$ is a strict double category, and hence it makes sense to take its tight opposite. In more detail we have the following:

\begin{lem}[Hom VDF]\label{lem:homVDFunctor}
	For every VDC $\Dbl{D}$, there exists a hom VDF $\Dbl{D}(-,-):\Cat{fc}(\Dbl{D})^{op_t}\times \Dbl{D}\to \Span$.
\end{lem}
\begin{proof}
	We define the hom VDF to consist of the following data:
	\begin{enumerate}
		\item A pair of objects $(d,e) \in \Cat{fc}(\Dbl{D})^{op_t}\times \Dbl{D}$ is mapped to the set of tight arrows $\uTight(\Dbl{D})_{d,e}$.
		\item A pair of tight morphisms $(d_0\xleftarrow{f}d_1,e_0\xrightarrow{g}e_1)$ is mapped to the function $\uTight(\Dbl{D})_{d_0,e_0}\xrightarrow{g\circ - \circ f}\uTight(\Dbl{D})_{d_1,e_1}$ given by using post- and pre-composition.
		\item A pair of loose arrows $(\prolist{\varphi}:a_0\proto a_n,\psi:d_0\proto d_1)$ is mapped to the span 
		\begin{equation*}
			\uTight(\Dbl{D})_{a_0,d_0}\xleftarrow{s}\MCell (\Dbl{D})(\prolist{\varphi},\psi)\xrightarrow{t}\uTight(\Dbl{D})_{a_n,d_1}
		\end{equation*}
		where the apex of the span is the set of multicells with loose source $\prolist{\varphi}$ and loose target $\psi$.
		\item Consider a pair of $n$-ary multicells $((\prolist{\alpha}_1 \cdots \prolist{\alpha}_n),\beta)$, where each $\prolist{\alpha}_i$ is a sequence of compatible multicells in $\Dbl{D}$ with loose target $\prolist{\varphi}_i$ and loose source $\prolist{\prolist{\varphi}'}_i$, and $\beta$ is an $n$-ary multicell with loose source $\prolist{\psi}$ and loose target $\psi'$. The resulting $n$-ary multicell in $\Span$ is given by the map of spans
    \[\begin{tikzcd}
      {\uTight(\mathbb{D})(b_0,c_0)} & P & {\uTight(\mathbb{D})(b_{k_n},c_n)} \\
      \\
      {\uTight(\mathbb{D})(a_{0,0},d_0)} & {\MCell(\mathbb{D})(\prolist{\varphi}',\psi')} & {\uTight(\mathbb{D})(a_{k_n,0},d_1)}
      \arrow["{{{{g_0\circ -\circ f_0}}}}"', from=1-1, to=3-1]
      \arrow["s"', from=1-2, to=1-1]
      \arrow["t", from=1-2, to=1-3]
      \arrow["{{{{\beta\circ - \circ (\prolist{\alpha}_1 \cdots \prolist{\alpha}_n)}}}}"{description}, draw=none, from=1-2, to=3-2]
      \arrow["{{{{g_1\circ - \circ f_{k_n}}}}}", from=1-3, to=3-3]
      \arrow["s", from=3-2, to=3-1]
      \arrow["t"', from=3-2, to=3-3]
    \end{tikzcd}\]
		induced by pasting multicells, where $P=\MCell(\mathbb{D})(\prolist{\varphi}_1,\psi_1)\times_{\uTight(\mathbb{D})(b_{k_1},c_1)}\cdots\times_{\uTight(\mathbb{D})(b_{k_{n-1}},c_{n-1})}\MCell(\mathbb{D})(\prolist{\varphi}_n,\psi_n).$
	\end{enumerate}
	The associativity and unitality of pasting in $\Dbl{D}$ ensures that this data defines a VDF 
  $\Dbl{D}(-,-):\Cat{fc}(\Dbl{D})^{op_t}\times \Dbl{D}\to \Span$, as desired.
\end{proof}

Alternatively, we can construct the VDC $\Cat{fc}(\Dbl{D})$ using the composite of left adjoints in the adjunctions
\begin{equation}\label{eq:defOfFc}
\begin{tikzcd}
	{\mathsf{Dbl}_s} && {\mathsf{Vdc}_n} && {\mathsf{Vdc}}
	\arrow[""{name=0, anchor=center, inner sep=0}, "\iota"', curve={height=18pt}, hook, from=1-1, to=1-3]
	\arrow[""{name=1, anchor=center, inner sep=0}, "{\mathbb{F}_c}"', curve={height=18pt}, from=1-3, to=1-1]
	\arrow[""{name=2, anchor=center, inner sep=0}, "U"', curve={height=18pt}, from=1-3, to=1-5]
	\arrow[""{name=3, anchor=center, inner sep=0}, "{\mathbb{F}_u}"', curve={height=18pt}, from=1-5, to=1-3]
	\arrow["\dashv"{anchor=center, rotate=-90}, draw=none, from=1, to=0]
	\arrow["\dashv"{anchor=center, rotate=-90}, draw=none, from=3, to=2]
\end{tikzcd}
\end{equation}
where $\Cat{Dbl}_s$ is the category of strict double categories and strict double functors between them. Explicitly, $\Dbl{F}_c$ sends a unital VDC $\Dbl{D}$ to the strict double category $\Dbl{F}_c(\Dbl{D})$ with the same underlying category as $\Dbl{D}$, and with loose arrows given by finite sequences of loose arrows in $\Dbl{D}$, while multicells are given by finite sequences of multicells in $\Dbl{D}$.

Since $\Cat{fc}(\Dbl{D})^{op_t}$ is a strict double category, by Proposition~\ref{prop:RepImpProRep} it is exponentiable in $\Vdc$, so we can transpose the hom VDF to a VDF $\rho_{\Dbl{D}}:\Dbl{D}\to \Span^{\Cat{fc}(\Dbl{D})^{op_t}}$. On underlying tight categories $\rho_{\Dbl{D}}$ restricts to the ordinary Yoneda embedding for $\Dbl{D}_0$. A loose arrow $\psi:y\proto z$ is sent by $\rho_\Dbl{D}$ to the functor $\Cat{fc}(\Dbl{D})_1^{op}\to \Span_1$ that maps a sequence of loose arrows $\prolist{\varphi}:x_0\proto x_n$ to the span with apex the set of multicells with loose source $\prolist{\varphi}$ and loose target $\psi$, while mapping a sequence of multicells to the function given by pre-composition by the multicells.

It remains to show that this VDF is a regular monomorphism in the category of large VDCs. 
The following result implies that it suffices to show it is fully-faithful, injective on objects, and injective on loose arrows. Here we say that a VDF $i:\Dbl{A}\to \Dbl{D}$ is fully-faithful if it is fully-faithful on underlying tight categories, and for every loose source $\prolist{\varphi}$ and loose target $\psi$ in $\Dbl{A}$, the induced map on multicells $\MCell (\Dbl{A})(\prolist{\varphi},\psi)\to \MCell (\Dbl{B})(i(\prolist{\varphi}),i(\psi))$ is bijective.

\begin{lem}[Fully-Faithful and Injective on Objects VDFs are Regular Monomorphisms]\label{lem:regMon}
	Let $i:\Dbl{A}\hookrightarrow \Dbl{B}$ be a fully-faithful, injective on objects, and injective on loose arrows VDF. Then $i$ is the equalizer of its cokernel pair.
\end{lem}
\begin{proof}
	Let $q_1,q_2:\Dbl{B}\rightrightarrows\Dbl{B}\cup_{\Dbl{A}}\Dbl{B}$ denote the cokernel pair for the fully-faithful, injective on objects, and injective on loose arrows VDF $i:\Dbl{A}\hookrightarrow \Dbl{B}$. Then if $F:\Dbl{C}\to \Dbl{B}$ is a VDF such that $q_1\circ F = q_2\circ F$, for each object $c$ in $\Dbl{C}$, $F(c)$ must be in the image of $\Dbl{A}$, since the only objects identified between the two copies of $\Dbl{B}$ in the pushout $\Dbl{B}\cup_{\Dbl{A}}\Dbl{B}$ are those in the image of $\Dbl{A}$. Similarly, for each loose arrow $\varphi$ in $\Dbl{C}$, $F(\varphi)$ must be in the image of $\Dbl{A}$. But since $i$ is fully-faithful, these two results together imply $F(\Dbl{C})\subseteq i(\Dbl{A})$, and as $i$ is also injective on objects and loose arrows $F$ uniquely specifies a VDF $F':\Dbl{C}\to \Dbl{A}$ such that $i\circ F' = F$, as desired.
\end{proof}

To prove that $\rho_{\Dbl{D}}$ is fully-faithful, injective on objects, and injective on loose arrows, we will first show an analogue of the Yoneda lemma. In the following statement $\rho_\Dbl{D}(\prolist{\varphi})$ for a sequence of loose arrows $\prolist{\varphi}$ denotes the sequence of loose arrows $(\rho_\Dbl{D}(\varphi_1) \cdots \rho_\Dbl{D}(\varphi_n))$ in $\Span^{\Cat{fc}(\Dbl{D})^{op_t}}$.

\begin{lem}[Yoneda Lemma for VDCs]\label{lem:Yoneda}
	For a VDC $\Dbl{D}$, and a sequence of loose arrows $\prolist{\varphi}:x_0\proto x_n$ in $\Dbl{D}$, 
  evaluation at the sequence of identity multicells 
  $\id_{\prolist{\varphi}}=(\id_{\varphi_1} \cdots \id_{\varphi_n})$ 
  viewed as a unary multicell in $\Cat{fc}(\Dbl{D})$ induces a natural isomorphism 
	\begin{equation*}
		\Span^{\Cat{fc}(\Dbl{D})^{op_t}}\scell{\rho_{\Dbl{D}}(\prolist{\varphi})}{\Phi}\xrightarrow{\text{ev}_{\id_{\prolist{\varphi}}}}\Phi(\prolist{\varphi})
	\end{equation*}
	where $\Phi:F\proto G$ is an arbitrary loose morphism in $\Span^{\Cat{fc}(\Dbl{D})^{op_t}}$.
\end{lem}

Note that Lemma~\ref{lem:Yoneda} applied to the case where $\Phi$ is of the form $\rho_\Dbl{D}(\psi)$ for some loose arrow $\psi:y_0\proto y_1$ in $\Dbl{D}$ gives fully-faithfullness of $\rho_{\Dbl{D}}$. We will first give a direct proof of Lemma~\ref{lem:Yoneda} using the classical Yoneda lemma for 1-categories.

\begin{proof}[Proof of Lemma~\ref{lem:Yoneda}]
	Using the description of multicells in $\Cat{fc}(\Dbl{D})$ as sequences of multicells in $\Dbl{D}$, multicells in 
  $\Span^{\Cat{fc}(\Dbl{D})^{op_t}}$ with loose source $\rho_{\Dbl{D}}(\prolist{\varphi})$ and loose target 
  $\Phi$ are equivalent to the data of natural transformations $\gamma_0:\uTight(\Dbl{D})(-,x_0)\Rightarrow F$, 
  $\gamma_1:\uTight(\Dbl{D})(-,x_n)\Rightarrow G$, and 
  $\Gamma:\Cat{fc}(\Dbl{D})_1(-,\prolist{\varphi})\Rightarrow \Phi$ which are compatible under whiskering 
  with source and target functors. 
  In particular, the claim now follows by the ordinary Yoneda lemma for the categories $\Dbl{D}_0$ and 
  $\Cat{fc}(\Dbl{D})_1$.
\end{proof}

More generally this result can be phrased and proved purely in the language of monoids in a VDC with restrictions. So as not to interrupt the current exposition on exponentiability, we leave this discussion to Appendix~\ref{sec:YonedaForMon}.

Next, observe that $\rho_{\Dbl{D}}$ is also injective on objects, and similarly for loose arrows, since for any two distinct objects $x$ and $x'$, 
the representable functors $\uTight(\Dbl{D})_{-,x}$ and $\uTight(\Dbl{D})_{-,x'}$ are distinct. Therefore, for each VDC $\Dbl{D}$, we have a fully-faithful, injective on objects, and injective on loose arrows VDF $\rho_{\Dbl{D}}:\Dbl{D}\hookrightarrow \Span^{\Cat{fc}(\Dbl{D})^{op_t}}$, which by Lemma~\ref{lem:regMon} is a regular monomorphism. The desired exponentiability result now follows from Proposition~\ref{prop:ExpByUnivExample}:

\begin{lem}[Exponentiability of VDCs via Span Exponential]\label{lem:ExpViaSpan}
	A VDC $\Dbl{D}$ is exponentiable in $\Vdc$ if and only if the VDC $\Span^{\Dbl{D}}$ exists in the category of large VDCs.
\end{lem}

This completes the proof of Theorem~\ref{thm:ExpChar}. 

%% file: Sections/6_ExamplesAndCounterexamples.tex
\subsection{Examples of Exponentiable VDCs}

As a consequence of Corollary~\ref{prop:RepImpProRep}, for any category $\Cat{C}$, the tight VDCs $\T(\Cat{C})$ and $\T_u(\Cat{C})$ are both exponentiable, 
and the VDC $\L(\BiCat{B})$ is exponentiable for any bicategory $\BiCat{B}$. For example, the walking object VDC $\Ob$ is exponentiable, 
and $\Dbl{E}^{\Ob}=\Dbl{C}\mathsf{haotic}(\Dbl{E}_0)$ is the chaotic VDC associated to the category $\Dbl{E}_0.$  
On the other hand, because $\Loose $ the non-unital loose arrow is not exponentiable (c.f.~Example~\ref{eg:looseCounter}),
the functor $(-)_1$ does \textit{not} admit a right adjoint, and so we don't get analogous chaotic VDCs of loose arrows in $\Dbl{E}.$

The following lemma allows us to construct some interesting examples of exponentiable VDCs by hand.
\begin{lem}\label{lem:expOfUnitalization}
	Let $\Dbl{D}$ be a VDC with no nullary multicells. If $\Dbl{D}$ satisfies both (P1) and (P3) of
	Definition~\ref{defn:proRep} with $T$ restricted to range over binary-unary trees,\footnote{That is, $\Dbl{D}$ admits 
  essentially unique decompositions whose shape is a binary-unary tree of height 2 (Definition \ref{defn:binaryNullaryTree}), and
  every height-2 decomposition of any shape essentially arises from one shaped by a binary-unary tree.}
	then the unitalization $\Dbl{F}_u(\Dbl{D})$ is exponentiable.
\end{lem}
\begin{proof}
	By definition of the multicells in the unitalization $\Dbl{F}_u(\Dbl{D})$, it still satisfies (P1) of 
	Definition~\ref{defn:proRep} for binary-unary trees $T$. 
	Thus, by the equivalence $(1)\iff(2)$ in Theorem~\ref{thm:ExpChar} it suffices to show 
	property (P1) is satisfied for height $2$ unary-nullary trees $T$, 
	and that property (P3) for arbitrary height $2$ trees $T'$ is not broken by applying $\Dbl{F}_u.$

	First, for any multicell $\alpha$, as below, the universal property of opcartesian multicells implies that for any partition of the loose source $\prolist{\varphi}=(\prolist{\varphi}_1,\prolist{\varphi}_2)$, we have the decomposition below right:
	\[
	\begin{tikzcd}
	{x_0} && {x_n} && {x_0} && {x_i} && {x_n} \\
	&&& {=} & {x_0} & {x_i} && {x_i} & {x_n} \\
	{y_0} && {y_1} && {y_0} &&&& {y_1}
	\arrow[""{name=0, anchor=center, inner sep=0}, "{{\prolist{\varphi}}}"{inner sep=.8ex}, "\shortmid"{marking}, from=1-1, to=1-3]
	\arrow["a"', from=1-1, to=3-1]
	\arrow["b", from=1-3, to=3-3]
	\arrow[""{name=1, anchor=center, inner sep=0}, "{{\prolist{\varphi}_1}}"{inner sep=.8ex}, "\shortmid"{marking}, from=1-5, to=1-7]
	\arrow[equals, from=1-5, to=2-5]
	\arrow[""{name=2, anchor=center, inner sep=0}, "{{\prolist{\varphi}_2}}"{inner sep=.8ex}, "\shortmid"{marking}, from=1-7, to=1-9]
	\arrow[equals, from=1-7, to=2-6]
	\arrow[equals, from=1-7, to=2-8]
	\arrow[equals, from=1-9, to=2-9]
	\arrow[""{name=3, anchor=center, inner sep=0}, "{{\prolist{\varphi}_1}}"'{inner sep=.8ex}, "\shortmid"{marking}, from=2-5, to=2-6]
	\arrow["a"', from=2-5, to=3-5]
	\arrow[""{name=4, anchor=center, inner sep=0}, "{{I_{x_i}}}"'{inner sep=.8ex}, "\shortmid"{marking}, from=2-6, to=2-8]
	\arrow[""{name=5, anchor=center, inner sep=0}, "{{\prolist{\varphi}_2}}"'{inner sep=.8ex}, "\shortmid"{marking}, from=2-8, to=2-9]
	\arrow["b", from=2-9, to=3-9]
	\arrow[""{name=6, anchor=center, inner sep=0}, "\psi"'{inner sep=.8ex}, "\shortmid"{marking}, from=3-1, to=3-3]
	\arrow[""{name=7, anchor=center, inner sep=0}, "\psi"'{inner sep=.8ex}, "\shortmid"{marking}, from=3-5, to=3-9]
	\arrow["\alpha"{description}, draw=none, from=0, to=6]
	\arrow["{\text{id}_{\prolist{\varphi}_1}}"{description}, draw=none, from=1, to=3]
	\arrow["{\mathsf{opcart}}"{description}, draw=none, from=1-7, to=4]
	\arrow["{\text{id}_{\prolist{\varphi}_2}}"{description}, draw=none, from=2, to=5]
	\arrow["{\alpha^\flat}"{description, pos=0.7}, draw=none, from=4, to=7]
\end{tikzcd}
	\]
	Additionally, as in the proof of Proposition~\ref{prop:RepImpProRep} these decompositions are terminal 
	in a sense that guarantees the essential uniqueness clause of (P1). Thus (P1) holds for $\Dbl{F}_u(\Dbl{D}).$ 

  Now it remains to prove that property (P3) holds for $\Dbl{F}_u(\Dbl{D}).$ Fix an arbitrary height $2$ tree $T'$.
	First, note that by definition of $\Dbl{F}_u(-)$, 
	each multicell in $\Dbl{F}_u(\Dbl{D})$ arises from a canonical multicell or tight arrow in $\Dbl{D}$ 
	after forgetting the freely added loose units.
	Then a decomposition $\frac{\prolist{\alpha}}{\beta}$ in $\Dbl{F}_u(\Dbl{D})$ 
	arises from either a decomposition of a multicell in $\Dbl{D}$ or from
	a factorization of a tight arrow, as in the following example, assuming $a=c\circ b$:	
\[\begin{tikzcd}
	&&&& x & x & x & \\
	x & x & x && z && z & z \\
	y && y && y &&& y
	\arrow["{I_x}"{inner sep=.8ex}, "\shortmid"{marking}, from=1-5, to=1-6]
	\arrow["b"', from=1-5, to=2-5]
	\arrow["{I_x}"{inner sep=.8ex}, "\shortmid"{marking}, from=1-6, to=1-7]
	\arrow["b"', from=1-7, to=2-7]
	\arrow["b", from=1-7, to=2-8]
	\arrow["{I_x}"{inner sep=.8ex}, "\shortmid"{marking}, from=2-1, to=2-2]
	\arrow["a"', from=2-1, to=3-1]
	\arrow["{I_x}"{inner sep=.8ex}, "\shortmid"{marking}, from=2-2, to=2-3]
	\arrow[""{name=0, anchor=center, inner sep=0}, "a", from=2-3, to=3-3]
	\arrow["{I_z}"{inner sep=.8ex}, "\shortmid"{marking}, from=2-5, to=2-7]
	\arrow[""{name=1, anchor=center, inner sep=0}, "c"', from=2-5, to=3-5]
	\arrow["{I_z}"{pos=0.4, inner sep=.8ex}, "\shortmid"{marking}, from=2-7, to=2-8]
	\arrow["c", from=2-8, to=3-8]
	\arrow["{I_y}"'{inner sep=.8ex}, "\shortmid"{marking}, from=3-1, to=3-3]
	\arrow["{I_y}"'{inner sep=.8ex}, "\shortmid"{marking}, from=3-5, to=3-8]
	\arrow[between={0.4}{0.6}, equals, from=0, to=1]
\end{tikzcd}\]
  
  We next further decompose $\frac{\prolist{\alpha}}{\beta}$ into three layers 
	$\threefrac{\prolist{\gamma}}{\prolist{\delta}}{\beta}$ such that the loose domains 
  of $\prolist{\delta}$ contain no identities $I_x,$ except once per 
  $I_y$ in the domains of $\beta,$ where such an occurrence is forced. 
  Such a decomposition involves some choices; one possible set of such 
  choices is exemplified below, where multicells in $\prolist{\alpha}$ with identity 
  codomain are factored canonically through a unary, while identity proarrows in other 
  multicells of $\prolist{\alpha}$ are absorbed in the nonidentity proarrow to their left, 
  if possible, and otherwise to their right. Note that these cases are exhaustive, due to the 
  assumption that $\Dbl{D}$ has no nullary multicells, so that if a cell in 
  $\prolist{\alpha}$ has an identity codomain, its domains must all be identites.
\[\begin{tikzcd}
	x & x & x & x & x & y & y & z & z \\
	y && y & w &&&&& u \\
	x & x & x & x & x & y & y & z & z \\
	x && x & x &&& y && z \\
	y && y & w &&&&& u
	\arrow["\shortmid"{marking}, equals, from=1-1, to=1-2]
	\arrow["f"', from=1-1, to=2-1]
	\arrow["\shortmid"{marking}, equals, from=1-2, to=1-3]
	\arrow["f", from=1-3, to=2-3]
	\arrow["\shortmid"{marking}, equals, from=1-4, to=1-5]
	\arrow["f"', from=1-4, to=2-4]
	\arrow["\varphi"{inner sep=.8ex}, "\shortmid"{marking}, from=1-5, to=1-6]
	\arrow["\shortmid"{marking}, equals, from=1-6, to=1-7]
	\arrow["\psi"{inner sep=.8ex}, "\shortmid"{marking}, from=1-7, to=1-8]
	\arrow["\shortmid"{marking}, equals, from=1-8, to=1-9]
	\arrow["g", from=1-9, to=2-9]
	\arrow[""{name=0, anchor=center, inner sep=0}, "\shortmid"{marking}, equals, from=2-1, to=2-3]
	\arrow[""{name=1, anchor=center, inner sep=0}, "\chi"{inner sep=.8ex}, "\shortmid"{marking}, from=2-4, to=2-9]
	\arrow["\shortmid"{marking}, equals, from=3-1, to=3-2]
	\arrow[equals, from=3-1, to=4-1]
	\arrow["\shortmid"{marking}, equals, from=3-2, to=3-3]
	\arrow[equals, from=3-3, to=4-3]
	\arrow["\shortmid"{marking}, equals, from=3-4, to=3-5]
	\arrow[equals, from=3-4, to=4-4]
	\arrow["\varphi"{inner sep=.8ex}, "\shortmid"{marking}, from=3-5, to=3-6]
	\arrow[""{name=2, anchor=center, inner sep=0}, "\shortmid"{marking}, equals, from=3-6, to=3-7]
	\arrow["\psi"{inner sep=.8ex}, "\shortmid"{marking}, from=3-7, to=3-8]
	\arrow[equals, from=3-7, to=4-7]
	\arrow["\shortmid"{marking}, equals, from=3-8, to=3-9]
	\arrow[equals, from=3-9, to=4-9]
	\arrow["\shortmid"{marking}, equals, from=4-1, to=4-3]
	\arrow["f"', from=4-1, to=5-1]
	\arrow["f", from=4-3, to=5-3]
	\arrow["\varphi"{inner sep=.8ex}, "\shortmid"{marking}, from=4-4, to=4-7]
	\arrow["f"', from=4-4, to=5-4]
	\arrow["\psi"{inner sep=.8ex}, "\shortmid"{marking}, from=4-7, to=4-9]
	\arrow["g", from=4-9, to=5-9]
	\arrow["\shortmid"{marking}, equals, from=5-1, to=5-3]
	\arrow["\chi"', from=5-4, to=5-9]
	\arrow[between={0.4}{0.9}, maps to, from=0, to=3-2]
	\arrow[between={0.3}{0.7}, maps to, from=1, to=2]
\end{tikzcd}\]

	where the top layer $\prolist{\gamma}$ consists only of opcartesian multicells
	and the center layer $\prolist{\delta}$ consists only of multicells in $\Dbl{D}$
	and possibly unary multicells of the form $I_a$ for $a$ a tight arrow in $\Dbl{D}$.
  
	Since $\Dbl{D}$ satisfies (P3), $\frac{\prolist{\delta}}{\beta}$ can
	be decomposed into a binary-unary tree shape.
	Further, since $\prolist{\gamma}$ only consists of opcartesian multicells,
	which necessarily all have identity tight boundaries,
	we can bring each non-identity opcartesian cell in $\prolist{\gamma}$
	into its own separate layer.
	Finally, we can decompose each opcartesian multicells of arity $\geq 3$ 
	as composites of opcartesian multicells of arity $2$ and $1$, producing a sequence of binary-unary trees
  which we graft together to show that $\Dbl{F}_u(\Dbl{D})$ satisfies property (P3) of 
  Definition~\ref{defn:proRep}, and hence is exponentiable, as desired.	
\end{proof}

As an immediate consequence we can determine which of the unitalizations $\Dbl{F}_u(\Sq{n})$ are exponentiable:

\begin{cor}\label{cor:expUnits}
	The VDC $\Dbl{F}_u(\Sq{n})$ is exponentiable for $0\leq n\leq 2$, and is non-exponentiable for $n\geq 3$.
\end{cor}
\begin{proof}
	First, observe that when $n\geq 3$, Theorem~\ref{thm:ExpViaDecomp} implies that the VDC $\Dbl{F}_u(\Sq{n})$ is 
  non-exponentiable since we can't factor out a binary multicell from the unique $n$-ary multicell coming from $\Sq{n}$. 
  As a consequence, we see that unital VDCs need not be exponentiable.

	For the first part of the claim note that $\Dbl{F}_u(\Sq{0})$ and $\Dbl{F}_u(\Sq{1})$ are both representable, 
  and hence are exponentiable by Proposition~\ref{prop:RepImpProRep}. However, the VDC $\Dbl{F}_u(\Sq{2})$ is not representable. 
  Nonetheless, $\Sq{2}$, as it only has identity multicells and a single non-identity binary multicell, 
  satisfies the hypotheses of Lemma~\ref{lem:expOfUnitalization}: the only binary-unary trees of height $2$ 
  whose composite is less than ternary are the three below, with respect to which every virtual double category admits essentially unique decompositions.
  Furthermore, the same shapes gives the only height-2 trees for which $\Sq{2}$ admits any cell decompositions whatsoever, so that 
  decompositions of height $2$ serve as their own binary-unary refinements.
  Thus, by Lemma~\ref{lem:expOfUnitalization} we have that $\Dbl{F}_u(\Sq{2})$ is exponentiable.
\[\begin{tikzcd}
	{\color{white}} & {\color{white}\bullet} && {\color{white}} & {\color{white}} && {\color{white}} \\
	\bullet && \bullet && \bullet && \bullet \\
	\bullet && \bullet &&& \bullet \\
	{\color{white}} && {\color{white}} &&& {\color{white}}
	\arrow[no head, from=1-1, to=2-1]
	\arrow[no head, from=1-2, to=2-3]
	\arrow[no head, from=1-4, to=2-3]
	\arrow[no head, from=1-5, to=2-5]
	\arrow[no head, from=1-7, to=2-7]
	\arrow[no head, from=2-1, to=3-1]
	\arrow[no head, from=2-3, to=3-3]
	\arrow[no head, from=2-5, to=3-6]
	\arrow[no head, from=2-7, to=3-6]
	\arrow[no head, from=3-1, to=4-1]
	\arrow[no head, from=3-3, to=4-3]
	\arrow[no head, from=3-6, to=4-6]
\end{tikzcd}\]
\end{proof}

\subsubsection*{VDCs of cospans}
An important class of exponentiable VDCs that do not arise as representable VDCs are the cospan VDCs.

\begin{prop}[Cospans give exponentiable VDCs]\label{prop:cospanPro}
    For any category $\Cat{E}$, the VDC $\Cospan(\Cat{E})$ is exponentiable, and is representable if and only if $\Cat{E}$ admits finite pushouts.
\end{prop}
\begin{proof}
    To begin consider an $n$-ary multicell in $\Cospan(\Cat{E})$ which equates to a commuting diagram of the form
\[\begin{adjustbox}{}\begin{tikzcd}
	{x_0} & {y_1} & {x_1} & \ldots & {x_{n-1}} & {y_n} & {x_n} \\
	&&& \ldots \\
	{z_0} &&& {w_1} &&& {z_1}
	\arrow["{{a_1}}", from=1-1, to=1-2]
	\arrow["{{c_0}}"', from=1-1, to=3-1]
	\arrow["{{e_1}}"', from=1-2, to=3-4]
	\arrow["{{b_1}}"', from=1-3, to=1-2]
	\arrow["{{a_2}}", from=1-3, to=1-4]
	\arrow["{{b_{n-1}}}"', from=1-5, to=1-4]
	\arrow["{{a_n}}", from=1-5, to=1-6]
	\arrow["{{e_n}}", from=1-6, to=3-4]
	\arrow["{{b_n}}"', from=1-7, to=1-6]
	\arrow["{{c_1}}", from=1-7, to=3-7]
	\arrow["{d_0}"', from=3-1, to=3-4]
	\arrow["{d_1}", from=3-7, to=3-4]
\end{tikzcd}\end{adjustbox}\]
and let $f_i=e_i\circ b_i=e_{i+1}\circ a_{i+1}:x_i\to w_1$ be the unique maps for the middle $x_i$'s. 
First let's show existence of height-2 decompositions. Let $n=k_1+\cdots+k_m$ for $k_i \geq 0$. 
If $m=1$ we can just insert an identity unary multicell below the multicell. If $m \geq 2$, and $k_1,k_m\geq 1$, then we can factor our multicell as
\[\begin{tikzcd}[column sep=tiny]
	{x_0} & {y_1} & \cdots & {y_{k_1}} & {x_{k_1}} & \cdots & {x_{m_{n-1}}} & {y_{m_{n-1}+1}} & \cdots & {y_{m_n}} & {x_{m_n}} \\
	&&&&& \cdots \\
	{z_0} && {w_1} && {w_1} & \cdots & {w_1} && {w_1} && {z_1} \\
	\\
	{z_0} &&&&& {w_1} &&&&& {z_1}
	\arrow[from=1-1, to=1-2]
	\arrow["{{{c_0}}}", from=1-1, to=3-1]
	\arrow["{{{e_1}}}"{description}, from=1-2, to=3-3]
	\arrow[from=1-3, to=1-2]
	\arrow[from=1-3, to=1-4]
	\arrow["{{{e_{k_1}}}}"{description}, from=1-4, to=3-3]
	\arrow[from=1-5, to=1-4]
	\arrow[from=1-5, to=1-6]
	\arrow["{{{f_{k_1}}}}"{description}, from=1-5, to=3-5]
	\arrow[from=1-7, to=1-6]
	\arrow[from=1-7, to=1-8]
	\arrow["{{{f_{m_{n-1}}}}}"{description}, from=1-7, to=3-7]
	\arrow["{{{e_{m_{n-1}+1}}}}"{description}, from=1-8, to=3-9]
	\arrow[from=1-9, to=1-8]
	\arrow[from=1-9, to=1-10]
	\arrow["{{{e_{m_n}}}}"{description}, from=1-10, to=3-9]
	\arrow[from=1-11, to=1-10]
	\arrow["{{{c_1}}}"', from=1-11, to=3-11]
	\arrow["{{{d_0}}}"', from=3-1, to=3-3]
	\arrow[equals, from=3-1, to=5-1]
	\arrow[equals, from=3-3, to=3-5]
	\arrow[equals, from=3-3, to=5-6]
	\arrow[equals, from=3-5, to=3-6]
	\arrow[equals, from=3-5, to=5-6]
	\arrow[equals, from=3-6, to=3-7]
	\arrow[equals, from=3-7, to=3-9]
	\arrow[equals, from=3-7, to=5-6]
	\arrow[equals, from=3-9, to=5-6]
	\arrow["{{{d_1}}}", from=3-11, to=3-9]
	\arrow[equals, from=3-11, to=5-11]
	\arrow["{{{d_0}}}"', from=5-1, to=5-6]
	\arrow["{{{d_1}}}", from=5-11, to=5-6]
\end{tikzcd}\]
where $m_i=\sum_{j=1}^ik_j$, and for $1 < i < m$, if $k_i=0$ we insert the nullary multicell
\[\begin{adjustbox}{}\begin{tikzcd}
	& {x_{m_{i-1}}} \\
	\\
	{y_1} & {y_1} & {y_1}
	\arrow["{{f_{m_{i-1}}}}"', from=1-2, to=3-1]
	\arrow["{{f_{m_{i-1}}}}"{description}, from=1-2, to=3-2]
	\arrow["{{f_{m_{i-1}}}}", from=1-2, to=3-3]
	\arrow[equals, from=3-1, to=3-2]
	\arrow[equals, from=3-2, to=3-3]
\end{tikzcd}\end{adjustbox}\]
If $k_1=0$ (or similarly for $k_m=0$) we insert the nullary multicell
\[\begin{adjustbox}{}\begin{tikzcd}
	& {x_0} \\
	\\
	{z_0} & {z_0} & {z_0}
	\arrow["{{c_0}}"', from=1-2, to=3-1]
	\arrow["{{c_0}}"{description}, from=1-2, to=3-2]
	\arrow["{{c_0}}", from=1-2, to=3-3]
	\arrow[equals, from=3-1, to=3-2]
	\arrow[equals, from=3-2, to=3-3]
\end{tikzcd}\end{adjustbox}\]
on the left (resp. right) so that the bottom multicell becomes
\[\begin{tikzcd}[column sep=tiny,row sep=large]
	{z_0} & {z_0} & {z_0} && {w_1} && {w_1} & \cdots & {w_1} && {w_1} && {z_1} & {z_1} & {z_1} \\
	\\
	{z_0} &&&&&&& {w_1} &&&&&&& {z_1}
	\arrow[equals, from=1-1, to=1-2]
	\arrow[equals, from=1-1, to=3-1]
	\arrow[equals, from=1-2, to=1-3]
	\arrow["{{d_0}}"{description}, from=1-2, to=3-8]
	\arrow["{{d_0}}"', from=1-3, to=1-5]
	\arrow[equals, from=1-5, to=1-7]
	\arrow[equals, from=1-5, to=3-8]
	\arrow[equals, from=1-7, to=1-8]
	\arrow[equals, from=1-7, to=3-8]
	\arrow[equals, from=1-8, to=1-9]
	\arrow[equals, from=1-9, to=1-11]
	\arrow[equals, from=1-9, to=3-8]
	\arrow[equals, from=1-11, to=3-8]
	\arrow["{{d_1}}", from=1-13, to=1-11]
	\arrow[equals, from=1-13, to=1-14]
	\arrow["{{d_1}}"{description}, from=1-14, to=3-8]
	\arrow[equals, from=1-15, to=1-14]
	\arrow[equals, from=1-15, to=3-15]
	\arrow["{{d_0}}"', from=3-1, to=3-8]
	\arrow["{{d_1}}", from=3-15, to=3-8]
\end{tikzcd}\]
So far we have proven that single multicells admit arbitrary decompositions. 
It remains to show that these decompositions are equivalent. 
However, this follows from the fact that any other decomposition is equivalent to this one by sliding the bottom multicell up. 

For example, in the $3=2+1$ case an arbitrary decomposition below left can be decomposed as in the lower diagram below:
\[\begin{tikzcd}
	{x_0} & {y_1} & {x_1} & {y_2} & {x_2} & {y_3} & {x_3} \\
	{x_0'} && {y_1'} && {x_1'} & {y_2'} & {x_2'} \\
	{z_0} &&& {w_1} &&& {z_1} \\
	\\
	{x_0} & {y_1} & {x_1} & {y_2} & {x_2} & {y_3} & {x_3} \\
	{x_0'} && {y_1'} && {x_1'} & {y_2'} & {x_2'} \\
	\\
	{z_0} && {w_1} && {w_1} & {w_1} & {z_1} \\
	{z_0} &&& {w_1} &&& {z_1}
	\arrow["{{{{a_1}}}}", from=1-1, to=1-2]
	\arrow["{{{{c_0'}}}}"', from=1-1, to=2-1]
	\arrow["{{{{e_1'}}}}"{description}, from=1-2, to=2-3]
	\arrow["{{{{b_1}}}}"', from=1-3, to=1-2]
	\arrow["{{{{a_2}}}}", from=1-3, to=1-4]
	\arrow["{{{{e_2'}}}}"{description}, from=1-4, to=2-3]
	\arrow["{{{{b_2}}}}"', from=1-5, to=1-4]
	\arrow["{{{{a_3}}}}", from=1-5, to=1-6]
	\arrow["{{{{c_1'}}}}"{description}, from=1-5, to=2-5]
	\arrow["{{{{e_3'}}}}"{description}, from=1-6, to=2-6]
	\arrow["{{{{b_3}}}}"', from=1-7, to=1-6]
	\arrow["{{{{c_2'}}}}", from=1-7, to=2-7]
	\arrow["{{{{a_1'}}}}", from=2-1, to=2-3]
	\arrow["{{{{c_0''}}}}"', from=2-1, to=3-1]
	\arrow["{{{{e_1''}}}}"{description}, from=2-3, to=3-4]
	\arrow["{{{{b_1'}}}}"', from=2-5, to=2-3]
	\arrow["{{{{a_2'}}}}", from=2-5, to=2-6]
	\arrow["{{{{e_2''}}}}"{description}, from=2-6, to=3-4]
	\arrow["{{{{b_2'}}}}"', from=2-7, to=2-6]
	\arrow["{{{{c_2''}}}}", from=2-7, to=3-7]
	\arrow["{{d_0}}"', from=3-1, to=3-4]
	\arrow[between={0.3}{0.7}, squiggly, tail reversed, from=3-4, to=5-4]
	\arrow["{{d_1}}", from=3-7, to=3-4]
	\arrow["{{{{a_1}}}}", from=5-1, to=5-2]
	\arrow["{{{{c_0'}}}}"', from=5-1, to=6-1]
	\arrow["{{{{e_1'}}}}"{description}, from=5-2, to=6-3]
	\arrow["{{{{b_1}}}}"', from=5-3, to=5-2]
	\arrow["{{{{a_2}}}}", from=5-3, to=5-4]
	\arrow["{{{{e_2'}}}}"{description}, from=5-4, to=6-3]
	\arrow["{{{{b_2}}}}"', from=5-5, to=5-4]
	\arrow["{{{{a_3}}}}", from=5-5, to=5-6]
	\arrow["{{{{c_1'}}}}"{description}, from=5-5, to=6-5]
	\arrow["{{{{e_3'}}}}"{description}, from=5-6, to=6-6]
	\arrow["{{{{b_3}}}}"', from=5-7, to=5-6]
	\arrow["{{{{c_2'}}}}", from=5-7, to=6-7]
	\arrow["{{{{a_1'}}}}"', from=6-1, to=6-3]
	\arrow["{{{{c_0''}}}}"', from=6-1, to=8-1]
	\arrow["{{{{e_1''}}}}"{description}, from=6-3, to=8-3]
	\arrow["{{{{b_1'}}}}", from=6-5, to=6-3]
	\arrow["{{{{a_2'}}}}"', from=6-5, to=6-6]
	\arrow["{{{{c_1'}}}}", from=6-5, to=8-5]
	\arrow["{{{{e_2''}}}}"{description}, from=6-6, to=8-6]
	\arrow["{{{{b_2'}}}}", from=6-7, to=6-6]
	\arrow["{{{{c_2''}}}}", from=6-7, to=8-7]
	\arrow["{{d_0}}"', from=8-1, to=8-3]
	\arrow[equals, from=8-1, to=9-1]
	\arrow[equals, from=8-3, to=8-5]
	\arrow[equals, from=8-3, to=9-4]
	\arrow[equals, from=8-6, to=8-5]
	\arrow[equals, from=8-6, to=9-4]
	\arrow["{{d_1}}", from=8-7, to=8-6]
	\arrow[equals, from=8-7, to=9-7]
	\arrow["{{d_0}}"', from=9-1, to=9-4]
	\arrow["{{d_1}}", from=9-7, to=9-4]
\end{tikzcd}\]
implying that it is equivalent to the canonical decomposition in the lower diagram below:
\[\begin{tikzcd}
	{x_0} & {y_1} & {x_1} & {y_2} & {x_2} & {y_3} & {x_3} \\
	{x_0'} && {y_1'} && {x_1'} & {y_2'} & {x_2'} \\
	\\
	{z_0} && {w_1} && {w_1} & {w_1} & {z_1} \\
	{z_0} &&& {w_1} &&& {z_1} \\
	{x_0} & {y_1} & {x_1} & {y_2} & {x_2} & {y_3} & {x_3} \\
	{z_0} && {w_1} && {w_1} & {w_1} & {z_1} \\
	{z_0} &&& {w_1} &&& {z_1}
	\arrow["{{{{a_1}}}}", from=1-1, to=1-2]
	\arrow["{{{{c_0'}}}}"', from=1-1, to=2-1]
	\arrow["{{{{e_1'}}}}"{description}, from=1-2, to=2-3]
	\arrow["{{{{b_1}}}}"', from=1-3, to=1-2]
	\arrow["{{{{a_2}}}}", from=1-3, to=1-4]
	\arrow["{{{{e_2'}}}}"{description}, from=1-4, to=2-3]
	\arrow["{{{{b_2}}}}"', from=1-5, to=1-4]
	\arrow["{{{{a_3}}}}", from=1-5, to=1-6]
	\arrow["{{{{c_1'}}}}"{description}, from=1-5, to=2-5]
	\arrow["{{{{e_3'}}}}"{description}, from=1-6, to=2-6]
	\arrow["{{{{b_3}}}}"', from=1-7, to=1-6]
	\arrow["{{{{c_2'}}}}", from=1-7, to=2-7]
	\arrow["{{{{a_1'}}}}"', from=2-1, to=2-3]
	\arrow["{{{{c_0''}}}}"', from=2-1, to=4-1]
	\arrow["{{{{e_1''}}}}"{description}, from=2-3, to=4-3]
	\arrow["{{{{b_1'}}}}", from=2-5, to=2-3]
	\arrow["{{{{a_2'}}}}"', from=2-5, to=2-6]
	\arrow["{{{{c_1'}}}}", from=2-5, to=4-5]
	\arrow["{{{{e_2''}}}}"{description}, from=2-6, to=4-6]
	\arrow["{{{{b_2'}}}}", from=2-7, to=2-6]
	\arrow["{{{{c_2''}}}}", from=2-7, to=4-7]
	\arrow["{{d_0}}"', from=4-1, to=4-3]
	\arrow[equals, from=4-1, to=5-1]
	\arrow[equals, from=4-3, to=4-5]
	\arrow[equals, from=4-3, to=5-4]
	\arrow[equals, from=4-6, to=4-5]
	\arrow[equals, from=4-6, to=5-4]
	\arrow["{{d_1}}", from=4-7, to=4-6]
	\arrow[equals, from=4-7, to=5-7]
	\arrow["{{d_0}}"', from=5-1, to=5-4]
	\arrow[between={0.1}{0.9}, squiggly, tail reversed, from=5-4, to=6-4]
	\arrow["{{d_1}}", from=5-7, to=5-4]
	\arrow["{{{{a_1}}}}", from=6-1, to=6-2]
	\arrow["{{{{c_0}}}}"', from=6-1, to=7-1]
	\arrow["{{{{e_1}}}}"{description}, from=6-2, to=7-3]
	\arrow["{{{{b_1}}}}"', from=6-3, to=6-2]
	\arrow["{{{{a_2}}}}", from=6-3, to=6-4]
	\arrow["{{{{e_2}}}}"{description}, from=6-4, to=7-3]
	\arrow["{{{{b_2}}}}"', from=6-5, to=6-4]
	\arrow["{{{{a_3}}}}", from=6-5, to=6-6]
	\arrow["{{{{c_2}}}}"{description}, from=6-5, to=7-5]
	\arrow["{{{{e_3}}}}"{description}, from=6-6, to=7-6]
	\arrow["{{{{b_3}}}}"', from=6-7, to=6-6]
	\arrow["{{{{c_1}}}}", from=6-7, to=7-7]
	\arrow["{{d_0}}"', from=7-1, to=7-3]
	\arrow[equals, from=7-1, to=8-1]
	\arrow[equals, from=7-3, to=8-4]
	\arrow[equals, from=7-5, to=7-3]
	\arrow[equals, from=7-5, to=7-6]
	\arrow[equals, from=7-6, to=8-4]
	\arrow["{{d_1}}", from=7-7, to=7-6]
	\arrow[equals, from=7-7, to=8-7]
	\arrow["{{d_0}}"', from=8-1, to=8-4]
	\arrow["{{d_1}}", from=8-7, to=8-4]
\end{tikzcd}\]
Thus, $\Cospan(\Cat{E})$ has decomposable multicells, and therefore by Theorem~\ref{thm:ExpChar} is exponentiable.
\end{proof}

\subsubsection*{Exponentiable multicategories as exponentiable VDCs}

As mentioned previously, the naming convention for pro-representable VDCs comes from their relation to promonoidal multicategories, 
which are precisely the exponentiable multicategories (c.f.~\cite{Pisani2014}), where monoidal categories are the representable multicategories. 
These exponentiable multicategories can be realized as a certain class of exponentiable VDCs through the fully-faithful embedding
\begin{equation*}
    B:\BiCat{MultCat}\hookrightarrow\BiCat{Vdc}
\end{equation*}
sending a multicategory to the VDC with trivial underlying tight category, one loose arrow for each object, and multicells precisely the multimorphisms in the multicategory. 
The essential image of this embedding consists precisely of those VDCs with trivial underlying tight category. Additionally, this embedding fits into a 2-adjunction
\[\begin{tikzcd}
	{\BiCat{MultCat}} && {\BiCat{Vdc}}
	\arrow[""{name=0, anchor=center, inner sep=0}, "B"', curve={height=18pt}, hook, from=1-1, to=1-3]
	\arrow[""{name=1, anchor=center, inner sep=0}, "{\mathsf{Coll}}"', curve={height=18pt}, from=1-3, to=1-1]
	\arrow["\dashv"{anchor=center, rotate=-90}, draw=none, from=1, to=0]
\end{tikzcd}\]
where $\mathsf{Coll}(\Dbl{D})=\Ob\cup_{\T(\Dbl{D})}\Dbl{D}$ is given by collapsing the underlying tight category of $\Dbl{D}$.

The fact that $B$ preserves exponentials follows from the more general result in~\cite[Prop 3.16]{arkor2025exponentiablevirtualdoublecategories} which shows that $\BiCat{MultCat}$ admits powers by exponentiable VDCs, and $B$ strictly preserves these exponentials. Thus, it remains only to prove that $B$ preserves and reflects exponentiable objects.

\begin{prop}[Exponentiable Multicategories in VDCs]\label{prop:multiCatPro}
    The 2-functor $B:\BiCat{MultCat}\hookrightarrow \BiCat{Vdc}$ preserves and reflects exponentiable objects as well as exponentials.
\end{prop}
\begin{proof}
    Recall that a multicategory $\Cat{M}$ is promonoidal, or equivalently exponentiable, if and only if for all $k_1,...,k_n\geq 0$, the composition function
    \begin{equation*}
        \int^{(c_1,...,c_n) \in \Cat{M}_0^{\times n}}\Cat{M}(c_1,...,c_n;a)\times \Pi_{i=1}^n \Cat{M}(b_{k_{i-1}+1},\ldots,b_{k_i};c_i)\xrightarrow{\circ}\Cat{M}(b_1,...,b_{m_n};a)
    \end{equation*}
    is a bijection, where $m_i = \sum_{j=1}^ik_i$ and we take $k_0=0.$

	However, under the identification $B$ of multicategories with VDCs having trivial underlying category, 
	these composition functions being bijections corresponds exactly to the statement that the associated VDC 
	$B(\Cat{M})$ admits arbitrary decompositions for multicells up to equivalence, or equivalently by Theorem~\ref{thm:ExpChar} is exponentiable. 
	Thus, $\Cat{M}$ is exponentiable as a multicategory if and only if $B(\Cat{M})$ is exponentiable as a VDC.
\end{proof}

\subsection{A non-characterization: pro-double categories \'a la Arkor}

In \cite[Definition 3.13]{arkor2025exponentiablevirtualdoublecategories}, Arkor defines a pro-double category as follows:
\begin{defn}[Arkor](\textrm{Pro-double category})\label{defn:ArkorProDoubleCat}
  A pro-double category is a pseudocategory object in the bicategory of small categories, profunctors, and natural transformations, with representable source and target.
\qed
\end{defn}
This involves giving a graph of categories $s,t:\Dbl{D}_1\rightrightarrows \Dbl{D}_0,$ an ``identity''
profunctor $I:\Dbl{D}_1\proto \Dbl{D}_0,$ and a ``composition'' profunctor
$\odot:\Dbl{D}_2\to \Dbl{D}_1,$ respecting sources and targets, 
where $\Dbl{D}_n$ is the iterated pullback of $\Dbl{D}_1$ over $\Dbl{D}_0$
and we take the pullback in $\Cat{Cat}$ rather than (where it need not exist) in $\Cat{Prof}.$
As profunctors, $I$ and $\odot$ comprise the data of nullary and binary multicells in a hypothetical 
exponentiable VDC associated to the pro-double category $\Dbl{D}$. 

As a pseudo-category object, such a $\Dbl{D}$ also contains unitors and, critically, an associator as 
below:
\[\begin{tikzcd}
	{\Dbl{D}_3} & {\Dbl{D}_2} & {\Dbl{D}_1} \\
	{\Dbl{D}_3} & {\Dbl{D}_2} & {\Dbl{D}_1}
	\arrow["{\odot\times_{\Dbl{D}_0}\Dbl{D}_1}"{inner sep=.8ex}, "\shortmid"{marking}, from=1-1, to=1-2]
	\arrow[equals, from=1-1, to=2-1]
	\arrow["\odot"{inner sep=.8ex}, "\shortmid"{marking}, from=1-2, to=1-3]
	\arrow["\alpha"{description}, draw=none, from=1-2, to=2-2]
	\arrow[equals, from=1-3, to=2-3]
	\arrow["{\Dbl{D}_1\times_{\Dbl{D}_0}\odot}"{inner sep=.8ex}, "\shortmid"{marking}, from=2-1, to=2-2]
	\arrow["\odot"{inner sep=.8ex}, "\shortmid"{marking}, from=2-2, to=2-3]
\end{tikzcd}.\]

It is left implicit in Arkor's definition that, when $\alpha$ is whiskered with $s$ and $t,$ 
the result should be the identity natural transformation, but this requirement appears clearly in (1.4) of 
\cite{MartinsFerreira2006}, on which Arkor, Bourke, and Ko rely in their treatment~\cite{arkor2024enhanced2categoricalstructurestwodimensional} of 
partly strict 2-categorial limit sketches. 

Thus if we draw a value of $\odot$ as a binary multicell and a value of the hom profunctor of $\Dbl{D}_1$ as a unary 
cell, then schematically, the associator$\alpha$ gives an equivalence between the two ways of building ternary 
cells out of binary ones: 
\[\begin{tikzcd}
	\bullet & \bullet & \bullet & \bullet && \bullet & \bullet & \bullet & \bullet \\
	\bullet && \bullet & \bullet && \bullet & \bullet && \bullet \\
	\bullet &&& \bullet && \bullet &&& \bullet
	\arrow["\shortmid"{marking}, from=1-1, to=1-2]
	\arrow[from=1-1, to=2-1]
	\arrow["\shortmid"{marking}, from=1-2, to=1-3]
	\arrow["\shortmid"{marking}, from=1-3, to=1-4]
	\arrow[from=1-3, to=2-3]
	\arrow[from=1-4, to=2-4]
	\arrow["\shortmid"{marking}, from=1-6, to=1-7]
	\arrow[from=1-6, to=2-6]
	\arrow["\shortmid"{marking}, from=1-7, to=1-8]
	\arrow[from=1-7, to=2-7]
	\arrow["\shortmid"{marking}, from=1-8, to=1-9]
	\arrow[from=1-9, to=2-9]
	\arrow["\shortmid"{marking}, from=2-1, to=2-3]
	\arrow[from=2-1, to=3-1]
	\arrow["\shortmid"{marking}, from=2-3, to=2-4]
	\arrow["\alpha", between={0.2}{0.8}, squiggly, tail reversed, from=2-4, to=2-6]
	\arrow[from=2-4, to=3-4]
	\arrow["\shortmid"{marking}, from=2-6, to=2-7]
	\arrow[from=2-6, to=3-6]
	\arrow["\shortmid"{marking}, from=2-7, to=2-9]
	\arrow[from=2-9, to=3-9]
	\arrow["\shortmid"{marking}, from=3-1, to=3-4]
	\arrow["\shortmid"{marking}, from=3-6, to=3-9]
\end{tikzcd}\]

Let us be clear what this means concretely:
if we apply $\alpha$ to a decomposed ternary multicell $\tau$ in $\Dbl{D}$ whose tight source is the path $f_1,f_2,$ then 
$\alpha(\tau)$ will have tight source $f_1',f_2'$ such that $f_1f_2=f_1'f_2',$ and similarly for the tight
targets. But the definition of pseudo-category has no way to ask, what might seem more natural, that $f_i=f_i'$,
since the individual components of the tight source and target paths are not available for a pseudocategory 
in a general bicategory. Indeed, these components are not even well-defined 
for the composition of profunctors, quotienting as it does by the internal action of unary multicells. 

Furthermore, $\alpha$ must satisfy a pentagon coherence condition. One of the sides of the pentagon
must be a map $\beta$ as below: 
\[\begin{tikzcd}
	{\Dbl{D}_4} && {\Dbl{D}_3} & {\Dbl{D}_2} & {\Dbl{D}_1} \\
	{\Dbl{D}_4} && {\Dbl{D}_3} & {\Dbl{D}_2} & {\Dbl{D}_1}
	\arrow["{{\odot\times_{\Dbl{D}_0}\Dbl{D}_1\times_{\Dbl{D}_0}\Dbl{D}_1}}"{inner sep=.8ex}, "\shortmid"{marking}, from=1-1, to=1-3]
	\arrow[equals, from=1-1, to=2-1]
	\arrow["{{\odot\times_{\Dbl{D}_0}\Dbl{D}_1}}"{inner sep=.8ex}, "\shortmid"{marking}, from=1-3, to=1-4]
	\arrow[draw=none, from=1-3, to=2-3]
	\arrow["{{\alpha\times_{\Dbl{D}_0}\Dbl{D}_1}}"{description}, shift left=3, draw=none, from=1-3, to=2-3]
	\arrow["\odot"{inner sep=.8ex}, "\shortmid"{marking}, from=1-4, to=1-5]
	\arrow[equals, from=1-4, to=2-4]
	\arrow[equals, from=1-5, to=2-5]
	\arrow["{{\Dbl{D}_1\times_{\Dbl{D}_0}\odot\times_{\Dbl{D}_0}\Dbl{D}_1}}"'{inner sep=.8ex}, "\shortmid"{marking}, from=2-1, to=2-3]
	\arrow["{{\odot\times_{\Dbl{D}_0}\Dbl{D}_1}}"'{inner sep=.8ex}, "\shortmid"{marking}, from=2-3, to=2-4]
	\arrow["\odot"{inner sep=.8ex}, "\shortmid"{marking}, from=2-4, to=2-5]
\end{tikzcd}\]
Schematically, $\beta$ applies the associator $\alpha$ in the upper-left corner of a quaternary multicell 
as below, without changing the remainder: 
\[\begin{tikzcd}[column sep=small]
	\bullet & \bullet & \bullet & \bullet & \bullet && \bullet & \bullet & \bullet & \bullet & \bullet \\
	\bullet && \bullet & \bullet & \bullet && \bullet & \bullet && \bullet & \bullet \\
	\bullet &&& \bullet & \bullet && \bullet &&& \bullet & \bullet \\
	\bullet &&&& \bullet && \bullet &&&& \bullet
	\arrow["\shortmid"{marking}, from=1-1, to=1-2]
	\arrow[from=1-1, to=2-1]
	\arrow["\shortmid"{marking}, from=1-2, to=1-3]
	\arrow["\shortmid"{marking}, from=1-3, to=1-4]
	\arrow[from=1-3, to=2-3]
	\arrow["\shortmid"{marking}, from=1-4, to=1-5]
	\arrow[from=1-4, to=2-4]
	\arrow[from=1-5, to=2-5]
	\arrow["\shortmid"{marking}, from=1-7, to=1-8]
	\arrow[from=1-7, to=2-7]
	\arrow["\shortmid"{marking}, from=1-8, to=1-9]
	\arrow[from=1-8, to=2-8]
	\arrow["\shortmid"{marking}, from=1-9, to=1-10]
	\arrow["\shortmid"{marking}, from=1-10, to=1-11]
	\arrow[from=1-10, to=2-10]
	\arrow[from=1-11, to=2-11]
	\arrow["\shortmid"{marking}, from=2-1, to=2-3]
	\arrow[from=2-1, to=3-1]
	\arrow["\shortmid"{marking}, from=2-3, to=2-4]
	\arrow["\shortmid"{marking}, from=2-4, to=2-5]
	\arrow[from=2-4, to=3-4]
	\arrow[""{name=0, anchor=center, inner sep=0}, from=2-5, to=3-5]
	\arrow["\shortmid"{marking}, from=2-7, to=2-8]
	\arrow[""{name=1, anchor=center, inner sep=0}, from=2-7, to=3-7]
	\arrow["\shortmid"{marking}, from=2-8, to=2-10]
	\arrow["\shortmid"{marking}, from=2-10, to=2-11]
	\arrow[from=2-10, to=3-10]
	\arrow[from=2-11, to=3-11]
	\arrow["\shortmid"{marking}, from=3-1, to=3-4]
	\arrow[from=3-1, to=4-1]
	\arrow["\shortmid"{marking}, from=3-4, to=3-5]
	\arrow[from=3-5, to=4-5]
	\arrow["\shortmid"{marking}, from=3-7, to=3-10]
	\arrow[from=3-7, to=4-7]
	\arrow["\shortmid"{marking}, from=3-10, to=3-11]
	\arrow[from=3-11, to=4-11]
	\arrow["\shortmid"{marking}, from=4-1, to=4-5]
	\arrow["\shortmid"{marking}, from=4-7, to=4-11]
	\arrow["\beta"{description}, between={0.2}{0.8}, squiggly, tail reversed, from=0, to=1]
\end{tikzcd}\]

However, a concern now arises. As discussed above, while $\alpha$ must preserve the \emph{composite} of 
the tight source and target paths of a ternary multicell, we are unable to ask that it preserve the individual
components of those paths. Consequently, it is unclear that the map $\beta$, as we have attempted to 
define it, even exists, since it must implicitly rewrite also the unary multicells. 

Abstractly, the problem is that, whereas the functor $(-)\times A$ defines a pseudo-double functor 
$\Dbl{P}\mathrm{rof}\to \Dbl{P}\mathrm{rof}$ for any category $A,$ insofar as if 
$H:B\proto C,K:C\proto D,$ then the canonical map 
\[(H\odot K)\times A\to (H\times A)\odot (K\times A):(A\times B\proto A\times D)\]
is an isomorphism, the same need not hold if we set $A,B,C,D,H,K$ over a fifth category $X$ and 
consider the functor $(-)\times_X A:\Dbl{P}\mathrm{rof}/X\to \Dbl{P}\mathrm{rof}.$
\footnote{Here, $\Dbllong{P}{rof}/X$ is the slice double category (\cite[1.7]{Grandis1999}), 
whose tight category is $\Cat/X$ and whose loose arrows are profunctors equipped with a cell to $\operatorname{Hom}_X.$}

This functor is, in fact, only lax:
\begin{exmp}[Pullback is lax]\label{ex:pullbackislax}
Consider the following presentation of a situation in $\Dbl{P}\mathrm{rof}/X,$ where 
$X$ is the path of length $2,$ $A$ the path of length $1,$ and we fix the functor $f:A\to X$ 
picking out the long arrow. Then the restriction cell makes the hom-profunctor $I_A$ into a profunctor over 
$I_X.$ Suppose furthermore that $B,C,D$ are all terminal and equipped with functors 
to $X$ picking out the three objects, and $H:B\proto C,K:C\proto D$ are both generated by a 
single heteromorpism, as drawn below:
\[\begin{tikzcd}
	\bullet & {x_1} & {a_1} \\
	\bullet & {x_2} \\
	\bullet & {x_3} & {a_2}
	\arrow[""{name=0, anchor=center, inner sep=0}, "b", from=1-1, to=1-2]
	\arrow[from=1-2, to=2-2]
	\arrow[""{name=1, anchor=center, inner sep=0}, from=1-3, to=1-2]
	\arrow[from=1-3, to=3-3]
	\arrow[""{name=2, anchor=center, inner sep=0}, "c"{description}, from=2-1, to=2-2]
	\arrow[from=2-2, to=3-2]
	\arrow[""{name=3, anchor=center, inner sep=0}, "d"', from=3-1, to=3-2]
	\arrow[""{name=4, anchor=center, inner sep=0}, from=3-3, to=3-2]
	\arrow["h", between={0.3}{0.7}, Rightarrow, squiggly, from=0, to=2]
	\arrow["f", between={0.2}{0.8}, Rightarrow, squiggly, from=1, to=4]
	\arrow["k", between={0.3}{0.7}, Rightarrow, squiggly, from=2, to=3]
\end{tikzcd}\]
We then compare the two profunctors $(H\odot K)\times_X A,(H\times_X A)\odot (K\times_X A):\{(\bullet,a_1)\}\proto \{(\bullet,a_2)\}.$ The one sends
has a single heteromorphism $(h\odot k,f),$ while the other is empty.
\end{exmp}
This means that, if we have a morphism $H\odot K\to H'\odot K'$ of profunctors over $X,$ 
we cannot in general derive a morphism $(H\times_X A)\odot (K\times_X A)\to (H'\times_X A)\odot (K'\times_X A),$
which is the shape of thing we would need to define $\beta$ as above. Thus it does not seem possible to fully 
realize Arkor's definition of pro-double category, which in a strange
way disproves that branch of the conjecture. 

More precisely, we will prove the following:

\begin{prop}[Not all exponentiable VDCs are pro-double]\label{prop:counterExToProDouble}
	There exists an exponentiable VDC $\Dbl{D}$ for which no morphism of profunctors of the signature of the supposed map $\alpha\times_{\Dbl{D}_0}\Dbl{D}_1$ above exists.
\end{prop}

\begin{proof}[Proof of Proposition~\ref{prop:counterExToProDouble}.]
	Let $T_l$, $I_2$, and $T_r$ be the trees below:
	\[\begin{tikzcd}
	{} && {} & {} & {} & {} & {} && {} \\
	& \bullet && \bullet & \bullet & \bullet && \bullet \\
	&& \bullet && \bullet && \bullet \\
	&& {} && {} && {}
	\arrow[no head, from=2-2, to=1-1]
	\arrow[no head, from=2-2, to=1-3]
	\arrow[no head, from=2-4, to=1-4]
	\arrow[no head, from=2-5, to=1-5]
	\arrow[no head, from=2-6, to=1-6]
	\arrow[no head, from=2-8, to=1-7]
	\arrow[no head, from=2-8, to=1-9]
	\arrow[no head, from=3-3, to=2-2]
	\arrow[no head, from=3-3, to=2-4]
	\arrow[no head, from=3-5, to=2-5]
	\arrow[no head, from=3-7, to=2-6]
	\arrow[no head, from=3-7, to=2-8]
	\arrow[no head, from=4-3, to=3-3]
	\arrow[no head, from=4-5, to=3-5]
	\arrow[no head, from=4-7, to=3-7]
\end{tikzcd}\]
	Then we define a VDC $\Dbl{C}$ as the colimit obtained by identifying ternary multicells in $\Dbl{J}(T_l)$ and 
  $\Dbl{J}(T_r)$,and identifying the tight targets in 
  $\Dbl{J}(T_l)$ with the tight sources in $\Dbl{J}(I_2)$:
	\[
	\begin{tikzcd}
	&& {\mathbb{T}\mathsf{ight}[2]} && {\mathbb{J}(I_2)} \\
	\\
	{\mathbb{S}\mathsf{q}_3} && {\mathbb{J}(T_l)} \\
	\\
	{\mathbb{J}(T_r)} &&&& {\mathbb{C}}
	\arrow["s", from=1-3, to=1-5]
	\arrow["t"', from=1-3, to=3-3]
	\arrow[from=1-5, to=5-5]
	\arrow["{\square(T_l)}", from=3-1, to=3-3]
	\arrow["{\square(T_r)}"', from=3-1, to=5-1]
	\arrow[from=5-1, to=5-5]
	\arrow["\lrcorner"{anchor=center, pos=0.125, rotate=180}, draw=none, from=5-5, to=3-3]
\end{tikzcd}
	\]
	The data of $\Dbl{C}$ consists of the non-identity multicells depicted below:
	\[
	\begin{tikzcd}[column sep=small]
	{x_0} & {x_1} & {x_2} & {x_3} & {x_4} & {x_0} & {x_1} & {x_2} & {x_3} \\
	{y_0} && {y_1} & {y_2} & {y_3} & {y_0'} & {y_1'} && {y_2'} \\
	{z_0} &&& {z_1} & {z_2} & {z_0} &&& {z_1}
	\arrow["{\varphi_1}"{inner sep=.8ex}, "\shortmid"{marking}, from=1-1, to=1-2]
	\arrow[""{name=0, anchor=center, inner sep=0}, "{a_0}"', from=1-1, to=2-1]
	\arrow["{\varphi_2}"{inner sep=.8ex}, "\shortmid"{marking}, from=1-2, to=1-3]
	\arrow[""{name=1, anchor=center, inner sep=0}, "{\varphi_3}"{inner sep=.8ex}, "\shortmid"{marking}, from=1-3, to=1-4]
	\arrow[""{name=2, anchor=center, inner sep=0}, "{a_1}"{description}, from=1-3, to=2-3]
	\arrow[""{name=3, anchor=center, inner sep=0}, "{\varphi_4}"{inner sep=.8ex}, "\shortmid"{marking}, from=1-4, to=1-5]
	\arrow["{a_2}"{description}, from=1-4, to=2-4]
	\arrow["{a_3}", from=1-5, to=2-5]
	\arrow["{\varphi_1}"{inner sep=.8ex}, "\shortmid"{marking}, from=1-6, to=1-7]
	\arrow[""{name=4, anchor=center, inner sep=0}, "{a_0'}"', from=1-6, to=2-6]
	\arrow["{\varphi_2}"{inner sep=.8ex}, "\shortmid"{marking}, from=1-7, to=1-8]
	\arrow[""{name=5, anchor=center, inner sep=0}, "{a_1'}"{description}, from=1-7, to=2-7]
	\arrow["{\varphi_3}"{inner sep=.8ex}, "\shortmid"{marking}, from=1-8, to=1-9]
	\arrow[""{name=6, anchor=center, inner sep=0}, "{a_2'}", from=1-9, to=2-9]
	\arrow["{\lambda_1}"'{inner sep=.8ex}, "\shortmid"{marking}, from=2-1, to=2-3]
	\arrow[""{name=7, anchor=center, inner sep=0}, "{b_0}"', from=2-1, to=3-1]
	\arrow[""{name=8, anchor=center, inner sep=0}, "{\lambda_2}"'{inner sep=.8ex}, "\shortmid"{marking}, from=2-3, to=2-4]
	\arrow[""{name=9, anchor=center, inner sep=0}, "{\lambda_3}"'{inner sep=.8ex}, "\shortmid"{marking}, from=2-4, to=2-5]
	\arrow[""{name=10, anchor=center, inner sep=0}, "{b_1}"{description}, from=2-4, to=3-4]
	\arrow[""{name=11, anchor=center, inner sep=0}, "{b_2}", from=2-5, to=3-5]
	\arrow["{\lambda_1'}"'{inner sep=.8ex}, "\shortmid"{marking}, from=2-6, to=2-7]
	\arrow[""{name=12, anchor=center, inner sep=0}, "{b_0'}"', from=2-6, to=3-6]
	\arrow["{\lambda_2'}"'{inner sep=.8ex}, "\shortmid"{marking}, from=2-7, to=2-9]
	\arrow[""{name=13, anchor=center, inner sep=0}, "{b_1'}"{description}, from=2-9, to=3-9]
	\arrow["{\psi_1}"'{inner sep=.8ex}, "\shortmid"{marking}, from=3-1, to=3-4]
	\arrow["{\psi_2}"'{inner sep=.8ex}, "\shortmid"{marking}, from=3-4, to=3-5]
	\arrow["{\psi_1}"'{inner sep=.8ex}, "\shortmid"{marking}, from=3-6, to=3-9]
	\arrow["{\alpha_1}"{description}, draw=none, from=0, to=2]
	\arrow["{\alpha_2}"{description}, draw=none, from=1, to=8]
	\arrow["{\alpha_3}"{description}, draw=none, from=3, to=9]
	\arrow["{\alpha_1'}"{description}, draw=none, from=4, to=5]
	\arrow["{\alpha_2'}"{description}, draw=none, from=5, to=6]
	\arrow["{\beta_1}"{description}, draw=none, from=7, to=10]
	\arrow["{\beta_2}"{description}, draw=none, from=10, to=11]
	\arrow["{\beta_1'}"{description}, draw=none, from=12, to=13]
\end{tikzcd}
	\]
	subject to the relation $\frac{(\alpha_1,\alpha_2)}{\beta_1}=\frac{(\alpha_1',\alpha_2')}{\beta_1'}$. 
  Note that by construction $\Dbl{C}$ has no nullary multicells, and every multicell is either unary, 
  binary, or the unique trinary multicell, which decomposes into binary multicells as above. 
  In particular, it follows that $\Dbl{C}$ satisfies property (P1) of Definition~\ref{defn:proRep} for binary-unary trees, 
  since these are the only binary-unary decompositions of a trinary multicell. 
  Additionally, every 2-layer decomposition in $\Dbl{C}$ is 
  already a binary-unary decomposition. 
  Thus, $\Dbl{C}$ satisfies the hypotheses of Lemma~\ref{lem:expOfUnitalization}, so its unitalization $\Dbl{D}:=\Dbl{F}_u(\Dbl{C})$ is exponentiable.

  However, no map of profunctors $\alpha\times_{\Dbl{D}_0}\Dbl{D}_1$ from
	\begin{equation*}
		\int^{(\varphi_1,\varphi_2,\varphi_3)\in \Dbl{D}_3}(\odot(-,\varphi_1)\times_{\Dbl{D}_0}\Dbl{D}_1(-,\varphi_2)\times_{\Dbl{D}_0}\Dbl{D}_1(-,\varphi_3))\times (\odot((\varphi_1,\varphi_2),-)\times_{\Dbl{D}_0}\Dbl{D}_1(\varphi_1,-))
	\end{equation*}
	to 
	\begin{equation*}
		\int^{(\varphi_1,\varphi_2,\varphi_3)\in \Dbl{D}_3}(\Dbl{D}_1(-,\varphi_1)\times_{\Dbl{D}_0}\odot(-,\varphi_2)\times_{\Dbl{D}_0}\Dbl{D}_1(-,\varphi_3))\times (\odot((\varphi_1,\varphi_2),-)\times_{\Dbl{D}_0}\Dbl{D}_1(\varphi_1,-))
	\end{equation*}
	preserving composites exists. Indeed, such a map would need to send the element of the source co-end represented by the multicells below left to the element of the target co-end represented by multicells in the shape below right:
\[\begin{tikzcd}[column sep=small]
	{x_0} & {x_1} & {x_2} && {x_3} && {x_4} & {x_0} && {x_1} & {x_2} & {x_3} && {x_4} \\
	{y_0} && {y_1} && {y_2} && {y_3} & {y_0''} && {y_1''} && {y_2''} && {y_3} \\
	{z_0} &&&& {z_1} && {z_2} & {z_0} &&&& {z_1} && {z_2}
	\arrow["{{{\varphi_1}}}"{inner sep=.8ex}, "\shortmid"{marking}, from=1-1, to=1-2]
	\arrow[""{name=0, anchor=center, inner sep=0}, "{{{a_0}}}"', from=1-1, to=2-1]
	\arrow["{{{\varphi_2}}}"{inner sep=.8ex}, "\shortmid"{marking}, from=1-2, to=1-3]
	\arrow[""{name=1, anchor=center, inner sep=0}, "{{{\varphi_3}}}"{inner sep=.8ex}, "\shortmid"{marking}, from=1-3, to=1-5]
	\arrow[""{name=2, anchor=center, inner sep=0}, "{{{a_1}}}"{description}, from=1-3, to=2-3]
	\arrow[""{name=3, anchor=center, inner sep=0}, "{{{\varphi_4}}}"{inner sep=.8ex}, "\shortmid"{marking}, from=1-5, to=1-7]
	\arrow["{{{a_2}}}"{description}, from=1-5, to=2-5]
	\arrow["{{{a_3}}}", from=1-7, to=2-7]
	\arrow["{{{\varphi_1}}}"{inner sep=.8ex}, "\shortmid"{marking}, from=1-8, to=1-10]
	\arrow[""{name=4, anchor=center, inner sep=0}, "{{{a_0''}}}"', from=1-8, to=2-8]
	\arrow["{{{\varphi_2}}}"{inner sep=.8ex}, "\shortmid"{marking}, from=1-10, to=1-11]
	\arrow[""{name=5, anchor=center, inner sep=0}, "{{{a_1''}}}"{description}, from=1-10, to=2-10]
	\arrow["{{{\varphi_3}}}"{inner sep=.8ex}, "\shortmid"{marking}, from=1-11, to=1-12]
	\arrow[""{name=6, anchor=center, inner sep=0}, "{{{\varphi_4}}}"{inner sep=.8ex}, "\shortmid"{marking}, from=1-12, to=1-14]
	\arrow[""{name=7, anchor=center, inner sep=0}, "{{{a_2''}}}"{description}, from=1-12, to=2-12]
	\arrow["{{{a_3''}}}", from=1-14, to=2-14]
	\arrow["{{{\lambda_1}}}"'{inner sep=.8ex}, "\shortmid"{marking}, from=2-1, to=2-3]
	\arrow[""{name=8, anchor=center, inner sep=0}, "{{{b_0}}}"', from=2-1, to=3-1]
	\arrow[""{name=9, anchor=center, inner sep=0}, "{{{\lambda_2}}}"'{inner sep=.8ex}, "\shortmid"{marking}, from=2-3, to=2-5]
	\arrow[""{name=10, anchor=center, inner sep=0}, "{{{\lambda_3}}}"'{inner sep=.8ex}, "\shortmid"{marking}, from=2-5, to=2-7]
	\arrow[""{name=11, anchor=center, inner sep=0}, "{{{b_1}}}"{description}, from=2-5, to=3-5]
	\arrow[between={0.1}{0.9}, maps to, from=2-7, to=2-8]
	\arrow[""{name=12, anchor=center, inner sep=0}, "{{{b_2}}}", from=2-7, to=3-7]
	\arrow["{{{\lambda_1''}}}"'{inner sep=.8ex}, "\shortmid"{marking}, from=2-8, to=2-10]
	\arrow[""{name=13, anchor=center, inner sep=0}, "{{{b_0''}}}"', from=2-8, to=3-8]
	\arrow["{{{\lambda_2''}}}"'{inner sep=.8ex}, "\shortmid"{marking}, from=2-10, to=2-12]
	\arrow[""{name=14, anchor=center, inner sep=0}, "{{{\lambda_3''}}}"'{inner sep=.8ex}, "\shortmid"{marking}, from=2-12, to=2-14]
	\arrow[""{name=15, anchor=center, inner sep=0}, "{{{b_1''}}}"{description}, from=2-12, to=3-12]
	\arrow[""{name=16, anchor=center, inner sep=0}, "{{{b_2''}}}", from=2-14, to=3-14]
	\arrow["{{{\psi_1}}}"'{inner sep=.8ex}, "\shortmid"{marking}, from=3-1, to=3-5]
	\arrow["{{{\psi_2}}}"'{inner sep=.8ex}, "\shortmid"{marking}, from=3-5, to=3-7]
	\arrow["{{{\psi_1}}}"'{inner sep=.8ex}, "\shortmid"{marking}, from=3-8, to=3-12]
	\arrow["{{{\psi_2}}}"'{inner sep=.8ex}, "\shortmid"{marking}, from=3-12, to=3-14]
	\arrow["{{{\alpha_1}}}"{description}, draw=none, from=0, to=2]
	\arrow["{{{\alpha_2}}}"{description}, draw=none, from=1, to=9]
	\arrow["{{{\alpha_3}}}"{description}, draw=none, from=3, to=10]
	\arrow["{{{\alpha_1''}}}"{description}, draw=none, from=4, to=5]
	\arrow["{{{\alpha_2''}}}"{description}, draw=none, from=7, to=5]
	\arrow["{{{\alpha_3''}}}"{description}, draw=none, from=6, to=14]
	\arrow["{{{\beta_1}}}"{description}, draw=none, from=8, to=11]
	\arrow["{{{\beta_2}}}"{description}, draw=none, from=11, to=12]
	\arrow["{{{\beta_1''}}}"{description}, draw=none, from=13, to=15]
	\arrow["{{{\beta_2''}}}"{description}, draw=none, from=15, to=16]
\end{tikzcd}\]
	such that $\frac{(\alpha_1,\alpha_2)}{\beta_1}=\frac{(\alpha_1'',\alpha_2'')}{\beta_1''}$ and $\frac{\alpha_3}{\beta_2}=\frac{\alpha_3''}{\beta_2''}$. However, note that the equality $\frac{(\alpha_1,\alpha_2)}{\beta_1}=\frac{(\alpha_1'',\alpha_2'')}{\beta_1''}$ is only possible if $\alpha_1''=\alpha_1'$, $\alpha_2''=\alpha_2'$ and $\beta_1''=\beta_1'$. Thus, since there are no unary multicells $\alpha_3''$ and $\beta_2''$ with the specified boundaries such that $\frac{\alpha_3}{\beta_2}=\frac{\alpha_3''}{\beta_2''}$, no such map can exist.
\end{proof}

\subsubsection*{Exponentiating Normal VDCs}

As the previous section of counter-examples illustrates, unitality is not sufficient for exponentiability of VDCs.
Nonetheless, unital VDCs provide a host of examples of exponentiable VDCs,
due to the fact that it is in general easier to detect when a unital VDC is exponentiable.
Before providing conditions for VDCs with units to be exponentiable, we'll begin by
explicating the behaviour of exponentials in $\Vdc_n$ under the forgetful functor to VDCs.

First, note that the since forgetful functor $U:\Vdc_n\to \Vdc$ is cocontinuous and conservative, it creates all colimits. 
In particular, it follows that for a unital VDC $\Dbl{D}$, the functor $\mathbb{F}_u(U(-))\times \Dbl{D}:\Vdc_n\to \Vdc_n$ commutes with all colimits if and only if the functor $U\mathbb{F}_u(-)\times U(\Dbl{D}):\Vdc\to \Vdc$ commutes with all colimits.
But $\Vdc_n$ is also locally finitely presentable, so we obtain that $U$ reflects exponentiable objects:

\begin{prop}\label{prop:expVDCnviaU}
	If $\Dbl{D}\in \Vdc_n$ is a normal VDC, and if $U(\Dbl{D})\in \Vdc$ is exponentiable, 
	then $\Dbl{D}$ is exponentiable in $\Vdc_n$.
\end{prop}

We will leave a full characterization of the exponentiable normal VDCs,
in analogy with Theorem~\ref{thm:ExpViaDecomp}, to future work.

For unital VDCs $\Dbl{D}$ and $\Dbl{E}$ such that the exponential exists in $\Vdc_n$, 
we will denote it by $\Dbl{V}\Cat{df}_n(\Dbl{D},\Dbl{E})$.
Using the exponentiation adjunction in $\Vdc_n$, and the descriptions in Lemma~\ref{lem:VdfVDC},
we have the following description of $\Dbl{V}\Cat{df}_n(\Dbl{D},\Dbl{E})$'s data:
\begin{enumerate}
	\item Objects are normal VDFs $\Dbl{D}\cong \Ob_u\times \Dbl{D}\to \Dbl{E}$.
	\item Tight arrows are normal VDFs $\Dbl{F}_u(\Tight)\times \Dbl{D}\to \Dbl{E}$, which are precisely tight transformations between normal VDFs.
	\item Loose arrows are normal VDFs $\Dbl{F}_u(\Loose)\times \Dbl{D}\to \Dbl{E}$, which are VDFs $\Dbl{F}_u(\Loose)\times \Dbl{D}\to \Dbl{E}$ whose restriction to either object in $\Dbl{F}_u(\Loose)$ yields a normal VDF.
	\item $n$-ary multicells are normal VDFs $\Dbl{F}_u(\Sq{n})\times \Dbl{D}\to \Dbl{E}$.
\end{enumerate}
Thus, the exponential is a sub-VDC of $\Dbl{V}\Cat{df}(\Dbl{D},\Dbl{E})$ with inherited loose units.
We can use the exponential in $\Vdc_n$ to internalize the adjunction $U\dashv \Mod$ in the following sense:

\begin{prop}\label{prop:InternalizeUModAdj}
	If $\Dbl{D}$ is a unital VDC which is exponentiable in $\Vdc_n$, then for any VDC $\Dbl{E}$ we have a natural isomorphism 
	\begin{equation*}
		\Dbl{V}\Cat{df}(U(\Dbl{D}),\Dbl{E})\cong \Dbl{V}\Cat{df}_n(\Dbl{D},\Mod(\Dbl{E}))
	\end{equation*}
\end{prop}
\begin{proof}
	Using the Yoneda lemma the result follows from the following sequence of natural isomorphisms for an arbitrary unital VDC $\Dbl{A}$:
	\begin{align*}
		\Vdc_n(\Dbl{A},\Dbl{V}\Cat{df}(U(\Dbl{D}),\Dbl{E})) &\cong \Vdc(U(\Dbl{A}), \Dbl{E}^{U(\Dbl{D})}) \\
		&\cong \Vdc(U(\Dbl{A})\times U(\Dbl{D}),\Dbl{E}) \\
		&\cong \Vdc(U(\Dbl{A}\times \Dbl{D}),\Dbl{E}) \\
		&\cong \Vdc_n(\Dbl{A}\times \Dbl{D},\Mod(\Dbl{E})) \\
		&\cong \Vdc_n(\Dbl{A},\Dbl{V}\Cat{df}_n(\Dbl{D},\Mod(\Dbl{E})))
	\end{align*}
	where in the middle we used that $U$ preserves limits. 
	Thus, we have a natural isomorphism
	\begin{equation*}
		\Dbl{V}\Cat{df}(U(\Dbl{D}),-)\cong \Dbl{V}\Cat{df}_n(\Dbl{D},\Mod(-))
	\end{equation*}
\end{proof}

Now that we have a description of exponentials in $\Vdc_n$, 
let's see how the pro-representability condition in Definition~\ref{defn:proRep}
simplifies for normal VDCs. 

\begin{prop}\label{prop:proRepForNormalVDCs}
	If $\Dbl{D}$ is a normal VDC, then $\Dbl{D}$ is pro-representable if and only if 
	$\Dbl{D}$ satisfies both (P1) and (P3) of Definition~\ref{defn:proRep} 
	with $T$ restricted to range over binary-unary trees.
\end{prop}
\begin{proof}
	The necessity is immediate, so suppose $\Dbl{D}$ 
	satisfies both (P1) and (P3) of Definition~\ref{defn:proRep} 
	with $T$ restricted to range over binary-unary trees. 
	The fact that $\Dbl{D}$ satisfies (P1) for unary-nullary trees follows
	from the existence of loose units, as in the proof of Lemma~\ref{lem:expOfUnitalization}.

	For property (P3) fix an arbitrary height 2 tree $T'$. 
	Then if $\frac{\prolist{\alpha}}{\beta}$ is a decomposition in $\Dbl{D}$
	with shape $T'$, we can use the universal property of loose units 
	to pull out nullary shapes from $\prolist{\alpha}$, 
	from left to right, as in the example below:
\[\begin{tikzcd}[column sep =small]
	& {x_0} & {x_1} & {x_2} &&& {x_0} & {x_1} & {x_2} & \\
	&&&&& {x_0} & {x_0} & {x_1} & {x_2} \\
	{z_0} & {z_1} && {z_2} & {z_n} & {x_0} & {x_0} & {x_1} & {x_2} & {x_2} \\
	&&&&& {z_0} & {z_1} && {z_2} & {z_n} \\
	{y_0} &&&& {y_1} & {y_0} &&&& {y_1}
	\arrow["{{{\varphi_1}}}"{inner sep=.8ex}, "\shortmid"{marking}, from=1-2, to=1-3]
	\arrow[""{name=0, anchor=center, inner sep=0}, "{{{a_0}}}"', from=1-2, to=3-1]
	\arrow["{{{a_1}}}"{description}, from=1-2, to=3-2]
	\arrow["{{{\varphi_2}}}"{inner sep=.8ex}, "\shortmid"{marking}, from=1-3, to=1-4]
	\arrow["{{{a_2}}}"{description}, from=1-4, to=3-4]
	\arrow[""{name=1, anchor=center, inner sep=0}, "{{{a_3}}}", from=1-4, to=3-5]
	\arrow["{{{\varphi_1}}}"{inner sep=.8ex}, "\shortmid"{marking}, from=1-7, to=1-8]
	\arrow[equals, from=1-7, to=2-6]
	\arrow[equals, from=1-7, to=2-7]
	\arrow["{{{\varphi_2}}}"{inner sep=.8ex}, "\shortmid"{marking}, from=1-8, to=1-9]
	\arrow[equals, from=1-8, to=2-8]
	\arrow[equals, from=1-9, to=2-9]
	\arrow[equals, from=2-6, to=2-7]
	\arrow[equals, from=2-6, to=3-6]
	\arrow["{{{\varphi_1}}}"{inner sep=.8ex}, "\shortmid"{marking}, from=2-7, to=2-8]
	\arrow[equals, from=2-7, to=3-7]
	\arrow["{{{\varphi_2}}}"{inner sep=.8ex}, "\shortmid"{marking}, from=2-8, to=2-9]
	\arrow[equals, from=2-8, to=3-8]
	\arrow[equals, from=2-9, to=3-9]
	\arrow[equals, from=2-9, to=3-10]
	\arrow[""{name=2, anchor=center, inner sep=0}, "{{{\chi_1}}}"'{inner sep=.8ex}, "\shortmid"{marking}, from=3-1, to=3-2]
	\arrow["{{{b_0}}}"', from=3-1, to=5-1]
	\arrow[""{name=3, anchor=center, inner sep=0}, "{{{\chi_2}}}"'{inner sep=.8ex}, "\shortmid"{marking}, from=3-2, to=3-4]
	\arrow[""{name=4, anchor=center, inner sep=0}, "{{{\chi_3}}}"'{inner sep=.8ex}, "\shortmid"{marking}, from=3-4, to=3-5]
	\arrow[between={0.2}{0.8}, equals, from=3-5, to=3-6]
	\arrow["{{{b_1}}}", from=3-5, to=5-5]
	\arrow[""{name=5, anchor=center, inner sep=0}, equals, from=3-6, to=3-7]
	\arrow["{{{a_0}}}"', from=3-6, to=4-6]
	\arrow["{{{\varphi_1}}}"{inner sep=.8ex}, "\shortmid"{marking}, from=3-7, to=3-8]
	\arrow["{{{a_1}}}"{description}, from=3-7, to=4-7]
	\arrow["{{{\varphi_2}}}"{inner sep=.8ex}, "\shortmid"{marking}, from=3-8, to=3-9]
	\arrow[""{name=6, anchor=center, inner sep=0}, equals, from=3-9, to=3-10]
	\arrow["{{{a_2}}}"{description}, from=3-9, to=4-9]
	\arrow["{{{a_3}}}", from=3-10, to=4-10]
	\arrow[""{name=7, anchor=center, inner sep=0}, "{{{\chi_1}}}"'{inner sep=.8ex}, "\shortmid"{marking}, from=4-6, to=4-7]
	\arrow["{{{b_0}}}"', from=4-6, to=5-6]
	\arrow[""{name=8, anchor=center, inner sep=0}, "{{{\chi_2}}}"'{inner sep=.8ex}, "\shortmid"{marking}, from=4-7, to=4-9]
	\arrow[""{name=9, anchor=center, inner sep=0}, "{{{\chi_3}}}"'{inner sep=.8ex}, "\shortmid"{marking}, from=4-9, to=4-10]
	\arrow["{{{b_1}}}", from=4-10, to=5-10]
	\arrow[""{name=10, anchor=center, inner sep=0}, "{{{{{{{\psi}}}}}}}"'{inner sep=.8ex}, "\shortmid"{marking}, from=5-1, to=5-5]
	\arrow[""{name=11, anchor=center, inner sep=0}, "{{{{{{{\psi}}}}}}}"'{inner sep=.8ex}, "\shortmid"{marking}, from=5-6, to=5-10]
	\arrow["{{{\alpha_1}}}"{description}, draw=none, from=0, to=2]
	\arrow["{{{\alpha_2}}}"{description}, draw=none, from=1-3, to=3]
	\arrow["{{{\alpha_3}}}"{description}, draw=none, from=1, to=4]
	\arrow["\beta"{description}, draw=none, from=3, to=10]
	\arrow["{{{\alpha_1^\flat}}}"{description}, draw=none, from=5, to=7]
	\arrow["{{{\alpha_2}}}"{description}, draw=none, from=3-8, to=8]
	\arrow["{{{\alpha_3^\flat}}}"{description}, draw=none, from=6, to=9]
	\arrow["\beta"{description, pos=0.7}, draw=none, from=8, to=11]
\end{tikzcd}\]
	The result is a decomposition where the second layer contains no nullary multicells 
	and above the second layer, each layer has a single nullary multicell
	while the rest of the multicells are unary. 
	We can then apply (P3) for binary-unary trees $T$ to decompose the bottom two layers 
	into a binary-unary tree, so that the whole decomposition has the shape of a
	binary-nullary tree. 
	Thus, $\Dbl{D}$ satisfies (P3), and hence is pro-representable, as desired.
\end{proof}

%% file: Sections/7_RepresentabilityOfExponential.tex
Our next goal is to determine sufficient conditions on an exponentiable VDC $\Dbl{D}$ and a VDC $\Dbl{E}$ so that the exponential $\Dbl{E}^\Dbl{D}$ has (weak) composites. Using these conditions we will also provide sufficient conditions for the VDC $\Dbl{V}\mathsf{df}(\Dbl{D},\Dbl{E})$ of virtual double functors to be representable.

\subsection{Relative Decomposition Conditions for VDCs}

In order to find a sufficient condition for representability of exponentials, we will consider a decomposition condition inspired by Par\'{e}'s AFP condition (c.f.~\cite{Pare2013}) and Behr et. al's cylindrical decomposition property (c.f.~\cite{ConvProd2023}). 
Namely, we will follow Par\'{e}'s approach in having the factorization condition be only up to essentially unique decompositions, but unlike in~\cite{Pare2013}, our globular decomposition conditions for VDCs do not involve vertical reduction relations (i.e.~loose reduction relations in our current terminology), as these explicitly involve composition of loose arrows, which is not possible in general VDCs. 
Nonetheless, even without this condition we can still prove representability results for exponentials, while also providing a number of VDCs which satisfy these properties.

Our globular decompositions will arise as a special case of the following more general notion of decompositions relative to a VDF $F:\Dbl{A}\to \Dbl{D}$, which is analogous to the right Conduch\'{e} condition in~\cite{Grandis2007}. Recall from Theorem~\ref{thm:ExpChar} that a VDC $\Dbl{D}$ is exponentiable if and only if the composition transformation 
\begin{equation*}
	\mathsf{fc}(\MCell(\mathbb{D}))\odot \MCell(\mathbb{D})\xrightarrow{\mathsf{comp}}\MCell(\mathbb{D})(\mu_\mathsf{fc},1):\mathsf{fc}(\mathsf{fc}(\Dbl{D}_1))^{op}\times \Dbl{D}_1\to \mathsf{Set}
\end{equation*}
is an isomorphism.

\begin{defn}[Relative Decomposition]\label{defn:relDecomp}
	Let $i:\Cat{C}\hookrightarrow \Cat{fc}(\Dbl{D}_1)$ be a wide subcategory and set $i_1:\Cat{C}_1\hookrightarrow \Dbl{D}_1$
  for the inclusion of $\Cat{C}\cap \Dbl{D}_1$ into $\Dbl{D}_1.$
  Next let $\Cat{Q}:\Cat{C}\proto \Cat{C}_1$ 
  be a sub-profunctor of the restriction $\MCell(i,i_1)$ of $\MCell:\fc(\Dbl{D}_1)\proto \Dbl{D}_1$ along $i$ and $i_1.$
  (Note that $\MCell(i,i_1)$ contains all the multicells of $\Dbl{D},$ since $\Cat{C}$ is wide, merely restricting the 
  actions of unary multicells on either side.) 

	Then for an integer $n\geq 0$, 
	we say $\Dbl{D}$ \emph{has $n$-ary $\Cat{Q}$-decompositions} if the composition transformation
	\begin{equation*}
		\mathsf{fc}(\MCell(\mathbb{D}))(1,i)\odot \Cat{Q}\xrightarrow{\mathsf{comp}}\MCell(\mathbb{D})(\mu_\mathsf{fc},i_1):\mathsf{fc}(\mathsf{fc}(\Dbl{D}_1))\proto \Cat{C}_1
	\end{equation*}
	is an isomorphism on the fiber over $n$-ary multicells.
	We say $\Dbl{D}$ \emph{has $\Cat{Q}$-decompositions} if it has $n$-ary $\Cat{Q}$-decompositions for all $n\geq 0$,
	and \emph{positive $\Cat{Q}$-decompositions} if it has $n$-ary $\Cat{Q}$-decompositions for $n\geq 1$.
\qed
\end{defn}

Diagrammatically, this definition says that any $n$-ary multicell below left admits an essentially unique decomposition, as below right,
given any partition of the length of its domain $\prolist{\varphi}$ to determine the arities in $\prolist{\alpha}^\#$:
\[
\begin{tikzcd}
	{x_0} & {x_n} & {x_0} & {x_n} \\
	&& {z_0} & {z_n} \\
	{y_0} & {y_1} & {y_0} & {y_1}
	\arrow[""{name=0, anchor=center, inner sep=0}, "{{{{\prolist{\varphi}}}}}"{inner sep=.8ex}, "\shortmid"{marking}, from=1-1, to=1-2]
	\arrow["a"', from=1-1, to=3-1]
	\arrow[""{name=1, anchor=center, inner sep=0}, "b", from=1-2, to=3-2]
	\arrow[""{name=2, anchor=center, inner sep=0}, "{{{{\prolist{\varphi}}}}}"{inner sep=.8ex}, "\shortmid"{marking}, from=1-3, to=1-4]
	\arrow["{{{{a'}}}}"', from=1-3, to=2-3]
	\arrow["{{{{b'}}}}", from=1-4, to=2-4]
	\arrow[""{name=3, anchor=center, inner sep=0}, "{{{{\prolist{\chi}}}}}"'{inner sep=.8ex}, "\shortmid"{marking}, from=2-3, to=2-4]
	\arrow["c"', from=2-3, to=3-3]
	\arrow["d", from=2-4, to=3-4]
	\arrow[""{name=4, anchor=center, inner sep=0}, "\psi"'{inner sep=.8ex}, "\shortmid"{marking}, from=3-1, to=3-2]
	\arrow[""{name=5, anchor=center, inner sep=0}, "{{{{\psi}}}}"'{inner sep=.8ex}, "\shortmid"{marking}, from=3-3, to=3-4]
	\arrow["\alpha"{description}, draw=none, from=0, to=4]
	\arrow[between={0.4}{0.6}, equals, from=1, to=2-3]
	\arrow["{{\prolist{\alpha}^\#}}"{description}, draw=none, from=2, to=3]
	\arrow["\beta"{description, pos=0.7}, draw=none, from=3, to=5]
\end{tikzcd}
\]
where $\beta \in \Cat{Q}(\prolist{\chi},\psi)$ and the essential uniqueness is up to
lists of unary multicells falling in $\Cat{C}\subseteq \Cat{fc}(\Dbl{D}_1)$.

\begin{defn}[Globular Decompositions]\label{defn:GlobDecomp}
	Let $i_g:\Cat{gl}(\Cat{fc}(\Dbl{D}_1))\hookrightarrow \Cat{fc}(\Dbl{D}_1)$ be the wide subcategory 
  whose morphisms are those sequences of unary multicells with outermost tight arrows given by identities, 
  and let $\Cat{gCell}(\Dbl{D})\subseteq \MCell (\Dbl{D})(i_g,i_{g1})$ be the sub-profunctor consisting of globular non-nullary multicells in $\Dbl{D}$. 
	Then for an integer $n\geq 0$ we say that $\Dbl{D}$ has $n$-ary globular decompositions 
	if it has $n$-ary $\Cat{gCell}(\Dbl{D})$ decompositions. 
	Further, if $\Dbl{D}$ has the $n$-ary globular decompositions for all $n\geq 0$, 
	we simply say it has globular decompositions, 
	and if it has $n$-ary globular decompositions for all $n >0$ we say it has positive globular decompositions.
\qed
\end{defn}

Having globular decompositions and positive globular decompostions 
will be our main generalizations of Par\'{e}'s AFP condition to VDCs. 
Note that in the definition we exclude decompositions where the bottom layer is a nullary multicell. 
The motivation for this restriction is that very few VDCs would have $0$-ary globular decompositions
if decompositions with nullary bottom layer were included. 
Indeed, if a VDC has $0$-ary globular decompositions 
while allowing for globular nullary multicells, 
then it can only have nullary multicells with loose target an endo-loose arrow $a\proto a$ 
and with tight boundary components being equal. 
This is also noted in Section 5 of~\cite{Pare2013}. 
Nonetheless, there are two important examples of VDCs which would
have $0$-ary globular decompositions with nullary bottom layer,
along with our current notion of globular decompositions.

\begin{exmp}
	All VDCs with discrete underlying tight category have globular decompositions,
	since every multicell is globular.
	For example:
	\begin{enumerate}
		\item The loose embedding of a bicategory $\BiCat{B}$ as the VDC $\L(\BiCat{B})$ has globular decompositions.
		\item The categorification of a multicategory $\Cat{M}$ as the VDC $B(\Cat{M})$ has globular decompositions.
	\end{enumerate}
\end{exmp}

In addition to these examples, cospan VDCs, which we saw in Proposition~\ref{prop:cospanPro} were exponentiable,
also have globular decompositions.

\begin{prop}[Cospan VDCs have Globular Decompositions]\label{rmk:CospGDP}
	For any category $\Cat{E}$, the VDC $\Cospan(\Cat{E})$ always has globular decompositions.
\end{prop}
\begin{proof}
	The proof follows as in Proposition~\ref{prop:cospanPro}. Namely, if we have a multicell 
	\[\begin{adjustbox}{}\begin{tikzcd}
		{x_0} & {y_1} & {x_1} & \ldots & {x_{n-1}} & {y_n} & {x_n} \\
		&&& \ldots \\
		{z_0} &&& {w_1} &&& {z_1}
		\arrow["{{a_1}}", from=1-1, to=1-2]
		\arrow["{{c_0}}"', from=1-1, to=3-1]
		\arrow["{{e_1}}"', from=1-2, to=3-4]
		\arrow["{{b_1}}"', from=1-3, to=1-2]
		\arrow["{{a_2}}", from=1-3, to=1-4]
		\arrow["{{b_{n-1}}}"', from=1-5, to=1-4]
		\arrow["{{a_n}}", from=1-5, to=1-6]
		\arrow["{{e_n}}", from=1-6, to=3-4]
		\arrow["{{b_n}}"', from=1-7, to=1-6]
		\arrow["{{c_1}}", from=1-7, to=3-7]
		\arrow["{d_0}"', from=3-1, to=3-4]
		\arrow["{d_1}", from=3-7, to=3-4]
	\end{tikzcd}\end{adjustbox}\]
	then if $f_i:x_i\to w_1$ denote the unique maps for the middle $x_i$'s, so 
  $f_i=b_i\cdot e_i=a_{i+1}\cdot e_{i+1}$, then we can slide the tight boundary maps up as in the decomposition 
	\[\begin{adjustbox}{}\begin{tikzcd}
		{x_0} & {y_1} & {x_1} & \cdots & {x_{n-1}} & {y_n} & {x_n} \\
		\\
		{z_0} & {w_1} & {w_1} & \cdots & {w_1} & {w_1} & {z_1} \\
		&&& \cdots \\
		{z_0} &&& {w_1} &&& {z_1}
		\arrow["{{{a_1}}}", from=1-1, to=1-2]
		\arrow["{{{c_0}}}"', from=1-1, to=3-1]
		\arrow["{{{e_1}}}"{description}, from=1-2, to=3-2]
		\arrow["{{{b_1}}}"', from=1-3, to=1-2]
		\arrow["{{{a_2}}}", from=1-3, to=1-4]
		\arrow["{c_1}", from=1-3, to=3-3]
		\arrow["{{{b_{n-1}}}}"', from=1-5, to=1-4]
		\arrow["{{{a_n}}}", from=1-5, to=1-6]
		\arrow["{c_{n-1}}"', from=1-5, to=3-5]
		\arrow["{{{e_n}}}"{description}, from=1-6, to=3-6]
		\arrow["{{{b_n}}}"', from=1-7, to=1-6]
		\arrow["{{{c_1}}}", from=1-7, to=3-7]
		\arrow["{d_0}"', from=3-1, to=3-2]
		\arrow[equals, from=3-1, to=5-1]
		\arrow[equals, from=3-2, to=5-4]
		\arrow[equals, from=3-3, to=3-2]
		\arrow[equals, from=3-3, to=3-4]
		\arrow[equals, from=3-3, to=5-4]
		\arrow[equals, from=3-4, to=3-5]
		\arrow[equals, from=3-5, to=5-4]
		\arrow[equals, from=3-6, to=3-5]
		\arrow[equals, from=3-6, to=5-4]
		\arrow["{d_1}", from=3-7, to=3-6]
		\arrow[equals, from=3-7, to=5-7]
		\arrow["{d_0}"', from=5-1, to=5-4]
		\arrow["{d_1}", from=5-7, to=5-4]
	\end{tikzcd}
	\end{adjustbox}\]
	Additionally, this is the terminal decomposition: indeed any other such decomposition
	\[
	\begin{adjustbox}{}
		\begin{tikzcd}
		{x_0} & {y_1} & {x_1} & \cdots & {x_{n-1}} & {y_n} & {x_n} \\
		\\
		{z_0} & {w_1'} & {z_1'} & \cdots & {z_{n-1}'} & {w_n'} & {z_1} \\
		&&& \cdots \\
		{z_0} &&& {w_1} &&& {z_1}
		\arrow["{{{a_1}}}", from=1-1, to=1-2]
		\arrow["{{{c_0}}}"', from=1-1, to=3-1]
		\arrow["{{{e_1}}'}"{description}, from=1-2, to=3-2]
		\arrow["{{{b_1}}}"', from=1-3, to=1-2]
		\arrow["{{{a_2}}}", from=1-3, to=1-4]
		\arrow["{c_1'}", from=1-3, to=3-3]
		\arrow["{{{b_{n-1}}}}"', from=1-5, to=1-4]
		\arrow["{{{a_n}}}", from=1-5, to=1-6]
		\arrow["{c_{n-1}'}"', from=1-5, to=3-5]
		\arrow["{{{e_n}}'}"{description}, from=1-6, to=3-6]
		\arrow["{{{b_n}}}"', from=1-7, to=1-6]
		\arrow["{{{c_1}}'}", from=1-7, to=3-7]
		\arrow["{f_1}"', from=3-1, to=3-2]
		\arrow[equals, from=3-1, to=5-1]
		\arrow["{e_1''}"', from=3-2, to=5-4]
		\arrow["{g_1}", from=3-3, to=3-2]
		\arrow["{f_2}"', from=3-3, to=3-4]
		\arrow["{g_{n-1}}", from=3-5, to=3-4]
		\arrow["{f_n}"', from=3-5, to=3-6]
		\arrow["{e_n''}", from=3-6, to=5-4]
		\arrow["{g_n}", from=3-7, to=3-6]
		\arrow[equals, from=3-7, to=5-7]
		\arrow["{d_0}"', from=5-1, to=5-4]
		\arrow["{d_1}", from=5-7, to=5-4]
	\end{tikzcd}
	\end{adjustbox}
	\]
	is equivalent to this one via the intermediary 3-layer decomposition 
	\[
	\begin{adjustbox}{}
		\begin{tikzcd}
		{x_0} & {y_1} & {x_1} & \cdots & {x_{n-1}} & {y_n} & {x_n} \\
		\\
		{z_0} & {w_1'} & {z_1'} & \cdots & {z_{n-1}'} & {w_n'} & {z_1} \\
		\\
		{z_0} & {w_1} & {w_1} & \cdots & {w_1} & {w_1} & {z_1} \\
		{z_0} &&& {w_1} &&& {z_1}
		\arrow["{{{a_1}}}", from=1-1, to=1-2]
		\arrow["{{{c_0}}}"', from=1-1, to=3-1]
		\arrow["{{{e_1}}'}"{description}, from=1-2, to=3-2]
		\arrow["{{{b_1}}}"', from=1-3, to=1-2]
		\arrow["{{{a_2}}}", from=1-3, to=1-4]
		\arrow["{c_1'}", from=1-3, to=3-3]
		\arrow["{{{b_{n-1}}}}"', from=1-5, to=1-4]
		\arrow["{{{a_n}}}", from=1-5, to=1-6]
		\arrow["{c_{n-1}'}"', from=1-5, to=3-5]
		\arrow["{{{e_n}}'}"{description}, from=1-6, to=3-6]
		\arrow["{{{b_n}}}"', from=1-7, to=1-6]
		\arrow["{{{c_1}}'}", from=1-7, to=3-7]
		\arrow["{f_1}"', from=3-1, to=3-2]
		\arrow[equals, from=3-1, to=5-1]
		\arrow["{e_1''}"{description}, from=3-2, to=5-2]
		\arrow["{g_1}", from=3-3, to=3-2]
		\arrow["{f_2}"', from=3-3, to=3-4]
		\arrow["{c_1''}", from=3-3, to=5-3]
		\arrow["{g_{n-1}}", from=3-5, to=3-4]
		\arrow["{f_n}"', from=3-5, to=3-6]
		\arrow["{c_{n-1}''}"', from=3-5, to=5-5]
		\arrow["{e_n''}"{description}, from=3-6, to=5-6]
		\arrow["{g_n}", from=3-7, to=3-6]
		\arrow[equals, from=3-7, to=5-7]
		\arrow["{d_0}"', from=5-1, to=5-2]
		\arrow[equals, from=5-1, to=6-1]
		\arrow[equals, from=5-2, to=5-3]
		\arrow[equals, from=5-2, to=6-4]
		\arrow[equals, from=5-3, to=5-4]
		\arrow[equals, from=5-4, to=5-5]
		\arrow[equals, from=5-5, to=5-6]
		\arrow[equals, from=5-6, to=6-4]
		\arrow["{d_1}", from=5-7, to=5-6]
		\arrow[equals, from=5-7, to=6-7]
		\arrow["{d_0}"', from=6-1, to=6-4]
		\arrow["{d_1}", from=6-7, to=6-4]
	\end{tikzcd}
	\end{adjustbox}
	\]
	as desired. The same argument applies when $n=0$, though in this case the terminal decomposition is
\[\begin{tikzcd}[column sep=small]
	&&&&&&&& {x_0} &&& \\
	&& {x_0} \\
	&&&&& {z_0} & {w_1} & {w_1} & \cdots & {w_1} & {w_1} & {z_1} \\
	{z_0} && {w_1} && {z_1} \\
	&&&&& {z_0} &&& {w_1} &&& {z_1}
	\arrow["{{c_0}}"', from=1-9, to=3-6]
	\arrow["{{e_1}}"{description}, from=1-9, to=3-8]
	\arrow["{{e_1}}"{description}, from=1-9, to=3-10]
	\arrow["{{c_1}}", from=1-9, to=3-12]
	\arrow["{{c_0}}"', from=2-3, to=4-1]
	\arrow[""{name=0, anchor=center, inner sep=0}, "{{c_1}}", from=2-3, to=4-5]
	\arrow["{{d_0}}"', from=3-6, to=3-7]
	\arrow[equals, from=3-6, to=5-6]
	\arrow[equals, from=3-7, to=3-8]
	\arrow[equals, from=3-7, to=5-9]
	\arrow[equals, from=3-8, to=3-9]
	\arrow[equals, from=3-9, to=3-10]
	\arrow[equals, from=3-10, to=3-11]
	\arrow[equals, from=3-11, to=5-9]
	\arrow["{{d_1}}", from=3-12, to=3-11]
	\arrow[equals, from=3-12, to=5-12]
	\arrow["{{d_0}}"', from=4-1, to=4-3]
	\arrow["{{d_1}}", from=4-5, to=4-3]
	\arrow["{{d_0}}"', from=5-6, to=5-9]
	\arrow["{{d_1}}", from=5-12, to=5-9]
	\arrow[between={0.4}{0.7}, equals, from=0, to=3-6]
\end{tikzcd}\]
\end{proof}

For any VDC $\Dbl{D}$, we can also show that the strict double category $\Cat{fc}(\Dbl{D})^{op_t}$ 
(See Equation \ref{eq:defOfFc}) has positive globular decompositions, which will be important for later embedding results.

\begin{lem}\label{lem:PosGDPforOpOfStrictification}
	Let $\Dbl{D}$ be an arbitrary VDC. Then the VDC $\Cat{fc}(\Dbl{D})^{op_t}$ has positive globular decompositions.
\end{lem}
\begin{proof}
	Recall that multicells in $\Cat{fc}(\Dbl{D})$ are given by compatible sequences of multicells in $\Dbl{D}$. 
	Thus, an $n$-ary multicell in $\Cat{fc}(\Dbl{D})^{op_t}$, for $n\geq 1$, takes the form
	\[
	\begin{tikzcd}
	{x_0} & {x_1} & \cdots & {x_n} \\
	{y_0} & {y_{k_1}} & \cdots & {y_{k_n}}
	\arrow[""{name=0, anchor=center, inner sep=0}, "{\varphi_1}"{inner sep=.8ex}, "\shortmid"{marking}, from=1-1, to=1-2]
	\arrow[""{name=1, anchor=center, inner sep=0}, "{\varphi_2}"{inner sep=.8ex}, "\shortmid"{marking}, from=1-2, to=1-3]
	\arrow[""{name=2, anchor=center, inner sep=0}, "{\varphi_n}"{inner sep=.8ex}, "\shortmid"{marking}, from=1-3, to=1-4]
	\arrow["\cdots"{description}, draw=none, from=1-3, to=2-3]
	\arrow["{a_0}", from=2-1, to=1-1]
	\arrow[""{name=3, anchor=center, inner sep=0}, "{\prolist{\psi}_1}"'{inner sep=.8ex}, "\shortmid"{marking}, from=2-1, to=2-2]
	\arrow["{a_1}"{description}, from=2-2, to=1-2]
	\arrow[""{name=4, anchor=center, inner sep=0}, "{\prolist{\psi}_2}"'{inner sep=.8ex}, "\shortmid"{marking}, from=2-2, to=2-3]
	\arrow[""{name=5, anchor=center, inner sep=0}, "{\prolist{\psi}_n}"'{inner sep=.8ex}, "\shortmid"{marking}, from=2-3, to=2-4]
	\arrow["{a_n}"', from=2-4, to=1-4]
	\arrow["{\alpha_1}"{description}, draw=none, from=0, to=3]
	\arrow["{\alpha_2}"{description}, draw=none, from=1, to=4]
	\arrow["{\alpha_n}"{description}, draw=none, from=2, to=5]
\end{tikzcd}
	\]
	where $\alpha_1,...,\alpha_n$ are multicells in $\Dbl{D}$, and 
  the list of lists $(\prolist{\psi}_1 \cdots \prolist{\psi}_n)$ 
  is viewed as a single loose arrow in $\Cat{fc}(\Dbl{D})^{op_t}$. 
	We then have a natural decomposition
	\[
	\begin{tikzcd}
	{x_0} & {x_1} & \cdots & {x_n} \\
	{y_0} & {y_{k_1}} & \cdots & {y_{k_n}} \\
	{y_0} & {y_{k_1}} & \cdots & {y_{k_n}}
	\arrow[""{name=0, anchor=center, inner sep=0}, "{\varphi_1}"{inner sep=.8ex}, "\shortmid"{marking}, from=1-1, to=1-2]
	\arrow[""{name=1, anchor=center, inner sep=0}, "{\varphi_2}"{inner sep=.8ex}, "\shortmid"{marking}, from=1-2, to=1-3]
	\arrow[""{name=2, anchor=center, inner sep=0}, "{\varphi_n}"{inner sep=.8ex}, "\shortmid"{marking}, from=1-3, to=1-4]
	\arrow["\cdots"{description}, draw=none, from=1-3, to=2-3]
	\arrow["{a_0}", from=2-1, to=1-1]
	\arrow[""{name=3, anchor=center, inner sep=0}, "{\prolist{\psi}_1}"'{inner sep=.8ex}, "\shortmid"{marking}, from=2-1, to=2-2]
	\arrow[equals, from=2-1, to=3-1]
	\arrow["{a_1}"{description}, from=2-2, to=1-2]
	\arrow[""{name=4, anchor=center, inner sep=0}, "{\prolist{\psi}_2}"'{inner sep=.8ex}, "\shortmid"{marking}, from=2-2, to=2-3]
	\arrow[equals, from=2-2, to=3-2]
	\arrow[""{name=5, anchor=center, inner sep=0}, "{\prolist{\psi}_n}"'{inner sep=.8ex}, "\shortmid"{marking}, from=2-3, to=2-4]
	\arrow["{a_n}"', from=2-4, to=1-4]
	\arrow[equals, from=2-4, to=3-4]
	\arrow["{\prolist{\psi}_1}"'{inner sep=.8ex}, "\shortmid"{marking}, from=3-1, to=3-2]
	\arrow["{\prolist{\psi}_2}"'{inner sep=.8ex}, "\shortmid"{marking}, from=3-2, to=3-3]
	\arrow["{\prolist{\psi}_n}"'{inner sep=.8ex}, "\shortmid"{marking}, from=3-3, to=3-4]
	\arrow["{\alpha_1}"{description}, draw=none, from=0, to=3]
	\arrow["{\alpha_2}"{description}, draw=none, from=1, to=4]
	\arrow["{\alpha_n}"{description}, draw=none, from=2, to=5]
\end{tikzcd}
	\]
	where in the center row each $\prolist{\psi}_i$ is viewed as a single loose in $\Cat{fc}(\Dbl{D})^{op_t}$, 
  so that the top row consists of unary multicells and the bottom row is an $n$-ary opcartesian multicell. 

	To show essential uniqueness suppose we had another globular decomposition
	\[
	\begin{tikzcd}
	{x_0} && {x_1} & \cdots && {x_n} \\
	{y_0} && {z_1} & \cdots && {y_{k_n}} \\
	{y_0} && {y_{k_1}} & \cdots && {y_{k_n}}
	\arrow[""{name=0, anchor=center, inner sep=0}, "{\varphi_1}"{inner sep=.8ex}, "\shortmid"{marking}, from=1-1, to=1-3]
	\arrow["{\varphi_2}"{inner sep=.8ex}, "\shortmid"{marking}, from=1-3, to=1-4]
	\arrow[""{name=1, anchor=center, inner sep=0}, "{\varphi_n}"{inner sep=.8ex}, "\shortmid"{marking}, from=1-4, to=1-6]
	\arrow["{a_0}", from=2-1, to=1-1]
	\arrow[""{name=2, anchor=center, inner sep=0}, "{(\chi_{1,1} \cdots \chi_{1,l_1})}"'{inner sep=.8ex}, "\shortmid"{marking}, from=2-1, to=2-3]
	\arrow[equals, from=2-1, to=3-1]
	\arrow["{c_1}"{description}, from=2-3, to=1-3]
	\arrow["\shortmid"{marking}, from=2-3, to=2-4]
	\arrow[""{name=3, anchor=center, inner sep=0}, "{(\chi_{n,1} \cdots \chi_{n,l_n})}"'{inner sep=.8ex}, "\shortmid"{marking}, from=2-4, to=2-6]
	\arrow["{a_n}"', from=2-6, to=1-6]
	\arrow[equals, from=2-6, to=3-6]
	\arrow[""{name=4, anchor=center, inner sep=0}, "{(\prolist{\psi}_{1,1} \cdots \prolist{\psi}_{1,l_1})}"'{inner sep=.8ex}, "\shortmid"{marking}, from=3-1, to=3-3]
	\arrow["{b_1}"{description}, from=3-3, to=2-3]
	\arrow["{\prolist{\psi}_2}"'{inner sep=.8ex}, "\shortmid"{marking}, from=3-3, to=3-4]
	\arrow[""{name=5, anchor=center, inner sep=0}, "{(\prolist{\psi}_{n,1} \cdots \prolist{\psi}_{n,l_n})}"'{inner sep=.8ex}, "\shortmid"{marking}, from=3-4, to=3-6]
	\arrow["{\beta_1}"{description}, draw=none, from=0, to=2]
	\arrow["{\beta_n}"{description}, draw=none, from=1, to=3]
	\arrow["{(\gamma_{1,1} \cdots \gamma_{1,l_1})}"{description, pos=0.7}, draw=none, from=2, to=4]
	\arrow["{(\gamma_{n,1} \cdots \gamma_{n,l_n})}"{description, pos=0.7}, draw=none, from=3, to=5]
\end{tikzcd}
	\]
	where $\prolist{\chi}_i=(\chi_{i,1} \cdots \chi_{i,l_i})$ is viewed as a single loose arrow in $\Cat{fc}(\Dbl{D})^{op_t}$. 
	Then we have the canonical decomposition
\[\begin{tikzcd}
	{x_0} && {x_1} & \cdots && {x_n} \\
	{y_0} && {z_1} & \cdots && {y_{k_n}} \\
	{y_0} && {y_{k_1}} & \cdots && {y_{k_n}} \\
	{y_0} && {y_{k_1}} & \cdots && {y_{k_n}}
	\arrow[""{name=0, anchor=center, inner sep=0}, "{{\varphi_1}}"{inner sep=.8ex}, "\shortmid"{marking}, from=1-1, to=1-3]
	\arrow["{{\varphi_2}}"{inner sep=.8ex}, "\shortmid"{marking}, from=1-3, to=1-4]
	\arrow[""{name=1, anchor=center, inner sep=0}, "{{\varphi_n}}"{inner sep=.8ex}, "\shortmid"{marking}, from=1-4, to=1-6]
	\arrow["{{a_0}}", from=2-1, to=1-1]
	\arrow[""{name=2, anchor=center, inner sep=0}, "{{(\chi_{1,1} \cdots \chi_{1,l_1})}}"'{inner sep=.8ex}, "\shortmid"{marking}, from=2-1, to=2-3]
	\arrow[equals, from=2-1, to=3-1]
	\arrow["{{c_1}}"{description}, from=2-3, to=1-3]
	\arrow["\shortmid"{marking}, from=2-3, to=2-4]
	\arrow[""{name=3, anchor=center, inner sep=0}, "{{(\chi_{n,1} \cdots \chi_{n,l_n})}}"'{inner sep=.8ex}, "\shortmid"{marking}, from=2-4, to=2-6]
	\arrow["{{a_n}}"', from=2-6, to=1-6]
	\arrow[equals, from=2-6, to=3-6]
	\arrow[""{name=4, anchor=center, inner sep=0}, "{{(\prolist{\psi}_{1,1} \cdots \prolist{\psi}_{1,l_1})}}"'{inner sep=.8ex}, "\shortmid"{marking}, from=3-1, to=3-3]
	\arrow["{{b_1}}"{description}, from=3-3, to=2-3]
	\arrow["{{\prolist{\psi}_2}}"'{inner sep=.8ex}, "\shortmid"{marking}, from=3-3, to=3-4]
	\arrow[""{name=5, anchor=center, inner sep=0}, "{{(\prolist{\psi}_{n,1} \cdots \prolist{\psi}_{n,l_n})}}"'{inner sep=.8ex}, "\shortmid"{marking}, from=3-4, to=3-6]
	\arrow[equals, from=3-6, to=4-6]
	\arrow[equals, from=4-1, to=3-1]
	\arrow["{{(\prolist{\psi}_{1,1} \cdots \prolist{\psi}_{1,l_1})}}"'{inner sep=.8ex}, "\shortmid"{marking}, from=4-1, to=4-3]
	\arrow[equals, from=4-3, to=3-3]
	\arrow["{{\prolist{\psi}_2}}"'{inner sep=.8ex}, "\shortmid"{marking}, from=4-3, to=4-4]
	\arrow["{{(\prolist{\psi}_{n,1} \cdots \prolist{\psi}_{n,l_n})}}"'{inner sep=.8ex}, "\shortmid"{marking}, from=4-4, to=4-6]
	\arrow["{{\beta_1}}"{description}, draw=none, from=0, to=2]
	\arrow["{{\beta_n}}"{description}, draw=none, from=1, to=3]
	\arrow["{{(\gamma_{1,1} \cdots \gamma_{1,l_1})}}"{description, pos=0.7}, draw=none, from=2, to=4]
	\arrow["{{(\gamma_{n,1} \cdots \gamma_{n,l_n})}}"{description, pos=0.7}, draw=none, from=3, to=5]
\end{tikzcd}\]
	where all the middle multicells $\prolist{\gamma}_i=(\gamma_{i,1} \cdots \gamma_{i,l_i})$ are viewed as unary multicells, completing the proof.
\end{proof}

To see a first application of VDCs having globular decompositions, 
we describe multicells in $\Dbl{X}^\Dbl{A}$ when $\Dbl{A}$ is an exponentiable VDC 
with $n$-ary globular decompositions for some $n$.

\begin{lem}[Multicells in $\Dbl{X}^\Dbl{A}$ for $\Dbl{A}$ Having Globular Decompositions]\label{lem:Uniqueglobular}
    If $\Dbl{A}$ is an exponentiable VDC with $n$-ary globular decompositions for some $n\geq 1$, and $\Dbl{X}$ is an arbitrary VDC, 
    then $n$-ary multicells in $\Dbl{X}^\Dbl{A}$ are uniquely determined by their action on globular multicells. 
    Additionally, any assignment on globular multicells which is compatible with the actions on the boundary loose arrows extends to such a multicell.
\end{lem}
\begin{proof}
    Let $\Gamma:(\alpha_0{\;}^{\prolist{\Phi}}_{\Psi}{\;}\alpha_1):(^{F_0 \cdots F_n}_{G_0\;\;\;\;G_1})$ be an $n$-ary multicell in $\Dbl{X}^\Dbl{A}$. Then for an $n$-ary multicell in $\Dbl{A}$ below left, we can factor it using
 	$n$-ary globular decompositions as below right:
\[\begin{tikzcd}
	x & y & x & y \\
	&& z & w \\
	z & w & z & w
	\arrow[""{name=0, anchor=center, inner sep=0}, "{{{{{\prolist{\varphi}}}}}}"{inner sep=.8ex}, "\shortmid"{marking}, from=1-1, to=1-2]
	\arrow["a"', from=1-1, to=3-1]
	\arrow[""{name=1, anchor=center, inner sep=0}, "b", from=1-2, to=3-2]
	\arrow[""{name=2, anchor=center, inner sep=0}, "{{{{{\prolist{\varphi}}}}}}"{inner sep=.8ex}, "\shortmid"{marking}, from=1-3, to=1-4]
	\arrow["a"', from=1-3, to=2-3]
	\arrow["{{{ b}}}", from=1-4, to=2-4]
	\arrow[""{name=3, anchor=center, inner sep=0}, "{{{{\prolist{\varphi}'}}}}"'{inner sep=.8ex}, "\shortmid"{marking}, from=2-3, to=2-4]
	\arrow[equals, from=2-3, to=3-3]
	\arrow[equals, from=2-4, to=3-4]
	\arrow[""{name=4, anchor=center, inner sep=0}, "\psi"'{inner sep=.8ex}, "\shortmid"{marking}, from=3-1, to=3-2]
	\arrow[""{name=5, anchor=center, inner sep=0}, "\psi"'{inner sep=.8ex}, "\shortmid"{marking}, from=3-3, to=3-4]
	\arrow["\alpha"{description}, draw=none, from=0, to=4]
	\arrow[between={0.4}{0.6}, equals, from=1, to=2-3]
	\arrow["{{\prolist{\beta}}}"{description}, draw=none, from=2, to=3]
	\arrow["{{\alpha'}}"{description}, draw=none, from=3, to=5]
\end{tikzcd}\]
    Functoriality of pasting for the multicell $\Gamma$ then gives the equality
\[\begin{tikzcd}
	{F_0x} & {F_ny} && {F_0x} & {F_ny} \\
	&&& {F_0z} & {F_nw} \\
	{G_0z} & {G_1w} && {G_0z} & {G_1 d}
	\arrow[""{name=0, anchor=center, inner sep=0}, "{{{{{\prolist{\Phi}({\prolist{\varphi}}})}}}}"{inner sep=.8ex}, "\shortmid"{marking}, from=1-1, to=1-2]
	\arrow["{{{{{f_0(a)}}}}}"', from=1-1, to=3-1]
	\arrow[""{name=1, anchor=center, inner sep=0}, "{{{{f_1(b)}}}}", from=1-2, to=3-2]
	\arrow[""{name=2, anchor=center, inner sep=0}, "{{{{{\prolist{\Phi}({\prolist{\varphi}}})}}}}"{inner sep=.8ex}, "\shortmid"{marking}, from=1-4, to=1-5]
	\arrow["{{{{{F_0a}}}}}"', from=1-4, to=2-4]
	\arrow["{{{{{F_nb}}}}}", from=1-5, to=2-5]
	\arrow[""{name=3, anchor=center, inner sep=0}, "{{{{{\prolist{\Phi}(\prolist{\varphi}'})}}}}"'{inner sep=.8ex}, "\shortmid"{marking}, from=2-4, to=2-5]
	\arrow["{{{{f_0(\text{id}_z)}}}}"', from=2-4, to=3-4]
	\arrow["{{{{f_1(\text{id}_w)}}}}", from=2-5, to=3-5]
	\arrow[""{name=4, anchor=center, inner sep=0}, "{{{{{\Psi(\psi)}}}}}"'{inner sep=.8ex}, "\shortmid"{marking}, from=3-1, to=3-2]
	\arrow[""{name=5, anchor=center, inner sep=0}, "{{{{{\Psi(\psi)}}}}}"'{inner sep=.8ex}, "\shortmid"{marking}, from=3-4, to=3-5]
	\arrow["{{\Gamma\alpha}}"{description}, draw=none, from=0, to=4]
	\arrow[between={0.4}{0.5}, equals, from=1, to=2-4]
	\arrow["{{\prolist{\Phi}\prolist{\beta}}}"{description}, draw=none, from=2, to=3]
	\arrow["{{\Gamma\alpha'}}"{description}, draw=none, from=3, to=5]
\end{tikzcd}\]
    Thus, $\Gamma$ is determined by its image on globular multicells, as desired.

    Conversely, suppose that for the same boundary we have an assignment $\Gamma$ on globular $n$-ary multicells such that for any pasting equality in $\Dbl{A}$ below left, we have the equality below right in $\Dbl{X}$
\[\begin{tikzcd}
	x & y & x & y & {F_0x} & {F_ny} & {F_0x} & {F_ny} \\
	z & w & {z'} & {w'} & {F_0z} & {F_nw} & {F_0z'} & {F_nw'} \\
	z & w & {z'} & {w'} & {G_0z} & {G_1w} & {G_0z'} & {G_1w'} \\
	u & v & u & v & {G_0u} & {G_1v} & {G_0u} & {G_1v}
	\arrow[""{name=0, anchor=center, inner sep=0}, "{{\prolist{\varphi}}}"{inner sep=.8ex}, "\shortmid"{marking}, from=1-1, to=1-2]
	\arrow["a"', from=1-1, to=2-1]
	\arrow["b", from=1-2, to=2-2]
	\arrow[""{name=1, anchor=center, inner sep=0}, "{{\prolist{\varphi}}}"{inner sep=.8ex}, "\shortmid"{marking}, from=1-3, to=1-4]
	\arrow["{{a'}}"', from=1-3, to=2-3]
	\arrow["{{b'}}", from=1-4, to=2-4]
	\arrow[""{name=2, anchor=center, inner sep=0}, "{{{\prolist{\Phi}({{{\prolist{\varphi}}}})}}}"{inner sep=.8ex}, "\shortmid"{marking}, from=1-5, to=1-6]
	\arrow["{{{F_0a}}}"', from=1-5, to=2-5]
	\arrow["{{{F_nb}}}", from=1-6, to=2-6]
	\arrow[""{name=3, anchor=center, inner sep=0}, "{{{\prolist{\Phi}({{{\prolist{\varphi}}}})}}}"{inner sep=.8ex}, "\shortmid"{marking}, from=1-7, to=1-8]
	\arrow["{{{F_0a'}}}"', from=1-7, to=2-7]
	\arrow["{{{F_nb'}}}", from=1-8, to=2-8]
	\arrow[""{name=4, anchor=center, inner sep=0}, "{{\prolist{\varphi'}}}"'{inner sep=.8ex}, "\shortmid"{marking}, from=2-1, to=2-2]
	\arrow[equals, from=2-1, to=3-1]
	\arrow[""{name=5, anchor=center, inner sep=0}, equals, from=2-2, to=3-2]
	\arrow[""{name=6, anchor=center, inner sep=0}, "{{\prolist{\varphi''}}}"'{inner sep=.8ex}, "\shortmid"{marking}, from=2-3, to=2-4]
	\arrow[""{name=7, anchor=center, inner sep=0}, equals, from=2-3, to=3-3]
	\arrow[equals, from=2-4, to=3-4]
	\arrow[""{name=8, anchor=center, inner sep=0}, "{{{\prolist{\Phi}({{\prolist{\varphi}'}})}}}"'{inner sep=.8ex}, "\shortmid"{marking}, from=2-5, to=2-6]
	\arrow["{{{f_0(z)}}}"', from=2-5, to=3-5]
	\arrow[""{name=9, anchor=center, inner sep=0}, "{{{f_1(w)}}}"{description}, from=2-6, to=3-6]
	\arrow[""{name=10, anchor=center, inner sep=0}, "{{{\prolist{\Phi}({{\prolist{\varphi}''}})}}}"'{inner sep=.8ex}, "\shortmid"{marking}, from=2-7, to=2-8]
	\arrow[""{name=11, anchor=center, inner sep=0}, "{{{f_0(z')}}}"{description}, from=2-7, to=3-7]
	\arrow["{{{f_1(w')}}}", from=2-8, to=3-8]
	\arrow[""{name=12, anchor=center, inner sep=0}, "\psi"'{inner sep=.8ex}, "\shortmid"{marking}, from=3-1, to=3-2]
	\arrow["c"', from=3-1, to=4-1]
	\arrow["d", from=3-2, to=4-2]
	\arrow[""{name=13, anchor=center, inner sep=0}, "{{\psi'}}"'{inner sep=.8ex}, "\shortmid"{marking}, from=3-3, to=3-4]
	\arrow["{{c'}}"', from=3-3, to=4-3]
	\arrow["{{d'}}", from=3-4, to=4-4]
	\arrow[""{name=14, anchor=center, inner sep=0}, "{{{\Psi\psi}}}"'{inner sep=.8ex}, "\shortmid"{marking}, from=3-5, to=3-6]
	\arrow["{{{G_0c}}}"', from=3-5, to=4-5]
	\arrow["{{{G_1d}}}", from=3-6, to=4-6]
	\arrow[""{name=15, anchor=center, inner sep=0}, "{{{\Psi\psi'}}}"'{inner sep=.8ex}, "\shortmid"{marking}, from=3-7, to=3-8]
	\arrow["{{{G_0c'}}}"', from=3-7, to=4-7]
	\arrow["{{{G_1d'}}}", from=3-8, to=4-8]
	\arrow[""{name=16, anchor=center, inner sep=0}, "\rho"', from=4-1, to=4-2]
	\arrow[""{name=17, anchor=center, inner sep=0}, "\rho"', from=4-3, to=4-4]
	\arrow[""{name=18, anchor=center, inner sep=0}, "{{{\Psi\rho}}}"'{inner sep=.8ex}, "\shortmid"{marking}, from=4-5, to=4-6]
	\arrow[""{name=19, anchor=center, inner sep=0}, "{{{\Psi\rho}}}"'{inner sep=.8ex}, "\shortmid"{marking}, from=4-7, to=4-8]
	\arrow["{{\prolist{\beta}}}"{description}, draw=none, from=0, to=4]
	\arrow["{{\prolist{\beta'}}}"{description}, draw=none, from=1, to=6]
	\arrow["{{\prolist{\Phi}\prolist{\beta}}}"{description}, draw=none, from=2, to=8]
	\arrow["{{\prolist{\Phi}\prolist{\beta}'}}"{description}, draw=none, from=3, to=10]
	\arrow[between={0.4}{0.6}, equals, from=5, to=7]
	\arrow["{{\Gamma\alpha}}"{description}, draw=none, from=8, to=14]
	\arrow[between={0.4}{0.6}, equals, from=9, to=11]
	\arrow["{{\Gamma \alpha'}}"{description}, draw=none, from=10, to=15]
	\arrow["\alpha"{description}, draw=none, from=12, to=4]
	\arrow["\gamma"{description}, draw=none, from=12, to=16]
	\arrow["{{\alpha'}}"{description}, draw=none, from=13, to=6]
	\arrow["{{\gamma'}}"{description}, draw=none, from=13, to=17]
	\arrow["{{\Psi\gamma}}"{description}, draw=none, from=14, to=18]
	\arrow["{{\Psi\gamma'}}"{description}, draw=none, from=15, to=19]
\end{tikzcd}\]
    Then these assignments assemble into an $n$-ary multicell in $\Dbl{X}^\Dbl{A}$.
\end{proof}

\subsection{Local colimits and the characterization theorem}

Next, we can define local colimits in a VDC in analogy with Definition 2.2.1 of~\cite{Pare2013}. 
First, for a VDC $\Dbl{D}$ with objects $x,y \in \Dbl{D}$, let $\Cat{gCell}_n(\Dbl{D})(x,y)$ 
denote the category whose objects are length $n$ sequences of loose arrows $x\proto y$, and 
whose morphisms are sequences of unary multicells, with the outermost tight arrows being identities.
Then for a category $\Cat{I}$, a functor $\prolist{\varphi}:\Cat{I}\to \Cat{gCell}_n(\Dbl{D})(x,y)$, 
tight arrows $a:x\to z$ and $b:y\to w$, and a loose arrow $\psi:z\proto w$, we will write 
$\uSq{n}(\Dbl{D})\cell{a}{\prolist{\varphi}}{\psi}{b}$ for the set of cocones under $\prolist{\varphi}$ with specified boundary. 
When $a$ and $b$ are not written explicitly, the boundary will be understood to be identity arrows. 
Here a cocone under $\prolist{\varphi}$ is a collection of multicells 
$\left\{\alpha_i:\cell{a}{\prolist{\varphi(i)}}{\psi}{b}\right\}_{i \in \Cat{I}}$ satisfying 
$\frac{\prolist{\varphi}(p)}{\alpha_j}=\alpha_i$ for all $p:i\to j$ in $\Cat{I}$.

We define a stable notion of local colimits in VDCs, so that colimits in $\Cat{gCell}(\Dbl{D})(x,y):=\Cat{gCell}_1(\Dbl{D})(x,y)$ 
must both exist and satisfy a property virtualizing the preservation of local colimits by composition of loose arrows.

\begin{defn}[Local $\Cat{I}$-Colimits]\label{defn:localColim}
    Let $\Cat{I}$ be a small category and let $\Dbl{D}$ be a VDC. Then $\Dbl{D}$ is said to have local $\Cat{I}$-colimits if the following are satisfied:
    \begin{itemize}
        \item[(L1)] For each $x,y \in \uOb (\Dbl{D})$, category $\Cat{gCell}(\Dbl{D})(x,y)$ has $\Cat{I}$-colimits.
        \item[(L2)] For every pair of sequences of loose arrows $\prolist{\chi}_1:x\proto y$ and $\prolist{\chi}_2:z\proto w$, 
        and every functor $\varphi:\Cat{I}\to \Cat{gCell}(\Dbl{D})(y,z)$ with colimit $\colim_\Cat{I}\,\varphi:y\proto z$ and 
        colimit cocone $\lambda_i:\varphi_i\Rightarrow \Phi$, pre-composition by the cone legs induces a bijection
		\begin{equation*}
			\uSq{|\prolist{\chi}_1|+1+|\prolist{\chi}_2|}(\Dbl{D})\cell{a}{\prolist{\chi}_1,\colim_\Cat{I}\,\varphi,\prolist{\chi}_2}{\psi}{b}\xrightarrow{\cong}\uSq{|\prolist{\chi}_1|+1+|\prolist{\chi}_2|}(\Dbl{D})\cell{a}{\prolist{\chi}_1,\varphi,\prolist{\chi}_2}{\psi}{b}
		\end{equation*}
		between multicells and cocones for any fixed tight arrows $a:x\to x'$ and $b:w\to w'$, and any fixed loose arrow $\psi:x'\proto w'$.
    \end{itemize}
	We will say a VDC has \emph{weak local $\Cat{I}$-colimits} if it only satisfies (L2) for $\prolist{\chi}_1$ and $\prolist{\chi}_2$ 
  being empty sequences of loose arrows, and in this case we refer to the condition as (wL2).	
\qed
\end{defn}

We will be most often considering local $\Cat{I}$-colimits in VDCs $\Dbl{D}$ with weak non-nullary composites, 
in which case (gL2) (i.e.~condition (L2) when $a$ and $b$ are identity tight arrows) can be rephrased as saying 
that the functors given by composing loose arrows preserve the $\Cat{I}$-colimits in (L1). On the other hand, 
(wL2) gives a tight-base change invariance of colimits. When the VDC being considered has restrictions (L2) follows from (gL2).

With Lemma~\ref{lem:Uniqueglobular} we can use local colimits to provide sufficient conditions for $\Dbl{X}^\Dbl{A}$ to have (weak) non-nullary composites. The arguments are largely analogous to those due to Par\'{e} in~\cite{Pare2013}.

\begin{thm}\label{thm:weakComp}
    If $\Dbl{A}$ is a small exponentiable VDC that has positive globular decompositions and $\Dbl{X}$ is a large (weakly) locally cocomplete VDC with (weak) non-nullary composites, then $\Dbl{X}^\Dbl{A}$ has (weak) non-nullary composites.
\end{thm}

Before proceeding to the proof, we'll construct a candidate for the weak composite in the exponential.

\begin{cons}\label{cons:CompCandidate}
	Let $\Dbl{A}$ be an exponentiable VDC that has $n$-ary globular decompositions, 
	for $n\geq 2$ some integer, and let $\Dbl{X}$ be a weakly locally cocomplete VDC with weak non-nullary composites. 
  Consider a length-$n$ sequence of loose arrows in $\Dbl{X}^\Dbl{A}$, $\prolist{\Phi}:F_0\proto F_n$. 
  To define a candidate for its composite $\odot(\prolist{\Phi}):F_0\proto F_n$ we need to define 
  $\odot(\prolist{\Phi})$ on loose arrows in $\Dbl{A}$, and on unary multicells. 
  Roughly speaking, we will define $\odot(\prolist{\Phi})$ on a loose arrow $\varphi:x\proto y$ in $\Dbl{A}$ by 
  approximating $\varphi$ by domains of globular $n$-ary multicells, applying $\prolist{\Phi}$ to their domains, and colimiting all the results.

	In more detail: given a loose arrow $\varphi:x\proto y$ in $\Dbl{A}$, let $\Cat{I}_{\varphi,n}$, $n \geq 2$, 
  be the category whose objects are globular $n$-ary multicells in $\Dbl{A}$ with loose target $\varphi$, 
  and whose non-identity morphisms $\alpha\to \beta$ consist of compatible sequences of multicells $(\gamma_1\cdots\gamma_n) \in \MCelln{n}(\Dbl{A})$ 
  such that $\alpha=\frac{\gamma_1\cdots \gamma_n}{\beta}$. Note that necessarily the tight source of $\gamma_1$ and the 
  tight target of $\gamma_n$ are identities, so $\prolist{\gamma}\in \Cat{gCell}_n(\Dbl{A})$.
	Then we define a functor $\Gamma^{\prolist{\Phi}}_{\varphi,n}:\Cat{I}_{\varphi,n}\to \Cat{gCell}(\Dbl{X})(F_0x,F_ny)$ by sending an object $\alpha:\scell{\psi_1 \cdots \psi_n}{\varphi}$ to the composite
	\begin{equation*}
		\Gamma_{\varphi,n}^{\prolist{\Phi}}(\alpha) := \Phi_1\psi_1\odot \cdots \odot \Phi_n\psi_n
	\end{equation*}
	and sending a morphism $\prolist{\gamma}:\scell{\varphi_1' \cdots \varphi_n'}{\varphi}\to \scell{\psi_1 \cdots \psi_n}{\varphi}$ to the unique multicell appearing in the factorization
	\[\begin{adjustbox}{}
\begin{tikzcd}
	{F_0x} && {F_ny} && {F_0x} && {F_ny} \\
	{F_0x} && {F_ny} & {=} & {F_0x} && {F_ny} \\
	{F_0x} && {F_ny} && {F_0x} && {F_ny}
	\arrow[""{name=0, anchor=center, inner sep=0}, "{{{{\prolist{\Phi}\prolist{\varphi}'}}}}"{inner sep=.8ex}, "\shortmid"{marking}, from=1-1, to=1-3]
	\arrow[equals, from=1-1, to=2-1]
	\arrow[equals, from=1-3, to=2-3]
	\arrow[""{name=1, anchor=center, inner sep=0}, "{{{{\prolist{\Phi}\prolist{\varphi}'}}}}"{inner sep=.8ex}, "\shortmid"{marking}, from=1-5, to=1-7]
	\arrow[equals, from=1-5, to=2-5]
	\arrow[equals, from=1-7, to=2-7]
	\arrow[""{name=2, anchor=center, inner sep=0}, "{{{{\prolist{\Phi}\prolist{\psi}}}}}"'{inner sep=.8ex}, "\shortmid"{marking}, from=2-1, to=2-3]
	\arrow[equals, from=2-1, to=3-1]
	\arrow[equals, from=2-3, to=3-3]
	\arrow[""{name=3, anchor=center, inner sep=0}, "{{{{\Phi \varphi_1'\odot \cdots \odot \Phi \varphi_n'}}}}"{inner sep=.8ex}, "\shortmid"{marking}, from=2-5, to=2-7]
	\arrow[equals, from=2-5, to=3-5]
	\arrow[equals, from=2-7, to=3-7]
	\arrow[""{name=4, anchor=center, inner sep=0}, "{{{{\Phi \psi_1\odot \cdots \odot \Phi \psi_n}}}}"'{inner sep=.8ex}, "\shortmid"{marking}, from=3-1, to=3-3]
	\arrow[""{name=5, anchor=center, inner sep=0}, "{{{{\Phi \psi_1\odot \cdots \odot \Phi \psi_n}}}}"'{inner sep=.8ex}, "\shortmid"{marking}, from=3-5, to=3-7]
	\arrow["{\prolist{\Phi}\prolist{\gamma}}"{description}, draw=none, from=0, to=2]
	\arrow["{\mathsf{opcart}}"{description}, draw=none, from=1, to=3]
	\arrow["{\mathsf{opcart}}"{description}, draw=none, from=2, to=4]
	\arrow["{\exists!\Gamma_{\varphi,n}^{\prolist{\Phi}}(\prolist{\gamma})}"{description}, draw=none, from=3, to=5]
\end{tikzcd}
\end{adjustbox}\]
	Then for any functor $Q:\Cat{gCell}(\Dbl{X})(F_0x,F_ny)\to \Cat{C}$, a cocone $\kappa:Q\Gamma_{\varphi,n}^{\prolist{\Phi}}\Rightarrow \prolist{c}$ with tip $c\in \Cat{C}$ is equivalent to a family of maps $\kappa(\alpha):Q(\Phi_1 \psi_1\odot \cdots\odot \Phi_n\psi_n)\to c$ indexed by objects of $\Cat{I}_{\varphi,n}$ such that if $\prolist{\gamma}:\alpha\to \beta$ is a morphism in $\Cat{I}_{\varphi,n}$, then $\kappa(\beta)\circ Q(\Phi_1\gamma_1\odot \cdots\odot \Phi_n\gamma_n) = \kappa(\alpha)$.

	Using the fact that $\Dbl{X}$ is locally cocomplete and $\Dbl{A}$ is small, 
	so that the categories $\Cat{I}_{\varphi,n}$ are also small, 
	we can define $\odot(\prolist{\Phi})(\varphi):=\colim_{\Cat{I}_{\varphi,n}}\;\Gamma_{\varphi,n}^{\prolist{\Phi}}$. 
	For an object $\beta:\scell{\prolist{\varphi}'}{\varphi}$ in $\Cat{I}_{\varphi,n}$, 
	let $\lambda_\beta:\scell{\prolist{\Phi}(\prolist{\varphi}')}{\odot(\prolist{\Phi})(\varphi)}$ 
	denote the associated cocone leg for the colimit. 
	To define $\odot(\prolist{\Phi})$ on a unary multicell $\sigma:\cell{a}{\varphi}{\psi}{b}$, 
	note that for any $\alpha \in \Cat{I}_{\varphi,n}$, as below left, we can use
	$n$-ary globular decompositions to obtain the essentially unique decomposition below right:
	\[\begin{adjustbox}{}
\begin{tikzcd}
	x & y && x & y \\
	x & y & {=} & z & w \\
	z & w && z & w
	\arrow[""{name=0, anchor=center, inner sep=0}, "{{\prolist{\rho}}}"{inner sep=.8ex}, "\shortmid"{marking}, from=1-1, to=1-2]
	\arrow[equals, from=1-1, to=2-1]
	\arrow[equals, from=1-2, to=2-2]
	\arrow[""{name=1, anchor=center, inner sep=0}, "{{\prolist{\rho}}}"{inner sep=.8ex}, "\shortmid"{marking}, from=1-4, to=1-5]
	\arrow["a"', from=1-4, to=2-4]
	\arrow["b", from=1-5, to=2-5]
	\arrow[""{name=2, anchor=center, inner sep=0}, "{{\varphi}}"'{inner sep=.8ex}, "\shortmid"{marking}, from=2-1, to=2-2]
	\arrow["a"', from=2-1, to=3-1]
	\arrow["b", from=2-2, to=3-2]
	\arrow[""{name=3, anchor=center, inner sep=0}, "{{\prolist{\chi}}}"'{inner sep=.8ex}, "\shortmid"{marking}, from=2-4, to=2-5]
	\arrow[equals, from=2-4, to=3-4]
	\arrow[equals, from=2-5, to=3-5]
	\arrow[""{name=4, anchor=center, inner sep=0}, "{{\psi}}"'{inner sep=.8ex}, "\shortmid"{marking}, from=3-1, to=3-2]
	\arrow[""{name=5, anchor=center, inner sep=0}, "{{\psi}}"'{inner sep=.8ex}, "\shortmid"{marking}, from=3-4, to=3-5]
	\arrow["\alpha"{description}, draw=none, from=0, to=2]
	\arrow["{\prolist{\delta}}"{description}, draw=none, from=1, to=3]
	\arrow["\sigma"{description}, draw=none, from=2, to=4]
	\arrow["\beta"{description}, draw=none, from=3, to=5]
\end{tikzcd}
\end{adjustbox}\]
	We can then define a cocone $\kappa(\sigma)$ with leg $\kappa(\sigma)_\alpha:= \frac{\prolist{\Phi}\prolist{\delta}}{\lambda_\beta}$ 
 	The essential uniqueness of the decomposition, and the fact that $\lambda$ constitutes a cocone ensures that this definition is independent of the choice of decomposition, and also defines a cocone. By the universal property of local colimits (specifically property (wL2)), the cocone $\kappa(\sigma)$ uniquely induces a multicell
	\[\begin{adjustbox}{}
\begin{tikzcd}
	{F_0x} && {F_ny} \\
	{F_0z} && {F_nw}
	\arrow[""{name=0, anchor=center, inner sep=0}, "{{(\Phi_1\odot \cdots \odot \Phi_n)(\varphi)}}"{inner sep=.8ex}, "\shortmid"{marking}, from=1-1, to=1-3]
	\arrow["{{F_0a}}"', from=1-1, to=2-1]
	\arrow["{{F_nb}}", from=1-3, to=2-3]
	\arrow[""{name=1, anchor=center, inner sep=0}, "{{(\Phi_1\odot \cdots \odot \Phi_n)(\psi)}}"'{inner sep=.8ex}, "\shortmid"{marking}, from=2-1, to=2-3]
	\arrow["{(\Phi_1\odot \cdots  \odot \Phi_n)(\sigma)}"{description}, draw=none, from=0, to=1]
\end{tikzcd}
\end{adjustbox}\]
	Uniqueness of these decompositions ensures that this definition is functorial in the multicells. Using these unique multicells we obtain our desired functor $\odot(\prolist{\Phi}):\Dbl{A}_1\to \Dbl{X}_1$ defining a loose arrow $F_0\proto F_n$ in $\Dbl{X}^\Dbl{A}$. We will also denote this functor by $\Phi_1\odot \cdots \odot \Phi_n$. In other words, we're defining the composite by the coend
	\begin{equation}\label{eq:coendDescOfComposite}
		\int^{\prolist{\varphi}\in \Cat{gl}(\Cat{fc}(\Dbl{A}_1))}\Cat{gCell}(\Dbl{A})(\prolist{\varphi},-)\star (\Phi_1(\varphi_1)\odot \cdots \odot \Phi_n(\varphi_n))
	\end{equation}
	where $\star$ denotes the tensoring of $\Dbl{X}_1$ over $\Cat{Set}$, which exists since $\Dbl{X}$ is locally cocomplete.

	Finally, in order to be a composition candidate we need a comparison multicell. The universal property of local colimits defining $\odot(\prolist{\Phi})$ on loose arrows and multicells, along with Lemma~\ref{lem:Uniqueglobular}, then allows us to define a globular multicell $\odot_{\prolist{\Phi}}:(\Phi_1 \cdots \Phi_n)\Rightarrow \odot(\prolist{\Phi})$, which at a globular $n$-ary multicell $\alpha:\scell{\psi_1 \cdots \psi_n}{\varphi}$ is given by the leg of the colimit cocone $\lambda_\alpha:\scell{\Phi_1\psi_1 \cdots \Phi_n\psi_n}{\odot(\prolist{\Phi})(\varphi)}$. 
	\qed
\end{cons}

In order to prove Theorem~\ref{thm:weakComp} we will show that this multicell defined in the construction is (weakly) opcartesian. We'll begin with proving the weak case:

\begin{proof}[Proof of the Weak Case for Theorem~\ref{thm:weakComp}.]
	Suppose $\Dbl{A}$ is a small exponentiable VDC that has positive globular decompositions, 
	and $\Dbl{X}$ is a large weakly locally cocomplete VDC with weak non-nullary composites. 
	Consider a multicell $\Gamma:\cell{f_0}{\prolist{\Phi}}{\Psi}{f_1}$ in $\Dbl{X}^\Dbl{A}$.
    Then for any globular $n$-cell $\alpha:\scell{\varphi_1 \cdots \varphi_n}{\psi}$ we obtain the multicell $\Gamma(\alpha):\cell{f_0(\text{id}_x)}{\prolist{\Phi}\prolist{\varphi}}{\Psi\psi}{f_1(\text{id}_y)}$
    with the functoriality of pasting ensuring that this defines a cocone, as in property (wL2). Thus, by the universal property of the colimit we obtain a unique multicell 
    \[\begin{adjustbox}{}
\begin{tikzcd}
	{F_0x} && {F_ny} \\
	{G_0x} && {G_1y}
	\arrow[""{name=0, anchor=center, inner sep=0}, "{{(\Phi_1\odot \cdots \odot \Phi_n)(\psi)}}"{inner sep=.8ex}, "\shortmid"{marking}, from=1-1, to=1-3]
	\arrow["{{f_0(\text{id}_{x})}}"', from=1-1, to=2-1]
	\arrow["{{f_1(\text{id}_{y})}}", from=1-3, to=2-3]
	\arrow[""{name=1, anchor=center, inner sep=0}, "{{\Psi \psi}}"'{inner sep=.8ex}, "\shortmid"{marking}, from=2-1, to=2-3]
	\arrow["{\exists!\Gamma^\flat(\text{id}_\psi)}"{description}, draw=none, from=0, to=1]
\end{tikzcd}
\end{adjustbox}\]
    through which the $\Gamma(\alpha)$ factor as $\frac{\odot_{\prolist{\Phi}}(\alpha)}{\Gamma^\flat(\text{id}_\psi)}$. It remains to show that this definition of $\Gamma^\flat$ is compatible with pasting from below, as in the proof of Lemma~\ref{lem:Uniqueglobular}. Observe that if we paste a unary multicell $\sigma:\cell{a}{\psi}{\psi'}{b}$ below our cocone defining $\Gamma^\flat(\text{id}_\psi)$, then we obtain a new cocone given by the family of multicells $\Gamma\left(\frac{\alpha}{\sigma}\right)=\frac{\Gamma(\alpha)}{\Psi\sigma}$.

    By property (wL2) for local colimits, this cocone induces a unique multicell $\Gamma^\flat(\sigma):\cell{f_0(a)}{(\Phi_1\odot \cdots \odot \Phi_n)(\psi)}{\Psi\psi'}{f_1(b)}$
    which the cocone legs factor through. Uniqueness then ensures that we have the pasting equalities

\[\begin{tikzcd}[column sep=small]
	{F_0x} &&&& {F_ny} && {F_0x} && {F_ny} && {F_0x} && {F_ny} \\
	{F_0z} &&&& {F_nw} & {=} &&&& {=} & {G_0x} && {G_1y} \\
	{G_0z} &&&& {G_1w} && {G_0z} && {G_1w} && {G_0z} && {G_1w}
	\arrow[""{name=0, anchor=center, inner sep=0}, "{{{{(\Phi_1\odot \cdots \odot \Phi_n)(\psi)}}}}"{inner sep=.8ex}, "\shortmid"{marking}, from=1-1, to=1-5]
	\arrow["{{{{F_0a}}}}"', from=1-1, to=2-1]
	\arrow["{{{{F_nb}}}}", from=1-5, to=2-5]
	\arrow[""{name=1, anchor=center, inner sep=0}, "{{{{(\Phi_1\odot \cdots \odot \Phi_n)(\psi)}}}}"{inner sep=.8ex}, "\shortmid"{marking}, from=1-7, to=1-9]
	\arrow["{{{{{f_0(a)}}}}}"', from=1-7, to=3-7]
	\arrow["{{{{{f_1(b)}}}}}", from=1-9, to=3-9]
	\arrow[""{name=2, anchor=center, inner sep=0}, "{{{{(\Phi_1\odot \cdots \odot \Phi_n)(\psi)}}}}"{inner sep=.8ex}, "\shortmid"{marking}, from=1-11, to=1-13]
	\arrow["{{{{{f_0(\text{id}_{x})}}}}}"', from=1-11, to=2-11]
	\arrow["{{{{{f_1(\text{id}_{y})}}}}}", from=1-13, to=2-13]
	\arrow[""{name=3, anchor=center, inner sep=0}, "{{{{(\Phi_1\odot \cdots \odot \Phi_n)(\psi')}}}}"'{inner sep=.8ex}, "\shortmid"{marking}, from=2-1, to=2-5]
	\arrow["{{{{f_0(\text{id}_{z})}}}}"', from=2-1, to=3-1]
	\arrow["{{{{f_1(\text{id}_{w})}}}}", from=2-5, to=3-5]
	\arrow[""{name=4, anchor=center, inner sep=0}, "{{{{\Psi \varphi}}}}"'{inner sep=.8ex}, "\shortmid"{marking}, from=2-11, to=2-13]
	\arrow["{{{{G_0a}}}}"', from=2-11, to=3-11]
	\arrow["{{{{G_1b}}}}", from=2-13, to=3-13]
	\arrow[""{name=5, anchor=center, inner sep=0}, "{{{{{\Psi \psi'}}}}}"'{inner sep=.8ex}, "\shortmid"{marking}, from=3-1, to=3-5]
	\arrow[""{name=6, anchor=center, inner sep=0}, "{{{{{\Psi \psi'}}}}}"'{inner sep=.8ex}, "\shortmid"{marking}, from=3-7, to=3-9]
	\arrow[""{name=7, anchor=center, inner sep=0}, "{{{{{\Psi \psi'}}}}}"'{inner sep=.8ex}, "\shortmid"{marking}, from=3-11, to=3-13]
	\arrow["{{{(\Phi_1\odot \cdots \odot \Phi_n)(\sigma)}}}"{description}, draw=none, from=0, to=3]
	\arrow["{{{\Gamma^\flat(\sigma)}}}"{description}, draw=none, from=1, to=6]
	\arrow["{{{\Gamma^\flat(\text{id}_\varphi)}}}"{description}, draw=none, from=2, to=4]
	\arrow["{{{\Gamma^\flat(\text{id}_\psi)}}}"{description, pos=0.7}, draw=none, from=3, to=5]
	\arrow["{{{\Psi(\sigma)}}}"{description}, draw=none, from=4, to=7]
\end{tikzcd}\]
    making $\Gamma^\flat$ a well-defined unary multicell in $\Dbl{X}^\Dbl{A}$. Further, by construction it follows that we have the desired equality of pastings in $\Dbl{X}^\Dbl{A}$:
    \[\begin{tikzcd}[column sep = small]
	{F_0} && {F_n} && {F_0} & {F_n} \\
	{F_0} && {F_n} & {=} \\
	{G_0} && {G_1} && {G_0} & {G_1}
	\arrow[""{name=0, anchor=center, inner sep=0}, "{{{\prolist{\Phi}}}}"{inner sep=.8ex}, "\shortmid"{marking}, from=1-1, to=1-3]
	\arrow[equals, from=1-1, to=2-1]
	\arrow[equals, from=1-3, to=2-3]
	\arrow[""{name=1, anchor=center, inner sep=0}, "{{{\prolist{\Phi}}}}"{inner sep=.8ex}, "\shortmid"{marking}, from=1-5, to=1-6]
	\arrow["{{{f_0}}}"', from=1-5, to=3-5]
	\arrow["{{{f_1}}}", from=1-6, to=3-6]
	\arrow[""{name=2, anchor=center, inner sep=0}, "{{\odot(\prolist{\varphi})}}"'{inner sep=.8ex}, "\shortmid"{marking}, from=2-1, to=2-3]
	\arrow["{{{f_0}}}"', from=2-1, to=3-1]
	\arrow["{{{f_1}}}", from=2-3, to=3-3]
	\arrow[""{name=3, anchor=center, inner sep=0}, "\Psi"'{inner sep=.8ex}, "\shortmid"{marking}, from=3-1, to=3-3]
	\arrow[""{name=4, anchor=center, inner sep=0}, "\Psi"'{inner sep=.8ex}, "\shortmid"{marking}, from=3-5, to=3-6]
	\arrow["{{\odot_{\prolist{\Phi}}}}"{description}, draw=none, from=0, to=2]
	\arrow["\Gamma"{description}, draw=none, from=1, to=4]
	\arrow["{{\Gamma^\flat}}"{description}, draw=none, from=2, to=3]
\end{tikzcd}\]
    Additionally, any other unary multicell in $\Dbl{X}^\Dbl{A}$ giving such a pasting equality is uniquely determined on globular multicells by the universal property of the colimit, so that by Lemma~\ref{lem:Uniqueglobular} $\Gamma^\flat$ is unique. Hence, $\odot_{\prolist{\Phi}}$ is weakly opcartesian.
\end{proof}

Finally, we prove the case for strong non-nullary composites.

\begin{proof}[Proof of the Strong Case for Theorem~\ref{thm:weakComp}.]
	Suppose $\Dbl{A}$ is a small exponentiable VDC that has positive globular decompositions, 
	and $\Dbl{X}$ is a large locally cocomplete VDC with non-nullary composites. 
	Consider a multicell $\Gamma:\cell{f_0}{(\prolist{\Omega},\prolist{\Phi},\prolist{\Xi})}{\Psi}{f_1}$
  in $\Dbl{X}^\Dbl{A}$. Then we need to find a unique multicell $\Gamma^\flat:\cell{f_0}{(\prolist{\Omega},\odot(\prolist{\Phi}),\prolist{\Xi})}{\Psi}{f_1}$
  which it factors through via $\odot_{\prolist{\Phi}}$. To define this multicell on a multicell $\delta:\cell{a}{(\prolist{\omega},\varphi,\prolist{\xi})}{\psi}{b}$
    in $\Dbl{A}$, we observe that for any globular multicell $\beta:\scell{\varphi_1 \cdots \varphi_n}{\varphi}$, we can apply $\Gamma$ to the composite $\frac{(\text{id}_{\prolist{\omega}},\beta,\text{id}_{\prolist{\xi}})}{\delta}$ to obtain the multicell
    \[\begin{adjustbox}{}
\begin{tikzcd}
	Hx & {F_0x'} && {F_ny'} & Ky \\
	{G_0z} &&&& {G_1w}
	\arrow["{{{\prolist{\Omega}}\prolist{\omega}}}"{inner sep=.8ex}, "\shortmid"{marking}, from=1-1, to=1-2]
	\arrow["{{{{f_0}}(a)}}"', from=1-1, to=2-1]
	\arrow[""{name=0, anchor=center, inner sep=0}, "{{\prolist{\Phi}\prolist{\varphi}}}"{inner sep=.8ex}, "\shortmid"{marking}, from=1-2, to=1-4]
	\arrow["{{{\prolist{\Xi}}\prolist{\xi}}}"{inner sep=.8ex}, "\shortmid"{marking}, from=1-4, to=1-5]
	\arrow["{{{{f_1}}(b)}}", from=1-5, to=2-5]
	\arrow[""{name=1, anchor=center, inner sep=0}, "{{\Psi \psi}}"'{inner sep=.8ex}, "\shortmid"{marking}, from=2-1, to=2-5]
	\arrow["{\Gamma\left(\frac{\text{id}_{\prolist{\omega}},\beta,\text{id}_{\prolist{\xi}}}{\delta}\right)}"{description}, draw=none, from=0, to=1]
\end{tikzcd}
\end{adjustbox}\]
    By functoriality of $\Gamma$ this procedure defines a cocone, which by property (L2) factors through a unique multicell
    \[\begin{adjustbox}{}
\begin{tikzcd}
	Hx & {F_0x'} && {F_ny'} & Ky \\
	{G_0z} &&&& {G_1w}
	\arrow["{{{\prolist{\Omega}}\prolist{\omega}}}"{inner sep=.8ex}, "\shortmid"{marking}, from=1-1, to=1-2]
	\arrow["{{{{f_0}}(a)}}"', from=1-1, to=2-1]
	\arrow[""{name=0, anchor=center, inner sep=0}, "{{\odot(\prolist{\Phi})(\varphi)}}"{inner sep=.8ex}, "\shortmid"{marking}, from=1-2, to=1-4]
	\arrow["{{{\prolist{\Xi}}\prolist{\xi}}}"{inner sep=.8ex}, "\shortmid"{marking}, from=1-4, to=1-5]
	\arrow["{{{{f_1}}(b)}}", from=1-5, to=2-5]
	\arrow[""{name=1, anchor=center, inner sep=0}, "{{\Psi \psi}}"'{inner sep=.8ex}, "\shortmid"{marking}, from=2-1, to=2-5]
	\arrow["{\exists!\Gamma^\flat(\delta)}"{description}, draw=none, from=0, to=1]
\end{tikzcd}
\end{adjustbox}\]
    Uniqueness along with Lemma~\ref{lem:Uniqueglobular} ensures that $\Gamma^\flat$ is a well-defined multicell in $\Dbl{X}^\Dbl{A}$.

    It remains to show that $\Gamma$ factors through $\Gamma^\flat$ and $\odot_{\prolist{\Phi}}$. 
    For a multicell in $\Dbl{A}$ below left, we first use the equivalent characterization of exponentiability in 
    Theorem~\ref{thm:ExpChar} in terms of VDCs which have decomposable multicells to factor it as below right:
    \[\begin{adjustbox}{}
\begin{tikzcd}
	x & {x'} & {y'} & y && x & {x'} & {y'} & y \\
	&&&& {=} & u & {u'} & {v'} & v \\
	z &&& w && z &&& w
	\arrow["{{\prolist{\omega}}}"{inner sep=.8ex}, "\shortmid"{marking}, from=1-1, to=1-2]
	\arrow["a"', from=1-1, to=3-1]
	\arrow[""{name=0, anchor=center, inner sep=0}, "{{\prolist{\varphi}}}"{inner sep=.8ex}, "\shortmid"{marking}, from=1-2, to=1-3]
	\arrow["{{\prolist{\xi}}}"{inner sep=.8ex}, "\shortmid"{marking}, from=1-3, to=1-4]
	\arrow["b", from=1-4, to=3-4]
	\arrow[""{name=1, anchor=center, inner sep=0}, "{{\prolist{\omega}}}"{inner sep=.8ex}, "\shortmid"{marking}, from=1-6, to=1-7]
	\arrow["{{a'}}"', from=1-6, to=2-6]
	\arrow[""{name=2, anchor=center, inner sep=0}, "{{\prolist{\varphi}}}"{inner sep=.8ex}, "\shortmid"{marking}, from=1-7, to=1-8]
	\arrow["c"{description}, from=1-7, to=2-7]
	\arrow[""{name=3, anchor=center, inner sep=0}, "{{\prolist{\xi}}}"{inner sep=.8ex}, "\shortmid"{marking}, from=1-8, to=1-9]
	\arrow["d"{description}, from=1-8, to=2-8]
	\arrow["{{b'}}", from=1-9, to=2-9]
	\arrow[""{name=4, anchor=center, inner sep=0}, "{{\prolist{\omega}'}}"'{inner sep=.8ex}, "\shortmid"{marking}, from=2-6, to=2-7]
	\arrow["{{a''}}"', from=2-6, to=3-6]
	\arrow[""{name=5, anchor=center, inner sep=0}, "{{\varphi'}}"{inner sep=.8ex}, "\shortmid"{marking}, from=2-7, to=2-8]
	\arrow[""{name=6, anchor=center, inner sep=0}, "{{\prolist{\xi}'}}"'{inner sep=.8ex}, "\shortmid"{marking}, from=2-8, to=2-9]
	\arrow["{{b''}}", from=2-9, to=3-9]
	\arrow[""{name=7, anchor=center, inner sep=0}, "{{\psi}}"'{inner sep=.8ex}, "\shortmid"{marking}, from=3-1, to=3-4]
	\arrow[""{name=8, anchor=center, inner sep=0}, "{{\psi}}"'{inner sep=.8ex}, "\shortmid"{marking}, from=3-6, to=3-9]
	\arrow["\chi"{description}, draw=none, from=0, to=7]
	\arrow["{\prolist{\delta}_1}"{description}, draw=none, from=1, to=4]
	\arrow["\beta"{description}, draw=none, from=2, to=5]
	\arrow["{\prolist{\delta}_2}"{description}, draw=none, from=3, to=6]
	\arrow["{\chi^\flat}"{description}, draw=none, from=5, to=8]
\end{tikzcd}
\end{adjustbox}\]
    Taking a positive globular decomposition of the bottom multicell if necessary, 
	we can choose this factorization such that the bottom multicell is a globular multicell:
\[\begin{tikzcd}
	x & {x'} & {y'} & y \\
	z & {x''} & {y''} & w \\
	z &&& w
	\arrow[""{name=0, anchor=center, inner sep=0}, "{{{\prolist{\omega}}}}"{inner sep=.8ex}, "\shortmid"{marking}, from=1-1, to=1-2]
	\arrow["a"', from=1-1, to=2-1]
	\arrow[""{name=1, anchor=center, inner sep=0}, "{{{\prolist{\varphi}}}}"{inner sep=.8ex}, "\shortmid"{marking}, from=1-2, to=1-3]
	\arrow["c"{description}, from=1-2, to=2-2]
	\arrow[""{name=2, anchor=center, inner sep=0}, "{{{\prolist{\xi}}}}"{inner sep=.8ex}, "\shortmid"{marking}, from=1-3, to=1-4]
	\arrow["d"{description}, from=1-3, to=2-3]
	\arrow["b", from=1-4, to=2-4]
	\arrow[""{name=3, anchor=center, inner sep=0}, "{{{\prolist{\omega}''}}}"'{inner sep=.8ex}, "\shortmid"{marking}, from=2-1, to=2-2]
	\arrow[equals, from=2-1, to=3-1]
	\arrow[""{name=4, anchor=center, inner sep=0}, "{{{\varphi''}}}"{inner sep=.8ex}, "\shortmid"{marking}, from=2-2, to=2-3]
	\arrow[""{name=5, anchor=center, inner sep=0}, "{{{\prolist{\xi}''}}}"'{inner sep=.8ex}, "\shortmid"{marking}, from=2-3, to=2-4]
	\arrow[equals, from=2-4, to=3-4]
	\arrow[""{name=6, anchor=center, inner sep=0}, "{{{\psi}}}"'{inner sep=.8ex}, "\shortmid"{marking}, from=3-1, to=3-4]
	\arrow["{{\prolist{\delta}_1'}}"{description}, draw=none, from=0, to=3]
	\arrow["\gamma"{description}, draw=none, from=1, to=4]
	\arrow["{{\prolist{\delta}_2}}"{description}, draw=none, from=2, to=5]
	\arrow["{{\chi'}}"{description}, draw=none, from=4, to=6]
\end{tikzcd}\]
    Finally, taking a positive globular decomposition of $\gamma$, 
	we obtain the factorization
    \[\begin{adjustbox}{}
\begin{tikzcd}
	x & {x'} & {y'} & y \\
	z & {x''} & {y''} & w \\
	z & {x''} & {y''} & w \\
	z &&& w
	\arrow[""{name=0, anchor=center, inner sep=0}, "{{\prolist{\omega}}}"{inner sep=.8ex}, "\shortmid"{marking}, from=1-1, to=1-2]
	\arrow["a"', from=1-1, to=2-1]
	\arrow[""{name=1, anchor=center, inner sep=0}, "{{\prolist{\varphi}}}"{inner sep=.8ex}, "\shortmid"{marking}, from=1-2, to=1-3]
	\arrow["c"{description}, from=1-2, to=2-2]
	\arrow[""{name=2, anchor=center, inner sep=0}, "{{\prolist{\xi}}}"{inner sep=.8ex}, "\shortmid"{marking}, from=1-3, to=1-4]
	\arrow["d"{description}, from=1-3, to=2-3]
	\arrow["{{b}}", from=1-4, to=2-4]
	\arrow[""{name=3, anchor=center, inner sep=0}, "{{\prolist{\omega}''}}"{inner sep=.8ex}, "\shortmid"{marking}, from=2-1, to=2-2]
	\arrow[equals, from=2-1, to=3-1]
	\arrow[""{name=4, anchor=center, inner sep=0}, "{{\prolist{\varphi}''}}"{inner sep=.8ex}, "\shortmid"{marking}, from=2-2, to=2-3]
	\arrow[equals, from=2-2, to=3-2]
	\arrow[""{name=5, anchor=center, inner sep=0}, "{{\prolist{\xi}''}}"{inner sep=.8ex}, "\shortmid"{marking}, from=2-3, to=2-4]
	\arrow[equals, from=2-3, to=3-3]
	\arrow[equals, from=2-4, to=3-4]
	\arrow[""{name=6, anchor=center, inner sep=0}, "{{\prolist{\omega}''}}"'{inner sep=.8ex}, "\shortmid"{marking}, from=3-1, to=3-2]
	\arrow[equals, from=3-1, to=4-1]
	\arrow[""{name=7, anchor=center, inner sep=0}, "{{\varphi''}}"{inner sep=.8ex}, "\shortmid"{marking}, from=3-2, to=3-3]
	\arrow[""{name=8, anchor=center, inner sep=0}, "{{\prolist{\xi}''}}"'{inner sep=.8ex}, "\shortmid"{marking}, from=3-3, to=3-4]
	\arrow[equals, from=3-4, to=4-4]
	\arrow[""{name=9, anchor=center, inner sep=0}, "{{\psi}}"'{inner sep=.8ex}, "\shortmid"{marking}, from=4-1, to=4-4]
	\arrow["{\prolist{\delta}_1'}"{description}, draw=none, from=0, to=3]
	\arrow["{\prolist{\gamma}^\sharp}"{description}, draw=none, from=1, to=4]
	\arrow["{\prolist{\delta}_2'}"{description}, draw=none, from=2, to=5]
	\arrow["{\text{id}_{\prolist{\omega}''}}"{description}, draw=none, from=3, to=6]
	\arrow["{\gamma^\flat}"{description}, draw=none, from=4, to=7]
	\arrow["{\text{id}_{\prolist{\xi}''}}"{description}, draw=none, from=5, to=8]
	\arrow["{\chi'}"{description}, draw=none, from=7, to=9]
\end{tikzcd}
\end{adjustbox}\]
    Then using the construction of $\Gamma^\flat$, we have the pasting equality 
\[\begin{tikzcd}[column sep=small, row sep=large]
	&&&&& Hx && {F_0x'} && {F_ny'} && Ky \\
	Hx & {F_0x'} & {F_ny'} & Ky && Hz && {F_0x''} && {F_ny''} && Kw \\
	{G_0z} &&& {G_1w} && Hz && {F_0x''} && {F_ny''} && Kw \\
	&&&&& {G_0z} &&&&&& {G_1w}
	\arrow[""{name=0, anchor=center, inner sep=0}, "{{{{{{\prolist{\Omega}{\prolist{\omega}}}}}}}}"{inner sep=.8ex}, "\shortmid"{marking}, from=1-6, to=1-8]
	\arrow["Ha"', from=1-6, to=2-6]
	\arrow[""{name=1, anchor=center, inner sep=0}, "{{{{{{\prolist{\Phi}\prolist{\varphi}}}}}}}"{inner sep=.8ex}, "\shortmid"{marking}, from=1-8, to=1-10]
	\arrow["{{{{{{F_0c}}}}}}"{description}, from=1-8, to=2-8]
	\arrow[""{name=2, anchor=center, inner sep=0}, "{{{{{{\prolist{\Xi}{\prolist{\xi}}}}}}}}"{inner sep=.8ex}, "\shortmid"{marking}, from=1-10, to=1-12]
	\arrow["{{{{{{F_nd}}}}}}"{description}, from=1-10, to=2-10]
	\arrow["Kb", from=1-12, to=2-12]
	\arrow["{{{{{{\prolist{\Omega}{\prolist{\omega}}}}}}}}"{inner sep=.8ex}, "\shortmid"{marking}, from=2-1, to=2-2]
	\arrow["{{{{{{f_0(a)}}}}}}"', from=2-1, to=3-1]
	\arrow[""{name=3, anchor=center, inner sep=0}, "{{{{{{\prolist{\Phi}\prolist{\varphi}}}}}}}"{inner sep=.8ex}, "\shortmid"{marking}, from=2-2, to=2-3]
	\arrow["{{{{{{\prolist{\Xi}{\prolist{\xi}}}}}}}}"{inner sep=.8ex}, "\shortmid"{marking}, from=2-3, to=2-4]
	\arrow[""{name=4, anchor=center, inner sep=0}, "{{{{{{f_1(b)}}}}}}", from=2-4, to=3-4]
	\arrow[""{name=5, anchor=center, inner sep=0}, "{{{{{{\prolist{\Omega}\prolist{\omega}''}}}}}}"'{inner sep=.8ex}, "\shortmid"{marking}, from=2-6, to=2-8]
	\arrow[""{name=6, anchor=center, inner sep=0}, equals, from=2-6, to=3-6]
	\arrow[""{name=7, anchor=center, inner sep=0}, "{{{{{{\prolist{\Phi}\prolist{\varphi}''}}}}}}"{inner sep=.8ex}, "\shortmid"{marking}, from=2-8, to=2-10]
	\arrow[equals, from=2-8, to=3-8]
	\arrow[""{name=8, anchor=center, inner sep=0}, "{{{{{{\prolist{\Xi}\prolist{\xi}''}}}}}}"'{inner sep=.8ex}, "\shortmid"{marking}, from=2-10, to=2-12]
	\arrow[equals, from=2-10, to=3-10]
	\arrow[equals, from=2-12, to=3-12]
	\arrow[""{name=9, anchor=center, inner sep=0}, "{{{{{{\Psi \psi}}}}}}"'{inner sep=.8ex}, "\shortmid"{marking}, from=3-1, to=3-4]
	\arrow[""{name=10, anchor=center, inner sep=0}, "{{{{{{\prolist{\Omega}\prolist{\omega}''}}}}}}"'{inner sep=.8ex}, "\shortmid"{marking}, from=3-6, to=3-8]
	\arrow["{{{{{{f_0(\text{id}_z)}}}}}}"', from=3-6, to=4-6]
	\arrow[""{name=11, anchor=center, inner sep=0}, "{{{{{{\odot(\prolist{\Phi})(\varphi'')}}}}}}"{inner sep=.8ex}, "\shortmid"{marking}, from=3-8, to=3-10]
	\arrow[""{name=12, anchor=center, inner sep=0}, "{{{{{{\prolist{\Xi}\prolist{\xi}''}}}}}}"'{inner sep=.8ex}, "\shortmid"{marking}, from=3-10, to=3-12]
	\arrow["{{{{{{f_1(\text{id}_w)}}}}}}", from=3-12, to=4-12]
	\arrow[""{name=13, anchor=center, inner sep=0}, "{{{{{{\Psi \psi}}}}}}"'{inner sep=.8ex}, "\shortmid"{marking}, from=4-6, to=4-12]
	\arrow["{{{\prolist{\Omega}\prolist{\delta}_1'}}}"{description}, draw=none, from=0, to=5]
	\arrow["{{{\prolist{\Phi}\prolist{\gamma}^\#}}}"{description, pos=0.3}, draw=none, from=1, to=7]
	\arrow["{{{\prolist{\Xi}\prolist{\delta}_2'}}}"{description}, draw=none, from=2, to=8]
	\arrow["{{{\Gamma(\chi)}}}"{description}, draw=none, from=3, to=9]
	\arrow[between={0.45}{0.55}, equals, from=4, to=6]
	\arrow["{{{\text{id}_{\prolist{\Omega}\prolist{\omega}''}}}}"{description}, draw=none, from=5, to=10]
	\arrow["{{{\lambda_{\gamma^\flat}}}}"{description}, draw=none, from=7, to=11]
	\arrow["{{{\text{id}_{\prolist{\Xi}\prolist{\xi}''}}}}"{description}, draw=none, from=8, to=12]
	\arrow["{{{\Gamma^\flat(\chi')}}}"{description}, draw=none, from=11, to=13]
\end{tikzcd}\]
    Since $\chi$ was arbitrary this implies that we have the desired factorization of multicells in $\Dbl{X}^\Dbl{A}$:
    \[\begin{adjustbox}{}
\begin{tikzcd}
	H & {F_0} & {F_n} & K && H & {F_0} & {F_n} & K \\
	&&&& {=} & H & {F_0} & {F_n} & K \\
	{G_0} &&& {G_1} && {G_0} &&& {G_1}
	\arrow["{{{\prolist{\Omega}}}}"{inner sep=.8ex}, "\shortmid"{marking}, from=1-1, to=1-2]
	\arrow["{{{{f_0}}}}"', from=1-1, to=3-1]
	\arrow[""{name=0, anchor=center, inner sep=0}, "{{{{\prolist{\Phi}}}}}"{inner sep=.8ex}, "\shortmid"{marking}, from=1-2, to=1-3]
	\arrow["{{{\prolist{\Xi}}}}"{inner sep=.8ex}, "\shortmid"{marking}, from=1-3, to=1-4]
	\arrow["{{{{f_1}}}}", from=1-4, to=3-4]
	\arrow[""{name=1, anchor=center, inner sep=0}, "{{{\prolist{\Omega}}}}"{inner sep=.8ex}, "\shortmid"{marking}, from=1-6, to=1-7]
	\arrow[equals, from=1-6, to=2-6]
	\arrow[""{name=2, anchor=center, inner sep=0}, "{{{{\prolist{\Phi}}}}}"{inner sep=.8ex}, "\shortmid"{marking}, from=1-7, to=1-8]
	\arrow[equals, from=1-7, to=2-7]
	\arrow[""{name=3, anchor=center, inner sep=0}, "{{{\prolist{\Xi}}}}"{inner sep=.8ex}, "\shortmid"{marking}, from=1-8, to=1-9]
	\arrow[equals, from=1-8, to=2-8]
	\arrow[equals, from=1-9, to=2-9]
	\arrow[""{name=4, anchor=center, inner sep=0}, "{{{\prolist{\Omega}}}}"'{inner sep=.8ex}, "\shortmid"{marking}, from=2-6, to=2-7]
	\arrow["{{{{f_0}}}}"', from=2-6, to=3-6]
	\arrow[""{name=5, anchor=center, inner sep=0}, "{{\odot({{\prolist{\Phi}}})}}"'{inner sep=.8ex}, "\shortmid"{marking}, from=2-7, to=2-8]
	\arrow[""{name=6, anchor=center, inner sep=0}, "{{{\prolist{\Xi}}}}"'{inner sep=.8ex}, "\shortmid"{marking}, from=2-8, to=2-9]
	\arrow["{{{{f_1}}}}", from=2-9, to=3-9]
	\arrow[""{name=7, anchor=center, inner sep=0}, "\Psi"'{inner sep=.8ex}, "\shortmid"{marking}, from=3-1, to=3-4]
	\arrow[""{name=8, anchor=center, inner sep=0}, "\Psi"'{inner sep=.8ex}, "\shortmid"{marking}, from=3-6, to=3-9]
	\arrow["\Gamma"{description}, draw=none, from=0, to=7]
	\arrow["{\text{id}_{\prolist{\Omega}}}"{description}, draw=none, from=1, to=4]
	\arrow["{\odot_{\prolist{\Phi}}}"{description}, draw=none, from=2, to=5]
	\arrow["{\text{id}_{\prolist{\Xi}}}"{description}, draw=none, from=3, to=6]
	\arrow["{\Gamma^\flat}"{description}, draw=none, from=5, to=8]
\end{tikzcd}
\end{adjustbox}\]
    completing the proof.
\end{proof}

\subsection*{Examples and special cases}
If $\Dbl{X}$ has a discrete underlying category, and we are just interested in weak composites, 
then we can drop the hypothesis that positive globular decompositions exist.

\begin{prop}[Weak Composites of Exponentials into Tightly Discrete VDCs]\label{prop:WeakCompForTightDiscVDC}
	If $\Dbl{A}$ is a small exponentiable VDC and $\Dbl{X}$ is a large locally cocomplete VDC with a discrete underlying tight category and weak $n$-ary composites for some $n \geq 2$, 
  then $\Dbl{X}^\Dbl{A}$ also has weak $n$-ary composites. 
  If in addition $\Dbl{X}$ has terminal underlying tight category and (weak) $0$-ary composites (i.e. (weak) units), then $\Dbl{X}^\Dbl{A}$ also has (weak) units.
\end{prop}
\begin{proof}
	The first claim follows exactly as in the proof of the weak case in Theorem~\ref{thm:weakComp}, where now $\Cat{I}_{\varphi,n}$ is replaced by the category $\Cat{I}_{\varphi,n}'$ whose objects are all $n$-ary multicells in $\Dbl{A}$ with loose target $\varphi$, and for which a morphism $\alpha\to \beta$ consists of a sequence of unary multicells $(\gamma_1 \cdots \gamma_n)$ such that $\alpha = \frac{\gamma_1\cdots\gamma_n}{\beta}$. The same argument carries through since under any VDF all multicells in $\Dbl{A}$ become globular multicells in $\Dbl{X}$.

	For the nullary composites in the case that $\Dbl{X}$ is a multicategory with (weakly) representable unit, we write $\bullet$ for the unique object in the underlying tight category and $I_\bullet:\bullet\proto \bullet$ for its loose unit. 
  Then $\Dbl{X}^\Dbl{A}$ also has terminal underlying tight category, and for the unique functor $!:\Dbl{A}_0\to \Dbl{X}_0$, we can define the (weak) unit 
  $I_!:!\proto !$ by sending every loose arrow $\varphi$ to the (weak) unit $I_\bullet$, and sending every unary multicell to the identity $\text{id}_{I_\bullet}$. 
  Then we have a natural nullary multicell $\eta:!\Rightarrow I_!$ given by sending any nullary multicell in $\Dbl{A}$ to a fixed opcartesian nullary multicell 
  witnessing $I_\bullet$ as the unit of $\bullet$ in $\Dbl{X}$. Then the universal property of the (weak) opcartesian multicell for the unit in $\Dbl{X}$ ensures that the nullary multicell $u$ 
  is (weakly) opcartesian as well, completing the proof.
\end{proof}

For instance, if $\Cat{C}$ is a cocomplete monoidal category, then the VDC $\L(B\Cat{C})$ is locally cocomplete, with discrete tight category, and the
monoidal product provides composites. Thus exponentials of the form $\L(B\Cat{C})^\Dbl{A}$ always have weak composites. Indeed, 
it is enough here that $\Cat{C}$ has a \emph{colax} monoidal structure; in fact we do not know whether 
we can eliminate the ``weak'' hypothesis above.
This leads to interesting 
examples of generalized Day convoluations for colax monoidal categories, as defined in~\cite{grandisOverviewColaxVirtual2025}, 
which are dual to Leinster's notion of unbiased lax monoidal categories in~\cite[Definition 3.1.1]{Leinster2004}:

\begin{exmp}[Colax Monoidal Convolution Structure]\label{eg:colaxConvStruct}
	If $(\Cat{C},\otimes,I)$ is a cocomplete colax monoidal category with $\otimes$ preserving colimits in either variable, then for any small exponentiable VDC $\Dbl{A}$, the functor category $\Cat{C}^{\Dbl{A}_1}$ has a colax monoidal convolution structure induced by the weak representability of $\L(B\Cat{C})^\Dbl{A}$. 
	The colax monoidal structure is given on a sequence of functors $F_1,...,F_n:\Dbl{A}_1\to \Cat{C}$ by the coend
	\begin{equation}\label{eq:convWeakRep}
		\int^{(\varphi_1 \cdots \varphi_n) \in \Cat{fc}_n(\Dbl{A}_1)}\MCell (\Dbl{A})(\prolist{\varphi},-)\star(F_1(\varphi_1)\otimes \cdots \otimes F_n(\varphi_n))
	\end{equation}
	where $\star$ denotes the tensoring of $\Cat{C}$ over $\Cat{Set}$, which exists since $\Cat{C}$ is cocomplete.
\end{exmp}

As an illustration, if $(\Cat{C},\otimes,I)=(\Cat{Set},\times,\{*\})=:\Cat{Set}_\times$, 
which we can view as the full sub VDC of $\Span$ with only object the one element set $\{*\}$, 
then for any pro-monoidal multicategory $\Cat{M}$ the exponential $(B \Cat{Set}_\times)^{B \Cat{M}}$ is always representable, 
which is equivalent to giving a bi-cocontinuous monoidal structure on the functor category $\Cat{Fun}(\Cat{M}_1,\Cat{Set})$.
(Here we use Proposition \ref{prop:WeakCompForTightDiscVDC} only for the units, and Theorem \ref{thm:repOfExp} for the strong representability.)
In particular, these observations along with Proposition~\ref{prop:multiCatPro} and characterization (6) in Theorem~\ref{thm:ExpChar} 
imply that a multicategory $\Cat{M}$ is pro-monoidal if and only if the exponential $\Cat{Set}_\times^{\Cat{M}}$ exists in $\Cat{MultCat}$, 
in which case it is necessarily representable, and hence uniquely determined by a bi-cocontinuous monoidal structure on $\Cat{Fun}(\Cat{M}_1,\Cat{Set})$.

\subsection{Representability of VDCs of VDFs}

Using Theorem~\ref{thm:weakComp} along with the following two lemmas, we will obtain sufficient conditions for 
$\Dbl{V}\mathsf{df}(\Dbl{A},\Dbl{X})$ to be (weakly) representable using Theorem A.1 in~\cite{Cruttwell2010} 
and the observation in Remark 6.4 of~\cite{arkor2025exponentiablevirtualdoublecategories}. 
In personal communication, Arkor has indicated that he's simultaneously developed some similar results, to appear in 
a future version of~\cite{arkor2025exponentiablevirtualdoublecategories}.

\begin{lem}[Restrictions in Exponential]\label{lem:restrInExp}
    If $\Dbl{A}$ is a small exponentiable VDC and $\Dbl{X}$ is a VDC with restrictions, then the VDC $\Dbl{X}^\Dbl{A}$ also has restrictions which are created by evaluation at loose arrows.
\end{lem}
\begin{proof}
    If $\Phi:F\proto G$ is a map of spans of categories,
    \begin{equation*}
        (F,\Phi,G):(\Dbl{A}_0\xleftarrow{}\Dbl{A}_1\xrightarrow{}\Dbl{A}_0)\to (\Dbl{X}_0\xleftarrow{}\Dbl{X}_1\xrightarrow{}\Dbl{X}_0)
    \end{equation*}
    corresponding to a loose arrow in $\Dbl{X}^\Dbl{A}$, and $f:H\Rightarrow F$, $g:K\Rightarrow G$ are natural transformations, then the restriction $\Phi(f,g):H\proto K$ is determined at a loose arrow $\varphi:x\proto y$ in $\Dbl{A}$ by a choice of cartesian multicell
    \[\begin{adjustbox}{}
\begin{tikzcd}
	Hx && Ky \\
	Fx && Gy
	\arrow[""{name=0, anchor=center, inner sep=0}, "{{\Phi(f,g)(\varphi)}}"{inner sep=.8ex}, "\shortmid"{marking}, from=1-1, to=1-3]
	\arrow["{{f(a)}}"', from=1-1, to=2-1]
	\arrow["{{g(b)}}", from=1-3, to=2-3]
	\arrow[""{name=1, anchor=center, inner sep=0}, "{{\Phi(\varphi)}}"'{inner sep=.8ex}, "\shortmid"{marking}, from=2-1, to=2-3]
	\arrow["{\mathsf{cart}}"{description}, draw=none, from=0, to=1]
\end{tikzcd}
\end{adjustbox}\]
    so that $\Phi(f,g)(\varphi)=\Phi(\varphi)(f(a),g(b))$. For a unary multicell $\gamma:\cell{a}{\varphi}{\psi}{b}$ in $\Dbl{A}$, $\Phi(f,g)(\gamma)$ is given using the unique factorization through the cartesian multicell below:
    \[\begin{adjustbox}{}
\begin{tikzcd}
	Hx & Ky && Hx && Ky \\
	Fx & Gy & {=} & Hz && Kw \\
	Fz & Gw && Fz && Gw
	\arrow[""{name=0, anchor=center, inner sep=0}, "{{\Phi(f,g)(\varphi)}}"{inner sep=.8ex}, "\shortmid"{marking}, from=1-1, to=1-2]
	\arrow["{{f(a)}}"', from=1-1, to=2-1]
	\arrow["{{g(b)}}", from=1-2, to=2-2]
	\arrow[""{name=1, anchor=center, inner sep=0}, "{{\Phi(f,g)(\varphi)}}"{inner sep=.8ex}, "\shortmid"{marking}, from=1-4, to=1-6]
	\arrow["Ha"', from=1-4, to=2-4]
	\arrow["Kb", from=1-6, to=2-6]
	\arrow[""{name=2, anchor=center, inner sep=0}, "{{\Phi(\varphi)}}"'{inner sep=.8ex}, "\shortmid"{marking}, from=2-1, to=2-2]
	\arrow["Fa"', from=2-1, to=3-1]
	\arrow["Gb", from=2-2, to=3-2]
	\arrow[""{name=3, anchor=center, inner sep=0}, "{{\Phi(f,g)(\psi)}}"'{inner sep=.8ex}, "\shortmid"{marking}, from=2-4, to=2-6]
	\arrow["{{f(c)}}"', from=2-4, to=3-4]
	\arrow["{{g(d)}}", from=2-6, to=3-6]
	\arrow[""{name=4, anchor=center, inner sep=0}, "{{\Phi(\psi)}}"'{inner sep=.8ex}, "\shortmid"{marking}, from=3-1, to=3-2]
	\arrow[""{name=5, anchor=center, inner sep=0}, "{{\Phi(\psi)}}"'{inner sep=.8ex}, "\shortmid"{marking}, from=3-4, to=3-6]
	\arrow["{\mathsf{cart}}"{description}, draw=none, from=0, to=2]
	\arrow["{\exists!\Phi(f,g)(\gamma)}"{description}, draw=none, from=1, to=3]
	\arrow["{\Phi(\gamma)}"{description}, draw=none, from=2, to=4]
	\arrow["{\mathsf{cart}}"{description}, draw=none, from=3, to=5]
\end{tikzcd}
\end{adjustbox}\]
    Uniqueness ensures that this is functorial in $\gamma$, and hence $\Phi(f,g)$ defines a loose arrow in $\Dbl{X}^\Dbl{A}$, with the associated cartesian multicell induced by the pointwise cartesian multicells. 
\end{proof}

\begin{lem}[Local Colimits in Exponential]\label{lem:ExpLocalColim}
    If $\Dbl{A}$ is a small exponentiable VDC with positive globular decompositions, and $\Dbl{X}$ is a large VDC with weak non-nullary composites which has local $\Cat{I}$-colimits, then $\Dbl{X}^\Dbl{A}$ also has local $\Cat{I}$-colimits, and they are created by evaluation at loose arrows.
\end{lem}
\begin{proof}
    Fix functors $F,G:\Dbl{A}_0\to \Dbl{X}_0$ and a diagram $D:\Cat{I}\to (\Dbl{X}^\Dbl{A})_1(F,G)$, where the objects in $(\Dbl{X}^\Dbl{A})_1(F,G)$ are maps of spans
    \[\begin{adjustbox}{}
\begin{tikzcd}
	{\Dbl{A}_0} && {\Dbl{A}_1} && {\Dbl{A}_0} \\
	\\
	{\Dbl{X}_0} && {\Dbl{X}_1} && {\Dbl{X}_0}
	\arrow["F"', from=1-1, to=3-1]
	\arrow["s"', from=1-3, to=1-1]
	\arrow["t", from=1-3, to=1-5]
	\arrow["\Phi"{description}, from=1-3, to=3-3]
	\arrow["G", from=1-5, to=3-5]
	\arrow["s", from=3-3, to=3-1]
	\arrow["t"', from=3-3, to=3-5]
\end{tikzcd}
\end{adjustbox}\]   
    and morphisms are natural transformations $\Phi\Rightarrow \Psi$ whose components are all globular unary multicells. Let $\colim_\Cat{I}\,D:\Dbl{A}_1\to \Dbl{X}_1$ denote the colimit in the functor category, which exists since for each loose arrow $\varphi:a\proto b$ in $\Dbl{A}$, the resulting functor $D(-,\varphi):\Cat{I}\to \Dbl{X}_1(Fa,Gb)$ lands in a category with $\Cat{I}$-colimits. Since $\colim_\Cat{I}\,D$ lies in $(\Dbl{X}^\Dbl{A})_1(F,G)$, it is also a colimit there. Further, property (wL2) is satisfied since the local $\Cat{I}$-colimits in $\Dbl{X}$ satisfy (wL2), and the colimit $\colim_\Cat{I}\,D$ is computed pointwise. Finally, since weak composites in $\Dbl{X}^\Dbl{A}$ can be constructed as a colimit of composites in $\Dbl{X}$, as in Theorem~\ref{thm:weakComp}, the local $\Cat{I}$-colimits in $\Dbl{X}^\Dbl{A}$ are preserved under weak composition by loose arrows on either side due to the fact that this is the case in $\Dbl{X}$ and colimits commute.
\end{proof}

Recall that Theorem A.1 in~\cite{Cruttwell2010} says that if $\Dbl{X}$ is a representable VDC with restrictions and local coequalizers, 
then $\Mod(\Dbl{X})$ is also representable. Note that here we do not need for $\Dbl{X}$ to have restrictions, as we just need property (wL2) 
for local coequalizers. However, as noted in~\cite{Cruttwell2010}, in practice most examples of such $\Dbl{X}$ satisfy (wL2) due to having restrictions. 
Additionally, Remark 6.4 of~\cite{arkor2025exponentiablevirtualdoublecategories} says that we can weaken the assumptions further by only requiring 
$\Dbl{X}$ to have non-nullary composites and local reflexive coequalizers, since these are the only things used in the proof of Theorem A.1 in~\cite{Cruttwell2010}. 
Putting Theorem~\ref{thm:weakComp} together with Lemma~\ref{lem:restrInExp} and Lemma~\ref{lem:ExpLocalColim}, we can now obtain our main representability result:

\begin{thm}[Representability of $\Dbl{V}\mathsf{df}(\Dbl{A},\Dbl{X})$]\label{thm:repOfExp}
    If $\Dbl{A}$ is a small exponentiable VDC with positive globular decompositions, and 
    $\Dbl{X}$ is a VDC with (weak) non-nullary composites and local colimits, then 
    $\Dbl{V}\mathsf{df}(\Dbl{A},\Dbl{X})$ is a (weakly) representable unital VDC which is locally cocomplete. 
    Further, if $\Dbl{X}$ has restrictions, so does $\Dbl{V}\mathsf{df}(\Dbl{A},\Dbl{X})$.
\end{thm}

Combining these results with Lemma~\ref{lem:PosGDPforOpOfStrictification} and Lemma~\ref{lem:Yoneda} we obtain the following embedding result.

\begin{cor}\label{cor:yonedaEmbeddProps}
	Let $\Dbl{D}$ be an arbitrary VDC. 
	Then the Yoneda embedding $\rho_\Dbl{D}:\Dbl{D}\hookrightarrow \Span^{\Cat{fc}(\Dbl{D})^{op_t}}$ embeds $\Dbl{D}$ into a large VDC which is locally cocomplete, has restrictions, and has non-nullary composites. 
	If $\Dbl{D}$ admits loose units, then a choice of such units corresponds to an extension of $\rho_\Dbl{D}$ to a unital embedding $\rho_\Dbl{D}':\Dbl{D}\hookrightarrow \Dbl{V}\mathsf{df}(\Cat{fc}(\Dbl{D})^{op_t},\Span)$ into a large locally cocomplete equipment.
\end{cor}

Using Proposition~\ref{prop:InternalizeUModAdj} it follows that for any unital VDC $\Dbl{D}$, a choice of units provides a unital embedding
\begin{equation}
	\rho_{\Dbl{D}}':\Dbl{D}\hookrightarrow \Dbl{V}\mathsf{df}(\Cat{fc}(\Dbl{D})^{op_t},\Span)\cong \Dbl{V}\Cat{df}_n(\Cat{fc}(\Dbl{D})^{op_t},\Prof)
\end{equation}
We'll explore the universal properties satisfied by these embeddings in future work.

%% file: Sections/A_ExponentiabilityOfMorphisms.tex
In this section we investigate exponentiable morphisms between VDCs. The arguments in this section are heavily inspired by the section on powerful morphisms in Street and Verity's paper on factorization and torsors~\cite{StreetVerity2010}. In particular, we use the fact that a morphism $p:E\to B$ in a finitely complete category $\Cat{E}$ is powerful, or exponentiable, if and only if the pullback functor $-\times_BE:\Cat{E}_{/B}\to \Cat{E}$ admits a right adjoint. Since we will be considering the case where $\Cat{E}=\Vdc$, Theorem~\ref{thm:univSketchProp} along with Proposition~\ref{prop:sliced-sketch} imply that such an adjoint exists if and only if the functor $-\times_BE$ preserves the colimiting cones described in Corollary~\ref{cor:exponentiable-morphisms-via-slice}.

Before generalizing our exponentiability results, we'll begin by explicating what a right adjoint to $-\times_BE$ would do on objects if it existed. Note that if $P:\Dbl{D}\to \Dbl{E}$ is an exponentiable VDF, so that we have a right adjoint $\Pi_P(-\times \Dbl{D}):\Vdc\to \Vdc_{/\Dbl{E}}$, then for a VDC $\Dbl{C}$, the structure of $s_\Dbl{C}:\Pi_P(\Dbl{C}\times \Dbl{D})\to \Dbl{E}$ can be determined by the adjunction
$$
\Vdc_{/\Dbl{E}}(Q:\Dbl{A}\to \Dbl{E},s_\Dbl{C})\cong \Vdc(\Dbl{A}\times_{Q,\Dbl{E},P}\Dbl{D},\Dbl{C})
$$
Thus, in order to use this adjunction to understand the structure of $\Pi_P(\Dbl{C}\times \Dbl{D})$ we must explicate what the VDC $\Dbl{A}\times_{Q,\Dbl{E},P}\Dbl{D}$ looks like for the various free walking VDCs.

Observe that if $x:\Ob\to \Dbl{E}$ is an object in $\Dbl{E}$, then since products with $\Ob$ give underlying tight categories, the fiber $P^{-1}(x)=\Ob\times_{x,\Dbl{E},P}\Dbl{D}$ is precisely the VDC $\T(P_0^{-1}(x))$. Similarly, if $a:x\to x'$ is a tight arrow in $\Dbl{E}$, corresponding to a VDF $a:\Tight \to \Dbl{E}$, then $P^{-1}(a)=\Tight \times_{a,\Dbl{E},P}\Dbl{D}$ can be identified with the collage of the fibers
\begin{equation*}
	\T(P_0^{-1}(x)+_{P_0^{-1}(a)}P_0^{-1}(x'))
\end{equation*}
That is, $P^{-1}(a)$ consists of disjoint copies of the categories $P_0^{-1}(x)$ and $P_0^{-1}(x')$, along with for each tight arrow $b:y\to y'$ in $\Dbl{D}$ with $P(b)=a$, a tight morphism $(t,b):(0,y)\to (1,y')$.

Next, if $\varphi:x\proto x'$ is a loose arrow in $\Dbl{E}$, corresponding to a VDF $\varphi:\Loose\to \Dbl{E}$, then $\Loose\times_{\varphi,\Dbl{E},P}\Dbl{D}=P^{-1}(\varphi)$ contains disjoint copies of the categories $P_0^{-1}(x)$ and $P_0^{-1}(x')$. Additionally, for each loose arrow $\psi:y\proto y'$ in $\Dbl{D}$ with $P(\psi)=\varphi$, $P^{-1}(\varphi)$ has a loose arrow $(\ell,\psi):(0,y)\proto (1,y')$, and for each unary multicell $\alpha$ in $\Dbl{D}$ lying over $\text{id}_{\varphi}$, $P^{-1}(\varphi)$ has a unary multicell $(\text{id}_{\ell},\alpha)$. In particular, the data in $P^{-1}(\varphi)$ is determined by the subspan of categories
\[\begin{adjustbox}{}
\begin{tikzcd}
	{P_0^{-1}(x)} && {P_1^{-1}(\varphi)} && {P_0^{-1}(x')} \\
	\\
	{\mathbb{D}_0} && {\mathbb{D}_1} && {\mathbb{D}_0}
	\arrow[hook, from=1-1, to=3-1]
	\arrow["s"', from=1-3, to=1-1]
	\arrow["t", from=1-3, to=1-5]
	\arrow[hook, from=1-3, to=3-3]
	\arrow[hook, from=1-5, to=3-5]
	\arrow["s", from=3-3, to=3-1]
	\arrow["t"', from=3-3, to=3-5]
\end{tikzcd}\end{adjustbox}
\]
Finally, for a multicell $\alpha:\cell{a}{\prolist{\varphi}}{\psi}{b}$
in $\Dbl{E}$, which can be given by a map $\alpha:\Sq{n}\to \Dbl{E}$, $P^{-1}(\alpha)$ contains $P^{-1}(a),P^{-1}(b),P^{-1}(\varphi_1),\ldots,P^{-1}(\varphi_n)$, and $P^{-1}(\psi)$ as sub-VDCs, along with a multicell for each multicell of $\Dbl{D}$ lying over $\alpha$.

Relativizing Proposition~\ref{prop:elementsOfExpVDC} we can obtain an analogous result for $\Pi_P(\Dbl{C}\times \Dbl{D})\to \Dbl{E}$. 
In the following proposition, if $\Cat{P}\to \Cat{Q}$ is a map of profunctors, and if $\alpha$ is a heteromorphism in $\Cat{Q}$, we write $\Cat{P}_{/\alpha}$ for the pullback in $\Prof_1$, which is a profunctor between fiber categories whose heteromorphisms are the heteromorphisms in $\Cat{P}$ lying over $\alpha$.

\begin{prop}[Elements of the exponential dependent product VDC]\label{prop:elementsOfDepProdVDC}
  If $P:\Dbl{D}\to \Dbl{E}$ is a VDF and $\Dbl{C}$ is a VDC such that the dependent product $s_\Dbl{C}:\Pi_P(\Dbl{C}\times \Dbl{D})\to \Dbl{E}$ exists, then the VDC $\Pi_P(\Dbl{C}\times \Dbl{D})$ consists of the following data:
\begin{enumerate}
    \item For $x\in \uOb (\Dbl{E})$, objects $F \in \uOb (\Pi_P(\Dbl{C}\times \Dbl{D}))_{/x}\cong \Vdc(\Ob\times_\Dbl{E} \Dbl{D},\Dbl{C})$ over $x$ correspond to functors $F:P_0^{-1}(x)\to \Dbl{C}_0$.
    \item For $a:x\to x'$ in $\Dbl{E}$, tight arrows $f \in \mathsf{Arr_T}(\Pi_P(\Dbl{C}\times \Dbl{D}))_{/a}\cong \Vdc(\Tight\times_\Dbl{E} \Dbl{D},\Dbl{C})$ over $a$ correspond to functors $\alpha:P_0^{-1}(x)+_{P_0^{-1}(a)}P_0^{-1}(x')\to \Dbl{C}_0$ extending those on objects.
    \item For $\varphi:x\proto x'$ in $\Dbl{E}$, Loose arrows $\Phi \in \uLoose(\Pi_P(\Dbl{C}\times \Dbl{D}))_{/\varphi}\cong \Vdc(\Loose \times_\Dbl{E} \Dbl{D},\Dbl{C})$ over $\varphi$ correspond to maps of spans of categories 
	\begin{equation*}
		(P_0^{-1}(x)\leftarrow P_1^{-1}(\varphi)\rightarrow P_0^{-1}(x'))\to (\Dbl{C}_0\leftarrow \Dbl{C}_1\rightarrow \Dbl{C}_0)
	\end{equation*}
    \item For a multicell $\alpha$ in $\Dbl{E}$, multicells $\Gamma \in \uSq{n}(\Pi_P(\Dbl{C}\times \Dbl{D}))_{/\alpha}\cong \Vdc(\Sq{n}\times_\Dbl{E} \Dbl{D},\Dbl{C}),$ over $\alpha$ are given by a map of profunctors $\MCelln{n}(\Dbl{D})_{/\alpha}\to \partial^*\MCelln{n}(\Dbl{C})$ over $\Dbl{C}_0^2,$ where 
    $\partial:\Cat{fc}_n(\Dbl{D}_1)_{/s(\alpha)}^{op}\times (\Dbl{D}_1)_{/t(\alpha)}\to \Cat{fc}_n(\Dbl{C}_1)^{op}\times \Dbl{C}_1$ is the functor which together with the natural transformations putting $\MCelln{n}(\Dbl{D})_{/\alpha}$ over $\Dbl{C}_0^2$ determines the  boundary data of $\Gamma.$
\end{enumerate}
\end{prop}

To begin characterizing exponentiable VDC morphisms, we begin by relativizing Definition~\ref{defn:DecompCells}.
Unlike in the absolute case, we'll need to consider factorizations of tight arrows as well as decompositions of multicells. 
For a tight arrow $a$ in a VDC $\Dbl{D}$, we will write $\Cat{Fact}_n(\Dbl{D};a)$ for the category whose objects
are factorizations of $a$ into $n$ composable tight arrows, and whose morphisms are maps of such factorizations.

\begin{defn}\label{defn:decompLifting}
	We say a VDF $P:\Dbl{D}\to \Dbl{E}$ lifts decompositions essentially uniquely if the following hold:
	\begin{itemize} 
		\item[(D1)] for every tight arrow $a$ in $\Dbl{D}$, every integer $n\geq 2$, and every factorization $\mathbf{b} \in \Cat{Fact}_n(\Dbl{E};P(a))$, the pullback $\mathsf{Fact}_n(\mathbb{D};a)\times_{\mathsf{Fact}_n(\mathbb{E};P(a))}\{\mathbf{b}\}$ is (nonempty and) connected.
		\item[(D2)] for every multicell $\alpha$ in $\Dbl{D}$, every integer $n\geq 2$, and every decomposition $\threefrac{\prolist{\gamma}_n}{\vdots}{\gamma_1} \in \Cat{Paste}_n(\mathbb{E};P(\alpha))$, the pullback $\mathsf{Paste}_n(\mathbb{D};\alpha)\times_{\mathsf{Paste}_n(\mathbb{E};P(\alpha))}\left\{\threefrac{\prolist{\gamma}_n}{\vdots}{\gamma_1} \right\}$ is (nonempty and) connected.
	\end{itemize}
\qed
\end{defn}

\begin{expl}[Decomposition Lifting]
    Condition (D1) in decomposition lifting for a VDF $P:\Dbl{D}\to \Dbl{E}$ says that the underlying 
	functor on tight categories is a Conduch\'{e} fibration, or equivalently that for each tight arrow $a$ in $\Dbl{D}$, 
	every factorization of its image $P(a)$ in $\Dbl{E}$ lifts to an essentially unique factorization of $a$ in $\Dbl{D}$, 
	where the maps appearing in the associativity relation are mapped to identities in $\Dbl{E}$. 
	Similarly, condition (D2) in decomposition lifting says that any multicell in $\Dbl{D}$, below upper left, lying over a multicell in $\Dbl{E}$ which admits a decomposition, below right, also admits a decomposition, as below bottom left:
    \[\begin{adjustbox}{}
\begin{tikzcd}[column sep=small]
	& {x_0} && {x_n} && {P(x_0)} && {P(x_n)} \\
	{\mathbb{D}\ni} &&&&&&&& {\in\mathbb{E}} \\
	& {y_0} && {y_1} && {P(y_0)} && {P(y_1)} \\
	& {x_0} & \cdots & {x_{k_m}} && {P(x_0)} & \cdots & {P(x_{k_m})} \\
	\exists & {w_0} & \cdots & {w_m} && {z_0} & \cdots & {z_m} \\
	& {y_0} && {y_1} && {P(y_0)} && {P(y_1)}
	\arrow[""{name=0, anchor=center, inner sep=0}, "{{{{{\prolist{\varphi}}}}}}"{inner sep=.8ex}, "\shortmid"{marking}, from=1-2, to=1-4]
	\arrow["a"', from=1-2, to=3-2]
	\arrow[""{name=1, anchor=center, inner sep=0}, "b", from=1-4, to=3-4]
	\arrow[""{name=2, anchor=center, inner sep=0}, "{{{P(\prolist{\varphi})}}}"{inner sep=.8ex}, "\shortmid"{marking}, from=1-6, to=1-8]
	\arrow[""{name=3, anchor=center, inner sep=0}, "{{{P(a)}}}"', from=1-6, to=3-6]
	\arrow["{{{P(b)}}}", from=1-8, to=3-8]
	\arrow[""{name=4, anchor=center, inner sep=0}, "\psi"'{inner sep=.8ex}, "\shortmid"{marking}, from=3-2, to=3-4]
	\arrow[""{name=5, anchor=center, inner sep=0}, "{{{P(\psi)}}}"'{inner sep=.8ex}, "\shortmid"{marking}, from=3-6, to=3-8]
	\arrow[""{name=6, anchor=center, inner sep=0}, "{{{\prolist{\varphi}_1}}}"{inner sep=.8ex}, "\shortmid"{marking}, from=4-2, to=4-3]
	\arrow["{{{c_0'}}}"', from=4-2, to=5-2]
	\arrow[""{name=7, anchor=center, inner sep=0}, "{{{\prolist{\varphi}_m}}}"{inner sep=.8ex}, "\shortmid"{marking}, from=4-3, to=4-4]
	\arrow["\cdots"{description}, draw=none, from=4-3, to=5-3]
	\arrow["{{{c_m'}}}", from=4-4, to=5-4]
	\arrow[""{name=8, anchor=center, inner sep=0}, "{{{{{P(\prolist{\varphi}_1)}}}}}"{inner sep=.8ex}, "\shortmid"{marking}, from=4-6, to=4-7]
	\arrow["{{{c_0}}}"', from=4-6, to=5-6]
	\arrow[""{name=9, anchor=center, inner sep=0}, "{{{{{P(\prolist{\varphi}_m)}}}}}"{inner sep=.8ex}, "\shortmid"{marking}, from=4-7, to=4-8]
	\arrow["\cdots"{description}, draw=none, from=4-7, to=5-7]
	\arrow["{{{c_m}}}", from=4-8, to=5-8]
	\arrow[""{name=10, anchor=center, inner sep=0}, "{{{\chi_1'}}}"'{inner sep=.8ex}, "\shortmid"{marking}, from=5-2, to=5-3]
	\arrow["{{{d_0'}}}"', from=5-2, to=6-2]
	\arrow[""{name=11, anchor=center, inner sep=0}, "{{{\chi_m'}}}"'{inner sep=.8ex}, "\shortmid"{marking}, from=5-3, to=5-4]
	\arrow["P", between={0.4}{0.6}, maps to, from=5-4, to=5-6]
	\arrow["{{{d_1'}}}", from=5-4, to=6-4]
	\arrow[""{name=12, anchor=center, inner sep=0}, "{{{\chi_1}}}"'{inner sep=.8ex}, "\shortmid"{marking}, from=5-6, to=5-7]
	\arrow["{{{d_0}}}"', from=5-6, to=6-6]
	\arrow[""{name=13, anchor=center, inner sep=0}, "{{{\chi_m}}}"'{inner sep=.8ex}, "\shortmid"{marking}, from=5-7, to=5-8]
	\arrow["{{{d_1}}}", from=5-8, to=6-8]
	\arrow[""{name=14, anchor=center, inner sep=0}, "\psi"'{inner sep=.8ex}, "\shortmid"{marking}, from=6-2, to=6-4]
	\arrow[""{name=15, anchor=center, inner sep=0}, "{{{{{P(\psi)}}}}}"'{inner sep=.8ex}, "\shortmid"{marking}, from=6-6, to=6-8]
	\arrow["\alpha"{description}, draw=none, from=0, to=4]
	\arrow["P", between={0.4}{0.6}, maps to, from=1, to=3]
	\arrow["{{P(\alpha)}}"{description}, draw=none, from=2, to=5]
	\arrow["{{\beta_1'}}"{description}, draw=none, from=6, to=10]
	\arrow["\circ"{pos=0.3}, between={0.1}{0.4}, maps to, from=4-3, to=4]
	\arrow["{{\beta_m'}}"{description}, draw=none, from=7, to=11]
	\arrow["{{\beta_1}}"{description}, draw=none, from=8, to=12]
	\arrow["\circ"{pos=0.3}, between={0.1}{0.4}, maps to, from=4-7, to=5]
	\arrow["{{\beta_m}}"{description}, draw=none, from=9, to=13]
	\arrow["{{\delta'}}"{description}, draw=none, from=5-3, to=14]
	\arrow["\delta"{description}, draw=none, from=5-7, to=15]
\end{tikzcd}
\end{adjustbox}\]
    Further, any two such decompositions are equivalent via unary multicells lying over identity unary multicells. That is, if we have two decompositions, as on the left and right below, we get a zig-zag of multicells (of length possibly $> 1$, though we just depict the length $1$ case) given by pulling out unary multicells in fibers from either the top or bottom:
\[\begin{tikzcd}
	&& {x_0} & \cdots & {x_{k_m}} && \\
	&& {w_0} & \cdots & {w_m} \\
	\\
	&& {w_0'} & \cdots & {w_m'} \\
	&& {y_0} && {y_1} \\
	{x_0} & \cdots & {x_{k_m}} && {x_0} & \cdots & {x_{k_m}} \\
	{w_0} & \cdots & {w_m} && {w_0'} & \cdots & {w_m'} \\
	{y_0} && {y_1} && {y_0} && {y_1}
	\arrow[""{name=0, anchor=center, inner sep=0}, "{{{\prolist{\varphi}_1}}}"{inner sep=.8ex}, "\shortmid"{marking}, from=1-3, to=1-4]
	\arrow["{{{c_0'}}}"', from=1-3, to=2-3]
	\arrow[""{name=1, anchor=center, inner sep=0}, "{{{\prolist{\varphi}_m}}}"{inner sep=.8ex}, "\shortmid"{marking}, from=1-4, to=1-5]
	\arrow["\cdots"{description}, draw=none, from=1-4, to=2-4]
	\arrow["{{{c_m'}}}", from=1-5, to=2-5]
	\arrow[""{name=2, anchor=center, inner sep=0}, "{{{\chi_1'}}}"'{inner sep=.8ex}, "\shortmid"{marking}, from=2-3, to=2-4]
	\arrow["{{{e_0}}}"', from=2-3, to=4-3]
	\arrow[""{name=3, anchor=center, inner sep=0}, "{{{\chi_m'}}}"'{inner sep=.8ex}, "\shortmid"{marking}, from=2-4, to=2-5]
	\arrow["\cdots"{description}, draw=none, from=2-4, to=4-4]
	\arrow["{{{e_m}}}", from=2-5, to=4-5]
	\arrow[""{name=4, anchor=center, inner sep=0}, "{{{\chi_1''}}}"'{inner sep=.8ex}, "\shortmid"{marking}, from=4-3, to=4-4]
	\arrow["{{{d_0''}}}"', from=4-3, to=5-3]
	\arrow[""{name=5, anchor=center, inner sep=0}, "{{{\chi_m''}}}"'{inner sep=.8ex}, "\shortmid"{marking}, from=4-4, to=4-5]
	\arrow["{{{d_1''}}}", from=4-5, to=5-5]
	\arrow[""{name=6, anchor=center, inner sep=0}, "\psi"'{inner sep=.8ex}, "\shortmid"{marking}, from=5-3, to=5-5]
	\arrow[""{name=7, anchor=center, inner sep=0}, "{{{\prolist{\varphi}_1}}}"{inner sep=.8ex}, "\shortmid"{marking}, from=6-1, to=6-2]
	\arrow["{{{c_0'}}}"', from=6-1, to=7-1]
	\arrow[""{name=8, anchor=center, inner sep=0}, "{{{\prolist{\varphi}_m}}}"{inner sep=.8ex}, "\shortmid"{marking}, from=6-2, to=6-3]
	\arrow["\cdots"{description}, draw=none, from=6-2, to=7-2]
	\arrow["{{{c_m'}}}", from=6-3, to=7-3]
	\arrow[""{name=9, anchor=center, inner sep=0}, "{{{\prolist{\varphi}_1}}}"{inner sep=.8ex}, "\shortmid"{marking}, from=6-5, to=6-6]
	\arrow["{{{e_0\circ c_0'}}}"', from=6-5, to=7-5]
	\arrow[""{name=10, anchor=center, inner sep=0}, "{{{\prolist{\varphi}_m}}}"{inner sep=.8ex}, "\shortmid"{marking}, from=6-6, to=6-7]
	\arrow["\cdots"{description}, draw=none, from=6-6, to=7-6]
	\arrow["{{{e_m\circ c_m'}}}", from=6-7, to=7-7]
	\arrow[""{name=11, anchor=center, inner sep=0}, "{{{\chi_1'}}}"'{inner sep=.8ex}, "\shortmid"{marking}, from=7-1, to=7-2]
	\arrow["{{{d_0'}}}"', from=7-1, to=8-1]
	\arrow[""{name=12, anchor=center, inner sep=0}, "{{{\chi_m'}}}"'{inner sep=.8ex}, "\shortmid"{marking}, from=7-2, to=7-3]
	\arrow["{{{d_1'}}}", from=7-3, to=8-3]
	\arrow[""{name=13, anchor=center, inner sep=0}, "{{{\chi_1''}}}"'{inner sep=.8ex}, "\shortmid"{marking}, from=7-5, to=7-6]
	\arrow["{{{d_0''}}}"', from=7-5, to=8-5]
	\arrow[""{name=14, anchor=center, inner sep=0}, "{{{\chi_m''}}}"'{inner sep=.8ex}, "\shortmid"{marking}, from=7-6, to=7-7]
	\arrow["{{{d_1''}}}", from=7-7, to=8-7]
	\arrow[""{name=15, anchor=center, inner sep=0}, "\psi"'{inner sep=.8ex}, "\shortmid"{marking}, from=8-1, to=8-3]
	\arrow[""{name=16, anchor=center, inner sep=0}, "\psi"'{inner sep=.8ex}, "\shortmid"{marking}, from=8-5, to=8-7]
	\arrow["{{\beta_1'}}"{description}, draw=none, from=0, to=2]
	\arrow["{{\beta_m'}}"{description}, draw=none, from=1, to=3]
	\arrow["{{\sigma_1}}"{description}, draw=none, from=2, to=4]
	\arrow["{{\sigma_m}}"{description}, draw=none, from=3, to=5]
	\arrow["{{\delta''}}"{description}, draw=none, from=4-4, to=6]
	\arrow[between={0.3}{0.8}, maps to, from=6, to=6-3]
	\arrow[between={0.3}{0.8}, maps to, from=6, to=6-5]
	\arrow["{{\beta_1'}}"{description}, draw=none, from=7, to=11]
	\arrow["{{\beta_m'}}"{description}, draw=none, from=8, to=12]
	\arrow["{{\frac{\beta_1'}{\sigma_1}}}"{description}, draw=none, from=9, to=13]
	\arrow["{{\frac{\beta_m'}{\sigma_m}}}"{description}, draw=none, from=10, to=14]
	\arrow["{{\delta'}}"{description}, draw=none, from=7-2, to=15]
	\arrow["{{\delta''}}"{description}, draw=none, from=7-6, to=16]
\end{tikzcd}\]
    where $P(\sigma_i) = \text{id}_{\chi_i}$ for all $1\leq i \leq m$. 
\end{expl}

The arguments in Lemma~\ref{lem:components} and Corollary~\ref{cor:decompSimp} extend directly to the relative setting to give the following:

\begin{cor}\label{cor:relDecompSimp}
	A VDF $P:\Dbl{D}\to \Dbl{E}$ lifts decompositions if and only if it lifts 2-layer decompositions.
\end{cor}

Under this reduction the condition for lifting decompositions can be phrased in terms of coend isomorphisms. Namely, $P:\Dbl{D}\to \Dbl{E}$ lifts decompositions if and only if for each multicell $\alpha \in \Dbl{E}$ with decomposition $\alpha=\frac{\prolist{\beta}}{\gamma}$, the composition map
\begin{equation*}
	\int^{\prolist{\varphi}\in P^{-1}(s(\gamma))}\mathsf{fc}(\MCell(\mathbb{D}))_{/\prolist{\beta}}(-,\prolist{\varphi})\times \MCell(\mathbb{D})_{/\gamma}(\prolist{\varphi},-)\to \MCell(\mathbb{D})_{/\alpha}(\mu_\mathsf{fc}(-),-)
\end{equation*}
is an isomorphism, and for each tight arrow $a\in \Dbl{E}$ with decomposition $a=c\circ b$, the composition map
\begin{equation*}
	\int^{y\in P^{-1}(s(c))}\Cat{Arr_T}(\mathbb{D})_{/b}(-,y)\times \Cat{Arr_T}(\mathbb{D})_{/c}(y,-)\to \Cat{Arr_T}(\mathbb{D})_{/a}(-,-)
\end{equation*}
is an isomorphism. Note that this second condition is precisely the condition that the functor on underlying tight categories $P_0$ is a Conduch\'{e} fibration, which is to say it is an exponentiable map of categories.

The last ingredient we need to relativize the argument in Theorem~\ref{thm:ExpViaDecomp} is a description of the colimits for the cocones.

\begin{rmk}[Relative Colimits for Sequences of Tight Arrows]\label{rmk:tightArrColimits}
	For a sequence of tight arrows $\prolist{a}:x_0\to x_n$ in $\Dbl{E}$, corresponding to a VDF $\prolist{a}:\Tight[n]\to \Dbl{E}$, we obtain the diagram
	\begin{equation*}
		\Cat{el}\downarrow \Tight[n]\xrightarrow{\prolist{a}} \Vdc_{/\Dbl{E}}\xrightarrow{\Dbl{D}\times_\Dbl{E}-}\Vdc
	\end{equation*}
	that we want to understand the colimit of. The resulting diagram is given by the wide span
	\begin{equation*}
		P^{-1}(a_1)\xleftarrow{t}P^{-1}(x_1)\xrightarrow{s}\cdots\xleftarrow{t}P^{-1}(x_{n-1})\xrightarrow{s}P^{-1}(a_n)
	\end{equation*}
	where each object lies in the image of $\T(-)$, so we can compute this colimit in $\Cat{Cat}$. The objects in the colimit are the objects in $\coprod_{i=0}^nP^{-1}(x_i)$, while a morphism between an object $y \in P^{-1}(x_i)$ and $z\in P^{-1}(x_j)$, for $i<j$, is an equivalence class of sequences of morphisms $(y,x_i)\to (y_1,x_{i+1})\to \cdots \to (z,x_j)$. As in Remark~\ref{rmk:cellInColimit}, associativity induces the equivalence relation, but now for a sequence of morphisms $(y,x_i)\to \cdots \to (z,x_j)$ we are only able to slide arrows along identities, such as $\text{id}_{x_k}$, if they lie in the fiber $P^{-1}(x_k)$.
\end{rmk}

\begin{rmk}[Relative Colimits for Trees]\label{rmk:relColimForTrees}
	Let $T$ be a tree with $m$-leaves, and let $\alpha \in \Dbl{E}$ be an $m$-ary multicell together with a decomposition $\beta:\Dbl{J}(T)\to \Dbl{E}$ relative to the tree $\alpha$. Then we obtain the associated diagram
	\begin{equation*}
		J(T)=\Cat{el}\downarrow \Dbl{J}(T)\xrightarrow{\beta} \Vdc_{/\Dbl{E}}\xrightarrow{\Dbl{D}\times_\Dbl{E} -}\Vdc
	\end{equation*}
	which we want to understand the colimit of. We can express $\colim_{i:\Dbl{c}\hookrightarrow \Dbl{J}(T)}(\Dbl{D}\times_\Dbl{E} \Dbl{c})$ as the colimit of a diagram obtained from the tree $T$ by replacing the products in Remark~\ref{rmk:cellInColimit} with pullbacks, or equivalently products in $\Vdc_{/\Dbl{E}}$ before forgetting to $\Vdc$. In particular, we get that objects and loose arrows in the colimit are pairs of such in $\Dbl{D}$ and $\Dbl{J}(T)$, such that the component in $\Dbl{D}$ must live in the appropriate fiber of $P$. Further, tight arrows are as described in Remark~\ref{rmk:tightArrColimits}.

	Finally, a multicell in the colimit corresponds to an equivalence class of tuples of a subtree $T'\subseteq T$ of $T$, together with for each non-leaf node in $T'$ a multicell in $\Dbl{D}$ living in the appropriate fiber of $P$, such that the multicells have compatible boundaries. The equivalence relation on the tuples is as in Remark~\ref{rmk:cellInColimit}, except that we can only move sequences of unary multicells that lie over identities in $\Dbl{E}$.
\end{rmk}

With these preliminaries, the proof technique for the absolute case in Theorem~\ref{thm:ExpViaDecomp} extends to the relative case. Indeed, extending the argument for Theorem~\ref{thm:ExpViaDecomp} to the relative context using the descriptions of the colimits in Remarks~\ref{rmk:tightArrColimits} and~\ref{rmk:relColimForTrees}, we see that Corollary~\ref{cor:exponentiable-morphisms-via-slice} implies a VDF $P:\Dbl{D}\to \Dbl{E}$ is exponentiable if and only if it lifts decompositions into 2 and 3 layers essentially uniquely. Corollary~\ref{cor:relDecompSimp} then ensures that this is equivalent to $P$ lifting all decompositions essentially uniquely.

\begin{thm}[Characterizing Exponentiable VDFs]\label{thm:expMorphViaDecomp}
    A VDF $P:\Dbl{D}\to \Dbl{E}$ is exponentiable if and only if it lifts decompositions essentially uniquely.
\end{thm}

Now, from Corollary 2.6 of~\cite{StreetVerity2010} we have that all isomorphisms of VDCs are exponentiable (in fact all equivalences of VDCs are exponentiable), all composites of exponentiable morphisms are exponentiable, and arbitrary pullbacks of exponentiable morphisms are exponentiable. In particular, as discussed previously, products of exponentiable VDCs are exponentiable, and the fibers of exponentiable morphisms are exponentiable VDCs. In fact, for any exponentiable morphism $P:\Dbl{D}\to \Dbl{E}$, and any VDF $F:\Dbl{C}\to \Dbl{E}$ out of an exponentiable VDC, the fiber product $\Dbl{C}\times_\Dbl{E}\Dbl{D}$ is exponentiable. Additionally, the characterization in Theorem~\ref{thm:expMorphViaDecomp} implies that VDC inclusions which are full with respect to decompositions of multicells are also exponentiable.

Generalizing the case of discrete opcartesian fibrations in~\cite{fujii2025}, Theorem~\ref{thm:expMorphViaDecomp} also implies that opcartesian fibrations of VDCs are exponentiable, though as exploring exponentiable morphisms of VDCs is not the central goal of the current work, we will not prove this here.

We'll conclude by observing that
the loose embedding does not reflect exponentiable morphisms. 

\begin{prop}\label{prop:functorsOfCategoriesAreExponentiable}
	If $F:\Cat{C}\to \Cat{D}$ is a functor between categories,
	then $\L(F):\L(\Cat{C})\to \L(\Cat{D})$ is an exponentiable VDF.
\end{prop}
\begin{proof}
	Since $\L(F)_0$ is a functor from a discrete category, 
	it is necessarily a Conduch\'{e} fibration.
	Further, the only multicells in $\L(\Cat{C})$ and $\L(\Cat{D})$ are unique
	opcartesian multicells witnessing the strict composites in $\Cat{C}$ and $\Cat{D}$.
	Decompositions of such multicells correspond to composing a sequence
	of arrows by first composing certain subsequences, and then composing the result. 
	In particular, all such decompositions lift uniquely to $\Cat{C}$, 
	so $\L(F)$ is an exponentiable VDF.
\end{proof}

%% file: Sections/Appendix.tex
\section{VDCs via Categories of Operators}\label{sec:VDCCatOfOper}

For extending the main exponentiability proofs in this article to other contexts in higher category theory, it will be beneficial to recontextualize the data of a VDC. This reframing is motivated by the category of operators perspective for operads (i.e.~multicategories) introduced by May and Thomason in~\cite{MAY1978205} which has found numerous applications in homotopy theory, and especially higher algebra, as well as the theory of algebraic patterns introduced in~\cite{Chu2021}.
In the context of $\infty$-categories, this approach to VDCs has also been used to discuss enrichment in $\infty$-categories~\cite{gepner2015}.

This algebraic patterns perspective was used in~\cite{blom2026dayconvolutionalgebraicpatterns} 
simultaneously to the development of the current work to characterize exponentiable morphisms of algebrads 
for rigid algebraic patterns in the context of $\infty$-categories. 
From~\cite{Chu2021} such algebraic patterns arise from parametric right adjoint monads $T$ on pre-sheaf categories, 
while the algebrads correspond to $T$-multicategories. 
Note that Proposition~\ref{prop:functorsOfCategoriesAreExponentiable} shows that a
classification of exponentiable morphisms of $T$-multicategories
cannot be used to classify exponentiable morphisms of $T$-algebras in general.
The work below shows that our characterization of exponentiable VDCs is equivalent to this 
characterization in terms of algebraic patterns in the context of 1-categories,
where the condition that the $\infty$-category of decompositions is weakly contractible
is replaced by the condition that the $1$-category of decompositions is connected.

In order to perform this reframing we use the display category equivalence observed by B{\'{e}}nabou in~\cite{benabou1972}, and phrased in terms of lax functors in~\cite{displayed}. Specifically, the display category equivalence can be phrased in terms of VDCs as an equivalence
\begin{equation}\label{eq:display}
    \Cat{Cat}_{/\Cat{A}}\simeq \BiCat{Vdc}(\L(\Cat{A})^{op_l},\Span)
\end{equation}
for every category $\Cat{A}$, where on the right we have the category of VDFs $\L(\Cat{A})^{op_l}\to \Span$ and tight transformations between them. Here the superscript $op_l$ indicates that we are reversing the direction of loose arrows in $\L(\Cat{A})$. Since $\L(\Cat{A})^{op_l}$ is a unital VDC, the adjunction $U\dashv \Mod$ gives an isomorphism $\BiCat{Vdc}(\L(\Cat{A})^{op_l},\Span)\cong \BiCat{Vdc}_n(\L(\Cat{A})^{op_l},\Prof)$, which will be used to describe VDCs as normal VDFs into $\Prof$ preserving certain opcartesian multicells.

To encode the data of a VDC $\Dbl{D}$ we take $\Cat{A}$ in the above equivalence to be the opposite of the simplex category $\Cat{\Delta}$. Recall that $\Cat{\Delta}$ has objects given by the finite (nonempty) linearly ordered sets $[n] := \{0 < 1 < \cdots < n\}$ for each natural number $n\geq 0$, and has morphisms given by order preserving maps $f:[n]\to [m]$. The category $\Cat{\Delta}$ admits an $(\mathsf{epi},\mathsf{mono})$ strict factorization system, where for each $n\geq 0$ and each $0\leq i \leq n$, we will write $\sigma^i:[n+1]\to [n]$ for the unique surjective order preserving map such that the pre-image of $i$ has two elements, and we will write $\delta^i:[n-1]\to [n]$ for the unique injective order preserving map whose image doesn't contain $i$. For a subset $T\subseteq [n]$, we will also write $\delta^{T}:[|T|]\to [n]$ for the unique injective order preserving map whose image is $T$.

The simplex category also comes equipped with another strict factorization system $(\Cat{act},\Cat{inrt})$ where the left class $\Cat{act}$ consists of those morphisms that preserve top and bottom elements, and the right class $\Cat{inrt}$ consists of all interval inclusions $\delta^{[i,i+k]}:[k]\hookrightarrow [n]$, for $[i,i+k]\subseteq [n]$ (c.f.~\cite[Example 3.3]{Chu2021}). We refer to morphisms in $\Cat{act}$ as \textit{active} and morphisms in $\Cat{inrt}$ as \textit{inert}. We can equivalently describe the data of an active morphism $f:[n]\to [m]$ in terms of a partition $m=k_1+\cdots+k_n$, for integers $k_1,...,k_n\geq 0$, where the correspondence is given by defining $k_i=f(i)-f(i-1)$. For each $n \geq 0$, we will write $a_n:[1]\to [n]$ for the unique active morphism from $[1]$ to $[n]$. Using this data we can now define the category of operators construction for VDCs.

\begin{cons}[Category of Operators for a VDC]\label{cons:CatOfOperators}
	Let $\Dbl{D}$ be a VDC. We define the \emph{category of operators} of $\Dbl{D}$ to be a category $\Cat{D}$ together with a functor $p_\Dbl{D}:\Cat{D}\to \Cat{\Delta}^{op}$ given by the following data:
	\begin{enumerate}
		\item An object in $\Cat{D}$ is a pair of an object $[n] \in \Cat{\Delta}$ together with a sequence of loose arrows $(\prolist{\varphi}:x_0\proto x_n)\in \Cat{fc}_n(\Dbl{D}_1)$. When $[n]=[0]$, $\Cat{fc}_0(\Dbl{D}_1)=\uOb (\Dbl{D})$ consists of the objects of $\Dbl{D}$.
		\item A morphism $([n],\prolist{\varphi}:x_0\proto x_n)\to ([m],\prolist{\psi}:y_0\proto y_m)$ in $\Cat{D}$ consists of a pair of a morphism $\sigma:[m]\to [n]$ in $\Cat{\Delta}$, 
    together with a chain of compatible multicells $\prolist{\chi}=(\chi_1 \cdots \chi_m)$ in $\Dbl{D}$ where $\chi_i$ has loose target $\psi_i:y_{i-1}\proto y_i$ and 
    loose source $(\varphi_{\sigma(i-1)+1} \cdots \varphi_{\sigma(i)}):x_{\sigma(i-1)}\proto x_{\sigma(i)}$:
		\[\begin{adjustbox}{}
		\begin{tikzcd}
	{x_{\sigma(i-1)}} && {x_{\sigma(i-1)+1}} && \cdots && {x_{\sigma(i)}} \\
	&&& {\chi_i} \\
	{y_{i-1}} &&&&&& {y_i}
	\arrow["{\varphi_{\sigma(i-1)+1}}"{inner sep=.8ex}, "\shortmid"{marking}, from=1-1, to=1-3]
	\arrow[from=1-1, to=3-1]
	\arrow["{\varphi_{\sigma(i-1)+2}}"{inner sep=.8ex}, "\shortmid"{marking}, from=1-3, to=1-5]
	\arrow["{\varphi_{\sigma(i)}}"{inner sep=.8ex}, "\shortmid"{marking}, from=1-5, to=1-7]
	\arrow[from=1-7, to=3-7]
	\arrow["{\psi_i}"'{inner sep=.8ex}, "\shortmid"{marking}, from=3-1, to=3-7]
\end{tikzcd}
\end{adjustbox}\]
		When $[m]=[0]$, so that $\prolist{\psi}$ corresponds to an object $y$ in $\Dbl{D}$, $\prolist{\chi}$ corresponds to a tight arrow $x_{\sigma(0)}\to y$.
	\end{enumerate}
	Composition is given by composition in $\Cat{\Delta}$ and composition of multicells in $\Dbl{D}$, while identity morphisms $([n],\prolist{\varphi})\to ([n],\prolist{\varphi})$ are given by pairs $(\text{id}_{[n]},\text{id}_{\prolist{\varphi}})$ of an identity morphism in $\Cat{\Delta}$ and a chain of identity multicells in $\Dbl{D}$. The functor $p_\Dbl{D}:\Cat{D}\to \Cat{\Delta}^{op}$ is the natural projection onto the first coordinate.

	Note that the fiber $\Cat{D}_n:=p_\Dbl{D}^{-1}([n])$ is isomorphic to the category of length $n$-chains of loose arrows in $\Dbl{D}$, $\Cat{fc}_n(\Dbl{D}_1)$, when $n\geq 1$, and is isomorphic to the underlying tight category for $\Dbl{D}$ when $n=0$.
\end{cons}

Given a VDF $F:\Dbl{D}\to \Dbl{E}$, we have a natural functor $p_F:\Cat{D}\to \Cat{E}$ over $\Cat{\Delta}^{op}$ given on a pair $([n],\prolist{\varphi})$ by $([n],F\prolist{\varphi})$, and given on a morphism $(f,\prolist{\chi}):([n],\prolist{\varphi})\to ([m],\prolist{\psi})$ by $(f,F\prolist{\chi})$. Since $F$ is a VDF, the definition of identities and composites in $\Cat{D}$ and $\Cat{E}$ imply that $p_F$ is indeed a functor, and further this assignment extends to a functor $\Vdc\to \Cat{Cat}_{/\Cat{\Delta}^{op}}$. We will now describe the essential image of this functor.

\begin{prop}[VDCs as Operator Categories over the Simplex]\label{prop:VDCsAsOperCats}
	A functor $q:\Cat{A}\to \Cat{\Delta}^{op}$ is isomorphic to the category of operators of a VDC if and only if
	\begin{enumerate}
		\item $q$ admits a functorial choice of cocartesian lifts of inert morphisms in $\Cat{\Delta}^{op}$;
		\item For each integer $n\geq 0$, the functor 
		\begin{equation*}
			(\delta^{[0,1]}_*,\ldots,\delta^{[n-1,n]}_*):\Cat{A}_n\to \underbrace{\Cat{A}_1\times_{\Cat{A}_0}\cdots\times_{\Cat{A}_0}\Cat{A}_1}_n
		\end{equation*}
		which is induced by the cocartesian lifts of the inert morphisms $\delta^{[i,i+1]}:[1]\cong [i,i+1]\subseteq [n]$ and $\delta^{\{i\}}:[0]\cong \{i\}\subseteq [n]$, is an isomorphism;
		\item For each $\sigma:[m]\to [n]$ in $\Cat{\Delta}$, and each $x\in \Cat{A}_n,y\in \Cat{A}_m$, the function
    \[\begin{tikzcd}
      {\Cat{A}_\sigma(x,y)} \\
      {      \Cat{A}_{\sigma\delta^{[0,1]}}(x,\delta^{[0,1]}_*y)\underset{\Cat{A}_{\sigma\delta^{\{1\}}}(x,\delta^{\{1\}}_*y)}{\times}\cdots\underset{\Cat{A}_{\sigma\delta^{\{m-1\}}}(x,\delta^{\{m-1\}}_*y)}{\times}\Cat{A}_{\sigma\delta^{[m-1,m]}}(x,\delta^{[m-1,m]}_*y)}
      \arrow[from=1-1, to=2-1]
    \end{tikzcd}\]
		induced by post-composing with the cocartesian lifts is a bijection, where the subscript on $\Cat{A}$ indicates what morphism in $\Cat{\Delta}^{op}$ the morphisms in $\Cat{A}$ are lying over.
	\end{enumerate}
\end{prop}

Before proving the proposition let's explicate the three conditions appearing in the statement. First, $q$ admitting cocartesian lifts of an inert morphism $i:[m]\cong [j,j+m]\subseteq [n]$ implies that we get a functor $i_*:\Cat{A}_n\to \Cat{A}_m$, which will play the role of projecting a sequence of loose arrows onto a subsequence, or an object in the sequence. Additionally, the first condition says that these functors can be chosen in such a way that if $[m]\xrightarrow{i}[n]\xrightarrow{j}[k]$ is a composable pair of inert morphisms, then $i_*j_*=(ji)_*:\Cat{A}_k\to \Cat{A}_m$, encoding the fact that projections should be strictly functorial. The second condition says that these cocartesian lifts do actually correspond to projections on the level of fibers, and $\Cat{A}_n$ corresponds to a limit over these projections. Finally, the third condition says that under our identification of the fibers of $\Cat{A}$ with the limit of the projections, a morphism $f:a\to b$ in $\Cat{A}$ lying over $\sigma:[m]\to [n]$ is equivalent to a sequence of compatible morphisms whose target is always an object either of $\Cat{A}_1$, or $\Cat{A}_0$ if $m=0$.

\begin{proof}[Proof of~\ref{prop:VDCsAsOperCats}.]
	First, let's show that the category of operators for a VDC $\Dbl{D}$ satisfies conditions (1) through (3).
	\begin{enumerate}
		\item If $i:[k]\cong [l,l+k]\subseteq [n]$ is an inert morphism, then we have a natural functor 
    $i_*:\Cat{D}_n\to \Cat{D}_k$ given by projecting from a sequence of loose arrows 
    $(\varphi_1 \cdots \varphi_n)$ (resp. multicells $(\alpha_1 \cdots \alpha_n)$) to the subsequence 
    $(\varphi_{l+1} \cdots \varphi_{l+k})$ (resp. $(\alpha_{l+1} \cdots \alpha_{l+k})$). 
    For a sequence of loose arrows $\prolist{\varphi}=(\varphi_1 \cdots \varphi_n)$ in $\Dbl{D}$, 
    the chosen cocartesian lift is then the morphism 
    $(i,(\text{id}_{\varphi_{l+1}} \cdots \text{id}_{\varphi_{l+k}})):([n],\prolist{\varphi})\to ([k],(\varphi_{l+1} \cdots \varphi_{l+k}))$. 
    This is cocartesian since any other morphism $(i\circ g,\prolist{\alpha}):([n],\prolist{\varphi})\to ([m],\prolist{\psi})$ 
    factors uniquely as $(i,(\text{id}_{\varphi_{l+1}} \cdots \text{id}_{\varphi_{l+k}}))$ followed by 
    $(g,\prolist{\alpha}): ([k],(\varphi_{l+1} \cdots \varphi_{l+k}))\to ([n],\prolist{\varphi})$. Further, this assignment on inert morphisms is functorial, proving (1).
		\item Since the functors $i_*$ are defined to be projections, (2) follows.
		\item Similarly, the functors $i_*$ being the natural projections also implies (3).
	\end{enumerate}

    Next, suppose $q:\Cat{A}\to \Cat{\Delta}^{op}$ is a functor satisfying properties (1)-(3), and write $(-)_*$ for the functorial choice of cocartesian lifts for inert morphisms. 
    In particular, this choice fixes an isomorphism $\Phi_n:=(\delta^{[0,1]}_*,\ldots,\delta^{[n-1,n]}_*):\Cat{A}_n\to \Cat{A}_1\times_{\Cat{A}_0}\cdots\times_{\Cat{A}_0}\Cat{A}_1$ 
    for each integer $n\geq 0$, as well as a natural isomorphism
    \[\begin{tikzcd}
      {\Cat{A}_\sigma(x,y)} \\
      {\Cat{A}_{\sigma\delta^{[0,1]}}(x,\delta^{[0,1]}_*y)\underset{\Cat{A}_{\sigma\delta^{\{1\}}}}{\times}(x,\delta^{\{1\}}_*y)\cdots\underset{\Cat{A}_{\sigma\delta^{\{m-1\}}}}{\times}\Cat{A}_{\sigma\delta^{[m-1,m]}}(x,\delta^{[m-1,m]}_*y)}
      \arrow["{\Psi^\sigma}"', from=1-1, to=2-1]
    \end{tikzcd}\]
	for each morphism $\sigma:[m]\to [n]$ in $\Cat{\Delta}$. Then we define a VDC $\Dbl{A}$ consisting of the following data:
	\begin{enumerate}
        \item The underlying tight category of $\Dbl{A}$ is the fiber category $\Cat{A}_0$.
        \item The loose arrows of $\Dbl{A}$ are the objects of the fiber category $\Cat{A}_1$.
        \item For $n\geq 0$, the set of $n$-ary multicells with loose source $\prolist{\varphi}:x_0\proto x_n$ and loose target $\psi:y_0\proto y_1$ is given by $\Cat{A}_{a_n}(\Phi_n^{-1}(\varphi_1 \cdots \varphi_n),\psi)$, while post composing with $\delta^{\{0\}}_*,\delta^{\{1\}}_*:\Cat{A}_1\to \Cat{A}_0$ serve as the tight source and target maps for the $n$-ary multicells.
        \item For a sequence of compatible multicells $(\alpha_1 \cdots \alpha_n)$, where $\alpha_i$ is $k_i$-ary, and an $n$-ary multicell $\beta$ whose loose source is given by the loose targets of $\prolist{\alpha}$, their composite is defined as 
		\begin{equation*}
			\dfrac{(\alpha_1 \cdots \alpha_n)}{\beta}:=\beta\circ (\Psi^{\tau_{(k_i)}})^{-1}(\alpha_1 \cdots \alpha_n)
		\end{equation*}
		where $\tau_{(k_i)}:[n]\to [k_1+\cdots+k_n]$ is given by $j\mapsto \sum_{i=1}^jk_j$.
    \end{enumerate}
    The unitality and associativity axioms for composition come from the associated axioms for $\Cat{A}$. Thus, $\Dbl{A}$ is a VDC, and by construction the operator category for $\Dbl{A}$ is naturally isomorphic to the originally category $\Cat{A}$ over $\Cat{\Delta}^{op}$.
\end{proof}

From~\cite[Lemma 9.10]{Chu2021} the active-inert factorization system on $\Cat{\Delta}^{op}$ extends to an active-inert orthogonal factorization system on $\Cat{D}$, where the active morphisms are those lying over active morphisms in $\Cat{\Delta}^{op}$, and the inert morphisms are the cocartesian lifts of inert morphisms in $\Cat{\Delta}^{op}$. 
Explicitly, an active morphism $(f,\prolist{\chi}):([n],\prolist{\varphi})\to ([m],\prolist{\psi})$ in $\Cat{D}$ is precisely either a tight arrow, if $m=n=0$, or a sequence of multimorphisms
\[\begin{adjustbox}{}
\begin{tikzcd}
	{x_0} && {x_{\sigma(1)}} && \cdots && {x_n} \\
	& {\chi_1} && {\chi_2} && {\chi_m} \\
	{y_0} && {y_1} && \cdots && {y_m}
	\arrow["{\prolist{\varphi}(0,\sigma(1))}"{inner sep=.8ex}, "\shortmid"{marking}, from=1-1, to=1-3]
	\arrow[from=1-1, to=3-1]
	\arrow["{\prolist{\varphi}(\sigma(1),\sigma(2))}"{inner sep=.8ex}, "\shortmid"{marking}, from=1-3, to=1-5]
	\arrow[from=1-3, to=3-3]
	\arrow["{\prolist{\varphi}(\sigma(m-1),\sigma(m))}"{inner sep=.8ex}, "\shortmid"{marking}, from=1-5, to=1-7]
	\arrow[from=1-7, to=3-7]
	\arrow["{\psi_1}"'{inner sep=.8ex}, "\shortmid"{marking}, from=3-1, to=3-3]
	\arrow["{\psi_2}"'{inner sep=.8ex}, "\shortmid"{marking}, from=3-3, to=3-5]
	\arrow["{\psi_m}"'{inner sep=.8ex}, "\shortmid"{marking}, from=3-5, to=3-7]
\end{tikzcd}
\end{adjustbox}\]
in $\Dbl{D}$, with the whole sequence $\prolist{\varphi}$ as loose source. An inert morphism $(i,\prolist{\chi}):([n],\prolist{\varphi})\to ([m],\prolist{\psi})$ in $\Cat{D}$ lying over the inclusion $i:[m]\cong [j,j+m]\subseteq [n]$ is a sequence of unary isomorphism multicells from a sub-interval of the loose source into the loose target:
\[\begin{adjustbox}{}
\begin{tikzcd}
	{x_j} && {x_{j+1}} && \cdots && {x_{j+m}} \\
	& \cong && \cong && \cong \\
	{y_0} && {y_1} && \cdots && {y_m}
	\arrow["{\varphi_{j+1}}"{inner sep=.8ex}, "\shortmid"{marking}, from=1-1, to=1-3]
	\arrow[from=1-1, to=3-1]
	\arrow["{\varphi_{j+2}}"{inner sep=.8ex}, "\shortmid"{marking}, from=1-3, to=1-5]
	\arrow[from=1-3, to=3-3]
	\arrow["{\varphi_{j+m}}"{inner sep=.8ex}, "\shortmid"{marking}, from=1-5, to=1-7]
	\arrow[from=1-7, to=3-7]
	\arrow["{\psi_1}"'{inner sep=.8ex}, "\shortmid"{marking}, from=3-1, to=3-3]
	\arrow["{\psi_{2}}"'{inner sep=.8ex}, "\shortmid"{marking}, from=3-3, to=3-5]
	\arrow["{\psi_m}"'{inner sep=.8ex}, "\shortmid"{marking}, from=3-5, to=3-7]
\end{tikzcd}
\end{adjustbox}\]

Next, composing with the display isomorphism $\Cat{Cat}_{/\Delta^{\op}}\simeq \BiCat{Vdc}(\L(\Cat{\Delta^{\op}})^{op_l},\Span)$, 
we see that a VDC $\Dbl{D}$ can be encoded by a normal VDF $\Cat{P}(\Dbl{D}):\L(\Cat{\Delta})\to \Prof$ given by the following data:
\begin{enumerate}
	\item For each integer $n \geq 0$, the object $[n]$ is mapped to the fiber $\Cat{P}(\Dbl{D})_n:= \Cat{D}_n$.
	\item For each morphism $\sigma:[n]\to [m]$ in the simplex category, the profunctor $\Cat{P}(\Dbl{D})(\sigma):\Cat{D}_n\times \Cat{D}_m^{op}\to \Cat{Set}$ is given by $\Cat{D}_\sigma(-,-)$.
	\item For each composable pair of morphisms $[n]\xrightarrow{f}[m]\xrightarrow{g}[k]$ in the simplex category, 
  the image of the cocartesian cell for $g\circ f$ is $\ell_{f,g}:\Cat{P}(\Dbl{D})(f)\odot \Cat{P}(\Dbl{D})(g)\Rightarrow \Cat{P}(\Dbl{D})(g\circ f)$ 
  given by composing sequences of multicells. 
  \item The rest of the data of the VDF follows, since every cell in $\L(\Cat{\Delta})$ arises from a sequence of composable morphisms.
\end{enumerate}
In particular, the description of the operator categories for VDCs translates to those normal VDFs $Q:\L(\Cat{\Delta})\to \Prof$ satisfying the following:
\begin{enumerate}
	\item For $i:[k]\hookrightarrow [n]$ inert, the profunctor $Q(i):Q_{[k]}\times Q_{[n]}^{op}\to\Cat{Set}$ is represented by a functor $i_*:Q_{[n]}\to Q_{[k]}$ so that $Q(i)\cong Q_{[k]}(i_*(-),-)$, and these representations can be chosen to be functorial in the inert morphisms (i.e.~for inert morphisms $[k] \xrightarrow{i} [n] \xrightarrow{j} [m]$, $i_*j_* = (j\circ i)_*$, and $(\text{id}_{[k]})_*=\text{id}_{Q_{[k]}}$). 
	\item For $i:[k]\hookrightarrow [n]$ inert and $g:[l]\to [k]$ arbitrary, the laxator cell
	\[\begin{adjustbox}{}
\begin{tikzcd}
{Q_{[l]}} && {Q_{[k]}} && {Q_{[n]}} \\
&& {\ell_{g,f}} \\
{Q_{[l]}} &&&& {Q_{[n]}}
\arrow["{Q(g)}"{inner sep=.8ex}, "\shortmid"{marking}, from=1-1, to=1-3]
\arrow[equals, from=1-1, to=3-1]
\arrow["{Q(i)}"{inner sep=.8ex}, "\shortmid"{marking}, from=1-3, to=1-5]
\arrow[equals, from=1-5, to=3-5]
\arrow["{Q(i\circ g)}"'{inner sep=.8ex}, "\shortmid"{marking}, from=3-1, to=3-5]
\end{tikzcd}
\end{adjustbox}\]
	is opcartesian.
	\item For each $[n] \in \Cat{\Delta}$, the natural map
	\begin{equation*}
		(\delta^{[0,1]}_*,\ldots,\delta^{[n-1,n]}_*):Q_{[n]}\to Q_{[1]}\times_{Q_{[0]}}\cdots\times_{Q_{[0]}}Q_{[1]}
	\end{equation*}
	in $\Cat{Cat}=\Prof_0$ induced by the functors from (1), for the inert morphisms $\delta^{[i,i+1]}:[1]\to [n]$, is an isomorphism.
	\item For each morphism $\phi:[n]\to [m]$, the natural map of profunctors 
\[\begin{tikzcd}[column sep=small]
	{Q(\phi)} \\
	{Q(\phi_{[0,1]})(-,\delta^{[0,1]}_*)\underset{{Q(\phi_{\{1\}})(-,\delta^{1}_*)}}{\times}
  \cdots \underset{{Q(\phi_{n-1})(-,\delta^{\{n-1\}}_*)}}{\times}Q(\phi_{[n-1,n]})(-,\delta^{[n-1,n]}_*)}
	\arrow[from=1-1, to=2-1]
\end{tikzcd}\]
	in $\Prof_1$ induced by the projections in (1) is an isomorphism.
\end{enumerate}

One of the benefits of this profunctor description of VDCs is the ability to translate it into the theory of pullback sketches in VDCs.

\begin{cons}[Pullback Sketch for a VDC]\label{cons:LimitSketch}
	Let $\Dbl{D}$ be a VDC, and let $\Dbl{Q}(\Cat{\Delta}^{op},\mathsf{inrt})\subseteq \Dbl{Q}(\Cat{\Delta}^{op})$ denote the sub-VDC of
   the quintet VDC whose only tight morphisms are inert morphisms. Then due to condition (1) in the description of VDCs as normal VDFs 
   into $\Prof$, we can extend the normal VDF $\Cat{P}(\Dbl{D}):\L(\Cat{\Delta})\to\Dbl{P}\Cat{rof}$ for a VDC $\Dbl{D}$ to a normal VDF 
   $\Cat{Q}(\Dbl{D}):\Dbl{Q}(\Cat{\Delta}^{op},\mathsf{inrt})^{op_l}\to \Dbl{P}\Cat{rof}$ where an inert morphism $i:[n]\hookrightarrow [m]$ 
   is mapped to the functor $i_*:\Cat{D}_m\to \Cat{D}_n$ representing the profunctor $\Cat{P}(\Dbl{D})(i):\Cat{D}_n\times \Cat{D}_m^{op}\to\Cat{Set}$. 
   Additionally, the unique multicell $\cell{i}{f_1 \cdots f_n}{g}{j}$ witnessing the composite $f_n\circ \cdots\circ f_1\circ i=j\circ g$ is mapped to the natural transformation 
	\begin{equation*}
		\Cat{P}(\Dbl{D})(f_1)\odot\cdots\odot \Cat{P}(\Dbl{D})(f_n)\Rightarrow \Cat{P}(\Dbl{D})(g)(j_*(-),i_*(-))
	\end{equation*}
	that takes in a sequence of towers of composable multicells, composes each tower, and then projects onto a subsequence of the multicells.

	Under this reformulation conditions (1) and (2) for normal VDFs $\L(\Cat{\Delta})\to \Prof$ representing VDCs become equivalent to the existence of such a normal VDF $\Cat{Q}(\Dbl{D}):\Dbl{Q}(\Cat{\Delta}^{op},\mathsf{inrt})^{op_l}\to \Dbl{P}\Cat{rof}$. Condition (3) then states that $\Cat{Q}(\Dbl{D})_0$ preserves pullbacks (i.e.~sends pushouts in $\mathsf{inrt}\subseteq \Cat{\Delta}$ to pullbacks in $\Cat{Cat}$), and condition (4) states that $\Cat{Q}(\Dbl{D})_1$ preserves pullbacks. 
  More explicitly, a colimit of a sequence of spans in $\mathsf{inrt}\subseteq \mathsf{\Delta}$ is canonically of the form
	\[
	\begin{tikzcd}[column sep=small]
		& {[m_0,k_0]} && {[m_1,k_1]} && {[m_{n-1},k_{n-1}]} \\
		{[0,k_0]} && {[m_0,k_1]} && \cdots && {[m_{n-1},k_n]} \\
		\\
		&&& {[k_n]}
		\arrow[from=1-2, to=2-1]
		\arrow[from=1-2, to=2-3]
		\arrow[from=1-4, to=2-3]
		\arrow[from=1-4, to=2-5]
		\arrow[from=1-6, to=2-5]
		\arrow[from=1-6, to=2-7]
		\arrow[from=2-1, to=4-4]
		\arrow[from=2-7, to=4-4]
	\end{tikzcd}\]
	where $0\leq m_0\leq k_0\leq m_1\leq k_1\leq \cdots\leq m_{n-1}\leq k_{n-1}\leq k_n$, while such a colimit in $\mathbb{Q}(\mathsf{\Delta},\mathsf{inrt})_1$ 
  is uniquely determined by a diagram which gives a colimit of a sequence of spans when taking source and target projections. 
  Using the explicit definition of $\Cat{Q}(\Dbl{D})$ from $\Cat{P}(\Dbl{D})$, these colimits are mapped exactly to the limits 
  appearing in condition (4) of the characterization, after we apply $\Cat{Q}(\Dbl{D})$ and use conditions (1) and (2) to re-express the limits in (4). 
  In particular, if $f:[n]\to [m]$ is an arbitrary morphism in $\Cat{\Delta}$, then using conditions (1) and (2), condition (4) becomes equivalent to the natural map
	\begin{equation*}
		\Cat{P}(\Dbl{D})(f)\to \Cat{P}(\Dbl{D})(f_{[0,1]})\times_{\Dbl{D}_0(-,-)}\cdots\times_{\Dbl{D}_0(-,-)}\Cat{P}(\Dbl{D})(f_{[n-1,n]})
	\end{equation*}
	in $\Prof_1$ being an isomorphism. Here $f_{[i,i+1]}:[1]\to [f(i+1)-f(i)]$ is the unique active morphism such that $f\circ \delta^{[i,i+1]}=\delta^{[f(i),f(i+1)]}\circ f_{[i,i+1]}$, 
  and the maps $\Cat{P}(\Dbl{D})(f_i)\leftarrow \Cat{P}(\Dbl{D})(f_{[i,i+1]})\to \Cat{P}(\Dbl{D})(f_{i+1})$ in $\Prof_1$ 
  are given by projecting from a multicell onto its tight source and target, respectively:
	\[\begin{adjustbox}{}
	\begin{tikzcd}
		&& {\mathsf{D}_1} \\
		\\
		{\mathsf{D}_0} && {\mathsf{D}_{f(i+1)-f(i)}} && {\mathsf{D}_0} \\
		\\
		{\mathsf{D}_0} &&&& {\mathsf{D}_0}
		\arrow["{\delta_*^{\{0\}}}"', from=1-3, to=3-1]
		\arrow[""{name=0, anchor=center, inner sep=0}, "{\mathsf{P}(\mathbb{D})(f_{[i,i+1]})}"{description, pos=0.8}, "\shortmid"{marking}, from=1-3, to=3-3]
		\arrow["{\delta^{\{1\}}_*}", from=1-3, to=3-5]
		\arrow[""{name=1, anchor=center, inner sep=0}, "{\mathsf{D}_0(-,-)}"'{inner sep=.8ex}, "\shortmid"{marking}, from=3-1, to=5-1]
		\arrow["{(f_i)_*}", from=3-3, to=5-1]
		\arrow["{(f_{i+1})_*}"', from=3-3, to=5-5]
		\arrow[""{name=2, anchor=center, inner sep=0}, "{\mathsf{D}_0(-,-)}"{inner sep=.8ex}, "\shortmid"{marking}, from=3-5, to=5-5]
		\arrow["s"{description}, between={0.3}{0.8}, Rightarrow, from=0, to=1]
		\arrow["t"{description}, between={0.3}{0.8}, Rightarrow, from=0, to=2]
	\end{tikzcd}
\end{adjustbox}\]
	where $f_i=f\circ \delta^{\{i\}}$. 
\end{cons}

We can also use the equivalence in Equation~\ref{eq:display} to describe a VDF $F:\Dbl{E}\to \Dbl{D}$ as a normal VDF $\Cat{P}_{/\Dbl{D}}(F):\L(\Cat{D})^{op_l}\to \Prof$ given by the following data:

\begin{enumerate}
    \item An object $\prolist{\varphi}$ in $\Cat{D}$ (i.e. a sequence of loose arrows in $\Dbl{D}$) is sent to the fiber category $F^{-1}(\prolist{\varphi})\subseteq \Cat{E}_{|\prolist{\varphi}|}$ whose objects are length $|\prolist{\varphi}|$-sequences of loose arrows in $\Dbl{E}$ lying over $\prolist{\varphi}$, and whose morphisms are sequences of unary multicells between such loose arrows lying over $\text{id}_{\prolist{\varphi}}$.
    \item A map $\beta:(H_1 \cdots H_m)\to (J_1 \cdots J_n)$ lying over $f:[n]\to [m]$ in $\Cat{D}$ is sent to the profunctor $F^{-1}(\beta,f):F^{-1}(\prolist{J})\proto F^{-1}(\prolist{H})$ (i.e. $F^{-1}(\prolist{J})\times F^{-1}(\prolist{H})^{op}\to \Cat{Set}$) that maps an object $(\prolist{I},\prolist{K})$ to the set of maps $\alpha:(K_1,...,K_m)\to (I_1,...,I_n)$ lying over $\beta$, and maps sequences of unary multicells to the set function given by pasting the unary multicells above and below the sequences of multicells in the image of $(\prolist{I},\prolist{K})$.
\end{enumerate}
This perspective will be beneficial for describing the exponentiable VDFs and exponentiable morphisms in more general contexts.

\subsection{Exponentiability in terms of Operator Categories}

In the operator category language, the exponentiable VDCs are precisely the active VDCs defined below:

\begin{defn}[Active VDC]\label{defn:active}
	A VDC $\Dbl{D}$ is said to be \emph{active} if and only if for every active morphism $f:[n]\to [m]$, the laxator
    \[\begin{adjustbox}{}
    \begin{tikzcd}
	{\mathsf{P}(\mathbb{D})_1} && {\mathsf{P}(\mathbb{D})_n} && {\mathsf{P}(\mathbb{D})_m} \\
	&& {\ell_{a_n,f}} \\
	{\mathsf{P}(\mathbb{D})_1} &&&& {\mathsf{P}(\mathbb{D})_m}
	\arrow["{\mathsf{P}(\mathbb{D})(a_n)}"{inner sep=.8ex}, "\shortmid"{marking}, from=1-1, to=1-3]
	\arrow[equals, from=1-1, to=3-1]
	\arrow["{\mathsf{P}(\mathbb{D})(f)}"{inner sep=.8ex}, "\shortmid"{marking}, from=1-3, to=1-5]
	\arrow[equals, from=1-5, to=3-5]
	\arrow["{\mathsf{P}(\mathbb{D})(a_m)}"'{inner sep=.8ex}, "\shortmid"{marking}, from=3-1, to=3-5]
\end{tikzcd}
\end{adjustbox}\]
    for the associated VDF $\mathsf{P}(\Dbl{D}):\L(\Cat{\Delta})\to \Dbl{P}\Cat{rof}$ is an opcartesian cell,
    where as above $a_n:[1]\to [n]$ denotes the unique active morphism of that signature in $\Cat{\Delta}.$
\qed
\end{defn}

In terms of coends a VDC $\Dbl{D}$ is active if and only if for every loose arrow $\psi$ and every $2$-fold sequence of loose arrows ${\prolist{\varphi}}$, the composition map
\begin{equation*}
	\int^{\prolist{\chi}\in \Cat{fc}_{|{\prolist{\varphi}}|_1}(\Dbl{D}_1)}\MCell (\Dbl{D})(\prolist{\chi},\psi)\times \MCell (\Dbl{D})({\prolist{\varphi}},\prolist{\chi})\xrightarrow{\circ} \MCell (\Dbl{D})({\prolist{\varphi}},\psi)
\end{equation*}
is an isomorphism. But by Corollary~\ref{cor:decompSimp} this is equivalent to $\Dbl{D}$ having decomposable multicells, and hence by Theorem~\ref{thm:ExpChar} it is equivalent to $\Dbl{D}$ being exponentiable. Similarly, we can encode exponentiable VDFs in the same language using the following notion of active VDFs:

\begin{defn}[Active VDF]\label{defn:activeVDF}
	A VDF $F:\Dbl{D}\to \Dbl{E}$ is said to be \emph{active} if and only if for every active morphism $(f^{op},\prolist{\chi}):([m],{\prolist{\varphi}})\to ([n],\prolist{\psi})$ in $\Cat{E}$, and every active morphism $(a_n^{op},\alpha):([n],\prolist{\psi})\to ([1],\lambda)$, the laxator
    \[\begin{adjustbox}{}
    \begin{tikzcd}
	{\Cat{P}_{/\Dbl{E}}(F)_{([1],\lambda)}} && {\Cat{P}_{/\Dbl{E}}(F)_{([n],\prolist{\psi})}} && {\Cat{P}_{/\Dbl{E}}(F)_{([m],{\prolist{\varphi}})}} \\
	&& {\ell_{a_n,f}} \\
	{\Cat{P}_{/\Dbl{E}}(F)_{([1],\lambda)}} &&&& {\Cat{P}_{/\Dbl{E}}(F)_{([m],{\prolist{\varphi}})}}
	\arrow["{\Cat{P}_{/\Dbl{E}}(F)(a_n^{op},\alpha)}"{inner sep=.8ex}, "\shortmid"{marking}, from=1-1, to=1-3]
	\arrow[equals, from=1-1, to=3-1]
	\arrow["{\Cat{P}_{/\Dbl{E}}(F)(f^{op},\prolist{\chi})}"{inner sep=.8ex}, "\shortmid"{marking}, from=1-3, to=1-5]
	\arrow[equals, from=1-5, to=3-5]
	\arrow["{\Cat{P}_{/\Dbl{E}}(F)\left(a_m^{op},\frac{\prolist{\chi}}{\alpha}\right)}"'{inner sep=.8ex}, "\shortmid"{marking}, from=3-1, to=3-5]
\end{tikzcd}
\end{adjustbox}\]
    for the associated VDF $\Cat{P}_{/\Dbl{E}}(F):\L(\Cat{E})^{op_l}\to \Dbl{P}\Cat{rof}$ is an opcartesian cell, and if $n=m=0$, then for all morphisms $(\text{id}_{[0]}^{op},a):([0],x)\to ([0],y)$ and $(\text{id}_{[0]}^{op},b):([0],y)\to ([0],z)$, the laxator
	\[\begin{adjustbox}{}
    \begin{tikzcd}
	{\Cat{P}_{/\Dbl{E}}(F)_{([0],z)}} && {\Cat{P}_{/\Dbl{E}}(F)_{([0],y)}} && {\Cat{P}_{/\Dbl{E}}(F)_{([0],x)}} \\
	&& {\ell_{a_n,f}} \\
	{\Cat{P}_{/\Dbl{E}}(F)_{([0],z)}} &&&& {\Cat{P}_{/\Dbl{E}}(F)_{([0],x)}}
	\arrow["{\Cat{P}_{/\Dbl{E}}(F)(\text{id}_{[0]}^{op},b)}"{inner sep=.8ex}, "\shortmid"{marking}, from=1-1, to=1-3]
	\arrow[equals, from=1-1, to=3-1]
	\arrow["{\Cat{P}_{/\Dbl{E}}(F)(\text{id}_{[0]}^{op},a)}"{inner sep=.8ex}, "\shortmid"{marking}, from=1-3, to=1-5]
	\arrow[equals, from=1-5, to=3-5]
	\arrow["{\Cat{P}_{/\Dbl{E}}(F)(\text{id}_{[0]}^{op},b\circ a)}"'{inner sep=.8ex}, "\shortmid"{marking}, from=3-1, to=3-5]
\end{tikzcd}
\end{adjustbox}\]
	is also an opcartesian cell.
\qed
\end{defn}

From Corollary~\ref{cor:ProRepImpliesDecomp}, 
to check that a VDC is active we need only check the laxator 
\[\begin{adjustbox}{}
    \begin{tikzcd}
	{\mathsf{P}(\mathbb{D})_1} && {\mathsf{P}(\mathbb{D})_n} && {\mathsf{P}(\mathbb{D})_m} \\
	&& {\ell_{a_n,f}} \\
	{\mathsf{P}(\mathbb{D})_1} &&&& {\mathsf{P}(\mathbb{D})_m}
	\arrow["{\mathsf{P}(\mathbb{D})(a_n)}"{inner sep=.8ex}, "\shortmid"{marking}, from=1-1, to=1-3]
	\arrow[equals, from=1-1, to=3-1]
	\arrow["{\mathsf{P}(\mathbb{D})(f)}"{inner sep=.8ex}, "\shortmid"{marking}, from=1-3, to=1-5]
	\arrow[equals, from=1-5, to=3-5]
	\arrow["{\mathsf{P}(\mathbb{D})(a_m)}"'{inner sep=.8ex}, "\shortmid"{marking}, from=3-1, to=3-5]
\end{tikzcd}
\end{adjustbox}\]
is opcartesian when $f$ is either an inner face map or a degeneracy map, and that for any active $f:[n]\to [m]$ and any decomposition $f_k\circ \cdots f_1=f$ where each $f_i$ is either an inner face map or a degeneracy map, the partial composite map 
\[\begin{adjustbox}{}
\begin{tikzcd}
	{\mathsf{P}(\mathbb{D})_1} && {\mathsf{P}(\mathbb{D})_n} & \cdots & {\mathsf{P}(\mathbb{D})_m} \\
	& {\text{id}} && {\ell_{f_1,...,f_k}} \\
	{\mathsf{P}(\mathbb{D})_1} && {\mathsf{P}(\mathbb{D})_n} && {\mathsf{P}(\mathbb{D})_m} \\
	&& {\mathsf{opcart}} \\
	{\mathsf{P}(\mathbb{D})_1} &&&& {\mathsf{P}(\mathbb{D})_m}
	\arrow["{{\mathsf{P}(\mathbb{D})(a_n)}}"{inner sep=.8ex}, "\shortmid"{marking}, from=1-1, to=1-3]
	\arrow[equals, from=1-1, to=3-1]
	\arrow["{{\mathsf{P}(\mathbb{D})(f_1)}}"{inner sep=.8ex}, "\shortmid"{marking}, from=1-3, to=1-4]
	\arrow[equals, from=1-3, to=3-3]
	\arrow["{{\mathsf{P}(\mathbb{D})(f_k)}}"{inner sep=.8ex}, "\shortmid"{marking}, from=1-4, to=1-5]
	\arrow[equals, from=1-5, to=3-5]
	\arrow["{{\mathsf{P}(\mathbb{D})(a_n)}}"{inner sep=.8ex}, "\shortmid"{marking}, from=3-1, to=3-3]
	\arrow[equals, from=3-1, to=5-1]
	\arrow["{{\mathsf{P}(\mathbb{D})(f)}}"{inner sep=.8ex}, "\shortmid"{marking}, from=3-3, to=3-5]
	\arrow[equals, from=3-5, to=5-5]
	\arrow["{{\mathsf{P}(\mathbb{D})(a_n)}\odot {\mathsf{P}(\mathbb{D})(f)}}"'{inner sep=.8ex}, "\shortmid"{marking}, from=5-1, to=5-5]
\end{tikzcd}
\end{adjustbox}\]
has surjective component maps. Using this observation and the normal VDF perspective $\Cat{Q}(\Dbl{D}):\Dbl{Q}(\Cat{\Delta}^{op},\mathsf{inrt})^{op_l}\to \Dbl{P}\Cat{rof}$ for a VDC $\Dbl{D}$, we can identify an exponentiable VDC with the following data:
\begin{enumerate}
	\item A graph of categories $s,t:\Cat{D}_1\rightrightarrows \Cat{D}_0$, from which we can define $\Cat{D}_n =\Cat{D}_1\times_{\Cat{D}_0}\cdots\times_{\Cat{D}_0}\Cat{D}_1$.
	\item Two graphs of profunctors
\[\begin{tikzcd}
	&& {\mathsf{D}_1} &&&& {\mathsf{D}_1} && \\
	\\
	{\mathsf{D}_0} && {\mathsf{D}_{0}} && {\mathsf{D}_0} && {\mathsf{D}_{2}} && {\mathsf{D}_0} \\
	&&&&& {\mathsf{D}_1} && {\mathsf{D}_1} \\
	{\mathsf{D}_0} &&&& {\mathsf{D}_0} &&&& {\mathsf{D}_0}
	\arrow["s"', from=1-3, to=3-1]
	\arrow[""{name=0, anchor=center, inner sep=0}, "{{{\mathsf{P}(\mathbb{D})(a_0)}}}"{description, pos=0.8}, "\shortmid"{marking}, from=1-3, to=3-3]
	\arrow["t", from=1-3, to=3-5]
	\arrow["s"', from=1-7, to=3-5]
	\arrow[""{name=1, anchor=center, inner sep=0}, "{{{\mathsf{P}(\mathbb{D})(a_2)}}}"{description, pos=0.8}, "\shortmid"{marking}, from=1-7, to=3-7]
	\arrow["t", from=1-7, to=3-9]
	\arrow[""{name=2, anchor=center, inner sep=0}, "{{{\mathsf{D}_0(-,-)}}}"{inner sep=.8ex}, "\shortmid"{marking}, from=3-1, to=5-1]
	\arrow[equals, from=3-3, to=5-1]
	\arrow[equals, from=3-3, to=5-5]
	\arrow[""{name=3, anchor=center, inner sep=0}, "{{{\mathsf{D}_0(-,-)}}}"{description, pos=0.7}, "\shortmid"{marking}, from=3-5, to=5-5]
	\arrow["{{\pi_1}}", from=3-7, to=4-6]
	\arrow["{{\pi_2}}"', from=3-7, to=4-8]
	\arrow[""{name=4, anchor=center, inner sep=0}, "{{{\mathsf{D}_0(-,-)}}}"'{pos=0.4, inner sep=.8ex}, "\shortmid"{marking}, from=3-9, to=5-9]
	\arrow["s", from=4-6, to=5-5]
	\arrow["t"', from=4-8, to=5-9]
	\arrow["s"{description}, between={0.3}{0.8}, Rightarrow, from=0, to=2]
	\arrow["t"{description}, between={0.3}{0.8}, Rightarrow, from=0, to=3]
	\arrow["s"{description}, between={0.3}{0.8}, Rightarrow, from=1, to=3]
	\arrow["t"{description}, between={0.3}{0.8}, Rightarrow, from=1, to=4]
\end{tikzcd}\]
	For $f:[n]\to [m]$ a degeneracy or inner face map, we define $\mathsf{P}(\Dbl{D})(f):=\mathsf{P}(\Dbl{D})(a_{k_1})\times_{\Dbl{D}_0(-,-)}\cdots\times_{\Dbl{D}_0(-,-)}\mathsf{P}(\Dbl{D})(a_{k_n})$ where $\prolist{k}=\prolist{1}\pm \delta_i$, so that all but one of the profunctors is $\Dbl{D}_1(-,-)$, and the remaining profunctor is either $\mathsf{P}(\Dbl{D})(a_0)$ or $\mathsf{P}(\Dbl{D})(a_2)$. 
  Additionally, for each binary-nullary tree $T$, which we can encode by a unique sequence of maps $f_k\circ \cdots f_1:[1]\to [m]$ with each $f_i$ either a degeneracy or inner face map, we define $\mathsf{P}(\Dbl{D})(T):=\mathsf{P}(\Dbl{D})(f_1)\odot \cdots \odot \mathsf{P}(\Dbl{D})(f_k)$.
	\item For any two binary-nullary trees $T$ and $T'$ with the same number of leaves, a natural isomorphism $\alpha_{T,T'}:\mathsf{P}(\Dbl{D})(T)\to \mathsf{P}(\Dbl{D})(T')$, satisfying the cocycle identities
	\begin{equation*}
		\alpha_{T,T}=\text{id},\;\;\alpha_{T',T''}\circ \alpha_{T,T'}=\alpha_{T,T''}
	\end{equation*}
\end{enumerate}

This presentation exhibits exponentiable VDCs as a kind of unbiased pro-double categories. 
Namely, we need only specify the data for nullary, unary, and binary multicells, as expected for a 
pro-double concept, but we need to provide a compatible family of associators for all possible 
towers of multicells that can be built from these. In particular, counterexamples like that 
in the proof of Proposition~\ref{prop:counterExToProDouble} will not satisfy this 
corrected notion of pro-double category, since we have explicitly added associators 
for trees with four and more leaves rather than attempting to build them out of the basic associators,
which as we showed there is impossible in general.

\section{Yoneda Lemma for monoids}\label{sec:YonedaForMon}

As discussed in Section~\ref{sec:OtherChar}, Lemma~\ref{lem:Yoneda} is a special case of a more general result for monoids in a VDC. For the remainder of this section we fix a VDC $\Dbl{D}$, together with a monoid $x=(x_0,x_1:x_0\proto x_0,\eta,\mu)$, where $\eta$ is the unit nullary multicell and $\mu$ is the multiplication binary multicell. We will phrase the Yoneda lemma for monoids in terms of left modules over the monoid, while dualizing would give an analogous statement for right modules.

\begin{defn}[Left Modules]\label{defn:leftMod}
	A left module over the monoid $x$ consists of an object $y \in \Dbl{D}$, a loose arrow $\varphi:x_0\proto y$, and a globular multicell $\chi:\frac{(x_1,\varphi)}{\varphi}$ compatible with the multiplication and unit multicells in the sense of Definition~\ref{defn:Modules}. A map of left modules $(x,\varphi,y)\to (z,\psi,w)$ consists of a map of monoids $x\to z$ together with a cell 
	\[\begin{adjustbox}{}
	\begin{tikzcd}
	{x_0} & y \\
	{z_0} & w
	\arrow[""{name=0, anchor=center, inner sep=0}, "\varphi"{inner sep=.8ex}, "\shortmid"{marking}, from=1-1, to=1-2]
	\arrow["{{a_0}}"', from=1-1, to=2-1]
	\arrow["b", from=1-2, to=2-2]
	\arrow[""{name=1, anchor=center, inner sep=0}, "\psi"'{inner sep=.8ex}, "\shortmid"{marking}, from=2-1, to=2-2]
	\arrow["\alpha"{description}, draw=none, from=0, to=1]
\end{tikzcd}
\end{adjustbox}\]
	which is compatible with the module multicells.
\qed
\end{defn}

Note that if $y\in \Dbl{D}$ admits a loose unit, a left $x$-module structure on $y$ is equivalent to a bimodule where the right action is the canonical one induced by a choice of loose unit. 
Now, the role of representable presheaves in the 1-categorical Yoneda lemma will be replaced by restrictions of the monoid $x$ along a tight morphism, as in the following construction.

\begin{cons}[Restricting Bimodules]\label{cons:restrBimoduleToLeftModule}
	Let $y=(y_0,y_1:y_0\proto y_0,\eta_y,\mu_y)$ and $z=(z_0,z_1:z_0\proto z_0,\eta_z,\mu_z)$ be monoids in $\Dbl{D}$, and let $(\chi_y,\varphi:y_0\proto z_0,\chi_z)$ be a bimodule between them. Additionally, let $w$ be an object in $\Dbl{D}$, and let $a:w\to z_0$ be a tight morphism such that the restriction $\varphi(1,a)$ exists in $\Dbl{D}$. Then $\varphi(1,a)$ has the natural structure of a left $y$-module using the action $\chi_{y,a}$ obtained by factoring through the cartesian cell of the restriction:
	\[
	\begin{tikzcd}
	{y_0} & {y_0} & w && {y_0} & {y_0} & w \\
	{y_0} & {y_0} & {z_0} & {=} & {y_0} && w \\
	{y_0} && {z_0} && {y_0} && {z_0}
	\arrow[""{name=0, anchor=center, inner sep=0}, "{{y_1}}"{inner sep=.8ex}, "\shortmid"{marking}, from=1-1, to=1-2]
	\arrow[equals, from=1-1, to=2-1]
	\arrow[""{name=1, anchor=center, inner sep=0}, "{{\varphi(1,a)}}"{inner sep=.8ex}, "\shortmid"{marking}, from=1-2, to=1-3]
	\arrow[equals, from=1-2, to=2-2]
	\arrow["a", from=1-3, to=2-3]
	\arrow["{{y_1}}"{inner sep=.8ex}, "\shortmid"{marking}, from=1-5, to=1-6]
	\arrow[equals, from=1-5, to=2-5]
	\arrow["{{\varphi(1,a)}}"{inner sep=.8ex}, "\shortmid"{marking}, from=1-6, to=1-7]
	\arrow[equals, from=1-7, to=2-7]
	\arrow[""{name=2, anchor=center, inner sep=0}, "{{y_1}}"'{inner sep=.8ex}, "\shortmid"{marking}, from=2-1, to=2-2]
	\arrow[equals, from=2-1, to=3-1]
	\arrow[""{name=3, anchor=center, inner sep=0}, "\varphi"'{inner sep=.8ex}, "\shortmid"{marking}, from=2-2, to=2-3]
	\arrow[equals, from=2-3, to=3-3]
	\arrow[""{name=4, anchor=center, inner sep=0}, "{{\varphi(1,a)}}"'{inner sep=.8ex}, "\shortmid"{marking}, from=2-5, to=2-7]
	\arrow[equals, from=2-5, to=3-5]
	\arrow["a", from=2-7, to=3-7]
	\arrow[""{name=5, anchor=center, inner sep=0}, "\varphi"'{inner sep=.8ex}, "\shortmid"{marking}, from=3-1, to=3-3]
	\arrow[""{name=6, anchor=center, inner sep=0}, "\varphi"'{inner sep=.8ex}, "\shortmid"{marking}, from=3-5, to=3-7]
	\arrow["{\text{id}_{y_1}}"{description}, draw=none, from=0, to=2]
	\arrow["{\mathsf{cart}}"{description}, draw=none, from=1, to=3]
	\arrow["{\exists!\chi_{y,a}}"{description}, draw=none, from=1-6, to=4]
	\arrow["{\chi_y}"{description}, draw=none, from=2-2, to=5]
	\arrow["{\mathsf{cart}}"{description}, draw=none, from=4, to=6]
\end{tikzcd}
\]
	The compatibility of $\chi_{y,a}$ with the unit and multiplication multicells for $y$ follows from the compatibility of $\chi_y$ and the uniqueness of factorizations through the cartesian cell. Additionally, if $y=z$, and $\varphi=y_1:y_0\proto y_0$, then we also obtain a unit type nullary multicell:
	\[
	\begin{tikzcd}
	& w &&&& w \\
	& {y_0} && {=} & {y_0} && w \\
	{y_0} && {y_0} && {y_0} && {y_0}
	\arrow["a"', from=1-2, to=2-2]
	\arrow["a"', from=1-6, to=2-5]
	\arrow[equals, from=1-6, to=2-7]
	\arrow[equals, from=2-2, to=3-1]
	\arrow[equals, from=2-2, to=3-3]
	\arrow[""{name=0, anchor=center, inner sep=0}, "{{y_1(1,a)}}"'{inner sep=.8ex}, "\shortmid"{marking}, from=2-5, to=2-7]
	\arrow[equals, from=2-5, to=3-5]
	\arrow["a", from=2-7, to=3-7]
	\arrow[""{name=1, anchor=center, inner sep=0}, "{{y_1}}"'{inner sep=.8ex}, "\shortmid"{marking}, from=3-1, to=3-3]
	\arrow[""{name=2, anchor=center, inner sep=0}, "{{y_1}}"'{inner sep=.8ex}, "\shortmid"{marking}, from=3-5, to=3-7]
	\arrow["{\exists!\eta_a}"{description}, draw=none, from=1-6, to=0]
	\arrow["{\eta_y}"{description}, draw=none, from=2-2, to=1]
	\arrow["{\mathsf{cart}}"{description}, draw=none, from=0, to=2]
\end{tikzcd}\]
	which will be important for our generalization of the Yoneda lemma. To see how this nullary multicell behaves like a unit observe that we have the pasting equalities
	\[
	\begin{tikzcd}[column sep=small]
	& {y_0} && w &&&& {y_0} && w \\
	{y_0} && {y_0} && w & {=} & {y_0} && {y_0} && w \\
	{y_0} &&&& w && {y_0} && {y_1} && {y_0} \\
	{y_0} &&&& {y_0} && {y_0} &&&& {y_0}
	\arrow[""{name=0, anchor=center, inner sep=0}, "{y_1(1,a)}"{inner sep=.8ex}, "\shortmid"{marking}, from=1-2, to=1-4]
	\arrow[equals, from=1-2, to=2-1]
	\arrow["a"', from=1-4, to=2-3]
	\arrow[equals, from=1-4, to=2-5]
	\arrow[""{name=1, anchor=center, inner sep=0}, "{y_1(1,a)}"{inner sep=.8ex}, "\shortmid"{marking}, from=1-8, to=1-10]
	\arrow[equals, from=1-8, to=2-7]
	\arrow["a"', from=1-10, to=2-9]
	\arrow[equals, from=1-10, to=2-11]
	\arrow[""{name=2, anchor=center, inner sep=0}, "{y_1}"'{inner sep=.8ex}, "\shortmid"{marking}, from=2-1, to=2-3]
	\arrow[equals, from=2-1, to=3-1]
	\arrow[""{name=3, anchor=center, inner sep=0}, "{{y_1(1,a)}}"'{inner sep=.8ex}, "\shortmid"{marking}, from=2-3, to=2-5]
	\arrow[equals, from=2-5, to=3-5]
	\arrow[""{name=4, anchor=center, inner sep=0}, "{y_1}"'{inner sep=.8ex}, "\shortmid"{marking}, from=2-7, to=2-9]
	\arrow[equals, from=2-7, to=3-7]
	\arrow[""{name=5, anchor=center, inner sep=0}, "{{y_1(1,a)}}"'{inner sep=.8ex}, "\shortmid"{marking}, from=2-9, to=2-11]
	\arrow[equals, from=2-9, to=3-9]
	\arrow["a", from=2-11, to=3-11]
	\arrow[""{name=6, anchor=center, inner sep=0}, "{y_1(1,a)}"'{inner sep=.8ex}, "\shortmid"{marking}, from=3-1, to=3-5]
	\arrow[equals, from=3-1, to=4-1]
	\arrow["a", from=3-5, to=4-5]
	\arrow[""{name=7, anchor=center, inner sep=0}, "{y_1}"'{inner sep=.8ex}, "\shortmid"{marking}, from=3-7, to=3-9]
	\arrow[equals, from=3-7, to=4-7]
	\arrow[""{name=8, anchor=center, inner sep=0}, "{y_1}"'{inner sep=.8ex}, "\shortmid"{marking}, from=3-9, to=3-11]
	\arrow[equals, from=3-11, to=4-11]
	\arrow[""{name=9, anchor=center, inner sep=0}, "{y_1}"'{inner sep=.8ex}, "\shortmid"{marking}, from=4-1, to=4-5]
	\arrow[""{name=10, anchor=center, inner sep=0}, "{y_1}"'{inner sep=.8ex}, "\shortmid"{marking}, from=4-7, to=4-11]
	\arrow["{\mathsf{cart}}"{description}, draw=none, from=0, to=2]
	\arrow["{\eta_a}"{description}, draw=none, from=1-4, to=3]
	\arrow["{\mathsf{cart}}"{description}, draw=none, from=1, to=4]
	\arrow["{\eta_a}"{description}, draw=none, from=1-10, to=5]
	\arrow["{\chi_{y,a}}"{description}, draw=none, from=2-3, to=6]
	\arrow["{\text{id}_{y_1}}"{description}, draw=none, from=4, to=7]
	\arrow["{\mathsf{cart}}"{description}, draw=none, from=5, to=8]
	\arrow["{\mathsf{cart}}"{description}, draw=none, from=6, to=9]
	\arrow["{\mu_y}"{description}, draw=none, from=3-9, to=10]
\end{tikzcd}
\]
	and
	\[\begin{adjustbox}{}
	\begin{tikzcd}[column sep=small]
	& {y_0} && w &&&& {y_0} && w \\
	{y_0} && {y_0} && w & {=} & {y_0} &&& {y_0} \\
	{y_0} && {y_1} && {y_0} && {y_0} && {y_1} && {y_0} \\
	{y_0} &&&& {y_0} && {y_0} &&&& {y_0}
	\arrow[""{name=0, anchor=center, inner sep=0}, "{y_1(1,a)}"{inner sep=.8ex}, "\shortmid"{marking}, from=1-2, to=1-4]
	\arrow[equals, from=1-2, to=2-1]
	\arrow["a"', from=1-4, to=2-3]
	\arrow[equals, from=1-4, to=2-5]
	\arrow[""{name=1, anchor=center, inner sep=0}, "{y_1(1,a)}"{inner sep=.8ex}, "\shortmid"{marking}, from=1-8, to=1-10]
	\arrow[equals, from=1-8, to=2-7]
	\arrow["a"', from=1-10, to=2-10]
	\arrow[""{name=2, anchor=center, inner sep=0}, "{y_1}"'{inner sep=.8ex}, "\shortmid"{marking}, from=2-1, to=2-3]
	\arrow[equals, from=2-1, to=3-1]
	\arrow[""{name=3, anchor=center, inner sep=0}, "{{y_1(1,a)}}"'{inner sep=.8ex}, "\shortmid"{marking}, from=2-3, to=2-5]
	\arrow[equals, from=2-3, to=3-3]
	\arrow["a", from=2-5, to=3-5]
	\arrow[""{name=4, anchor=center, inner sep=0}, "{y_1}"'{inner sep=.8ex}, "\shortmid"{marking}, from=2-7, to=2-10]
	\arrow[equals, from=2-7, to=3-7]
	\arrow[equals, from=2-10, to=3-9]
	\arrow[equals, from=2-10, to=3-11]
	\arrow[""{name=5, anchor=center, inner sep=0}, "{y_1}"'{inner sep=.8ex}, "\shortmid"{marking}, from=3-1, to=3-3]
	\arrow[equals, from=3-1, to=4-1]
	\arrow[""{name=6, anchor=center, inner sep=0}, "{y_1}"'{inner sep=.8ex}, "\shortmid"{marking}, from=3-3, to=3-5]
	\arrow[equals, from=3-5, to=4-5]
	\arrow[""{name=7, anchor=center, inner sep=0}, "{y_1}"'{inner sep=.8ex}, "\shortmid"{marking}, from=3-7, to=3-9]
	\arrow[equals, from=3-7, to=4-7]
	\arrow[""{name=8, anchor=center, inner sep=0}, "{y_1}"'{inner sep=.8ex}, "\shortmid"{marking}, from=3-9, to=3-11]
	\arrow[equals, from=3-11, to=4-11]
	\arrow[""{name=9, anchor=center, inner sep=0}, "{y_1}"'{inner sep=.8ex}, "\shortmid"{marking}, from=4-1, to=4-5]
	\arrow[""{name=10, anchor=center, inner sep=0}, "{y_1}"'{inner sep=.8ex}, "\shortmid"{marking}, from=4-7, to=4-11]
	\arrow["{\mathsf{cart}}"{description}, draw=none, from=0, to=2]
	\arrow["{\eta_a}"{description}, draw=none, from=1-4, to=3]
	\arrow["{\mathsf{cart}}"{description}, draw=none, from=1, to=4]
	\arrow["{\text{id}_{y_1}}"{description}, draw=none, from=2, to=5]
	\arrow["{\mathsf{cart}}"{description}, draw=none, from=3, to=6]
	\arrow["{\text{id}_{y_1}}"{description}, draw=none, from=4, to=7]
	\arrow["{\eta_y}"{description}, draw=none, from=2-10, to=8]
	\arrow["{\mu_y}"{description}, draw=none, from=3-3, to=9]
	\arrow["{\mu_y}"{description}, draw=none, from=3-9, to=10]
\end{tikzcd}
\end{adjustbox}\]
	where the right hand side reduces simply to the cartesian cell, so by uniqueness we have the equality 
	\begin{equation}\label{eq:restrUnit}
		\begin{tikzcd}
	& {y_0} && w &&& {y_0} && w \\
	{y_0} && {y_0} && w & {=} && {\text{id}_{y_1(1,a)}} \\
	{y_0} &&&& w && {y_0} && w
	\arrow[""{name=0, anchor=center, inner sep=0}, "{y_1(1,a)}"{inner sep=.8ex}, "\shortmid"{marking}, from=1-2, to=1-4]
	\arrow[equals, from=1-2, to=2-1]
	\arrow["a"', from=1-4, to=2-3]
	\arrow[equals, from=1-4, to=2-5]
	\arrow["{y_1(1,a)}"{inner sep=.8ex}, "\shortmid"{marking}, from=1-7, to=1-9]
	\arrow[equals, from=1-7, to=3-7]
	\arrow[equals, from=1-9, to=3-9]
	\arrow[""{name=1, anchor=center, inner sep=0}, "{y_1}"'{inner sep=.8ex}, "\shortmid"{marking}, from=2-1, to=2-3]
	\arrow[equals, from=2-1, to=3-1]
	\arrow[""{name=2, anchor=center, inner sep=0}, "{{y_1(1,a)}}"'{inner sep=.8ex}, "\shortmid"{marking}, from=2-3, to=2-5]
	\arrow[equals, from=2-5, to=3-5]
	\arrow[""{name=3, anchor=center, inner sep=0}, "{y_1(1,a)}"'{inner sep=.8ex}, "\shortmid"{marking}, from=3-1, to=3-5]
	\arrow["{y_1(1,a)}"'{inner sep=.8ex}, "\shortmid"{marking}, from=3-7, to=3-9]
	\arrow["{\mathsf{cart}}"{description}, draw=none, from=0, to=1]
	\arrow["{\eta_a}"{description}, draw=none, from=1-4, to=2]
	\arrow["{\chi_{y,a}}"{description}, draw=none, from=2-3, to=3]
\end{tikzcd}
	\end{equation}
\end{cons}

We can now state and prove the version of the Yoneda lemma for left modules in $\Dbl{D}$.

\begin{thm}[Yoneda Lemma for Left Modules in a VDC]\label{thm:YonedaForLeftModules}
	Let $\Dbl{D}$ be a VDC and let $x=(x_0,x_1:x_0\proto x_0,\eta,\mu)$ be a monoid in $\Dbl{D}$. 
  If $(\varphi:x_0\proto y,\chi_x)$ is a left $x$-module, then for any tight arrow 
  $a:z\to x_0$ for which the restriction $x_1(1,a)$ exists, we have a natural bijection 
  between maps of left modules below left and nullary multicells below right:
	\[\begin{adjustbox}{}
	\begin{tikzcd}
	{x_0} & z && z \\
	{x_0} & y & {x_0} && y
	\arrow[""{name=0, anchor=center, inner sep=0}, "{{x_1(1,a)}}"{inner sep=.8ex}, "\shortmid"{marking}, from=1-1, to=1-2]
	\arrow[equals, from=1-1, to=2-1]
	\arrow[""{name=1, anchor=center, inner sep=0}, "b", from=1-2, to=2-2]
	\arrow[""{name=2, anchor=center, inner sep=0}, "a"', from=1-4, to=2-3]
	\arrow["b", from=1-4, to=2-5]
	\arrow[""{name=3, anchor=center, inner sep=0}, "\varphi"'{inner sep=.8ex}, "\shortmid"{marking}, from=2-1, to=2-2]
	\arrow[""{name=4, anchor=center, inner sep=0}, "\varphi"'{inner sep=.8ex}, "\shortmid"{marking}, from=2-3, to=2-5]
	\arrow["\alpha"{description}, draw=none, from=0, to=3]
	\arrow[between={0.4}{0.6}, squiggly, tail reversed, from=1, to=2]
	\arrow["{\alpha(\eta_{a,b})}"{description}, draw=none, from=1-4, to=4]
\end{tikzcd}
\end{adjustbox}\]
\end{thm}
\begin{proof}
	For the forward direction we send a map of left modules $\alpha$ to the composite multicell 
	\[\begin{adjustbox}{}
	\begin{tikzcd}
	& z &&&& z \\
	&&& {:=} & {x_0} && z \\
	{x_0} && y && {x_0} && y
	\arrow["a"', from=1-2, to=3-1]
	\arrow["b", from=1-2, to=3-3]
	\arrow["a"', from=1-6, to=2-5]
	\arrow[equals, from=1-6, to=2-7]
	\arrow[""{name=0, anchor=center, inner sep=0}, "{{x_1(1,a)}}"'{inner sep=.8ex}, "\shortmid"{marking}, from=2-5, to=2-7]
	\arrow[equals, from=2-5, to=3-5]
	\arrow["b", from=2-7, to=3-7]
	\arrow[""{name=1, anchor=center, inner sep=0}, "\varphi"'{inner sep=.8ex}, "\shortmid"{marking}, from=3-1, to=3-3]
	\arrow[""{name=2, anchor=center, inner sep=0}, "\varphi"'{inner sep=.8ex}, "\shortmid"{marking}, from=3-5, to=3-7]
	\arrow["{\alpha(\eta_{a,b})}"{description}, draw=none, from=1-2, to=1]
	\arrow["{\eta_a}"{description}, draw=none, from=1-6, to=0]
	\arrow["\alpha"{description}, draw=none, from=0, to=2]
\end{tikzcd}
\end{adjustbox}\]
	Using Equation~\ref{eq:restrUnit} and the fact that $\alpha$ is a map of left modules, we can observe that $\alpha$ is uniquely determined by $\alpha(\eta_{a,b})$, since these properties give the equalities
	\[
	\begin{adjustbox}{}
		\begin{tikzcd}
	& {x_0} & z &&&& {x_0} & z \\
	{x_0} & {x_0} && z && {x_0} & {x_0} && z \\
	{x_0} &&& z & {=} & {x_0} & {x_0} && y \\
	{x_0} &&& y && {x_0} &&& y
	\arrow[""{name=0, anchor=center, inner sep=0}, "{{x_1(1,a)}}"{inner sep=.8ex}, "\shortmid"{marking}, from=1-2, to=1-3]
	\arrow[equals, from=1-2, to=2-1]
	\arrow["a"{description}, from=1-3, to=2-2]
	\arrow[equals, from=1-3, to=2-4]
	\arrow[""{name=1, anchor=center, inner sep=0}, "{{x_1(1,a)}}"{inner sep=.8ex}, "\shortmid"{marking}, from=1-7, to=1-8]
	\arrow[equals, from=1-7, to=2-6]
	\arrow["a"{description}, from=1-8, to=2-7]
	\arrow[equals, from=1-8, to=2-9]
	\arrow[""{name=2, anchor=center, inner sep=0}, "{{x_1}}"{inner sep=.8ex}, "\shortmid"{marking}, from=2-1, to=2-2]
	\arrow[equals, from=2-1, to=3-1]
	\arrow[""{name=3, anchor=center, inner sep=0}, "{{x_1(1,a)}}"'{inner sep=.8ex}, "\shortmid"{marking}, from=2-2, to=2-4]
	\arrow[equals, from=2-4, to=3-4]
	\arrow[""{name=4, anchor=center, inner sep=0}, "{{x_1}}"{inner sep=.8ex}, "\shortmid"{marking}, from=2-6, to=2-7]
	\arrow[""{name=5, anchor=center, inner sep=0}, "{{x_1(1,a)}}"'{inner sep=.8ex}, "\shortmid"{marking}, from=2-7, to=2-9]
	\arrow[equals, from=2-7, to=3-7]
	\arrow["b", from=2-9, to=3-9]
	\arrow[""{name=6, anchor=center, inner sep=0}, "{{x_1(1,a)}}"'{inner sep=.8ex}, "\shortmid"{marking}, from=3-1, to=3-4]
	\arrow[equals, from=3-1, to=4-1]
	\arrow["b", from=3-4, to=4-4]
	\arrow[equals, from=3-6, to=2-6]
	\arrow[""{name=7, anchor=center, inner sep=0}, "{{x_1}}"{inner sep=.8ex}, "\shortmid"{marking}, from=3-6, to=3-7]
	\arrow[equals, from=3-6, to=4-6]
	\arrow[""{name=8, anchor=center, inner sep=0}, "\varphi"{inner sep=.8ex}, "\shortmid"{marking}, from=3-7, to=3-9]
	\arrow[equals, from=3-9, to=4-9]
	\arrow[""{name=9, anchor=center, inner sep=0}, "\varphi"'{inner sep=.8ex}, "\shortmid"{marking}, from=4-1, to=4-4]
	\arrow[""{name=10, anchor=center, inner sep=0}, "\varphi"'{inner sep=.8ex}, "\shortmid"{marking}, from=4-6, to=4-9]
	\arrow["{{\mathsf{cart}}}"{description}, draw=none, from=0, to=2]
	\arrow["{\eta_a}"{description}, draw=none, from=1-3, to=3]
	\arrow["{{\mathsf{cart}}}"{description}, draw=none, from=1, to=4]
	\arrow["{\eta_a}"{description}, draw=none, from=1-8, to=5]
	\arrow["{\chi_{x,a}}"{description}, draw=none, from=2-2, to=6]
	\arrow["{\text{id}_{x_1}}"{description}, draw=none, from=4, to=7]
	\arrow["\alpha"{description}, draw=none, from=5, to=8]
	\arrow["\alpha"{description}, draw=none, from=6, to=9]
	\arrow["{\chi_x}"{description}, draw=none, from=3-7, to=10]
\end{tikzcd}
	\end{adjustbox}
	\]
	and
	\[
	\begin{adjustbox}{}
		\begin{tikzcd}
	& {x_0} & z \\
	{x_0} & {x_0} && z &&& {x_0} & z \\
	&& \alpha && {=} & {x_0} & {x_0} && y \\
	{x_0} & {x_0} && y && {x_0} &&& y \\
	{x_0} &&& y
	\arrow[""{name=0, anchor=center, inner sep=0}, "{{x_1(1,a)}}"{inner sep=.8ex}, "\shortmid"{marking}, from=1-2, to=1-3]
	\arrow[equals, from=1-2, to=2-1]
	\arrow["a"{description}, from=1-3, to=2-2]
	\arrow[equals, from=1-3, to=2-4]
	\arrow[""{name=1, anchor=center, inner sep=0}, "{{x_1}}"'{inner sep=.8ex}, "\shortmid"{marking}, from=2-1, to=2-2]
	\arrow[""{name=2, anchor=center, inner sep=0}, "{{x_1(1,a)}}"'{inner sep=.8ex}, "\shortmid"{marking}, from=2-2, to=2-4]
	\arrow[equals, from=2-2, to=4-2]
	\arrow["b", from=2-4, to=4-4]
	\arrow[""{name=3, anchor=center, inner sep=0}, "{{x_1(1,a)}}"{inner sep=.8ex}, "\shortmid"{marking}, from=2-7, to=2-8]
	\arrow[equals, from=2-7, to=3-6]
	\arrow["a"{description}, from=2-8, to=3-7]
	\arrow["b", from=2-8, to=3-9]
	\arrow[""{name=4, anchor=center, inner sep=0}, "{{x_1}}"'{inner sep=.8ex}, "\shortmid"{marking}, from=3-6, to=3-7]
	\arrow[equals, from=3-6, to=4-6]
	\arrow[""{name=5, anchor=center, inner sep=0}, "\varphi"'{inner sep=.8ex}, "\shortmid"{marking}, from=3-7, to=3-9]
	\arrow[equals, from=3-9, to=4-9]
	\arrow[equals, from=4-1, to=2-1]
	\arrow[""{name=6, anchor=center, inner sep=0}, "{{x_1}}"{inner sep=.8ex}, "\shortmid"{marking}, from=4-1, to=4-2]
	\arrow[equals, from=4-1, to=5-1]
	\arrow["\varphi"{inner sep=.8ex}, "\shortmid"{marking}, from=4-2, to=4-4]
	\arrow[equals, from=4-4, to=5-4]
	\arrow[""{name=7, anchor=center, inner sep=0}, "\varphi"'{inner sep=.8ex}, "\shortmid"{marking}, from=4-6, to=4-9]
	\arrow[""{name=8, anchor=center, inner sep=0}, "\varphi"'{inner sep=.8ex}, "\shortmid"{marking}, from=5-1, to=5-4]
	\arrow["{{\mathsf{cart}}}"{description}, draw=none, from=0, to=1]
	\arrow["{\eta_a}"{description}, draw=none, from=1-3, to=2]
	\arrow["{\text{id}_{x_1}}"{description}, draw=none, from=1, to=6]
	\arrow["{{\mathsf{cart}}}"{description}, draw=none, from=3, to=4]
	\arrow["{\alpha(\eta_{a,b})}"{description}, draw=none, from=2-8, to=5]
	\arrow["{\chi_x}"{description}, draw=none, from=3-7, to=7]
	\arrow["{\chi_x}"{description}, draw=none, from=4-2, to=8]
\end{tikzcd}
	\end{adjustbox}
	\]
	It remains to show that for an arbitrary nullary multicell $\alpha(\eta_{a,b})$, the composite above right defines a map of left modules. First, by definition of $\chi_{x,a}$ we have the equality:
	\[
		\begin{tikzcd}[column sep=small]
	{x_0} && {x_0} & z && {x_0} & {x_0} && z \\
	{x_0} &&& z && {x_0} & {x_0} & {x_0} && y \\
	{x_0} & {x_0} && y & {=} & {x_0} && {x_0} && y \\
	{x_0} &&& y && {x_0} &&&& y
	\arrow["{{x_1}}"{inner sep=.8ex}, "\shortmid"{marking}, from=1-1, to=1-3]
	\arrow[equals, from=1-1, to=2-1]
	\arrow["{{x_1(1,a)}}"{inner sep=.8ex}, "\shortmid"{marking}, from=1-3, to=1-4]
	\arrow[equals, from=1-4, to=2-4]
	\arrow[""{name=0, anchor=center, inner sep=0}, "{{x_1}}"{inner sep=.8ex}, "\shortmid"{marking}, from=1-6, to=1-7]
	\arrow[equals, from=1-6, to=2-6]
	\arrow[""{name=1, anchor=center, inner sep=0}, "{{x_1(1,a)}}"{inner sep=.8ex}, "\shortmid"{marking}, from=1-7, to=1-9]
	\arrow[equals, from=1-7, to=2-7]
	\arrow["a"', from=1-9, to=2-8]
	\arrow["b", from=1-9, to=2-10]
	\arrow[""{name=2, anchor=center, inner sep=0}, "{{x_1(1,a)}}"'{inner sep=.8ex}, "\shortmid"{marking}, from=2-1, to=2-4]
	\arrow[equals, from=2-1, to=3-1]
	\arrow[""{name=3, anchor=center, inner sep=0}, "a"{description}, from=2-4, to=3-2]
	\arrow["b", from=2-4, to=3-4]
	\arrow[""{name=4, anchor=center, inner sep=0}, "{{x_1}}"{inner sep=.8ex}, "\shortmid"{marking}, from=2-6, to=2-7]
	\arrow[equals, from=2-6, to=3-6]
	\arrow[""{name=5, anchor=center, inner sep=0}, "{{x_1}}"{inner sep=.8ex}, "\shortmid"{marking}, from=2-7, to=2-8]
	\arrow[""{name=6, anchor=center, inner sep=0}, "\varphi"{inner sep=.8ex}, "\shortmid"{marking}, from=2-8, to=2-10]
	\arrow[equals, from=2-8, to=3-8]
	\arrow[equals, from=2-10, to=3-10]
	\arrow[""{name=7, anchor=center, inner sep=0}, "{{x_1}}"'{inner sep=.8ex}, "\shortmid"{marking}, from=3-1, to=3-2]
	\arrow[equals, from=3-1, to=4-1]
	\arrow[""{name=8, anchor=center, inner sep=0}, "\varphi"'{inner sep=.8ex}, "\shortmid"{marking}, from=3-2, to=3-4]
	\arrow[equals, from=3-4, to=4-4]
	\arrow[""{name=9, anchor=center, inner sep=0}, "{{x_1}}"'{inner sep=.8ex}, "\shortmid"{marking}, from=3-6, to=3-8]
	\arrow[equals, from=3-6, to=4-6]
	\arrow[""{name=10, anchor=center, inner sep=0}, "\varphi"'{inner sep=.8ex}, "\shortmid"{marking}, from=3-8, to=3-10]
	\arrow[equals, from=3-10, to=4-10]
	\arrow[""{name=11, anchor=center, inner sep=0}, "\varphi"'{inner sep=.8ex}, "\shortmid"{marking}, from=4-1, to=4-4]
	\arrow[""{name=12, anchor=center, inner sep=0}, "\varphi"'{inner sep=.8ex}, "\shortmid"{marking}, from=4-6, to=4-10]
	\arrow["{\chi_{x,a}}"{description}, draw=none, from=1-3, to=2]
	\arrow["{{\text{id}_{x_1}}}"{description}, draw=none, from=0, to=4]
	\arrow["{{\mathsf{cart}}}"{description}, draw=none, from=1, to=5]
	\arrow["{{\alpha(\eta_{a,b})}}"{description}, draw=none, from=1-9, to=6]
	\arrow["{{\mathsf{cart}}}"{description}, draw=none, from=2, to=7]
	\arrow["{{\alpha(\eta_{a,b})}}"{description}, draw=none, from=3, to=8]
	\arrow["\mu"{description}, draw=none, from=2-7, to=9]
	\arrow["{{\text{id}_\varphi}}"{description}, draw=none, from=6, to=10]
	\arrow["{\chi_x}"{description}, draw=none, from=3-2, to=11]
	\arrow["{\chi_x}"{description}, draw=none, from=3-8, to=12]
\end{tikzcd}
	\]
	Second, using the associativity of the action $\chi_x$ we have the equality:
	\[
		\begin{tikzcd}[column sep=small]
	{x_0} & {x_0} &&& z && {x_0} & {x_0} &&& z \\
	{x_0} & {x_0} & {x_0} && y && {x_0} & {x_0} & {x_0} && y \\
	{x_0} && {x_0} && y & {=} & {x_0} & {x_0} &&& y \\
	{x_0} &&&& y && {x_0} &&&& y
	\arrow[""{name=0, anchor=center, inner sep=0}, "{{x_1}}"{inner sep=.8ex}, "\shortmid"{marking}, from=1-1, to=1-2]
	\arrow[equals, from=1-1, to=2-1]
	\arrow[""{name=1, anchor=center, inner sep=0}, "{{x_1(1,a)}}"{inner sep=.8ex}, "\shortmid"{marking}, from=1-2, to=1-5]
	\arrow[equals, from=1-2, to=2-2]
	\arrow[""{name=2, anchor=center, inner sep=0}, "a"', from=1-5, to=2-3]
	\arrow["b", from=1-5, to=2-5]
	\arrow[""{name=3, anchor=center, inner sep=0}, "{{x_1}}"{inner sep=.8ex}, "\shortmid"{marking}, from=1-7, to=1-8]
	\arrow[equals, from=1-7, to=2-7]
	\arrow[""{name=4, anchor=center, inner sep=0}, "{{x_1(1,a)}}"{inner sep=.8ex}, "\shortmid"{marking}, from=1-8, to=1-11]
	\arrow[equals, from=1-8, to=2-8]
	\arrow[""{name=5, anchor=center, inner sep=0}, "a"', from=1-11, to=2-9]
	\arrow["b", from=1-11, to=2-11]
	\arrow[""{name=6, anchor=center, inner sep=0}, "{{x_1}}"'{inner sep=.8ex}, "\shortmid"{marking}, from=2-1, to=2-2]
	\arrow[equals, from=2-1, to=3-1]
	\arrow[""{name=7, anchor=center, inner sep=0}, "{{x_1}}"'{inner sep=.8ex}, "\shortmid"{marking}, from=2-2, to=2-3]
	\arrow[""{name=8, anchor=center, inner sep=0}, "\varphi"'{inner sep=.8ex}, "\shortmid"{marking}, from=2-3, to=2-5]
	\arrow[equals, from=2-3, to=3-3]
	\arrow[equals, from=2-5, to=3-5]
	\arrow[""{name=9, anchor=center, inner sep=0}, "{{x_1}}"'{inner sep=.8ex}, "\shortmid"{marking}, from=2-7, to=2-8]
	\arrow[equals, from=2-7, to=3-7]
	\arrow[""{name=10, anchor=center, inner sep=0}, "{{x_1}}"'{inner sep=.8ex}, "\shortmid"{marking}, from=2-8, to=2-9]
	\arrow[equals, from=2-8, to=3-8]
	\arrow[""{name=11, anchor=center, inner sep=0}, "\varphi"'{inner sep=.8ex}, "\shortmid"{marking}, from=2-9, to=2-11]
	\arrow[equals, from=2-11, to=3-11]
	\arrow[""{name=12, anchor=center, inner sep=0}, "{{x_1}}"'{inner sep=.8ex}, "\shortmid"{marking}, from=3-1, to=3-3]
	\arrow[equals, from=3-1, to=4-1]
	\arrow[""{name=13, anchor=center, inner sep=0}, "\varphi"'{inner sep=.8ex}, "\shortmid"{marking}, from=3-3, to=3-5]
	\arrow[equals, from=3-5, to=4-5]
	\arrow[""{name=14, anchor=center, inner sep=0}, "{{x_1}}"'{inner sep=.8ex}, "\shortmid"{marking}, from=3-7, to=3-8]
	\arrow[equals, from=3-7, to=4-7]
	\arrow[""{name=15, anchor=center, inner sep=0}, "\varphi"'{inner sep=.8ex}, "\shortmid"{marking}, from=3-8, to=3-11]
	\arrow[equals, from=3-11, to=4-11]
	\arrow[""{name=16, anchor=center, inner sep=0}, "\varphi"'{inner sep=.8ex}, "\shortmid"{marking}, from=4-1, to=4-5]
	\arrow[""{name=17, anchor=center, inner sep=0}, "\varphi"'{inner sep=.8ex}, "\shortmid"{marking}, from=4-7, to=4-11]
	\arrow["{{\text{id}_{x_1}}}"{description}, draw=none, from=0, to=6]
	\arrow["{{\mathsf{cart}}}"{description}, draw=none, from=1, to=7]
	\arrow["{{\alpha(\eta_{a,b})}}"{description}, draw=none, from=2, to=8]
	\arrow["{{\text{id}_{x_1}}}"{description}, draw=none, from=3, to=9]
	\arrow["{{\mathsf{cart}}}"{description}, draw=none, from=4, to=10]
	\arrow["{{\alpha(\eta_{a,b})}}"{description}, draw=none, from=5, to=11]
	\arrow["\mu"{description}, draw=none, from=2-2, to=12]
	\arrow["{{\text{id}_\varphi}}"{description}, draw=none, from=8, to=13]
	\arrow["{{\text{id}_{x_1}}}"{description}, draw=none, from=9, to=14]
	\arrow["{{\chi_x}}"{description}, draw=none, from=11, to=15]
	\arrow["{\chi_x}"{description}, draw=none, from=3-3, to=16]
	\arrow["{\chi_x}"{description}, draw=none, from=3-8, to=17]
\end{tikzcd}
	\]
	which completes the proof.
\end{proof}

When $\Dbl{D}=\Span$, $x=\Cat{C}$ is a category, and a left module is a set $A$ together with an $A$-indexed family of presheaves $(P_a:\Cat{C}^{op}\to \Cat{Set})_{a\in A}$. Then Theorem~\ref{thm:YonedaForLeftModules} says that if we have a set function $g:B\to A$, together with a family of objects $(c_b)_{b\in B}$ in $\Cat{C}$, then we have a bijection between families of natural transformations below left, and families of elements below right:
\begin{equation*}
	(\alpha_b:\Cat{C}(-,c_b)\Rightarrow P_{g(b)}(c_b))_{b\in B}\leftrightsquigarrow (x_b \in P_{g(b)}(c_b))_{b\in B}
\end{equation*}
given by evaluating at identity arrows. In order to obtain an alternate proof of Lemma~\ref{lem:Yoneda} we can apply Theorem~\ref{thm:YonedaForLeftModules} to the case where the virtual double category is $\Dbl{L}\Cat{Kl}(\Span(\Cat{Grph}),\Cat{fc})$, the monoid $x$ is the VDC $\Cat{fc}(\Dbl{D})$, $y=z=\Cat{Edge}$ is the graph with a single edge between two distinct vertices, and the tight arrow $a$ picks out the sequence of loose arrows $\prolist{\varphi}$ in $\Dbl{D}$. Then a left module $\varphi:x_0\proto y$ corresponds to a loose arrow $\Phi:F\proto G$ in $\Span^{\Cat{fc}(\Dbl{D})^{op_t}}$, and module maps correspond to multicells in $\Span^{\Cat{fc}(\Dbl{D})^{op_t}}\scell{\rho_{\Dbl{D}}(\prolist{\varphi})}{\Phi}$ while elements of $\Phi(\prolist{\varphi})$ correspond precisely to nullary multicells
\[
\begin{tikzcd}
	& {\mathsf{Edge}} \\
	{\mathsf{fc}(\mathbb{D})_0} && {\mathsf{Edge}}
	\arrow["{{\prolist{\varphi}}}"', from=1-2, to=2-1]
	\arrow[equals, from=1-2, to=2-3]
	\arrow[""{name=0, anchor=center, inner sep=0}, "\Phi"'{inner sep=.8ex}, "\shortmid"{marking}, from=2-1, to=2-3]
	\arrow["\beta"{description}, draw=none, from=1-2, to=0]
\end{tikzcd}\]
in $\Dbl{L}\Cat{Kl}(\Span(\Cat{Grph}),\Cat{fc})$.